\documentclass[11pt]{amsart}
\usepackage[a4paper,centering]{geometry}
\usepackage[latin1]{inputenc}
\usepackage{aeguill}
\usepackage{enumerate}
\usepackage{epsfig}
\usepackage{graphicx}
\usepackage{psfrag}
\usepackage{color}
\usepackage{url}

\newcommand{\Ufive}{\Delta_5}
\newcommand{\caseone}[2]{{\cal A}(#1,#2)}
\newcommand{\casetwo}[2]{{\cal B}(#1,#2)}
\newcommand{\ignore}[1]{}

\newcommand{\un}{1} 
\newcommand{\de}{2}
\newcommand{\tr}{3}
\newcommand{\unb}{\overline{1}} 
\newcommand{\deb}{\overline{2}}
\newcommand{\trb}{\overline{3}}
\newcommand{\one}{1}
\newcommand{\two}{2}
\newcommand{\thr}{3}
\newcommand{\oneb}{\overline{1}}
\newcommand{\twob}{\overline{2}}
\newcommand{\thrb}{\overline{3}}
\newcommand{\AAA}{1}
\newcommand{\BBB}{2}
\newcommand{\CCC}{3}
\newcommand{\DDD}{4}
\newcommand{\cside}{\mu}
\newcommand{\dside}{\overline{\mu}}
\newcommand{\iAAA}{i}
\newcommand{\jBBB}{j}
\newcommand{\AAAb}{\overline{i}}
\newcommand{\BBBb}{\overline{j}}
\newcommand{\ptg}{{\bullet}}
\newcommand{\bcut}[2]{{\cal D}_#2}
\newcommand{\cut}[1]{{\cal D}_#1}

\newcommand{\rev}[1]{{#1}^{-1}}

\newcommand{\node}[3]{#1#2#3}
\newcommand{\nodeone}[2]{#1#2}
\newcommand{\nodetwo}[2]{\jbar{#1}#2}
\newcommand{\nodethr}[2]{#1\jbar{#2}}
\newcommand{\nodefou}[2]{\jbar{#1}\jbar{#2}}
\newcommand{\setnode}[2]{\{\nodeone{#1}{#2},\nodetwo{#1}{#2},\nodethr{#1}{#2},\nodefou{#1}{#2}\}}
\newcommand{\setnodebis}[2]{\{\nodeone{#1}{#2},\nodetwo{#1}{#2},\nodethr{#1}{#2},\nodefou{#1}{#2} \mid #2 \in I\setminus #1\}}

\newcommand{\OO}{1}
\newcommand{\uu}{2}
\newcommand{\vv}{3}
\newcommand{\ww}{4}
\newcommand{\xx}{k}
\newcommand{\yy}{k'}
\newcommand{\GOO}{\Gamma_{\OO}}
\newcommand{\Guu}{\Gamma_{\uu}}
\newcommand{\Gvv}{\Gamma_{\vv}}
\newcommand{\Gww}{\Gamma_{\ww}}
\newcommand{\Gxx}{\Gamma_{\xx}}
\newcommand{\nA}{X}
\newcommand{\nB}{Y}
\newcommand{\nC}{A}
\newcommand{\nD}{B}

\newcommand{\gtriangle}{$\gamma$-triangle}
\newcommand{\gtriangles}{$\gamma$-triangles}
\newcommand{\base}[1]{\operatorname{Int}_\gamma(#1)}
\newcommand{\cobase}[1]{\operatorname{Ext}_\gamma(#1)}
\newcommand{\oldflagop}[1]{\sigma_{#1}}
\newcommand{\oldflagdoublureop}[1]{\widehat{\sigma}_{#1}}
\newcommand{\OldFlag}[1]{{\cal F}(#1)}
\newcommand{\OldFlagDoublure}[1]{\widehat{{\cal F}}(#1)}

\newcommand{\XX}[1]{{\cal N}(#1)}

\newcommand{\interval}[2]{[#1,#2]}

\newcommand{\TopPlane}{{\cal P}}
\newcommand{\Lines}{\mathbb{L}}
\renewcommand{\Lines}{{\cal L}}
\newcommand{\OLines}{\widetilde{\Lines}}
\newcommand{\aline}{\ell}
\renewcommand{\aline}{L}
\newcommand{\apoint}{P}
\newcommand{\apointbis}{Q}

\newcommand{\bline}{k}
\renewcommand{\bline}{K}

\newcommand{\gcurve}{\operatorname{curve}_\gamma}
\newcommand{\UU}{\gamma_+}

\newcommand{\VV}{\gamma_-}
\newcommand{\tracecurve}{$\gamma$-curve}
\newcommand{\tracecurves}{$\gamma$-curves}

\newcommand{\TD}{{\cal D}}
\newcommand{\MS}{{\cal M}}
\newcommand{\OPC}{\overline{\MS}}

\newcommand{\GPP}{{\cal G}}
\newcommand{\PP}{{\cal P}}

\newcommand{\CC}{{\cal C}}

\newcommand{\Map}[3]{#1 : #2 \rightarrow #3}
\newcommand{\MapLight}[3]{#2 \rightarrow #3}
\newcommand{\lift}[1]{\widetilde{#1}}
\newcommand{\stripbis}[1]{C(#1)}
\newcommand{\twin}[1]{#1_*}

\newtheorem{rem}{Remark}
\newenvironment{remark}{\begin{rem}\em}{\end{rem}}
\newtheorem{exa}{Example}
\newenvironment{example}{\begin{exa}\em}{\end{exa}}

\newtheorem{theorem}{Theorem}
\newtheorem{lemma}[theorem]{Lemma}

\newcommand{\cal}{\mathcal}

\newcommand{\onecrossR}[2]{\overline{#2}}
\newcommand{\twocrossR}[2]{\overline{#2}}
\newcommand{\thrcrossR}[2]{#2}
\newcommand{\foucrossR}[2]{#2}

\newcommand{\flagone}[4]{(\{\nodeone{#1}{#2}\},{#3},{#4})}

\newcommand{\cyclefouR}[4]{
\foucrossR{#1}{#2}
\onecrossR{#1}{#2}
\foucrossR{#1}{#3}
\onecrossR{#1}{#3}
\foucrossR{#1}{#4}
\onecrossR{#1}{#4}
\twocrossR{#1}{#2}
\thrcrossR{#1}{#2}
\twocrossR{#1}{#3}
\thrcrossR{#1}{#3}
\twocrossR{#1}{#4}
\thrcrossR{#1}{#4}
}

\newcommand{\cyclethrTHRTWOtwoR}[3]
{
\onecrossR{#2}{#3}
\foucrossR{#2}{#1}
\onecrossR{#2}{#1}
\twocrossR{#2}{#3}
\thrcrossR{#2}{#3}
\twocrossR{#2}{#1}
\thrcrossR{#2}{#1}
\foucrossR{#2}{#3}
}
\newcommand{\cyclethrTHRTWOthrR}[3]
{
\foucrossR{#3}{#1}
\onecrossR{#3}{#1}
\thrcrossR{#3}{#2}
\foucrossR{#3}{#2}
\twocrossR{#3}{#1}
\thrcrossR{#3}{#1}
\onecrossR{#3}{#2}
\twocrossR{#3}{#2}
}

\def\aone{a}\def\bone{b}\def\cone{c}\def\done{d}\def\eone{e}\def\fone{f}
\def\gone{g}\def\hone{h}\def\ione{i}\def\jone{j}\def\kone{k}\def\lone{l}
\def\atwo{\hat{a}}\def\btwo{\hat{b}}\def\ctwo{\hat{c}}\def\dtwo{\hat{d}}\def\etwo{\hat{e}}\def\ftwo{\hat{f}}
\def\gtwo{\hat{g}}\def\htwo{\hat{h}}\def\itwo{\hat{i}}\def\jtwo{\hat{j}}\def\ktwo{\hat{k}}\def\ltwo{\hat{l}}
\def\athr{\tilde{a}}\def\bthr{\tilde{b}}\def\cthr{\tilde{c}}\def\dthr{\tilde{d}}\def\ethr{\tilde{e}}\def\fthr{\tilde{f}}
\def\gthr{\tilde{g}}\def\hthr{\tilde{h}}\def\ithr{\tilde{i}}\def\jthr{\tilde{j}}\def\kthr{\tilde{k}}\def\lthr{\tilde{l}}
\def\afou{\check{a}}\def\bfou{\check{b}}\def\cfou{\check{c}}\def\dfou{\check{d}}\def\efou{\check{e}}\def\ffou{\check{f}}
\def\gfou{\check{g}}\def\hfou{\check{h}}\def\ifou{\check{i}}\def\jfou{\check{j}}\def\kfou{\check{k}}\def\lfou{\check{l}}

\def\rdone{\rev{d}}\def\rfone{\rev{f}}
\def\rgone{\rev{g}}\def\rione{\rev{i}}\def\rkone{\rev{k}}
\def\ratwo{\rev{\hat{a}}}\def\rbtwo{\rev{\hat{b}}}\def\rctwo{\rev{\hat{c}}}\def\retwo{\rev{\hat{e}}}
\def\rgtwo{\rev{\hat{g}}}\def\rhtwo{\rev{\hat{h}}}\def\rjtwo{\rev{\hat{j}}}\def\rltwo{\rev{\hat{l}}}
\def\rathr{\rev{\tilde{a}}}\def\rcthr{\rev{\tilde{c}}}\def\rdthr{\rev{\tilde{d}}}\def\rfthr{\rev{\tilde{f}}}
\def\rhthr{\rev{\tilde{h}}}\def\rjthr{\rev{\tilde{j}}}\def\rlthr{\rev{\tilde{l}}}
\def\rbfou{\rev{\check{b}}}\def\rdfou{\rev{\check{d}}}\def\rffou{\rev{\check{f}}}
\def\rhfou{\rev{\check{h}}}\def\rjfou{\rev{\check{j}}}\def\rlfou{\rev{\check{l}}}

\def\anormal{{a}}\def\bnormal{{b}}\def\cnormal{{c}}\def\dnormal{{d}}
\def\enormal{{e}}\def\fnormal{{f}}\def\gnormal{{g}}\def\hnormal{{h}}

\def\ansour{1}\def\bnsour{2}\def\cnsour{3}\def\dnsour{4}
\def\ensour{5}\def\fnsour{6}\def\gnsour{7}\def\hnsour{8}

\def\achapeau{\hat{a}}\def\bchapeau{\hat{b}}\def\cchapeau{\hat{c}}\def\dchapeau{\hat{d}}
\def\echapeau{\hat{e}}\def\fchapeau{\hat{f}}\def\gchapeau{\hat{g}}\def\hchapeau{\hat{h}}

\def\acsour{\hat{1}}
\def\bcsour{\hat{2}}
\def\ccsour{\hat{3}}
\def\dcsour{\hat{4}}
\def\ecsour{\hat{5}}
\def\fcsour{\hat{6}}
\def\gcsour{\hat{7}}
\def\hcsour{\hat{8}}

\def\atilde{\tilde{a}}
\def\btilde{\tilde{b}}
\def\ctilde{\tilde{c}}
\def\dtilde{\tilde{d}}
\def\etilde{\tilde{e}}
\def\ftilde{\tilde{f}}
\def\gtilde{\tilde{g}}
\def\htilde{\tilde{h}}

\def\atsour{\tilde{1}}
\def\btsour{\tilde{2}}
\def\ctsour{\tilde{3}}
\def\dtsour{\tilde{4}}
\def\etsour{\tilde{5}}
\def\ftsour{\tilde{6}}
\def\gtsour{\tilde{7}}
\def\htsour{\tilde{8}}

\newcommand{\barre}[1]{\overline{#1}}
\newcommand{\bbpp}[2]{#1#2\ptg\ptg}
\newcommand{\BBPP}[2]{#1#2\ptg\ptg}
\newcommand{\BBBPBPF}[3]{#1#2#3\ptg\barre{#2}\ptg}
\newcommand{\BBBPBPS}[3]{#1\barre{#2}#3\ptg#2\ptg}
\newcommand{\BBBBPPF}[3]{#1#2#3\barre{#1}\ptg\ptg}
\newcommand{\BBBBPPS}[3]{\barre{#1}#2#3#1\ptg\ptg}

\newcommand{\nameM}{M}

\newcommand{\nBR}{{\cal B}}
\newcommand{\nRR}{{\cal R}}

\newcommand{\roll}[2]{#1_{\otimes #2}}

\newcommand{\jind}[1]{#1}
\newcommand{\jbar}[1]{\overline{#1}}

\newcommand{\cyclePAsimple}[2]{#1#2\overline{#1#2}}
\newcommand{\cyclePAnonsim}[2]{\overline{#1}#1\overline{#2}#2}

\newcommand{\name}[1]{C_{#1}}
\newcommand{\bname}[2]{\name{#1}(#2)}

\newcommand{\ii}{1}
\newcommand{\jj}{3}
\newcommand{\kk}{2}
\newcommand{\bii}{\overline{1}}
\newcommand{\bjj}{\overline{3}}
\newcommand{\bkk}{\overline{2}}

\newcommand{\SSJ}{{\cal P}}

\newcommand{\red}{\textcolor{red}{\bigcirc}}
\newcommand{\blu}{\textcolor{blue}{\bullet}}

\newcommand{\subarrang}{\zeta}

\newcommand{\trace}[3]{[#1,#2]}

\newcommand{\polar}[1]{#1^{\diamond}}
\newcommand{\tang}[1]{#1^{*}}
\newcommand{\dual}[1]{#1^{*}}
\newcommand{\EEE}{K}
\newcommand{\FFF}{K'}
\newcommand{\GGG}{G}
\newcommand{\HHH}{H}
\newcommand{\CB}{U}
\newcommand{\OCB}{V}
\newcommand{\FCB}{\widetilde{\CB}}
\newcommand{\FCBP}{\FCB_+}
\newcommand{\FCBM}{\FCB_-}
\newcommand{\CBP}{\CB_+}
\newcommand{\CBM}{\CB_-}
\newcommand{\OCBP}{\OCB_+}
\newcommand{\OCBM}{\OCB_-}
\newcommand{\II}{I}
\newcommand{\FII}{\widetilde{\II}}
\newcommand{\IP}{\II_+}

\newcommand{\FIP}{\FII_+}
\newcommand{\FIM}{\FII_-}
\newcommand{\CL}{L}
\newcommand{\FCL}{\widetilde{L}}

\newcommand{\mylabel}{\nu}
\newcommand{\vijone}{\nodeone{i}{j}}
\newcommand{\vijtwo}{\nodetwo{i}{j}}
\newcommand{\vijthr}{\nodethr{i}{j}}
\newcommand{\vijfou}{\nodefou{i}{j}}

\newcommand{\pcplane}{\widehat{\TopPlane}}
\newcommand{\pclines}{\widehat{\Lines}}

\newcommand{\pp}{{\cal P}}
\newcommand{\rr}{{\cal R}}
\newcommand{\lpp}{{\cal L}}
\newcommand{\spp}{\mathbb{RP}^2}
\newcommand{\transmap}{\mu}
\newcommand{\opposite}[1]{\widehat{#1}}

\newcommand{\ot}{\opposite{t}}

\newcommand{\wo}{\opposite{w}} 
\newcommand{\uo}{\opposite{u}}
\newcommand{\renewmeet}[3]{L_{#1#2#3}} 
\newcommand{\QNodes}{{\cal V}}
\newcommand{\interior}[1]{\operatorname{Int}(#1)}
\newcommand{\calC}[1]{{\cal C}_{#1}}
\newcommand{\calCstar}[1]{{\cal C}^*_{#1}}

\DeclareGraphicsExtensions{.eps,.eps.gz}
\graphicspath{{./xfigone/}{./xfigtwo/}{./xfigthr/}}

\title[\today]{{LR} characterization of chirotopes of finite planar families  of  pairwise disjoint convex bodies} 
\author{Luc Habert and Michel Pocchiola}
\thanks{MP was partially supported by the TEOMATRO grant ANR-10-BLAN 0207.}
\address{Luc Habert}
\email{Luc.Habert@normalesup.org}
\address{Michel Pocchiola\\ 
Universit{\'e} Pierre \&  Marie Curie \\
Institut de Math{\'e}matiques de Jussieu (UMR 7586) \\
4 place Jussieu\\
75252 Paris Cedex\\
France}
\email{pocchiola@math.jussieu.fr}

\date{\today}
\begin{document}

\maketitle

\begin{abstract} 
We extend the classical LR characterization of chirotopes of finite planar families of points to chirotopes of finite planar families of pairwise disjoint 
convex bodies: 
a map $\chi$ on the set of $3$-subsets of a finite set $I$  is a chirotope of finite planar families of  pairwise disjoint convex bodies 
if and only if  for every $3$-, $4$-, and $5$-subset $J$ of $I$ the restriction of $\chi$ to the set of $3$-subsets  
of $J$ is a chirotope of finite planar families of pairwise disjoint convex bodies. 
Our main tool is the polarity map, i.e., 
the map that  assigns to a convex body the set of lines missing its interior, from which we  derive the  key notion of arrangements of 
double pseudolines, introduced for the first time in this paper. 

\medskip
\noindent
{\bf Keywords.} Convexity, discrete geometry, projective planes, pseudoline arrangements, chirotopes. 

\bigskip
\bigskip
\noindent
Abbreviated versions in 
Abstracts of the 12th European Workshop Comput. Geom. pages 211--214, March 2006, Delphes, Greece,
in the poster session of the Workshop on Geometric and Topological Combinatorics
(satellite conference of ICM 2006), September 2006, Alcala de Henares, Spain,
and in Proc. 25th Annu. ACM Sympos. Comput. Geom. (SCG09), pages 314--323, June 2009, Aahrus, Denmark.
\end{abstract}


\clearpage
\tableofcontents
\phantom{sautdepage}
\clearpage

\section{Introduction}\label{secone}
The term planar in the title makes reference to real two-dimensional projective planes. 
We review what we need of the basics of real two-dimensional  projective planes and especially the notion of convex body before introducing the notion of chirotope, 
explaining the main result of the paper and the main lines of its proof. 
\subsection{Cross surfaces and  projective planes}
We assume that the reader is familiar with basic notions of algebraic and combinatorial topology like
homeomorphism, homotopy, fundamental group, covering, etc.,
found, for example, in~\cite[chap. 0 and 1]{h-at-01}.
The following standard notions, basic results and terminology associated with projective planes will be used throughout the paper; they are 
mainly taken from~\cite{sbghl-cpp-95,gpwz-atp-94,ps-gs-01,s-tp-68}
\begin{enumerate}
\item
A {\it closed (open)  topological disk} or {\it closed (open)  two-cell} is a topological space homeomorphic to the unit closed (open)  disk of $\mathbb{R}^2$.
An {\it orientation} of a 
topological disk is a one-to-one parametrization of the topological disk by the unit disk of $\mathbb{R}^2$, defined up to direct homeomorphism,
 and an {\it oriented topological disk} is a topological disk endowed with an orientation. 
Orientations will be 
indicated in our drawings by a little oriented circle in the interior of the disk or by an arrow on its boundary.
\item A  {\it cross surface}{}\footnote{We follow the J.~H.~Conway's proposition to call a sphere with one crosscap a {\it cross surface}; cf.~\cite{fw-czipp-99}.}
 is a topological space homeomorphic to the ``standard'' cross surface~$\spp$, quotient of the unit sphere 
$\mathbb{S}^2$ of $\mathbb{R}^3$ under 
identification of antipodal points;
cross surfaces will be represented in our drawings by circular diagrams with antipodal boundary points identified, as illustrated in Fig.~\ref{projectiveplane}a.
\item An  {\it open crosscap} or {\it open M{\"o}bius strip} is a  topological space homeomorphic to  a cross surface with one point or one closed topological disk deleted; 
an open  crosscap is a noncompact surface and its one-point compactification
(the space obtained by adding to the crosscap a point at infinity) 
is a cross surface. 
\item A {\it pseudocircle} is a simple closed curve  embedded in a cross surface; 
the connected components of the complement of a pseudocircle in its underlying cross surface are called its {\it open sides},  or simply its {\it sides}. 
An {\it oriented pseudocircle} is a pseudocircle endowed with an {\it orientation} (i.e., a one-to-one parametrization of the pseudocircle by $\mathbb{S}^1$, defined up to direct homeomorphism), 
indicated in our drawings by an arrow; as usual the intersection of two oriented pseudocircles 
is the intersection of their unoriented versions.
\item A {\it pseudoline} is a non-separating pseudocircle and a {\it double pseudoline} or {\it pseudo-oval} is a separating pseudocircle; 
cf. Fig.~\ref{projectiveplane}b and~\ref{projectiveplane}c.  There is a unique isomorphism class of pseudolines, i.e., given two pseudolines, one is
the image of the other by a homeomorphism of their underlying cross surfaces; 
in particular the complement of a pseudoline 
is an open two-cell.
Similarly for double pseudolines: there is a unique isomorphism class of double pseudolines and  
the complement of a  double pseudoline has two connected components (an open two-cell and an open  crosscap). 
The {\it core} pseudolines of a double pseudoline are the pseudolines contained in its crosscap side; cf. Fig.~\ref{projectiveplane}c, and~\ref{projectiveplane}d. 
\begin{figure}[!htb]
\psfrag{aa}{(a)}
\psfrag{bb}{(b)}
\psfrag{cc}{(c)}
\psfrag{dd}{(d)}
\centering
\includegraphics[width=0.9875\linewidth]{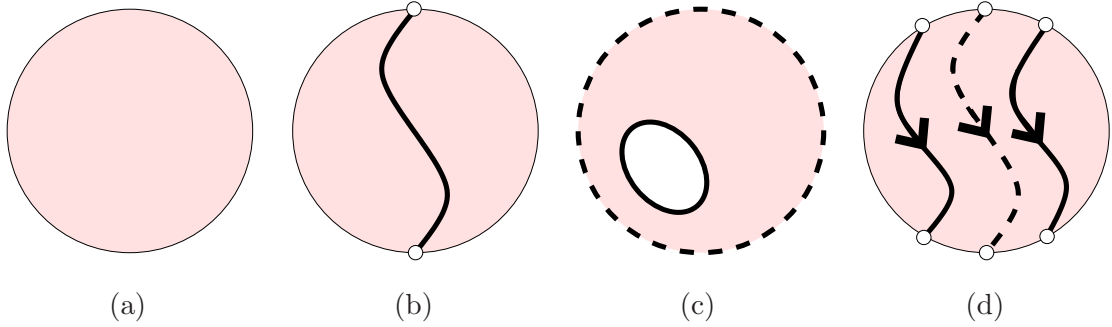}
\caption{ (a) A cross surface represented by a circular diagram with antipodal
boundary points identified; (b) a pseudoline; (c) a double
pseudoline with one of its core pseudolines drawn in dashed and with its disk side in white; (d) 
an oriented double pseudoline with one of its core oriented pseudolines drawn in dashed
\label{projectiveplane}}
\end{figure}

\item A {\it projective plane } is  a  topological point-line incidence geometry $(\pp,\lpp)$  whose point space $\pp$ 
is a cross surface, whose line space $\lpp$ is a subspace of the space of pseudolines of $\pp$, and whose incidence relations are the membership relations; as usual the dual of a point $p$ 
of a projective plane is denoted $\dual{p}$ and is defined as its set of incident lines.  
The {\it duality principle} for projective planes asserts that the dual $(\lpp,\dual{\pp})$ of a projective plane  $(\pp,\lpp)$
is still a projective plane, i.e., $\lpp$ is a cross surface and $\dual{\pp}$, the set of $\dual{p}$ as $p$ ranges over $\pp$, 
 is a subspace of the space of pseudolines of $\lpp$.  
In particular the dual of a finite set of points is an {\it arrangement of pseudolines}, i.e., a finite set of pseudolines (living in the same cross surface) 
that intersect pairwise in exactly one point; the basics of pseudoline arrangements used in the paper are reviewed in  Appendix~\ref{appendix:apl}. 
A projective plane is isomorphic to its bidual via the map that assigns to a point its dual and to a line the set of duals of its points. 
\item The standard projective plane is defined as the standard cross surface~$\spp$ together with the image 
of the space of great circles of $\mathbb{S}^2$
under the canonical projection $\MapLight{}{\mathbb{S}^2}{\spp}$. (Equivalently the standard projective plane can be defined as the projective completion of the Euclidean plane.)  
The standard projective plane is isomorphic to its dual via the map 
$\varphi$ that assigns to the point $(u,v,w)\in \mathbb{S}^2$ the great circle with equation $ux+vy+wz=0$ and that assigns to 
the great circle with equation $ux+vy+wz =0$, for $(u,v,w)\in \mathbb{S}^2$, the pencil of circles through the point $(u,v,w).$ 
\end{enumerate}
A {\it convex body} is a closed subset of the point space of a projective plane whose intersection with any line of the 
plane is a (necessarily closed) line segment; 
the {\it polar} of a convex body $U$, denoted $\polar{U}$, is the set of lines of the plane missing the interior of the convex body 
and its {\it dual}, denoted $\tang{U}$, is the set of lines of the plane intersecting 
the body but not its interior, {\it tangents} for short. 
For example, for $(u,v,w) \in \mathbb{S}^2$ and $h \in (0,1)$, the disk in the standard projective plane with equation 
\begin{equation}
|ux+vy+wz| \geq (1-h^2)^{1/2}
\end{equation}
is a convex body, its polar is the disk with equation  $|ux+vy+wz| \geq h$, and its dual 
is the circle with equation $|ux+vy+wz| = h$.
Similarly for finitely generated (pointed and full-dimensional) 
cones of the standard projective plane: the polar of the cone generated by the vectors $(u_i,v_i,w_i) \in \mathbb{S}^2$, $w_i >0$, 
is the polyhedral cone intersection of the half-spaces $u_ix+v_i y+ w_i z \geq 0$, $z\geq 0$.
As illustrated in these examples, 
a convex body of a projective plane is a closed topological disk, 
its polar is a convex body of the dual projective plane, and 
its dual is the boundary of its polar, hence a double pseudoline. Furthermore, polarity extends to oriented convex bodies: the polar of an oriented convex body has a natural orientation, inherited from the orientation of the body, 
compatible with the reorientation  operation.  (In the case where there is exactly one tangent through each boundary point and only one touching point per tangent 
the orientation of the polar inherited from the orientation 
$f$  is simply defined as the extension to the unit closed disk  of the map that assigns to $u\in \mathbb{S}^1$ the tangent to the convex body through the boundary point $f(u)$ of $U$.
 The general case follows once it is observed that the set of boundary points through which passes  a proper interval of tangents and the set of proper line segments included 
in the boundary are both countable.)  
Last but not least, we take for granted that, up to homeomorphism, the dual arrangement
of a pair of disjoint convex bodies of a projective plane is the unique arrangement of two double pseudolines that intersect transversely  in four points and induce a cellular decomposition of their underlying cross surface.
\begin{theorem}\label{theoone}
A convex body of a projective plane is a topological disk, its polar  is a convex body of the dual projective plane, and its dual  
is the boundary of its polar (hence a double pseudoline).
Furthermore, up to homeomorphism, the dual arrangement of a pair of disjoint convex bodies of a projective plane 
 is  the unique cellular arrangement of two double pseudolines that intersect transversely  in four points;
 in particular, two disjoint convex bodies share exactly four common tangents, the arrangement of these four tangents is simple, 
and the set of lines missing the two bodies is nonempty. 
\end{theorem} 
\begin{proof}
No proofs of these basic properties are available in the literature on convexity in projective planes that we became aware~\cite{c-gcI-74,c-gcII-78,ck-gcIII-78, gv-cspg-58, d-crpns-55, dghps-ctap-07, d-cepp-76,d-gcsp-52, hmos-ccs-08,kw-actrb-71,s-opgfg-91}. For completeness we offer proofs in Appendix~\ref{sectwo}.
\end{proof}

Fig.~\ref{gpmrfig}a  shows  a pair of disjoint convex bodies with the arrangement of their four common tangents. 
\begin{figure}[!htb]
\centering
\psfrag{U}{$U$}\psfrag{V}{$V$}
\psfrag{US}{$\tang{U}$}\psfrag{VS}{$\tang{V}$}
\psfrag{us}{$\dual{u}$}\psfrag{vs}{$\dual{v}$}
\psfrag{u}{$u$}\psfrag{v}{$v$}
\psfrag{aa}{(a)}\psfrag{bb}{(b)}
\includegraphics[width=0.75\linewidth]{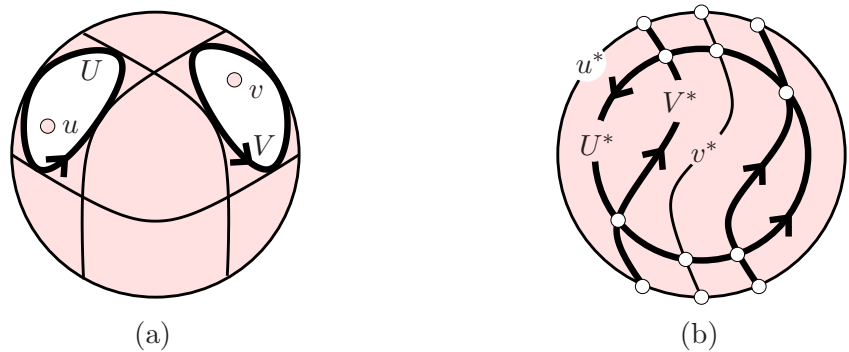}
\caption{(a) Two disjoint oriented convex bodies  with their common tangents  and (b) their dual arrangement \label{gpmrfig}}
\end{figure}
Each body is indexed, oriented and marked with an interior point. Fig.~\ref{gpmrfig}b  shows its dual arrangement. 
The automorphism group of the dual arrangement is trivial, the permutation group of a $2$-set or the dihedral group of order $8$ (group of automorphisms of the square)
depending on whether orientation and indexing of the pseudocircles are both taken into account, 
the orientation of the pseudocircles is taken into account but not their indexing, or  neither
the orientation nor the indexing are taken into account. 
Observe that the dual arrangement does not encode the nature of the contacts between the convex bodies and their common tangents. 
Thereafter, the four common tangents of two disjoint convex bodies will be called their {\it bitangents}.

\subsection{Definitions and main results} 
Throughout the paper, we use the words {\it configuration of convex bodies} for a finite family of pairwise disjoint convex bodies of a projective plane and we use, unless specified otherwise, the words {\it arrangement of double pseudolines} for a finite family of double pseudolines of a cross surface with the property that its 
subfamilies of size two are homeomorphic to the dual arrangement of a (hence any) configuration of two convex bodies; cf. Theorem~\ref{theoone}.
The {\it rhombicubeoctahedron} or {\it hemi- rhombicubeoctahedron arrangement} is the arrangement of double pseudolines composed of the $3$ circles of the standard projective plane 
with centers $(1,0,0)$, $(0,1,0)$, $(0,0,1)$ and radius $\arccos 1/2$ or, to say it differently, with equations are $|x| =1/2$, $|y|=1/2$ and $|z| = 1/2$.  
Its face poset is that of the projective version of the rhombicubeoctahedron (hence the name), one of the 13 Archimedean solids.
The {\it cube} or {\it hemi-cube arrangement} is the arrangement of double pseudolines composed of the $3$ circles of the standard projective plane with equations
$|x| = 1/\sqrt{3}$, 
$|y|= 1/\sqrt{3}$ and  
$|z| = 1/\sqrt{3}$.  Its face poset is 
obtained from that of the projective version of the cube (the hemi-cube) by replacing its $1$-cells by digons; cf.  Fig.~\ref{defarrangdoublelinePP}.
\begin{figure}[!htb]
\centering
\psfrag{C}{$\alpha$}
\psfrag{G}{$\gamma$}
\psfrag{MSG}{$\MS(\gamma)$}
\psfrag{PP}{$\PP$}\psfrag{infty}{$\infty$}
\psfrag{1}{$1$-cell}\psfrag{2}{$2$-cell}\psfrag{0}{$0$-cell}\psfrag{3}{$\cal M$}\psfrag{-1}{$\emptyset$}
\psfrag{aa}{$(a)$} \psfrag{bb}{$(b)$} \psfrag{cc}{$(c)$} \psfrag{dd}{$(d)$}
\psfrag{aa}{(a)} \psfrag{bb}{(b)} \psfrag{cc}{(c)} \psfrag{dd}{(d)}
\psfrag{one}{$1$}
\psfrag{two}{$2$}
\psfrag{thr}{$3$}
\psfrag{splitting}{\footnotesize splitting}
\psfrag{merging}{\footnotesize merging}
\psfrag{arrow}{$\rightarrow$}
\psfrag{barrow}{$\leftarrow$}
\includegraphics[width=0.9875\linewidth]{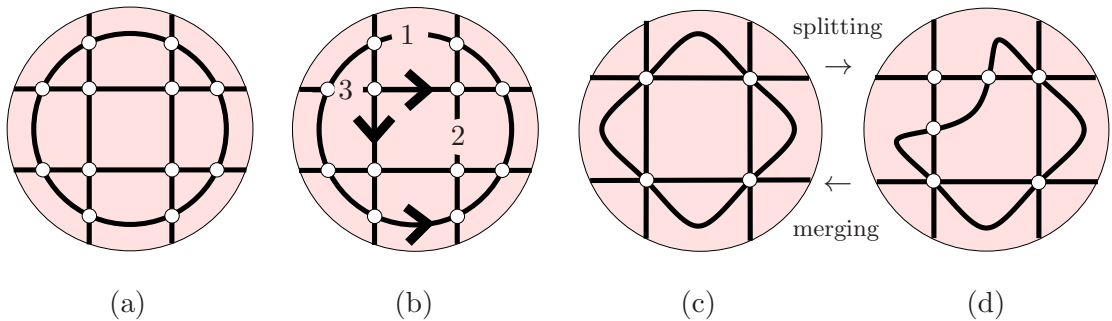}
\caption{\label{defarrangdoublelinePP}
(a) The rhombicubeoctahedron arrangement; 
(b) an indexed and oriented version of the rhombicubeoctahedron arrangement;
(c) the cube (or hemi-cube) arrangement;
(d) an arrangement of three double pseudolines obtained from the hemi-cube arrangement  by a splitting mutation 
} 
\end{figure}
We extend in the natural way the basic terminology of  arrangements of pseudolines to arrangements of double pseudolines.  In particular we use the following terminology.  
\begin{enumerate}
\item A vertex of an arrangement is {\it ordinary} if exactly two curves of the arrangement meet at that vertex.
An  arrangement is {\it simple} if all vertices of it are ordinary. 
Three vertices of the arrangement of Fig.~\ref{defarrangdoublelinePP}d are ordinary; three are not.
The rhombicubeoctahedron  arrangement is simple; the hemi-cube arrangement is not.
\item A {\it mutation} is a  homotopy of arrangements 
during which only one of the curves of the arrangement is moving and only one of the incidence relations between the moving curve and the vertices of
the cell complex induced by the other curves changes its value, swapping from false to true (first case) or from true to false (second case):
In the first case we speak of a {\it merging} 
mutation and in second  case we  speak of a {\it splitting} mutation. 
Fig.~\ref{defarrangdoublelinePP}c and ~\ref{defarrangdoublelinePP}d  show arrangements of three double pseudolines that are related by mutations of merging and splitting.
\item The {\it flag diagram} of an arrangement is the $3$-valent graph on its set of {\it flags} (maximal simplices of its first barycentric subdivision)
whose edges are the pairs of adjacent flags, each edge being labeled by the numeral $0,1$ or $2$ depending on whether the flags of the edge differ by their 
 $0$-, $1$-, or $2$-cells;
one can also think a flag diagram as  the Cayley graph of the group generated by the $0$-, $1$- and $2$-flag operators, denoted $\sigma_0, \sigma_1$ and $\sigma_2$ in the sequel, which are  the involutive operators on the set of flags
that exchange two  adjacent flags that differ by their $0$-, $1$-, or $2$-cells, respectively.
Fig~\ref{flagdiagrambis}a and \ref{flagdiagrambis}b show the (geometric version of the) first barycentric subdivision and the flag diagram of an arrangement of two double pseudolines.  
Fig.~\ref{flagdiagrambis}c and \ref{flagdiagrambis}d show this for the hemi-cube arrangement.
\item An {\it indexed arrangement of oriented double pseudolines} is a one-to-one map  
that assigns to each index of a finite set of indices an oriented double pseudoline of a cross surface such that the image of the map is an arrangement of oriented double pseudolines.
\item 
The {\it isomorphism class} of an arrangement is its set of homeomorphic images: in other words, two arrangements are called {\it isomorphic} if one is the image of the other by a homeomorphism 
of their underlying cross surfaces. The {\it isomorphism class} of an indexed arrangement of oriented double pseudolines is defined in a similar way.
\item Let $\Delta$ be a finite  abstract simplicial complex.
A {\it $\Delta$-chirotope} is a map on $\Delta$ that assigns to the simplex $J$ an isomorphism class of arrangements of oriented double pseudolines indexed by $J$ 
with the property that if $J'$ is a subset of $J$ then $\chi(J')$ is a subarrangement of $\chi(J)$. The $\chi(J)$, $J\in \Delta$, are called the 
{\it entries} of the $\Delta$-chirotope~$\chi$.  
A {\it $k$-chirotope on the indexing set $I$} is a  $\Delta$-chirotope whose domain $\Delta$ is the complex of subsets of size at most $k$ of $I$, and a {\it  chirotope} is 
the restriction of a $3$-chirotope to the set of $3$-subsets of its domain.
\item 
For any indexed arrangement of oriented double pseudolines $\Gamma$ and any simplicial complex $\Delta$ on the  indexing set of $\Gamma$ 
the  {\it $\Delta$-chirotope} of $\Gamma$ is  the map $\chi_{\Gamma}$ on $\Delta$ 
that assigns to $J \in \Delta$ the isomorphism class of the subarrangement of $\Gamma$ indexed by $J$.  
\end{enumerate}
Arrangements of double pseudolines are conveniently 
represented  by their flag diagrams,
in view of the following two properties: 
\begin{enumerate}
\item two arrangements are isomorphic if and only if their flag diagrams are isomorphic, cf. ~\cite[Appendix 4.7]{blswz-om-99}; and 
\item 
the group of automorphisms of an arrangement (by definition quotient of the group of self-homeomorphisms of the arrangement by its subgroup of self-homeo\-morphisms isotopic to the identity map) 
is isomorphic to the group of automorphisms of its flag diagram or, equivalently, to the centralizer of the flag operators in the group of permutations of the flags.
Note that an automorphism is defined by the image of one flag since the face poset of an arrangement is flag-connected.
\end{enumerate}
\begin{example} 
The automorphism group of an arrangement of two double pseudolines is the dihedral of order $8$, the group of automorphisms of the square.
The automorphisms $\tau_1$ and $\tau_2$ defined by $\tau_1(F) = \sigma_1(F)$ and  $\tau_2(F) = \sigma_0(F)$
where $F$ is any one of the $8$ flags of the
tetragon intersection of the crosscap sides of the double pseudolines of the arrangement 
are an example of pair of generators of this group, 
for which $\tau_1^2=\tau_2^2 = 1$ and $(\tau_2\tau_1)^4 = 1$  is a complete set of relations; cf. Fig.~\ref{flagdiagrambis}a and Fig.~\ref{flagdiagrambis}b.  
\end{example}
\begin{figure}[!htb]
\psfrag{un}{$1$}\psfrag{de}{$2$}\psfrag{tr}{$3$}\psfrag{qu}{$4$}\psfrag{ci}{$5$}\psfrag{si}{$6$}\psfrag{se}{$7$}\psfrag{hu}{$8$}
\psfrag{unp}{$1'$}\psfrag{dep}{$2'$}\psfrag{trp}{$3'$}\psfrag{qup}{$4'$}\psfrag{cip}{$5'$}\psfrag{sip}{$6'$}\psfrag{sep}{$7'$}\psfrag{hup}{$8'$}
\psfrag{unq}{\tiny $1''$}\psfrag{deq}{$2''$}\psfrag{trq}{$3''$}\psfrag{quq}{$4''$}\psfrag{ciq}{\tiny $5''$}\psfrag{siq}{\footnotesize $6''$}\psfrag{seq}{$7''$}\psfrag{huq}{\tiny $8''$}
\psfrag{un}{$F$}\psfrag{de}{}\psfrag{tr}{}\psfrag{qu}{}\psfrag{ci}{}\psfrag{si}{}\psfrag{se}{}\psfrag{hu}{}
\psfrag{unp}{}\psfrag{dep}{}\psfrag{trp}{}\psfrag{qup}{}\psfrag{cip}{}\psfrag{sip}{}\psfrag{sep}{}\psfrag{hup}{}
\psfrag{unq}{\tiny}\psfrag{deq}{}\psfrag{trq}{}\psfrag{quq}{}\psfrag{ciq}{\tiny}\psfrag{siq}{\footnotesize }\psfrag{seq}{}\psfrag{huq}{\tiny }
\psfrag{oga}{$24$}\psfrag{ogabasic}{$8$}
\psfrag{oga}{}\psfrag{ogabasic}{}
\psfrag{nrr}{$2$}
\psfrag{namebasic}{$A$} \psfrag{name}{$Z$}
\psfrag{namebasic}{}\psfrag{name}{}
\psfrag{cycles}{}
\psfrag{cyclesbis}{}
\psfrag{on}{\footnotesize \tiny $1$}
\psfrag{tw}{\footnotesize \tiny $2$}
\psfrag{ze}{\footnotesize \tiny $0$}
\psfrag{F}{\tiny $F$}
\psfrag{Fone}{\tiny$\tau_1(F)$}
\psfrag{Fzer}{\tiny $\tau_2(F)$}
\psfrag{Fthr}{\tiny $\tau_3(F)$}
\centering
\psfrag{aa}{$(a)$} \psfrag{bb}{$(b)$} \psfrag{cc}{$(c)$} \psfrag{dd}{$(d)$}
\psfrag{aa}{(a)} \psfrag{bb}{(b)} \psfrag{cc}{(c)} \psfrag{dd}{(d)}
\includegraphics[width=0.875\linewidth]{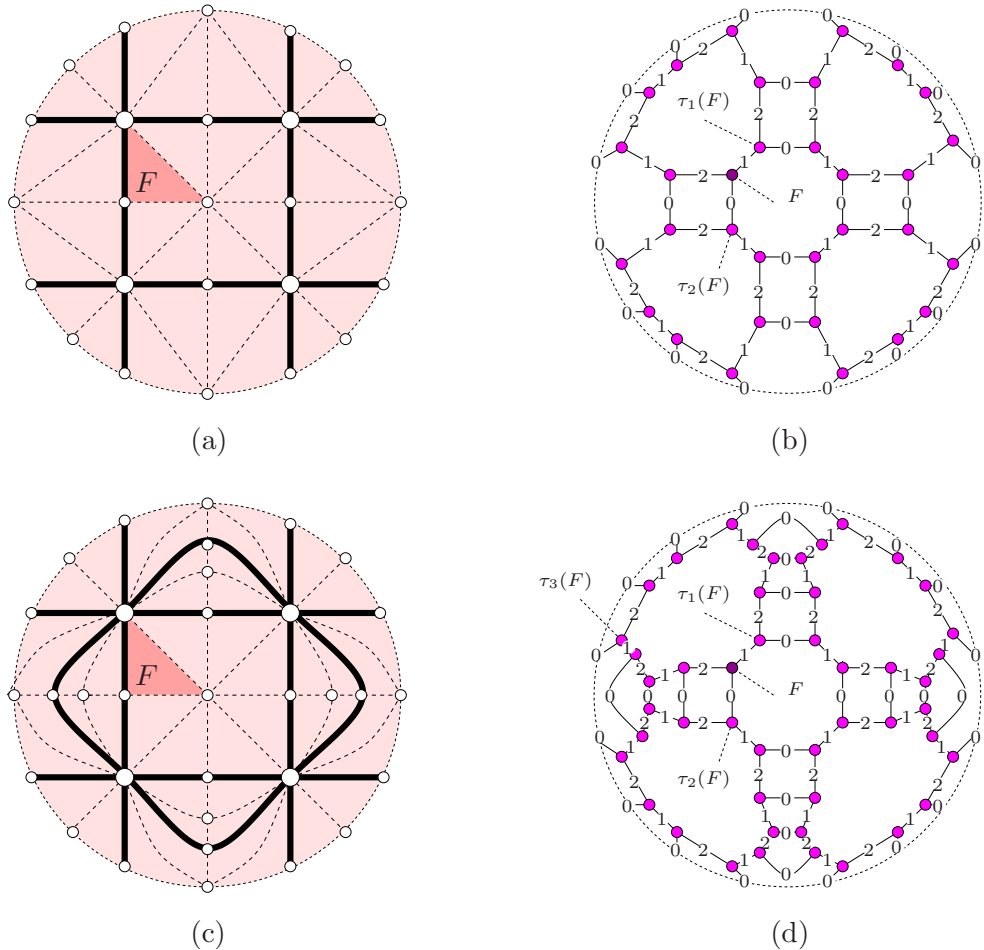}
\caption{\label{flagdiagrambis} 
(a) The first barycentric subdivision of an arrangement of two double pseudolines 
and (b) its flag diagram together with a pair $\tau_1,\tau_2$ of generators of its automorphism group, implicitly defined by the images of the flag $F$; 
(c)  the first barycentric subdivision of the hemi-cube arrangement 
and (d) its flag diagram together with a triple $\tau_1,\tau_2, \tau_3$
 of generators of its automorphism group, implicitly defined by the images of the flag $F$}
\end{figure}

\begin{example} 
The automorphism group of the hemi-cube arrangement is the permutation group of a $4$-set. 
The automorphisms $\tau_1, \tau_2$ and $\tau_3$ defined by $\tau_1(F) = \sigma_1(F), \tau_2(F) = \sigma_0(F)$, and 
$\tau_3(F) = \sigma_{1}\sigma_{2}\sigma_{1}\sigma_{2}(F)$ 
where $F$ is any one of the $3\times 8$ flags of the $3$ tetragons 
of the arrangement (each tetragon is the intersection of the crosscap sides of a pair of double pseudolines) 
are an example of triple of generators of this group,
for which $\tau_1^2 = \tau_2^2 = \tau_3^3 = 1$, $(\tau_{1}\tau_2)^4 = 1$, $\tau_1 \tau_3 = \tau_3^2 \tau_1$ and 
$\tau_2\tau_3 = \tau_3 (\tau_1\tau_2)^2$ is a complete set of relations; cf. Fig.~\ref{flagdiagrambis}c and Fig.~\ref{flagdiagrambis}d.  
\end{example}

Besides (appropriately labeled) flag diagrams, two other codings of indexed arrangements of oriented double pseudolines are used in the paper. Both  are 
 defined using the idea of {\it signed indices}, namely the original indices $i_1,i_2,\ldots,i_n$ and their complements 
$\overline{i}_1, \ldots, \overline{i}_n$;
the original indices are said to be positive, their complements are said to be negative, and the complement of a
negative index is its positive version; cf~\cite[page 12]{k-ah-92}.
Indexed arrangements of oriented double pseudolines 
are now extended to negative indices
by assigning to a negative index the reoriented version of the double pseudoline assigned to its complement.
In this introduction we only give the definition of one of these two codings, namely the coding by {\it side cycles}.  

Let $\Gamma$ be an indexed arrangement of oriented double pseudolines. Its coding by side cycles assigns to each (positive and negative) index  of $\Gamma$ 
two circular words on the set of indices: the first one is called its {\it side cycle of disk type} and the second one is called its {\it side cycle of crosscap type}.
The side cycle of disk type assigned to the index $i$  is the circular sequence of indices 
of the double pseudolines crossed by the side wheel of a sidecar rolling on $\Gamma_i$, side wheel on the disk side of $\Gamma_i$, 
that are (locally) oriented away from~$\Gamma_i$.  
Similarly  the side cycle of crosscap type assigned to the index $i$ is the circular sequence of indices 
of the double pseudolines crossed by the side wheel of a sidecar rolling on $\Gamma_i$, side wheel on the crosscap side of $\Gamma_i$,
that are (locally) oriented away from~$\Gamma_i$.  
Note that the side cycles of disk (crosscap)  type assigned to an index and to its complement are reverse to one another 
and that for simple  arrangements the side cycle of disk type assigned to an index is the complement of its side cycle of crosscap type and vice versa.
We show in Section~\ref{secfou} that the isomorphism class of an indexed arrangement of oriented double pseudolines depends only on its side cycles.
\begin{example}
The side cycles of disk type and crosscap type of the rhombicubeoctahedron arrangement of Fig.~\ref{defarrangdoublelinePP}b are  
$$
\begin{array}{cl}
1:&              \thrcrossR{1}{2}
                 \foucrossR{1}{2}
                 \thrcrossR{1}{3}
                 \foucrossR{1}{3}
                 \onecrossR{1}{2}
                 \twocrossR{1}{2}
                 \onecrossR{1}{3}
                 \twocrossR{1}{3}
\\
2:&              \thrcrossR{2}{3}
                 \foucrossR{2}{3}
                 \thrcrossR{2}{1}
                 \foucrossR{2}{1}
                 \onecrossR{2}{3}
                 \twocrossR{2}{3}
                 \onecrossR{2}{1}
                 \twocrossR{2}{1}
\\
3:&              \thrcrossR{3}{1}
                 \foucrossR{3}{1}
                 \thrcrossR{3}{2}
                 \foucrossR{3}{2}
                 \onecrossR{3}{1}
                 \twocrossR{3}{1}
                 \onecrossR{3}{2}
                 \twocrossR{3}{2}
\end{array}
\qquad \text{and} \qquad
\begin{array}{cl}
1:&              
                 \onecrossR{1}{2}
                 \twocrossR{1}{2}
                 \onecrossR{1}{3}
                 \twocrossR{1}{3}
                 \thrcrossR{1}{2}
                 \foucrossR{1}{2}
                 \thrcrossR{1}{3}
                 \foucrossR{1}{3}
\\
2:&              
                 \onecrossR{2}{3}
                 \twocrossR{2}{3}
                 \onecrossR{2}{1}
                 \twocrossR{2}{1}
                 \thrcrossR{2}{3}
                 \foucrossR{2}{3}
                 \thrcrossR{2}{1}
                 \foucrossR{2}{1}
\\
3:&              
                 \onecrossR{3}{1}
                 \twocrossR{3}{1}
                 \onecrossR{3}{2}
                 \twocrossR{3}{2}
                 \thrcrossR{3}{1}
                 \foucrossR{3}{1}
                 \thrcrossR{3}{2}
                 \foucrossR{3}{2}
\end{array}
$$

\end{example}
\begin{example}  
The side cycles of disk type and crosscap type of the hemi-cube arrangement of Fig.~\ref{CodingADPTER}a are 
$$
\begin{array}{cl}
1: & \jind{3}\jind{2}\jbar{2}\jind{3}\jbar{3}\jbar{2}\jind{2}\jbar{3}\\ 
2: & \jbar{3}\jind{1}\jbar{1}\jbar{3}\jind{3}\jbar{1}\jind{1}\jind{3}\\ 
3: & \jbar{2}\jbar{1}\jind{1}\jbar{2}\jind{2}\jind{1}\jbar{1}\jind{2}
\end{array}
\qquad \text{and} \qquad
\begin{array}{cl}
1: & \jbar{2}\jbar{3}\jbar{3}\jind{2}\jind{2}\jind{3}\jind{3}\jbar{2}\\ 
2: & \jbar{1}\jind{3}\jind{3}\jind{1}\jind{1}\jbar{3}\jbar{3}\jbar{1}\\ 
3: & \jind{1}\jind{2}\jind{2}\jbar{1}\jbar{1}\jbar{2}\jbar{2}\jind{1}.
\end{array}
$$
Similarly the side cycles of disk type and crosscap type of the arrangement of Fig.~\ref{CodingADPTER}b (obtained from that of Fig.~\ref{CodingADPTER}a
 by a splitting mutation) 
are 
$$
\begin{array}{cl}
1: & \jind{3}\jind{2}\jbar{2}\jind{3}\jbar{3}\jbar{2}\jind{2}\jbar{3}\\ 
2: & \jbar{3}\jind{1}\jbar{1}\jbar{3}\jind{3}\jbar{1}\jind{1}\jind{3}\\ 
3: & \jbar{2}\jbar{1}\jind{1}\jbar{2}\jind{2}\jind{1}\jbar{1}\jind{2}
\end{array}
\qquad \text{and} \qquad
\begin{array}{cl}
1: & \jbar{3}\jbar{2}\jbar{3}\jind{2}\jind{2}\jind{3}\jind{3}\jbar{2}\\ 
2: & \jind{3}\jbar{1}\jind{3}\jind{1}\jind{1}\jbar{3}\jbar{3}\jbar{1}\\ 
3: & \jind{2}\jind{1}\jind{2}\jbar{1}\jbar{1}\jbar{2}\jbar{2}\jind{1}.
\end{array}
$$
Note that these  two arrangements have the same side cycles of disk type but differ in their side cycles of crosscap type.
\end{example}
\begin{figure}[!htb]
\psfrag{un}{1}\psfrag{deu}{2}\psfrag{tr}{3}
\psfrag{unb}{$\jbar{1}$}\psfrag{deb}{$\jbar{2}$}\psfrag{trb}{$\jbar{3}$}
\psfrag{onecd}{1}\psfrag{twocd}{2}\psfrag{thrcd}{3}
\psfrag{A}{$A$}\psfrag{B}{$B$}\psfrag{C}{$C$}\psfrag{D}{$D$}
\psfrag{A}{}\psfrag{B}{}\psfrag{C}{}\psfrag{D}{}
\psfrag{aa}{(a)}\psfrag{bb}{(b)}
\centering
\includegraphics[width=0.875\linewidth]{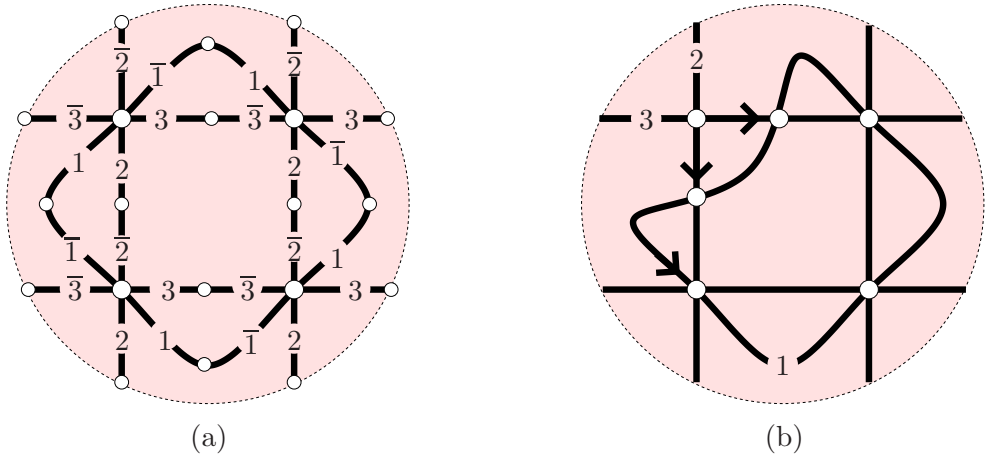}
\caption{ (a) The first barycentric subdivision of the one-skeleton of an indexed and oriented version of the hemi-cube arrangement : each edge of the subdivision is labeled 
with the index of the signed supporting curve of the edge that is,
locally on the edge, oriented away from the vertex of the arrangement to which the edge
is incident; (b) an indexed and oriented version of the arrangement of Fig.~\ref{defarrangdoublelinePP}d  \label{CodingADPTER}
}
\end{figure}

We are now ready to state the first main result of the paper.  It is a direct extension of the rank three case or
 pseudoline case of the Folkman-Lawrence LR characterization (LR for local realizability) of chirotopes of
arrangements of pseudohyperplanes~\cite{fl-om-78}.
\begin{theorem}
\label{theoADP}
The map that assigns to an isomorphism class of indexed arrangements of oriented double pseudolines its chirotope is one-to-one and its range 
is the set of maps~$\chi$  on the set of $3$-subsets  of a finite set  $I$ such that for every $3$-, $4$-, and $5$-subset $J$ of $I$
the restriction of $\chi$ to the set of $3$-subsets of $J$ is a chirotope of arrangements of double pseudolines. 
In other terms, the map which assigns to an isomorphism class of indexed arrangements of double pseudolines its $3$ chirotope is one-to-one and that which assigns 
its $5$-chirotope is (one-to-one and) onto.  \qed 
\end{theorem}
The main lines of its proof are the following. 

Concerning the range part we proceed in three steps. 
First, we extend the notion of double pseudoline arrangements by relaxing the condition on the arrangement which says that the genus of its underlying 
nonorientable surface is~$1$ while retaining locally in the vicinity of a curve of the arrangement, the notions of disk side and crosscap side. 
(The underlying surface of a subarrangement of size at least $2$  is the one obtained by gluing topological disks along the boundaries of a closed tubular neighborhood
 of the curves of the subarrangement in the underlying surface of the whole arrangement. Thus, a subarrangement does not  necessarily live
in a surface whose genus is that  of the underlying surface of the whole arrangement. By convention the underlying surface of a subarrangement of size zero or one is a cross surface.) 
Second, we show that {\it the map that assigns to an isomorphism class of indexed arrangements of oriented double pseudolines its  $5$-chirotope is one-to-one and onto.} 
Third, we  characterize among these arrangements {\it those living in a cross surface as those whose subarrangements of size at most $5$ live in a cross surface.} 
To prove that the arrangements living in a cross surface are those whose subarrangements of size at most $5$ live in a cross surface it can be argued 
that the mutation graphs of the latter are connected or 
that for any pair $FF'$ of distinct faces of an arrangement living in a cross surface, there exists a subarrangement of size at most $3$ whose faces containing $F$ and $F'$ are distinct, 
the {\it separation property} for short. 
Thus, a byproduct of our study is the following direct extension of the Ringel homotopy theorem for arrangements of pseudolines~\cite{r-tegtg-56}. 
\begin{theorem}
\label{theoHT} 
Any two arrangements of double pseudolines of the same size and living in the same cross surface
are homotopic via a finite sequence of mutations followed by an isotopy; in other words,  mutation graphs are connected. \qed
\end{theorem}
Further analysis of the separation property leads us to prove that an arrangement of double pseudolines whose subarrangements of size at most $4$ live in cross surfaces, 
lives in a cross surface or its subarrangements of size $4$ belong to a well-defined class of few tens of arrangements. Therefore, a computer check of the conjecture 
that the arrangements of double pseudolines living in cross surfaces are those whose subarrangements of size at most $4$ live in cross surfaces is doable with modest computing ressources. 
This computer check will the subject of another paper.
That's all for the range part.

Concerning the one-to-one part we proceed by induction on the number of double pseudolines, the crucial case being  the base case of $4$ double pseudolines and, more specifically,
the base case of $4$ double pseudolines with a chosen one whose intersections with each of the others are ordinary and occur in consecutive runs, 
{\it T{\"u}rkenbund} or {\it martagons}{}\footnote{``Da stehn sie also, die Geschwisterkinder, links bl{\"u}t der T{\"u}rkenbund, bl{\"u}t wild, bl{\"u}t wie nirgends, und rechts, da steht die Rapunzel,
und Dianthus superbus, die Prachtnelke, steht nicht weit davon." Gespr{\"a}ch im Gebirg, Paul Celan.} for short. 
Fortunately the list of martagons on $4$ double pseudolines is easily calculated by hand from the exhaustive list of simple arrangements 
of $3$ double pseudolines which in turn is (less) easily  calculated by hand  using the connectedness of mutation graphs. 
It turns out that there are only two martagons on four double pseudolines and that each depends only on its chirotope.

We come now to the definition of chirotopes of configurations of convex bodies. 
Our definition is a natural extension  of the classical 
definition of chirotopes of configurations of points of the standard projective plane; cf. Appendix~\ref{appendix:CFPFP}. 
As for arrangements of double pseudolines, indexed configurations of oriented convex bodies are extended to negative indices 
by assigning to a negative index the reoriented version of the convex body assigned to its complement.

Let $\Delta$ be an  indexed configuration of oriented convex bodies of a projective plane $(\pp,\lpp)$, let $\tau$ be a line of $(\pp,\lpp)$,
let $\rr_\tau$ be the equivalence relation on $\pp$ generated by the pairs of points belonging to a same line segment of $\Delta\cap \tau$, and 
let $\Map{\omega_\tau}{\pp}{\pp/\rr_\tau}$ be the  associated quotient map.
We define   
\begin{enumerate}  
\item the {\it cocycle of $\Delta$ at $\tau$} or the {\it cocycle of $\tau$ with respect to $\Delta$} or the {\it cocycle of the pair $(\Delta,\tau)$} as 
the homeomorphism class of the image of the pair $(\Delta,\tau)$ under 
$\omega_\tau$, i.e., the set of $(\varphi \omega_\tau \Delta, \varphi \omega_\tau \tau)$ as $\varphi$ ranges over the set of homeomorphisms with domain $\pp/\rr_\tau$; 
\item  a {\it bitangent cocycle} or {\it zero-cocycle}  as a cocycle at a bitangent;
\item the {\it isomorphism class of  $\Delta$}
 as  the set of configurations that have the same set of bitangent cocycles as $\Delta$, hence the same set of cocycles as $\Delta$ (use a simple perturbation argument); and 
\item the {\it chirotope of $\Delta$} as the map that assigns to each $3$-subset $J$ of the indexing set of $\Delta$ the isomorphism
class of the subfamily indexed by $J$.
\end{enumerate}
Fig.~\ref{FinalPiercing} depicts the bitangent cocycles of an indexed configuration of three oriented convex bodies. 
\begin{figure}[!htb]
\psfrag{LA}{} \psfrag{LB}{} \psfrag{LC}{} \psfrag{LD}{}
\psfrag{one}{$1$}
\psfrag{two}{$2$}
\psfrag{thr}{$3$}
\psfrag{onecd}{$1$}
\psfrag{twocd}{$2$}
\psfrag{thrcd}{$3$}
\psfrag{A}{$A$}
\psfrag{B}{$B$}
\psfrag{C}{$C$}
\psfrag{D}{$D$}
\psfrag{PA}{$A$}
\psfrag{PB}{$B$}
\psfrag{PC}{$C$}
\psfrag{PD}{$D$}
\psfrag{BBBP}{$\bkk\ii\kk\ptg$}
\psfrag{BPvB}{$\ii\ptg,\kk$}
\psfrag{BvB}{$\ii,\kk$}
\psfrag{BBBB}{$\ii\kk\bii\bkk$}
\psfrag{BBvB}{$\ii\bii,\bkk$}
\psfrag{LA}{$\ii\ptg\bkk\ptg\bjj\ptg$}
\psfrag{LB}{$\ii\ptg\bjj\ptg\kk\ptg$}
\psfrag{LC}{$\ii\ptg\kk\ptg\jj\ptg$}
\psfrag{LD}{$\ii\ptg\jj\ptg\bkk\ptg$}
\psfrag{aa}{(a)}
\psfrag{bb}{(b)}
\psfrag{cc}{(c)}
\centering
\includegraphics[width=0.9875\linewidth]{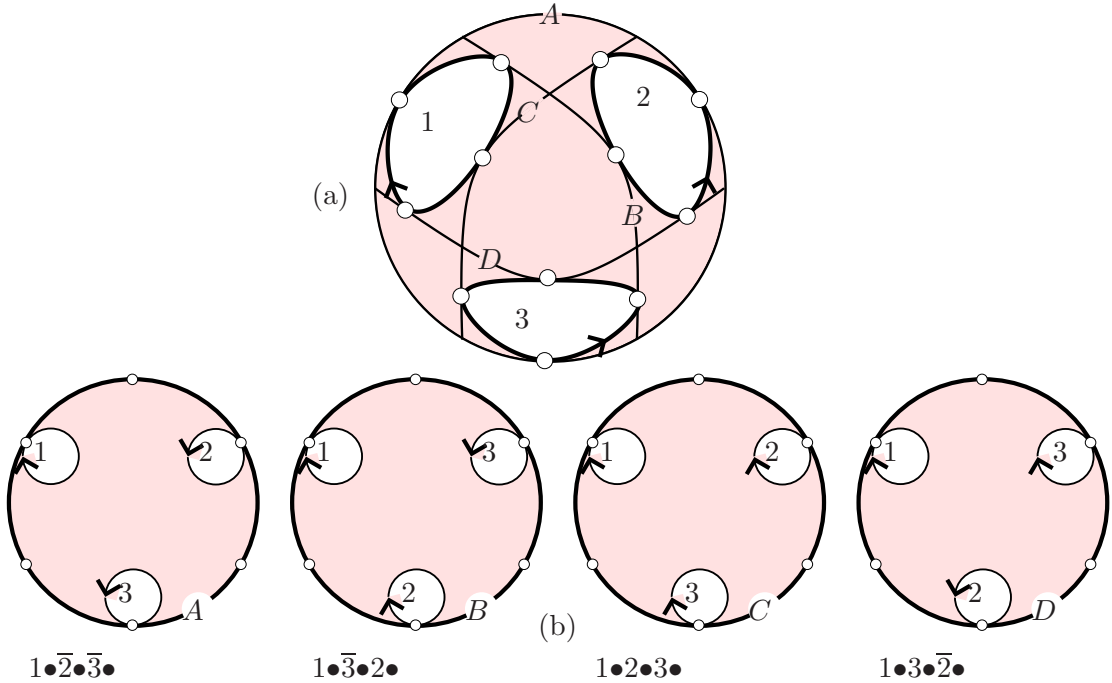}
\caption{(a) The dual configuration of an indexed and oriented version of the hemi-cube arrangement  
together with its bitangents
(these bitangents are all tritangents and are labeled $A,B,C$ and $D$ to ease the correspondance between the diagrams);
 (b)  its bitangent cocycles with their signatures
\label{FinalPiercing}}
\end{figure}
Observe that the cocycle of a tangent to a body does not encode the nature, line segment or point, of the intersection between the tangent and 
the body since the map $\omega_\tau$ reduces this intersection to a point.
A cocycle is conveniently represented by its {\it signature}: a set of words on the signed indices plus the extra symbol $\ptg$ that is defined  as follows. 
Let $\bcut{\pp}{\tau}$ be the closed 2-cell obtained by cutting $\pp/\rr_\tau$ along the line $\tau/\rr_\tau$,
let $\Map{\nu_{\tau}}{\bcut{\pp}{\tau}}{\pp/\rr_\tau}$
be the induced canonical projection,
let $\Sigma_\tau$ be the set of connected components of the pre-images under $\nu_\tau$ of the (indexed and oriented) convex bodies
of $\Delta$ 
(the cardinality of $\Sigma_\tau$ is twice the number of convex bodies  intersected by $\tau$ plus  the number of  bodies missed  by $\tau$), 
let $\epsilon$ be an orientation of $\cut{\tau}$ and let $\overline{\epsilon}$ be its opposite.
The {\it signature of the pair $(\Delta,\tau)$} or the {\it signature of $\Delta$ at $\tau$} is then defined as the pair of {\it signatures of the triples $(\Delta, \tau,\epsilon)$
 and $(\Delta, \tau,\overline{\epsilon})$} where the {\it signature of a triple $(\Delta,\tau, \epsilon)$}
is the set of indices of the elements of $\Sigma_\tau$ with orientation $\epsilon$ contained in the interior of $\cut{\tau}$ plus
the circular sequence of indices of the elements of $\Sigma_\tau$ with orientation $\epsilon$
encountered when walking along the boundary of $\bcut{\pp}{\tau}$ according to the orientation $\epsilon$
with the convention that the indices indexing points are replaced by the extra symbol $\ptg$.
Since
the signature of the triple  $(\Delta,\tau,\overline{\epsilon})$
is obtained from that of the triple $(\Delta,\tau ,\epsilon)$ by replacing each of its elements  by the reversal of its complement (with the convention that $\ptg$ is its own complement) 
the signature of $(\Delta,\tau)$ can be represented by any of its two elements.
Clearly the  cocycle of a pair $(\Delta,\tau)$ depends only on its signature and vice-versa.
 Fig.~\ref{FinalZeroCocycles} depicts the bitangent cocycles  of configurations of two and three convex bodies together with their signatures.
\begin{figure}[!htb]
\psfrag{lone}{}
\psfrag{ltwo}{}
\psfrag{lthr}{}
\psfrag{lfiv}{}
\psfrag{lsix}{}
\psfrag{lsev}{}
\psfrag{one}{}
\psfrag{two}{}
\psfrag{thr}{}
\psfrag{one}{$\ii$}
\psfrag{two}{$\kk$}
\psfrag{thr}{$\jj$}
\psfrag{lone}{$\ii\kk\ptg\ptg$}
\psfrag{ltwo}{$\ii\kk\jj\ptg\ptg\ptg$}
\psfrag{lthr}{$\ii\ptg \kk\ptg \jj\ptg$}
\psfrag{lfiv}{$\ii\kk\ptg\ptg,\jj$}
\psfrag{lsix}{$\ii\bjj\kk \ptg \jj \ptg$}
\psfrag{lsev}{$\bjj\ii\kk\jj \ptg \ptg$}
\psfrag{o6}{$8$}
\psfrag{o2}{$24$}
\psfrag{o22}{$4$}
\centering
\includegraphics[width=0.75\linewidth]{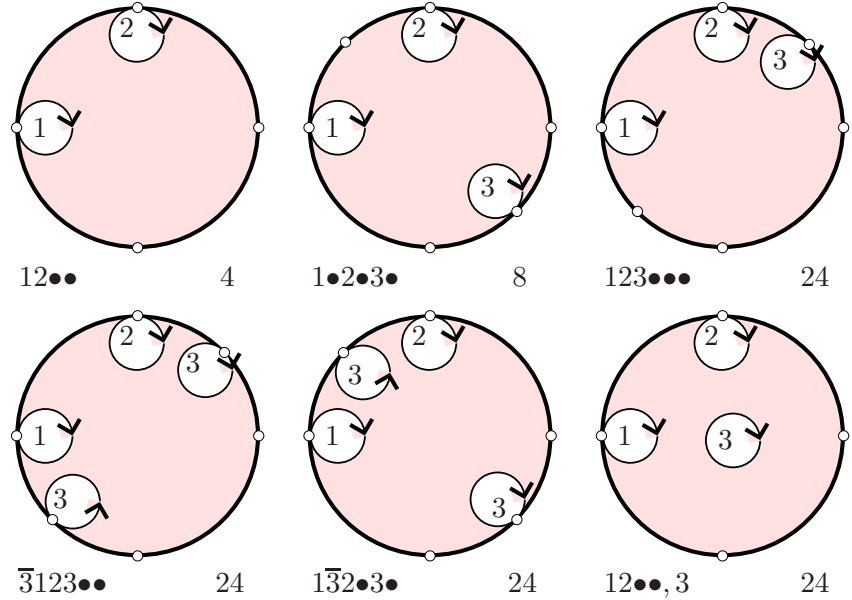}
\caption{The bitangent cocycles on the indexing set $\{1,2\}$ and $\{1,2,3\}$:
each bitangent cocycle is labeled at
at its bottom left  with its signature and
at its bottom right with its number of reoriented and reindexed versions:
thus the number of bitangent cocycles
on a given set of two indices is exactly the number ($4$) of bitangents of a pair of disjoint convex bodies,
 and  the number of bitangent cocycles on a given set of three indices is $8 + 4 \times 24 = 104$.  \label{FinalZeroCocycles}}
\end{figure}

We are now ready to state the second main result of the paper. Its onto part is called the {\it geometric representation theorem for arrangements of double pseudolines} thereafter.
\begin{theorem}
\label{theoPM} 
The map that assigns to an indexed configuration of oriented convex bodies the isomorphism class of its dual arrangement 
is compatible with the isomorphism relation on indexed configurations of oriented convex bodies. 
Furthermore the induced quotient map is one-to-one and onto, i.e., any arrangement of double pseudolines is isomorphic to the dual arrangement of a configuration of convex bodies.  \qed
\end{theorem}
The main lines of its proof are the following. 

Compatibility and one-to-one parts are easy consequences of two basic properties of cocycles: first,
the injectivity of the map that assigns to each cell of the dual arrangement of an indexed configuration of two oriented convex bodies the cocycle of the configuration at some 
(hence any) element of the cell, and, second,
the injectivity of the map that assigns to a bitangent cocycle of an indexed family of at least three oriented convex bodies 
the  sub-cocycles obtained by removing in turn  each of the convex bodies. 
Concerning the onto part we show that the property of being isomorphic to the dual arrangement of a configuration of convex bodies is invariant under mutation
(mutation graphs being connected the result follows). 
To this end we first show that the isomorphism class of the dual arrangement of an indexed configuration of oriented convex bodies depends only on 
the isomorphism class of its (appropriately) indexed arrangement of bitangents, its {\it Rapunzel} or {\it raiponce} for short. 
Then we easily characterize the class of raiponces, using the
{\it enlargement theorem for pseudoline arrangements} of Goodman, Pollack, Wenger and Zamfirescu~\cite{gpwz-atp-94}; cf. Appendix~\ref{appendix:apl}.
And finally  we explain how to push back a mutation at the level of raiponces (the resulting operation is not a mutation). 
Combining Theorems~\ref{theoADP} and~\ref{theoPM} we get the result announced in the abstract, namely:

\begin{theorem}
\label{theoFCB} 
The map that assigns to an isomorphism class of indexed configurations of oriented convex bodies its chirotope is one-to-one and its range 
is the set of maps~$\chi$ on the set of $3$-subsets of a finite set~$I$ such that for every $3$-, $4$-, and $5$-subset~$J$ of~$I$  
the restriction of $\chi$ to the set of $3$-subsets of $J$ is a chirotope of configurations of convex bodies. 
\qed \end{theorem}

\subsection{Organization of the paper}
The paper is organized as follows. 
In Section~\ref{sec:homotopy} we prove that mutation graphs are connected (Theorem~\ref{theoHT}) 
and we use this connectedness result to compute the isomorphism classes of simple arrangements of three 
double pseudolines and 
the martagons on three and four double pseudolines.
In Section~\ref{secthr}  we show, again using the connectedness of mutation graphs proved in Section~\ref{sec:homotopy}, that  
any arrangement of double pseudolines is isomorphic to the dual arrangement of a  configuration of convex  bodies (onto part of Theorem~\ref{theoPM}). 
In Section~\ref{secfou} we prove that the isomorphism class of an indexed arrangement of oriented double pseudolines depends only on its chirotope (Theorem~\ref{theoADP}) 
and we prove that the map that assigns to a configuration of convex bodies the isomorphism class of its dual arrangement is compatible with the isomorphism relation 
on configurations of convex bodies and that the induced quotient map is one-to-one (Theorem~\ref{theoPM}).
In Section~\ref{secfiv} we introduce the arrrangements of double pseudolines living in nonorientable surfaces of arbitrary genus and we prove the LR characterization of chirotopes of indexed arrangements of oriented double pseudolines living in cross surfaces (Theorem~\ref{theoADP}). 
Still in Section~\ref{secfiv} we offer results in strong support of 
the conjecture that the arrangements living in cross surfaces are those whose subarrangements of size at most $4$ live in cross surfaces.
In Section~\ref{secsix} we discuss {\it arrangements of pseudocircles} 
(as natural extensions of both arrangements of pseudolines and arrangements of double pseudolines), {\it crosscap or M{\"o}bius arrangements} 
and their {\it fibrations} (as dual arrangements of affine configurations of convex bodies with, in particular, a positive answer to 
a question of Goodman and Pollack about the realizability of their double permutation sequences by affine configurations of pairwise disjoint convex bodies).
We conclude in the seventh and last section with  a list of open problems suggested by this research.

\section{Homotopy theorem}\label{sec:homotopy}
In this section we prove  that any two  arrangements of double pseudolines with the same number
of double pseudolines and living in the same cross surface are homotopic via a finite sequence of mutations
followed by an isotopy; cf. Theorem~\ref{theoHT}. We proceed into two steps:
\begin{enumerate}
\item firstly, in order to benefit from Ringel's homotopy theorem for arrangements of pseudolines, 
 we embed the collection of isomorphism classes of simple arrangements of pseudolines into the collection of isomorphism classes of arrangements of 
double pseudolines; 
the embedding is canonical and is based on the notion of {\it thin} arrangement of double pseudolines;
\item secondly (and this is the core of our proof) we introduce a `pumping' device to come down to the  
case of arrangements of pseudolines.
\end{enumerate}
We also provide representatives of the isomorphism classes of simple arrangements of three double pseudolines and we use these representatives 
to  compute the full list of martagons on three and four double pseudolines. Recall that martagons are arrangements 
that play a special r{\^o}le in the proof that 
the isomorphism class of an indexed arrangement of oriented double pseudolines depends only on its chirotope.


\subsection{Thin arrangements of double pseudolines} 
A simple arrangement of double pseu\-do\-li\-nes is  
{\it thin} if the crosscap sides of  its double pseudolines are free of vertices.  A 
thin arrangement of double pseudolines $\Gamma^*$  is a {\it double} of a simple arrangement of pseudolines $\Gamma$ (or $\Gamma$ is a {\it core} arrangement of pseudolines of $\Gamma^*$) if
there exists a one-to-one correspondence between $\Gamma$ and $\Gamma^*$ such that any pseudoline of $\Gamma$ is a core pseudoline of its corresponding 
double pseudoline in~$\Gamma^*$.
\begin{figure}[!htb]
\centering
\psfrag{C}{$\alpha$}
\psfrag{G}{$\gamma$}
\psfrag{MSG}{$\MS(\gamma)$}
\psfrag{PP}{$\PP$}\psfrag{infty}{$\infty$}
\psfrag{1}{$1$-cell}\psfrag{2}{$2$-cell}\psfrag{0}{$0$-cell}\psfrag{3}{$\cal M$}\psfrag{-1}{$\emptyset$}
\psfrag{aa}{$(a)$} \psfrag{bb}{$(b)$} \psfrag{cc}{$(c)$} \psfrag{dd}{$(d)$}
\psfrag{aa}{(a)} \psfrag{bb}{(b)} \psfrag{cc}{(c)} \psfrag{dd}{(d)}
\psfrag{one}{$1$}
\psfrag{two}{$2$}
\psfrag{thr}{$3$}
\psfrag{splitting}{\footnotesize splitting}
\psfrag{merging}{\footnotesize merging}
\psfrag{arrow}{$\rightarrow$}
\psfrag{barrow}{$\leftarrow$}
\includegraphics[width=0.65\linewidth]{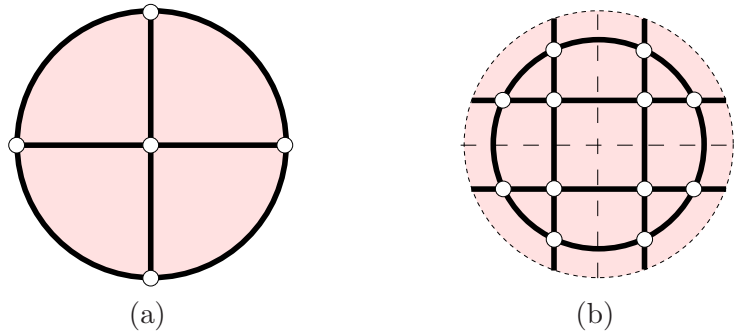}
\caption{\label{finaldouble}
(a) The octahedron arrangement and (b) its double, the rhombicubeoctahedron arrangement} 
\end{figure}
For example the rhombicubeoctahedron arrangement is thin and is the double of the octahedron arrangement, the unique simple arrangement of $3$ pseudolines; 
cf. Fig.~\ref{finaldouble}.
The following two lemmas are simple consequences of the definitions. 
 
\begin{lemma}\label{EmbeddingSimple} 
The map that assigns to a simple arrangement of pseudolines its set of doubles induces a one-to-one 
and onto 
correspondence between the set of isomorphism classes 
of simple arrangements of pseudolines and the set of isomorphism  classes of thin arrangements of double pseudolines.\qed 
\end{lemma}
\begin{lemma} \label{relevementdesmutations}
Let $\Gamma$ and $\Gamma'$ be two simple arrangements of pseudolines and let $\Gamma^*$ and $\Gamma'^*$ be 
double versions of $\Gamma$ and $\Gamma'.$ 
Assume that $\Gamma$ and $\Gamma'$ are connected by a  sequence of two mutations (a merging mutation followed by its `symmetric' splitting mutation) during which the
moving pseudoline is $\Gamma_i$. Then
$\Gamma^*$ and $\Gamma'^*$ are homotopic via a sequence of sixteen mutations during which the only moving double pseudoline is~$\Gamma_i^*$.  \qed
\end{lemma}

\subsection{The pumping lemma}
We come now to the statement of our pumping lemma and 
to the proof of our homotopy theorem.
\begin{lemma}[Pumping Lemma] \label{mainresult}
Let $\Gamma$ be a simple arrangement   
of double pseudolines, and $\gamma \in \Gamma$.  
Assume that there is a vertex
of the arrangement $\Gamma$ lying in the  crosscap side of $\gamma$. 
Then there is a triangular two-cell of the arrangement $\Gamma$
 contained in the crosscap side of  $\gamma$  with 
a side supported by~$\gamma$.\qed
\end{lemma}
\begin{proof}
Let ${\cal P}$ be the underlying cross surface of $\Gamma$ and let $\Map{p}{\lift{{\cal P}}}{{\cal P}}$ be a $2$-sheeted unbranched covering of ${\cal P}.$ 
For example the two relations 
$$\left\{\begin{array}{c}
\alpha_1 \alpha_2 =1\\
\alpha_2\alpha_1 =1
\end{array} \right.
$$
define a $2$-sheeted unbranched covering of the cross surface defined by the relation~$\alpha\alpha =1$; cf.~\cite{m-bcat-91,j-cts-79}.  
The two lifts under $p$ of a curve $\tau$ of $\Gamma$ are denoted $\tau_+$ and~$\tau_-$, and the set of lifts of the curves of $\Gamma$ is denoted $\lift{\Gamma}$.  
Fig.~\ref{TwoCoveringXX}a shows a subarrangement of two double pseudolines and Fig.~\ref{TwoCoveringXX}b shows its $2$-sheeted unbranched covering. 
\begin{figure}[!htb]
\centering
\psfrag{UU}{$\UU$}\psfrag{VV}{$\VV$}
\psfrag{alpha}{$\alpha$}
\psfrag{alphaone}{$\alpha_1$}
\psfrag{alphatwo}{$\alpha_2$}
\psfrag{B}{$B$}\psfrag{Bs}{$B_*$}
\psfrag{Ud}{$\tau$}\psfrag{Vd}{$\tau$}
\psfrag{Ud'}{$\tau'$}\psfrag{Vd'}{$\tau'$}
\psfrag{U}{$\tau_+$}\psfrag{V}{$\tau_-$}
\psfrag{U'}{$\tau'_+$}\psfrag{V'}{$\tau'_-$}
\psfrag{Up}{$B$} \psfrag{Upp}{$B'$} 
\psfrag{aa}{(a)} \psfrag{bb}{(b)} \psfrag{cc}{} \psfrag{dd}{(c)} \psfrag{ee}{}
\includegraphics[width=0.8500075\linewidth]{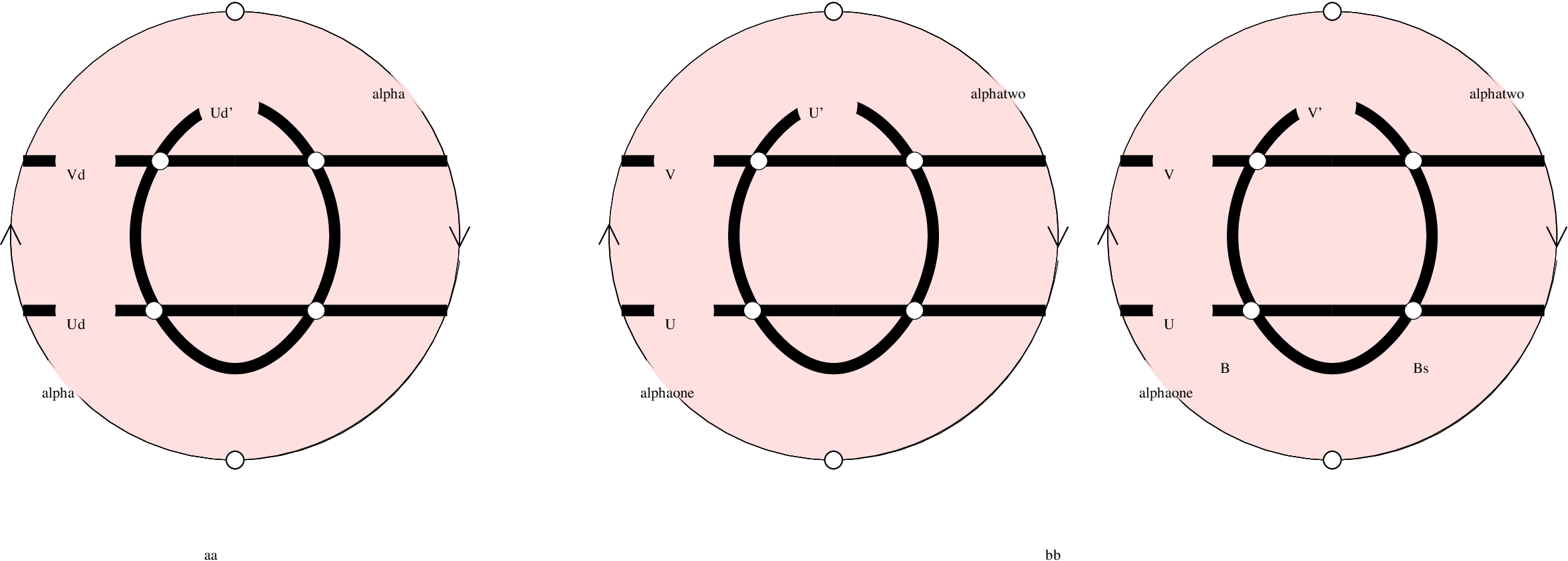}
\caption{
(a) An arrangement of two double pseudolines; (b) its $2$-sheeted unbranched covering 
\label{TwoCoveringXX}}
\end{figure}
We note that two curves of $\lift{\Gamma}$ have exactly $0$ or $2$ intersection 
points depending on whether they are the lifts of the same 
curve in $\Gamma$, or not. By convention if $B$ is one of the two intersection points of 
two crossing curves of $\lift{\Gamma}$  then the other one is denoted~$B_*$, as illustrated in Fig.~\ref{TwoCoveringXX}b. 
Let $C$ be the cylinder of $\lift{{\cal P}}$ 
bounded by $\UU$ and $\VV$.  We introduce  the following terminology.
\begin{enumerate}
\item A {\it \tracecurve\ supported by} $\gamma' \in \Gamma$, $\gamma' \neq \gamma$, is a maximal subcurve of $\gamma'_+$ or $\gamma'_-$ contained in the cylinder $C$.
Observe that there are four \tracecurves\ supported by  $\gamma'$ (two per lift of $\gamma'$)  and that a \tracecurve\ has an endpoint on $\UU$ and 
the other one on $\VV.$
The \tracecurve\ with endpoint $B$ on $\UU$ is denoted $\gcurve(B).$  
\item An {\it arrangement of \tracecurves}  is a set of \tracecurves\ embedded in the cylinder $C.$ 
The cell  decomposition of the cylinder $C$ induced by an arrangement of two \tracecurves\ depends only on the number of intersection points, 
as depicted in Fig.~\ref{TwoCoveringYY}.
\begin{figure}[!htb]
\centering
\psfrag{UU}{$\UU$}\psfrag{VV}{$\VV$}
\psfrag{alpha}{$\alpha$}
\psfrag{alphaone}{$\alpha_1$}
\psfrag{alphatwo}{$\alpha_2$}
\psfrag{B}{$B$}\psfrag{Bs}{$B_*$}
\psfrag{Ud}{$\tau$}\psfrag{Vd}{$\tau$}
\psfrag{Ud'}{$\tau'$}\psfrag{Vd'}{$\tau'$}
\psfrag{U}{$\tau_+$}\psfrag{V}{$\tau_-$}
\psfrag{U'}{$\tau'_+$}\psfrag{V'}{$\tau'_-$}
\psfrag{Up}{$B$} \psfrag{Upp}{$B'$} 
\psfrag{aa}{(a)} \psfrag{bb}{(b)} \psfrag{cc}{} \psfrag{dd}{} \psfrag{ee}{}
\includegraphics[width=0.8500075\linewidth]{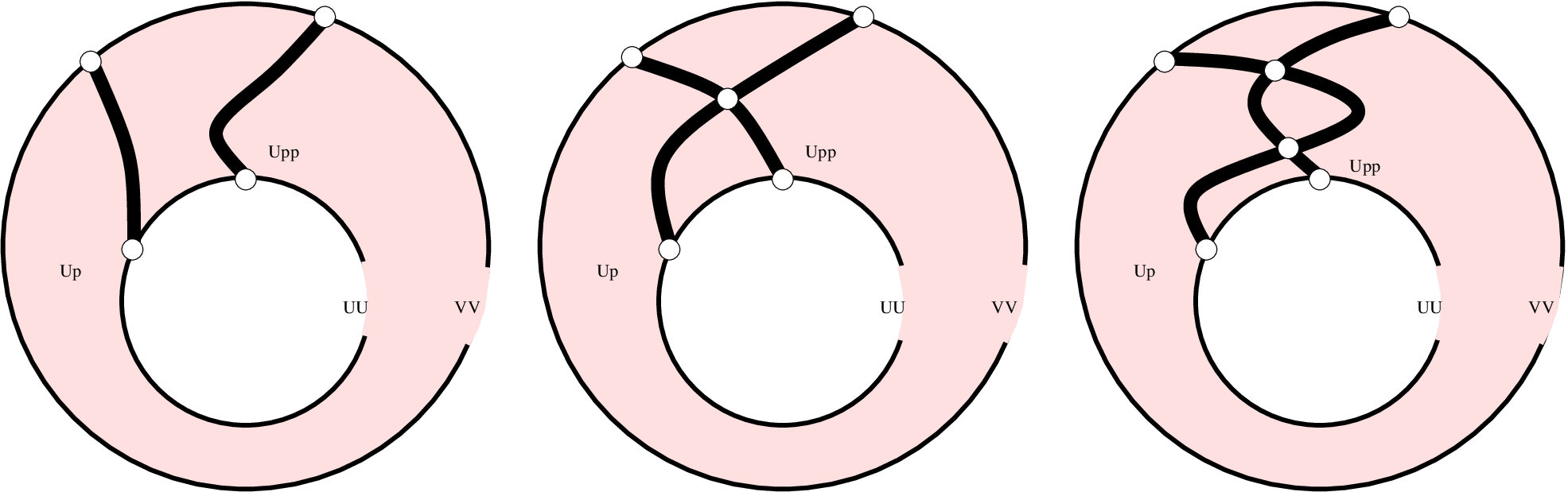}
\caption{The $3$ possible arrangements of two  $\gamma$-curves 
\label{TwoCoveringYY}}
\end{figure}
\item A {\it \gtriangle} is  a triangular face of the 
arrangement of  two crossing \tracecurves\ 
with a side supported by $\UU$;
the vertex of a \gtriangle\ not on $\UU$ is called its {\it apex} and the side
of a \gtriangle\ supported by $\UU$ is called its {\it base side}. 
The interior and the exterior of the base side of a  \gtriangle\ $T$, considered as a subset of $\UU$, are denoted $\base{T}$ and $\cobase{T}$, respectively. 
\item A  \gtriangle\ is {\it admissible} if one of its two sides
with the apex as an endpoint is an edge of $\lift{\Gamma}$. 
\begin{figure}[!htb]
\centering
\psfrag{aa}{\normalsize (a)} \psfrag{bb}{\normalsize (b)} \psfrag{cc}{\normalsize (c)}
\tiny
\footnotesize
\psfrag{Delta}{$\Delta$}
\psfrag{CYZ}{\normalsize $\stripbis{Y,Y'}$}
\psfrag{CBBY}{\normalsize $\stripbis{B,B',Y}$}
\psfrag{CBBZ}{\normalsize $\stripbis{B,B',Y'}$}
\psfrag{CBBB}{\normalsize $\stripbis{B,B',B''}$}
\psfrag{x}{$X$}
\psfrag{y1}{$Y'$}
\psfrag{y0}{$Y$}
\psfrag{t0}{$T$} \psfrag{t1}{$T'$} \psfrag{t2}{$T''$} \psfrag{a0}{$A$}
\psfrag{a1}{$A'$} \psfrag{a2}{$A''$} \psfrag{a3}{$A'''$} \psfrag{a4}{$A^{(4)}$}
\psfrag{b0}{$B$} \psfrag{b1}{$B'$} \psfrag{b2}{$B''$} \psfrag{b3}{$B'''$}
\psfrag{b4}{$B^{(4)}$} \psfrag{b5}{$B^{(5)}$} \psfrag{b1s}{$\twin{B}'$}
\psfrag{b2s}{$\twin{B}''$} \psfrag{bp3}{$\twin{B}'''$} \psfrag{T1}{$T'$}
\psfrag{a1s}{$\twin{A'}$}
\psfrag{T3}{$T'''$}
\psfrag{casetwo}{$B'' \in \cobase{T}$}
\psfrag{caseone}{$B'' \in \base{T}$}
\psfrag{aa}{(a)}
\psfrag{bb}{(b)}
\psfrag{cc}{(c)}
\psfrag{dd}{(d)}
\psfrag{ee}{(e)}
\psfrag{ff}{(f)}
\psfrag{gg}{(g)}
\includegraphics[width = 0.850 \linewidth]{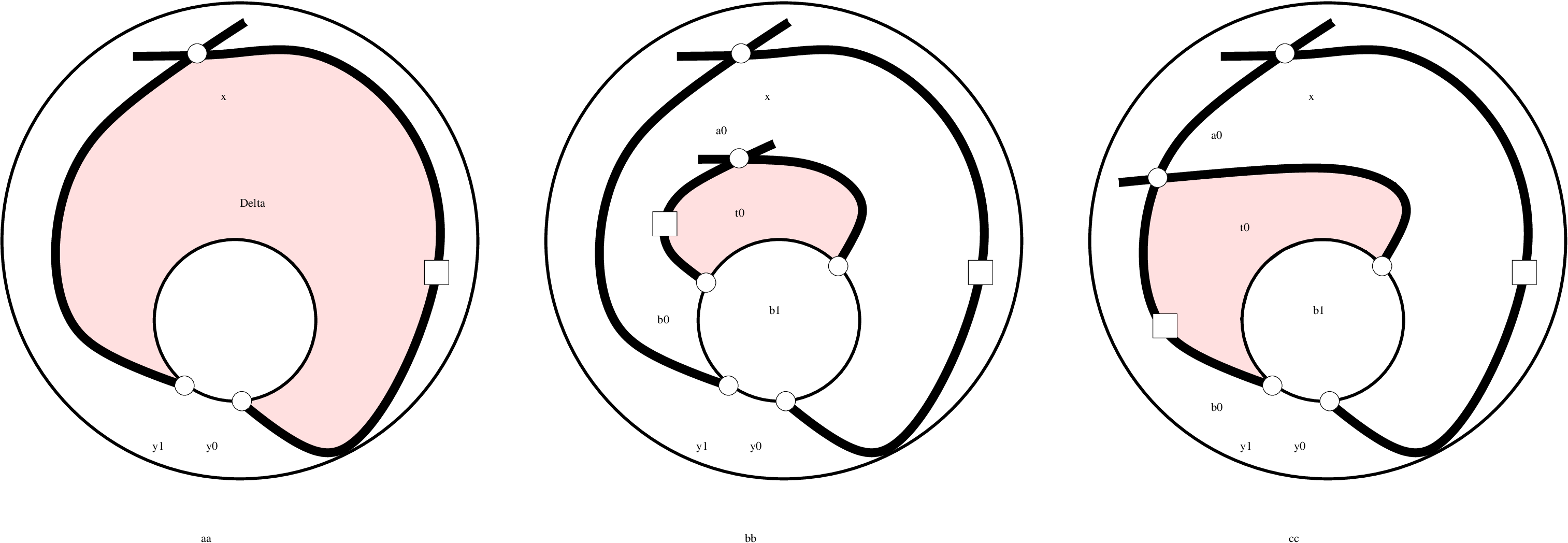}
\caption{The admissible \gtriangle\ $\Delta$ encloses the admissible \gtriangle\ $T$ \label{finalencloseone}}
\end{figure}
\item 
An admissible \gtriangle\ $\Delta = XYY'$ with apex $X$ and edge side $XY$ is said to enclose an admissible \gtriangle\ $T = ABB'$ with apex $A$ and edge side $AB$ if $T$ is included in $\Delta$ and walking along the base side of $\Delta$ from $Y$ to $Y'$ we encounter $B'$ before $B$, thus 
the arrangement of the four \tracecurves\ $\gcurve(Y)$, $\gcurve(Y')$, $\gcurve(B)$, $\gcurve(B')$ is, up to homeomorphism, one of those implicitly depicted in 
Fig.~\ref{finalencloseone}a in case $B\neq Y'$ or one of those implicitly depicted in 
Fig.~\ref{finalencloseone}b in case $B= Y'$.
\end{enumerate}
\begin{lemma}\label{existence}
There is at least one admissible \gtriangle.
\end{lemma}
\begin{proof}
Since by  assumption there is a vertex of $\Gamma$ in the crosscap side of  the double pseudoline 
$\gamma$, there is  a \gtriangle, say $T=ABB'$ with apex $A$. Let 
$A'$ be the vertex of $\lift{\Gamma}$ that follows 
$B'$ on the side $B'A$ of $T$. Then $A'$ is the apex 
of an  admissible \gtriangle\ $T' = A'B'B''$ with edge side $A'B'$. 
This proves that there is at least one admissible \gtriangle.
\end{proof}

\begin{figure}[!htb]
\centering
\psfrag{aa}{\normalsize (a)}
\psfrag{bb}{\normalsize (b)}
\psfrag{cc}{\normalsize (c)}
\psfrag{dd}{\normalsize (d)}
\psfrag{x}{$X$}
\psfrag{y1}{$Y'$}
\psfrag{y0}{$Y$}
\psfrag{t01}{$T,T'$}
\psfrag{a01}{$A,A'$}
\psfrag{t0}{$T$} \psfrag{t1}{$T'$} \psfrag{t2}{$T''$} \psfrag{a0}{$A$}
\psfrag{a1}{$A'$} \psfrag{a2}{$A''$} \psfrag{a3}{$A'''$} \psfrag{a4}{$A^{(4)}$}
\psfrag{b0}{$B$} \psfrag{b1}{$B'$} \psfrag{b2}{$B''$} \psfrag{b3}{$B'''$}
\psfrag{b4}{$B^{(4)}$} \psfrag{b5}{$B^{(5)}$} \psfrag{b1s}{$\twin{B}'$}
\psfrag{b2s}{$\twin{B}''$} \psfrag{bp3}{$\twin{B}'''$} \psfrag{T1}{$T'$}
\psfrag{a1s}{$\twin{A'}$}
\psfrag{T3}{$T'''$}
\psfrag{delta}{$\Delta$}
\psfrag{UU}{$\gamma^+$}
\psfrag{VV}{$\gamma^-$}

\includegraphics[width = 0.875 \linewidth]{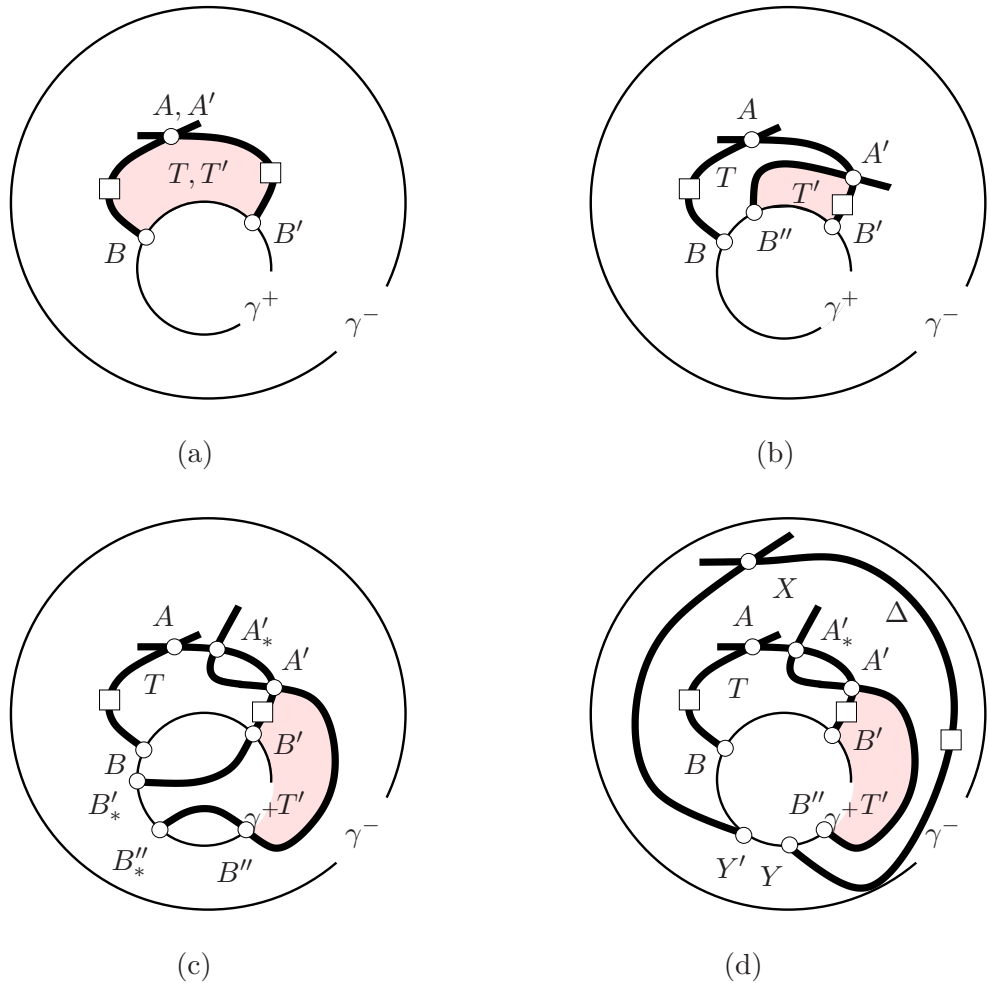}
\caption{Relative positions of an admissible \gtriangle\ $T$ and its derived  admissible \gtriangle~$T'$  
\label{caseoneBis}}
\end{figure} 

Now let $T=ABB'$ be an  admissible \gtriangle\ with apex $A$ and edge side $AB$,
 let $A'$ be the vertex of $\lift{\Gamma}$ that follows 
$B'$ on the side $B'A$ of $T$, and let $T' = A'B'B''$ be the admissible
\gtriangle\ with apex $A'$ and with edge side $A'B'$.
A simple use of the Jordan curve theorem leads to the  following three lemmas concerning 
the relative positions $T$  with respect to  $T'$, possibly in the presence  
of a third admissible \gtriangle\ $\Delta$ enclosing~$T$.  Fig.~\ref{caseoneBis}a, \ref{caseoneBis}b, \ref{caseoneBis}c, and \ref{caseoneBis}d illustrate these lemmas.

\begin{lemma} \label{positionthr} 
Assume that $T=T'$. Then $T$ is a triangular two-cell of $\widetilde{\Gamma}$. \qed
\end{lemma}

\begin{lemma}
\label{positiontwo}
Assume that $T\neq T'$ and that $B'' \in \base{T}$.  Then 
$\gcurve(B'')$ crosses the side $B'A$ of $T$ exactly once (at $A'$) and 
$\base{T'}$ is contained in $\base{T}$. \qed
\end{lemma}

\begin{lemma}\label{positionone}
Assume that $T\neq T'$ and that $B'' \in \cobase{T}$.  Then 
\begin{enumerate}
\item $\gcurve(B')$ and $\gcurve(B'')$ cross twice (at $A'$ and $\twin{A'}$) on the side $B'A$ of $T$, 
\item $\base{T}$ and $\base{T'}$ are disjoint, 
\item $\twin{B}'$ and $\twin{B}'' \in \cobase{T} \cap \cobase{T'}$, and 
\item walking along $\cobase{T} \cap \cobase{T'}$ from $B''$ to $B$ we encounter successively the points $\twin{B}''$ and $\twin{B}'$. 
\end{enumerate} 
Furthermore if  $\Delta$ encloses $T$ then $\Delta$ encloses $T'$. \qed
\end{lemma}

Consider now the sequence of admissible \gtriangles\ 
$T_0,T_1,T_2,\ldots $  
defined inductively by $T_0=T$ and $T_{k+1} = T'_k$  for $k\geq 0$. 
A simple combination of Lemmas~\ref{positionone}, and ~\ref{positiontwo} leads to the conclusion that the sequence $T_k$ is stationary.
According to Lemma~\ref{positionthr} the pumping lemma  follows. 
\end{proof}

\begin{remark} The  proof of the pumping lemma involves only subarrangements of size at most $6$; cf. Fig.~\ref{caseoneBis}d. 
A slightly more careful analysis shows that only the subarrangements of size at most $5$ are relevant.  
This key feature is exploited in Section~\ref{secfiv} to extend the classical LR characterization of chirotopes of arrangements of pseudolines 
to chirotopes of arrangements of double pseudolines; cf. Theorem~\ref{theoADP}.
\end{remark}
\begin{remark} \label{bravo}
The pumping lemma  asserts that a certain instance of the problem of sweeping a spherical arrangement of pseudocircles crossing pairwise in $0$ or $2$ points 
has a positive answer. 
This problem is studied in full generality by J.~Snoeyink and J. Hershberger~\cite{sh-sac-91} and, as pointed to us by an anonymous referee,  the
pumping lemma can be derived from their results. (It is necessary to use both Theorem 3.1  and Lemma 5.2  of~\cite{sh-sac-91}.)
\end{remark}

We are now ready for the proof of our homotopy theorem.

\setcounter{theorem}{2}
\begin{theorem}
Any two arrangements of double pseudolines of the same size and living in the same cross surface
are homotopic via a finite sequence of mutations followed by an isotopy; in other words,  mutation graphs are connected. 
\end{theorem}
\setcounter{theorem}{12}
\begin{proof}
Clearly any arrangement of double pseudolines is homotopic, via a finite sequence of splitting mutations, 
to a simple one. 
Now by a repeated application of the  pumping lemma
we see easily that any simple arrangement of double pseudolines 
is homotopic, via a finite sequence of mutations, to a simple thin one.
It remains to use Lemma~\ref{EmbeddingSimple}, Lemma~\ref{relevementdesmutations} and the homotopy theorem of Ringel for arrangements of pseudolines
to conclude the proof.
 \end{proof}

For the sake of completeness, we mention that one of the standard ways to prove the Ringel's homotopy theorem for arrangements of pseudolines is to 
show that any arrangement of pseudolines is homotopic, via a finite sequence of mutations 
followed by an isotopy, to a cyclic arrangement of pseudolines using {\it avant la lettre} the following specialization 
to arrangements of pseudolines of our pumping lemma for arrangements of double pseudolines (think of a pair of pseudolines as a pinched double pseudoline).
\begin{lemma}[Pumping Lemma for Arrangements of Pseudolines] \label{merging}
Let $\Gamma$ be a simple arrangement of pseudolines, let $\gamma, \gamma'\in \Gamma$, $\gamma \neq \gamma'$,
 and let $M(\gamma, \gamma')$ be one of the two two-cells of the subarrangement $\{\gamma,\gamma'\}$.  Assume that there exists a vertex of
the arrangement
$\Gamma$ lying in  $M(\gamma, \gamma')$.
Then there exists a triangular two-cell of the arrangement $\Gamma$ contained in
$M(\gamma,\gamma')$ with a side supported by $\gamma'$ and a vertex contained in $M(\gamma,\gamma')$. 
\end{lemma}
\begin{proof} The proof is standard and will not be repeated here; see e.g.~\cite{c-mor3a-82}.\end{proof}
\begin{remark} The proof of the pumping lemma for arrangements of pseudolines involves only subarrangements of size $4$. 
This observation will be used in Section~\ref{secfiv} to give a new proof of the classical LR characterization of chirotopes of 
indexed arrangements of oriented pseudolines. (For historical comments on the various proofs of the LR characterization of chirotopes of       
indexed arrangements of oriented pseudolines and, more generaly, pseudo-hyperplanes, we refer to~\cite{b-com-06,bkms-trom-05,bms-fltrt-01}.)
\end{remark}

\begin{remark} At this point it is natural to ask if the space of one-extensions of an arrangement of double pseudolines is connected under mutations, 
as is the space of 
one-extensions of an arrangement of pseudolines~\cite{g-pa-04,blswz-om-99,sz-esom-93}. (A one-extension of an arrangement of $n$ pseudolines $\Gamma$ is a arrangement of 
$n+1$ pseudolines $\Gamma'$ of which $\Gamma$ is a sub-arrangement.)
A positive answer to that question, providing the key to a practical enumeration algorithm for simple arrangements of at most $5$ double pseudolines, 
is given in~\cite{fpp-nsafd-11}. 
The proof presented in~\cite{fpp-nsafd-11} of this connectedness result is based on an enhanced version of the pumping lemma which says that, given a double pseudoline $\gamma$ of an 
arrangement $\Gamma$ with the property that the vertices of the arrangement $\Gamma$ lying on the curve $\gamma$ are ordinary,
either there are (at least) two fans contained in the crosscap side of the double pseudoline $\gamma$ with base sides supported by $\gamma$ or there are no vertices of the arrangement 
contained in the crosscap side of $\gamma$. The enhanced version of the pumping lemma can be easily proved using the geometric representation theorem for 
arrangements of double pseudolines.
It will be interesting to have a direct proof of it since, as explained in~\cite{fpp-nsafd-11}, the geometric representation theorem for 
arrangements of double pseudolines can be derived from it. 
\end{remark}
\subsection{Martagons}\label{enumeration}
The exhaustive list of isomorphism classes of simple arrangements of
three double pseudolines is depicted in Fig.~\ref{fulllist}.  
This list was first established by hand, using the connectedness of the corresponding mutation graph. 
The adjacency list representation of this graph is the following:
$$ \begin{array}{lll}
\name{04} & \text{adjacent to} & \name{07} \\
\name{07} & :& \name{04}, \name{15},\name{18}\\
\name{15} & :& \name{07},\name{25_1},\name{25_2}\\
\name{18} & :& \name{07},\name{25_1},\name{37} \\
\name{22} & : & \name{25_2}\\
\name{25_1}& :&\name{15},\name{18},\name{32},\name{33},\name{43} \\
\name{25_2}& :&\name{15},\name{22},\name{33},\name{36}\\
\name{32} & :&\name{25_1}\\
\name{33} & :& \name{25_1},\name{25_2}\\
\name{36}&:&\name{25_2}\\
\name{37}& :& \name{18},\name{43},\name{64}\\
\name{43}& :& \name{25_1},\name{37}\\
\name{64}& :& \name{37}
\end{array}
$$
where $\name{\alpha}$ denotes the arrangement whose $2$-sequence of its numbers of $2$-cells of size~2 and~3 is $\alpha$.
Such a sequence identifies a unique isomorphism class of arrangements, with one exception: the
sequence $25$ identifies two isomorphism classes (which have also the same numbers of two-cells of size $4$,$5$,$6$, etc).
To distinguish them we use the sequences  $25_1$ and $25_2$, where the subscript stands for the order of the automorphism group of the corresponding arrangement. 
The orders of the automorphism groups of the arrangements are reported at the bottom right of the arrangements in Fig.~\ref{fulllist}. 
\begin{figure}[!htb]
\def\factor{0.200019315015000023}
\centering
\psfrag{8}{\small 8} \psfrag{7}{\small 7} \psfrag{6}{\small 6} \psfrag{5}{\small 5} \psfrag{4}{\small 4} \psfrag{3}{\small 3} \psfrag{2}{\small 2}
\psfrag{8}{} \psfrag{7}{} \psfrag{6}{} \psfrag{5}{} \psfrag{4}{} \psfrag{3}{} \psfrag{2}{}
\psfrag{A}{$04\scriptstyle9000$}
\psfrag{B}{$07\scriptstyle3300$}
\psfrag{C}{$18\scriptstyle1030$}
\psfrag{D}{$25\scriptstyle2310$} 
\psfrag{F}{$07\scriptstyle3300$} 
\psfrag{G}{$37\scriptstyle0003$}
\psfrag{H}{$15\scriptstyle4300$}
\psfrag{I}{\textcolor{red}{$252310$}}
\psfrag{J}{$43\scriptstyle3021$}
\psfrag{K}{$25\scriptstyle2310$}
\psfrag{L}{$33\scriptstyle3310$}
\psfrag{M}{$32\scriptstyle6020$}
\psfrag{N}{$25\scriptstyle2310$}
\psfrag{Nstar}{$25\scriptstyle2310$}
\psfrag{O}{$32\scriptstyle6020$}
\psfrag{P}{$22\scriptstyle8010$}
\psfrag{Q}{$25\scriptstyle2310$}
\psfrag{R}{$36\scriptstyle0040$}
\psfrag{Z}{$64{\scriptstyle 00003}$}
\psfrag{i1}{$A$}
\psfrag{i2}{$B$}
\psfrag{i3}{$C$}
\psfrag{o24}{24}
\psfrag{o12}{12}
\psfrag{o2}{2}
\psfrag{o4}{4}
\psfrag{o6}{6}
\psfrag{o1}{1}
\psfrag{dummy}{}
\psfrag{A}{$\name{04\scriptstyle9000}$}
\psfrag{B}{$\name{07\scriptstyle3300}$}
\psfrag{C}{$\name{18\scriptstyle1030}$}
\psfrag{D}{$\name{25\scriptstyle2310}$} 
\psfrag{F}{$\name{07\scriptstyle3300}$} 
\psfrag{G}{$\name{37\scriptstyle0003}$}
\psfrag{H}{$\name{15\scriptstyle4300}$}
\psfrag{I}{\textcolor{red}{$\name{252310}$}}
\psfrag{J}{$\name{43\scriptstyle3021}$}
\psfrag{K}{$\name{25\scriptstyle2310}$}
\psfrag{L}{$\name{33\scriptstyle3310}$}
\psfrag{M}{$\name{32\scriptstyle6020}$}
\psfrag{N}{$\name{25\scriptstyle2310}$}
\psfrag{Nstar}{$\name{25\scriptstyle2310}$}
\psfrag{O}{$\name{32\scriptstyle6020}$}
\psfrag{P}{$\name{22\scriptstyle8010}$}
\psfrag{Q}{$\name{25\scriptstyle2310}$}
\psfrag{R}{$\name{36\scriptstyle0040}$}
\psfrag{Z}{$\name{64{\tiny\scriptstyle 00003}}$}

\psfrag{A}{$\name{04}$}
\psfrag{B}{$\name{07}$}
\psfrag{C}{$\name{18}$}
\psfrag{D}{$\name{25}$} 
\psfrag{F}{$\name{07}$} 
\psfrag{G}{$\name{37}$}
\psfrag{H}{$\name{15}$}
\psfrag{I}{\textcolor{red}{$\name{252310}$}}
\psfrag{J}{$\name{43}$}
\psfrag{K}{$\name{25}$}
\psfrag{L}{$\name{33}$}
\psfrag{M}{$\name{32}$}
\psfrag{N}{$\name{25_2}$}
\psfrag{Nstar}{$\name{25_1}$}
\psfrag{O}{$\name{32}$}
\psfrag{P}{$\name{22}$}
\psfrag{Q}{$\name{25}$}
\psfrag{R}{$\name{36}$}
\psfrag{Z}{$\name{64}$}

\includegraphics[width = \factor\linewidth]{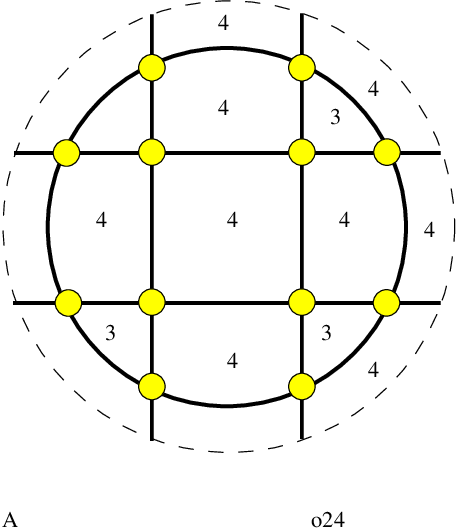}
\includegraphics[width = \factor\linewidth]{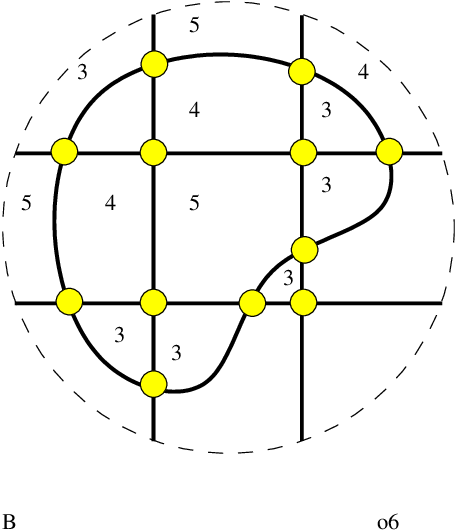}
\bigskip
\includegraphics[width = \factor\linewidth]{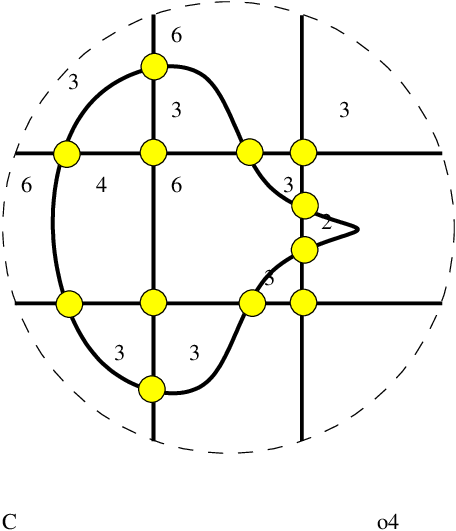}
\includegraphics[width = \factor\linewidth]{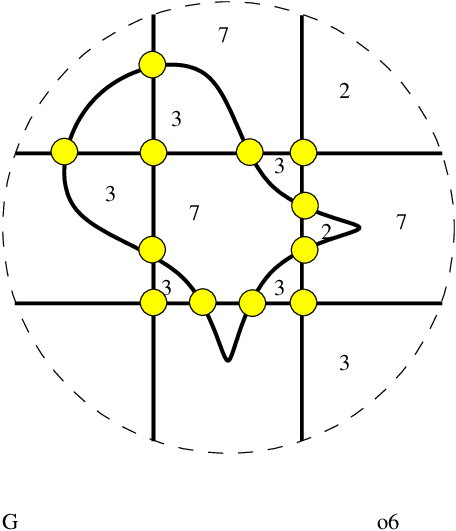}
\bigskip
\includegraphics[width = \factor\linewidth]{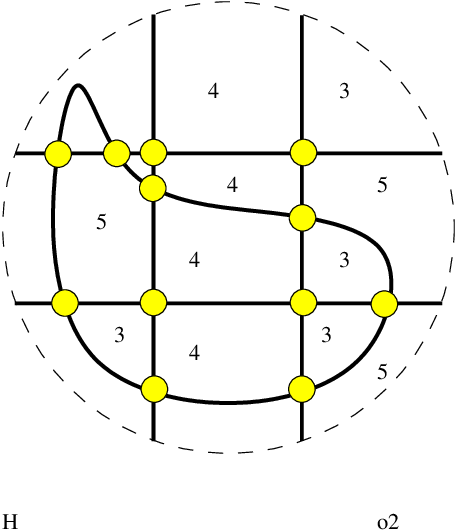}
\includegraphics[width = \factor\linewidth]{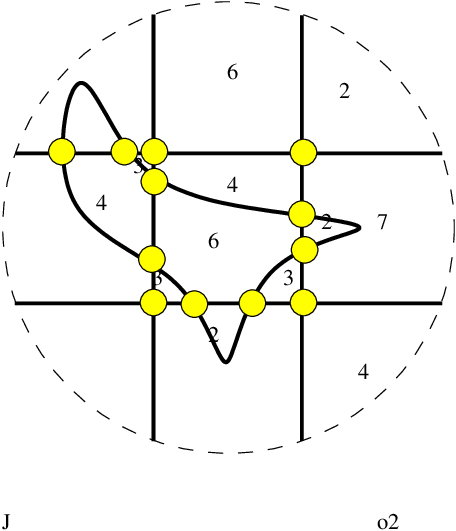}
\includegraphics[width = \factor\linewidth]{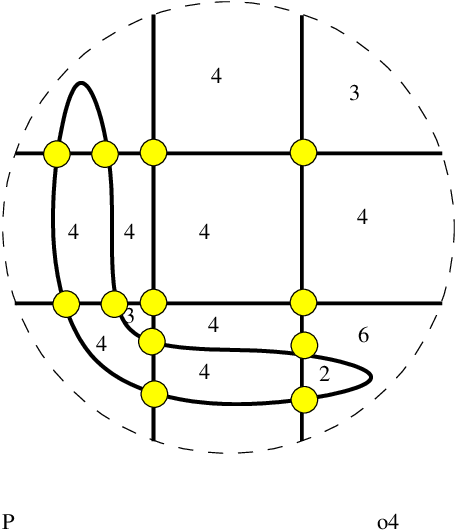}
\includegraphics[width = \factor\linewidth]{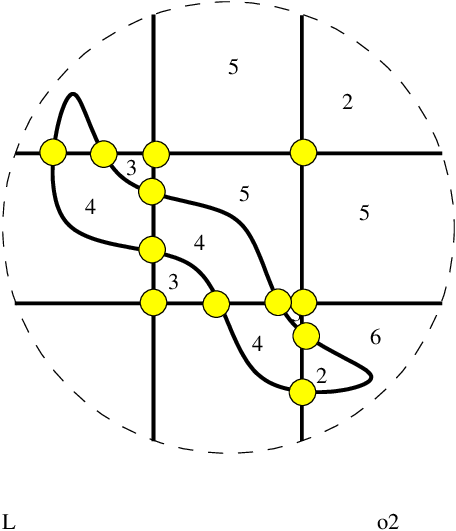}
\includegraphics[width = \factor\linewidth]{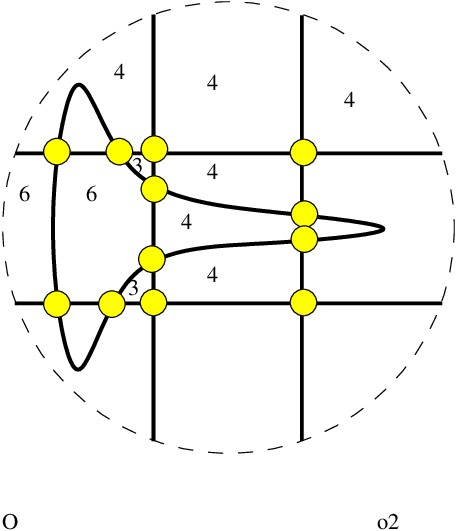}
\bigskip
\includegraphics[width = \factor\linewidth]{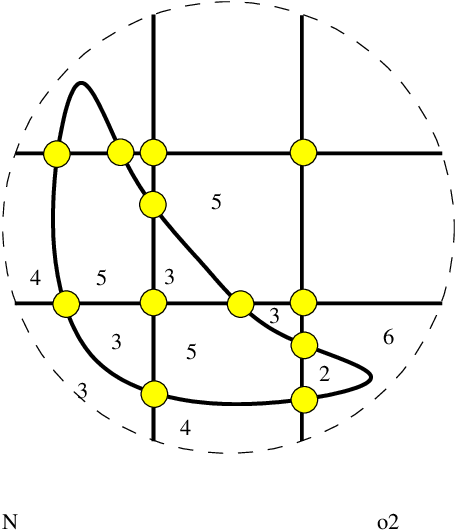}
\includegraphics[width = \factor\linewidth]{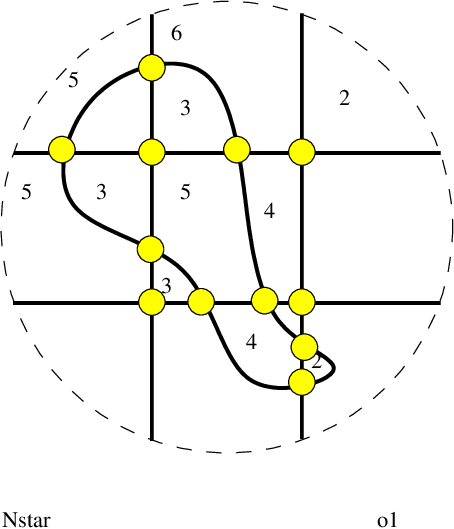}
\includegraphics[width = \factor\linewidth]{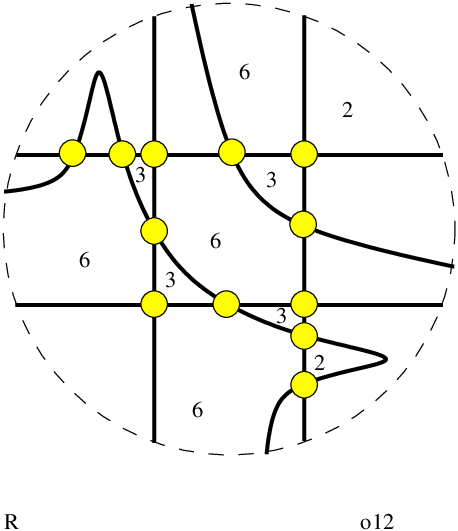}
\includegraphics[width = \factor\linewidth]{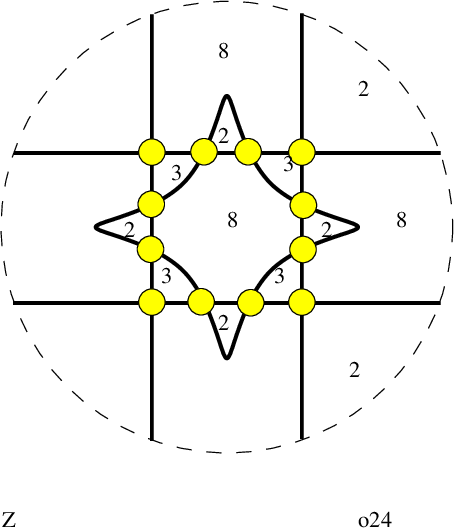}
\caption{Representatives of the thirteen isomorphism classes of simple arrangements of three double pseudolines.
In this figure each isomorphism class is labeled at its bottom left with a symbol to name the arrangement
and at its bottom right with the order of its automorphism group\label{fulllist}}
\end{figure}
Thus there are $13$ isomorphism classes of arrangements of three double pseudolines and 
$216$ isomorphism classes of indexed arrangements of three oriented double pseudolines on a given set of three indices (and not $214$ as indicated by error in~\cite{fpp-nsafd-11}). 
This latter number is computed as the sum 
$$\sum_{k\geq 1} \frac{3!2^3}{k} g_k  
$$
where $g_k$ is the number of arrangements with group of automorphisms  of order~$k$.  
For the number of isomorphism classes of arrangements of four double pseudolines and for the number of isomorphism classes of simple arrangements of five double pseudolines we refer to~\cite{fpp-nsafd-11}.

Using the exhaustive list of simple arrangements of three double pseudolines we now  compute the martagons on three and four double pseudolines.
Recall the definition of martagons. 
An arrangement of $n\geq 3$ double pseudolines $\Gamma$ is called  a {\it martagon with respect to} a double pseudoline $\gamma$ of $\Gamma$ if  
the vertices of the arrangement on the curve $\gamma$ are ordinary 
and if for any $\gamma' \in \Gamma$, $\gamma'\neq  \gamma$, no pair of distinct elements  $v,v'$  of  
\begin{equation}
\bigcup_{\gamma''\in \Gamma: \gamma'' \neq \gamma', \gamma}\gamma''\cap \gamma
\end{equation}
is  separated on the curve $\gamma$ by a pair of distinct elements $u,u'$ of $\gamma' \cap \gamma$;  in other words, the four intersection points of $\gamma'$ and $\gamma$ 
are ordinary and appear consecutively on the curve $\gamma$.
For example the arrangements $\name{22}$ and $\name{32}$ of Fig.~\ref{fulllist} are martagons with respect to the curved double pseudoline. 
Fig.~\ref{FinalMartagon} depicts examples of martagons 
\begin{figure}[!htb]
\centering
\psfrag{4}{} \psfrag{6}{} \psfrag{3}{} \psfrag{2}{}
\psfrag{O}{$326020$} \psfrag{P}{$228010$}
\psfrag{PNS1}{$\Gamma^{1}(X,Y,Z)$}
\psfrag{PNS10}{$\Gamma^{10}(X,Y,Z)$} \psfrag{PNS2}{$\Gamma^2(X,Y,Z)$}
\psfrag{PNS11}{$\Gamma^{11}(X,Y,Z)$}
\psfrag{U}{$\uu$} \psfrag{V}{$\vv$} \psfrag{W}{$\ww$} \psfrag{OO}{$\OO$}
\psfrag{U}{$X$} \psfrag{V}{$Y$} \psfrag{W}{$Z$}
\psfrag{one}{$\name{22}$}
\psfrag{two}{$\name{32}$}
\psfrag{thr}{$M_1$} 
\psfrag{fou}{$M_2$} 
\psfrag{oneor}{$4$}
\psfrag{twoor}{$2$}
\psfrag{thror}{$6$} 
\psfrag{fouor}{$2$} 
\includegraphics[width = 0.9999 \linewidth]{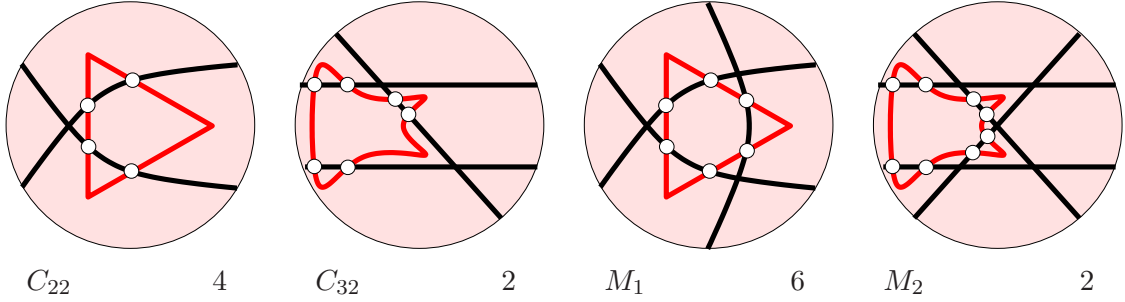}
\caption{Martagons with respect to the double pseudoline that do not intersect the dashed pseudoline, red in colored pdf, on three and four double pseudolines.
In this figure each double pseudoline whose crosscap side is free of vertices is  
simply represented by one of its core pseudolines
\label{FinalMartagon}}
\end{figure}
on three and four double pseudolines.
Observe that the subarrangements of size three of $M_1$ are $\name{22}$ (3 times) and $\name{04}$,
 and those of $M_2$ are $\name{22}$  and $\name{32}$, both  2 times.  The reader will have no difficulties adding to these examples martagons of arbitrary  size.
We leave the verification of the following lemma to the reader.
\begin{lemma} \label{martagonthr}
The only martagons on three and four double pseudolines are the arrangements of Fig.~\ref{FinalMartagon}. \qed
\end{lemma}

\section{Geometric representation theorem\label{secthr}}
In this section we prove the {\it Geometric Representation Theorem} for double pseudoline arrangements announced in the introduction: 
any arrangement of double pseudolines is isomorphic to the dual arrangement of a configuration of convex bodies. 
The main idea of the proof is to show that the property on the set of arrangements of double pseudolines of being {\it the dual arrangement of a  configuration of convex bodies}, 
is stable under mutations. 
The main ingredients of the proof are 
\begin{enumerate}
\item the connectedness of mutation graphs;
\item the coding of the isomorphism class of an indexed arrangement of oriented double pseudolines by its family of {\it node cycles}; 
\item the {\it raiponces}: we name thus the (appropriately) indexed arrangements of bitangents of 
indexed configurations of oriented convex bodies;
\item the existence of a projective plane extension for any arrangement of pseudolines~\cite{gpwz-atp-94}.
\end{enumerate}

\subsection{Nodes and node cycles of an arrangement}\label{sec:cycles}

Let $\Gamma$ be an indexed arrangement of oriented double pseudolines 
and let 
$v(\Gamma)$ be the indexed family of vertices of $\Gamma$ defined by the following three conditions:
\begin{enumerate}
\item the indexing set of $v(\Gamma)$ is the set of unordered pairs $ij (=ji)$ of signed indices of~$\Gamma$
 with the property that $i \neq \overline{j}$; 
\item the $v_{\alpha}(\Gamma)$, $\alpha \in \setnode{i}{j}$, are the four intersection
points of the double pseudolines $\Gamma_i$ and~$\Gamma_j$;
\item walking along the double pseudoline 
$\Gamma_i$ we encounter the $v_{\alpha}(\Gamma)$, $\alpha \in \setnode{i}{j}$, in cyclic order 
$v_{\nodeone{i}{j}}(\Gamma),v_{\nodetwo{i}{j}}(\Gamma),v_{\nodethr{i}{j}}(\Gamma),v_{\nodefou{i}{j}}(\Gamma)$, 
as illustrated in Fig.~\ref{CodingADP}a. 
 \end{enumerate}
\begin{figure}[!htb]
\psfrag{aa}{(a)}
\psfrag{bb}{(b)}
\psfrag{cc}{(c)}
\psfrag{firs}{$i$}
\psfrag{seco}{$j$}
\psfrag{one}{$\nodeone{i}{j}$}
\psfrag{two}{$\nodetwo{i}{j}$}
\psfrag{thr}{$\nodethr{i}{j}$}
\psfrag{fou}{$\nodefou{i}{j}$}
\psfrag{un}{1}\psfrag{deu}{2}\psfrag{tr}{3}
\psfrag{A}{$A$}\psfrag{B}{$B$}\psfrag{C}{$C$}\psfrag{D}{$D$}
\psfrag{A}{}\psfrag{B}{}\psfrag{C}{$$}\psfrag{D}{$$}
\psfrag{AA}{$\{\nodeone{1}{2},\nodefou{1}{3},\nodetwo{2}{3}\}$}
\psfrag{BB}{$\{\nodetwo{1}{2},\nodeone{1}{3},\nodefou{2}{3}\}$}
\psfrag{CC}{$\{\nodethr{1}{2},\nodetwo{1}{3},\nodethr{2}{3}\}$}
\psfrag{DD}{$\{\nodefou{1}{2},\nodethr{1}{3},\nodeone{2}{3}\}$}
\psfrag{AAA}{$\nodetwo{2}{3}$}
\psfrag{EE}{$\nodeone{1}{2}$}
\psfrag{FF}{$\nodefou{1}{3}$}
\psfrag{AA}{$\nodeone{1}{2},\nodefou{1}{3},\nodetwo{2}{3}$}
\psfrag{BB}{$\nodetwo{1}{2},\nodeone{1}{3},\nodefou{2}{3}$}
\psfrag{CC}{$\nodethr{1}{2},\nodetwo{1}{3},\nodethr{2}{3}$}
\psfrag{DD}{$\nodefou{1}{2},\nodethr{1}{3},\nodeone{2}{3}$}
\centering
\includegraphics[width=0.9975\linewidth]{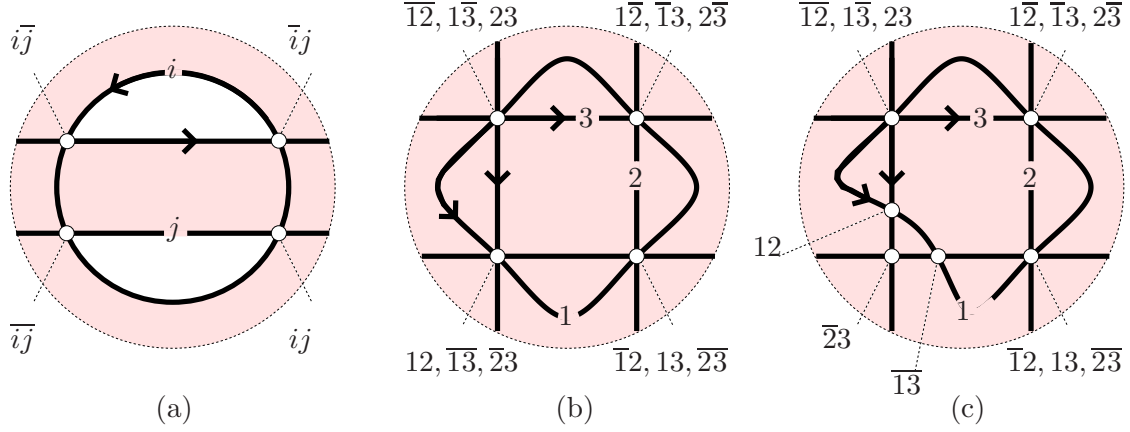}
\caption{Indexed families of vertices of indexed arrangements of two and three oriented double pseudolines \label{CodingADP}}
\end{figure}


The reader will easily check that the family $v(\Gamma)$ is well-defined.

The set of  {\it nodes} of $\Gamma$, denoted  $\QNodes(\Gamma)$, is the quotient of the indexing set of $v(\Gamma)$ under the relation ``to index the same vertex of $\Gamma$''
and the indexed family of {\it node cycles} of $\Gamma$, denoted $\CC(\Gamma)$,  is the indexed family 
of circular sequences of nodes of $\Gamma$ that correspond to the circular sequences of vertices of $\Gamma$ encountered when
walking along the double pseudolines of $\Gamma$, each circular sequence being indexed by the index of the double pseudoline on which is done the walk. 
Note that the cycles assigned to an index and its  complement are reverse to one another.
For example for the hemi-cube arrangement of  Fig.~\ref{CodingADP}b one has 
$\QNodes(\Gamma) = \{A,B,C,D\}$,
 $\CC_1(\Gamma)  =  ABCD$, $\CC_2(\Gamma)  =  ACBD$, and $\CC_3(\Gamma)  =  ABDC$ 
where 
$$
\begin{array}{llll}
A =  & \{\nodeone{1}{2},\nodefou{1}{3},\nodetwo{2}{3}\} & B =& \{\nodetwo{1}{2},\nodeone{1}{3},\nodefou{2}{3}\} \\
C = & \{\nodethr{1}{2},\nodetwo{1}{3},\nodethr{2}{3}\} & D =& \{\nodefou{1}{2},\nodethr{1}{3},\nodeone{2}{3}\}.
\end{array}
$$
Similarly for the arrangement of  Fig.~\ref{CodingADP}c, obtained from the hemi-cube arrangement of  Fig.~\ref{CodingADP}b by a splitting mutation,  one has 
$\QNodes(\Gamma) = \{A,B,C,D,E,F\}$,
 $\CC_1(\Gamma)  =  EFBCD$, $\CC_2(\Gamma)  =  EACBD$, and $\CC_3(\Gamma)  =  AFBDC$ 
where 
\begin{eqnarray*}
A & =& \{\nodetwo{2}{3}\}\\ 
E & =& \{\nodeone{1}{2}\}\\ 
F & =& \{\nodefou{1}{3}\}\\ 
B & =& \{\nodetwo{1}{2},\nodeone{1}{3},\nodefou{2}{3}\} \\
C & =& \{\nodethr{1}{2},\nodetwo{1}{3},\nodethr{2}{3}\}\\ 
D & =& \{\nodefou{1}{2},\nodethr{1}{3},\nodeone{2}{3}\}.
\end{eqnarray*}
The family $\CC(\Gamma)$ turns out to be a coding of the isomorphism class of $\Gamma.$

\begin{theorem}\label{coding}
Two indexed arrangements of oriented double pseudolines are isomorphic if and
only if  they
have the same indexed family of node cycles. \qed
\end{theorem}
\begin{proof}
Let $\Gamma$ be an indexed arrangement of oriented double pseudolines, 
let $\OldFlag{\Gamma}$ be the set of flags of the cell poset  $X(\Gamma)$ of $\Gamma$ 
and let $\Map{\oldflagop{i}(\Gamma)}{\OldFlag{\Gamma}}{\OldFlag{\Gamma}}$, $i\in \{0,1,2\}$, be its flag operators. 
The {\it node, index} and {\it side} of  a flag $F \in \OldFlag{\Gamma}$ are 
\begin{enumerate} 
\item the node of $\Gamma$ corresponding to the zero-cell of $F$;
\item the index of the supporting double pseudoline  of the one-cell of $F$ that is outgoing at the zero-cell of $F$; 
\item the symbol $\cside$ or its complement $\dside$
depending on whether the two-cell of $F$ is contained in the crosscap side of 
 the supporting double pseudoline of the one-cell of $F$ or is contained in its disk side.
\end{enumerate}
Fig.~\ref{flagcode} shows the first barycentric subdivision of an indexed arrangement of two oriented double pseudolines where each flag is labeled, using the obvious
convention,  with its node, index and side. 
\begin{figure}[!htb]
\psfrag{ii}{$\iAAA$}
\psfrag{jj}{$\jBBB$}
\psfrag{jb}{$\BBBb$}
\psfrag{ib}{$\AAAb$}
\psfrag{ij}{$\iAAA\jBBB$}
\psfrag{ibj}{$\AAAb\jBBB$}
\psfrag{ijb}{$\iAAA\BBBb$}
\psfrag{ibjb}{$\AAAb\BBBb$}
\psfrag{ds}{$\dside$}
\psfrag{cs}{$\cside$}
\centering
\includegraphics[width=0.5\linewidth]{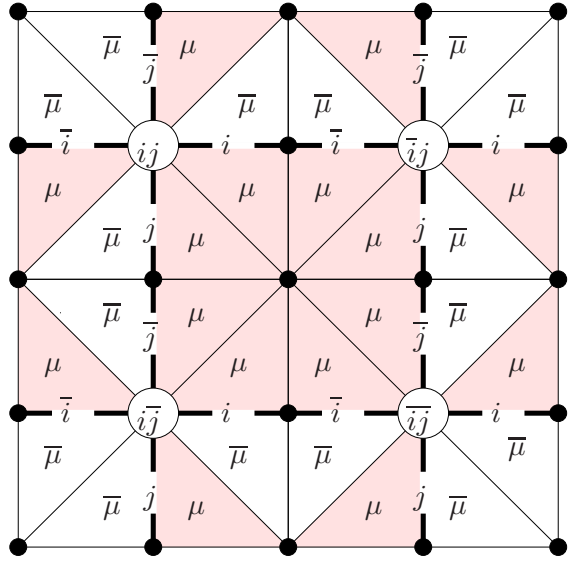}
\caption{\label{flagcode} The first barycentric subdivision  of an indexed arrangement of two oriented double pseudolines where each flag is labeled 
with its node, index and side}
\end{figure}
Let $I$ be the set of positive indices of $\Gamma$, let  $\OldFlagDoublure{\Gamma} = \{ (A, \nu, \eta) \mid \nu\in I \cup \overline{I}, A \in \CC_\nu(\Gamma), \eta \in \{\cside,\dside\}\}$,
 let $\Map{\omega(\Gamma)}{\OldFlag{\Gamma}}{\OldFlagDoublure{\Gamma}}$ be the (one-to-one and onto) map that assigns to the flag $F$ the triple 
composed of the node, index and side of $F$  and,  for $i\in \{0,1,2\}$, let  
 $$\oldflagdoublureop{i}(\Gamma) = \omega(\Gamma) \oldflagop{i}(\Gamma) \omega(\Gamma)^{-1}.$$ 
Table~\ref{sigmaoneoperator} gives the table of the operator $\oldflagdoublureop{1}(\Gamma)$ in the case where $\Gamma$ is an arrangement of two double pseudolines.

\begin{table}[!htb]
$$
\begin{array}{c||c|c|c|c}
F           
& \flagone{i}{j}{i}{\cside} 
& \flagone{i}{j}{i}{\dside}
& \flagone{i}{j}{\overline{i}}{\cside} 
& \flagone{i}{j}{\overline{i}}{\dside} 
\\
\oldflagdoublureop{1}(\Gamma)(F) 
& \flagone{i}{j}{j}{\cside} 
& \flagone{i}{j}{\overline{j}}{\cside} 
& \flagone{i}{j}{j}{\dside} 
& \flagone{i}{j}{\overline{j}}{\dside} 
\\
\end{array}
$$
\caption{\label{sigmaoneoperator} 
Table of the operator $\oldflagdoublureop{1}(\Gamma)$ in the case where $\Gamma$ is an indexed arrangement of two oriented double pseudolines with
 signed indexing set $\{i, \overline{i}, j, \overline{j}\}$}
\end{table}


Clearly two arrangements of oriented double pseudolines $\Gamma$ and $\Gamma'$ are isomorphic if
and only if for any $i \in \{0,1,2\}$ the operators $\oldflagdoublureop{i}(\Gamma)$ and $\oldflagdoublureop{i}(\Gamma')$ coincide. 
Therefore proving our theorem comes down to proving that the operators $\oldflagdoublureop{i}(\Gamma)$, $i \in \{0,1,2\}$, 
depend only on the indexed family $\CC(\Gamma).$ We define $\overline{\dside} = \cside$.
Clearly 
\begin{enumerate} 
\item $\oldflagdoublureop{2}(\Gamma) (A, \nu, \eta) = (A,  \nu, \overline{\eta})$;
\item $\oldflagdoublureop{0}(\Gamma) (A, \nu, \eta)=  (A', \overline{\nu}, \eta)$ 
where $A'$ is the successor of $A$ in the cycle $\CC_\nu(\Gamma)$. 
\end{enumerate}
Thus it remains to explain why $\oldflagdoublureop{1}(\Gamma)$ depends only on $\CC(\Gamma).$ (Actually it depends only on $\QNodes(\Gamma)$.)
For $J \subseteq I$ with at least two elements let  $\Gamma|J$ be the restriction of $\Gamma$ to $J$ and let $\Map{i_J}{\OldFlagDoublure{\Gamma|J}}{\OldFlagDoublure{\Gamma}}$
be the induced canonical injection (note that $i_J$ is the identity map on the two last coordinates).
For $F \in \OldFlagDoublure{\Gamma}$, let $U(F)$ be the set of $F_J = i_J \,\oldflagdoublureop{1}(\Gamma|J)(i_J)^{-1}(F)$ 
where $J$ ranges over the set of $2$-subsets of $I$  composed of the index  of $F$ and one of the indices occuring in its node,
 and endow $U(F)$ with the dominance relation $\prec_F$ defined 
by $F_J \prec_F F_K$ if  
$F_K = (\oldflagdoublureop{2}(\Gamma)(F_J))_{J\Delta K}$ where as usual $J\Delta K$ denotes the set symmetric difference operator. 
Clearly $\prec_F$ is a total order and $\oldflagdoublureop{1}(\Gamma)(F) = \min_{\prec_F} U(F).$
It follows that we can restrict our attention to the case where the size of the set of indices is two. 
The theorem follows.
\end{proof}

\begin{remark} In the preliminary versions~\cite{G-hp-htadp-06,G-hp-adp-09} of the paper we used the notations 
$v_{ij1}(\Gamma),v_{ij2}(\Gamma),v_{ij3}(\Gamma)$ and $v_{ij4}(\Gamma)$ for the vertices 
 $v_{\nodeone{i}{j}}(\Gamma),v_{\nodetwo{i}{j}}(\Gamma),v_{\nodethr{i}{j}}(\Gamma)$ and $v_{\nodefou{i}{j}}(\Gamma)$ of the arrangement $\Gamma$. 
The new notations are better in that they are compatible with the operation of changing sign.  
\end{remark}

\subsection{Raiponces} \label{raiponces}
Recall that a {\it cyclic} arrangement of pseudolines is a simple arrangement of pseudolines with the property that the maximum of its two-cell sizes is 
its number of pseudolines. The simple arrangements of size at most $5$ are cyclic.
Fig.~\ref{cyclic} shows cyclic arrangements of three, four, five and six pseudolines.
\begin{figure}[!htb]
\psfrag{A}{$A$} \psfrag{B}{$B$} \psfrag{C}{$C$} \psfrag{D}{$D$}
\psfrag{firs}{$i$} \psfrag{seco}{$j$} \psfrag{firscd}{$\nabla_{ij}(L)$} \psfrag{secocd}{$\nabla_{ji}(L)$}
\psfrag{onecd}{$\nabla_{1}(L^0)$} \psfrag{twocd}{$\nabla_{2}(L^0)$} \psfrag{thrcd}{$\nabla_{3}(L^0)$}
\psfrag{one}{$\renewmeet{i}{j}{1}$} \psfrag{two}{$\renewmeet{i}{j}{2}$} \psfrag{thr}{$\renewmeet{i}{j}{3}$} \psfrag{fou}{$\renewmeet{i}{j}{4}$}
\psfrag{un}{1} \psfrag{deu}{2} \psfrag{tr}{3}
\centering
\includegraphics[width = 0.98750 \linewidth]{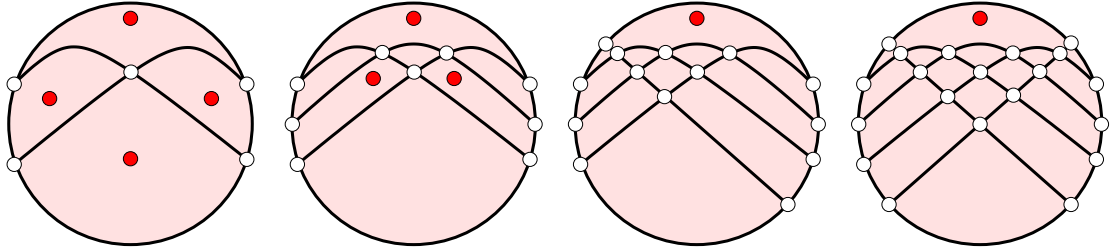}
\caption{Cyclic arrangements of 3, 4, 5 and 6 pseudolines  \label{cyclic}
 with their central cells marked with a little black bullet (red in  colored pdf)}
\end{figure}
The isomorphism class of a cyclic arrangement of pseudolines depends only of its number of pseudolines;  
in particular the number of two-cells realizing the maximum of their sizes is $2$, $4$, $3$ or $1$ depending on whether the number of pseudolines 
of the arrangement is $2$, $3$, $4$, or larger than $4$. 
A two-cell realizing the maximum of the two-cell sizes of a cyclic arrangement of pseudolines is called a {\it central cell} of the arrangement.
We isolate a simple lemma that will be repeatedly used in the sequel.
\begin{lemma}\label{relaxation} 
Let $L$ be a cyclic arrangement of $n\geq 3$ pseudolines, let $\nabla$ be a central cell of $L$ and let $L_1,L_2,\ldots,L_n$ be the circular list of pseudolines of $L$ 
encountered when walking along the boundary of $\nabla$. Let $K$ be a pseudoline such that (1) $K$ is tangent to $\nabla$ at the intersection point of $L_1$ and $L_2$  and (2) 
the family  $ L ' = L \setminus \{L_1,L_2\} \cup \{K\}$ 
is an arrangement of pseudolines. 
Then  (1) $L'$ is cyclic and (2) $\nabla$ is contained in a central cell $\nabla'$ of $L'$ such that walking along the boundary of $\nabla'$ we encounter the pseudolines of $L'$
in the circular order $K,L_3,L_4,\ldots,L_n$.  \qed
\end{lemma}

We are now ready to define the raiponces.

A {\it raiponce $L$ on a finite set of indices $I$} is a simple indexed arrangement of pseudolines such that 
\begin{enumerate}
\item the indexing set of $L$ is the set of unordered pairs  $ij (=ji)$ of signed indices of $I$ with the property that $i\neq \overline{j}$;
\item for any $i\in I$ and any $j\in I$, $i\neq j$, the subarrangement of $L$ whose pseudolines are the $L_{\alpha}$, $\alpha \in \setnode{i}{j}$,  
is an arrangement of four pseudolines;  we denote by $\nabla_{i,j}$ 
its unique oriented two-cell such that walking along its boundary we encounter the pseudolines $L_{\alpha}$, $\alpha \in \setnode{i}{j}$, 
in the circular order $L_{\vijone},L_{\vijtwo},L_{\vijthr}, L_{\vijfou}$, as illustrated in Fig.~\ref{gather}a;
\begin{figure}[!htb]
\psfrag{A}{$A$} \psfrag{B}{$B$} \psfrag{C}{$C$} \psfrag{D}{$D$}\psfrag{E}{$E$} \psfrag{F}{$F$}
\psfrag{firs}{$i$} \psfrag{seco}{$j$} \psfrag{firscd}{$\nabla_{i,j}$} \psfrag{secocd}{$\nabla_{j,i}$}
\psfrag{onecd}{$\nabla_{1}(L)$} \psfrag{twocd}{$\nabla_{2}(L)$} \psfrag{thrcd}{\small $\nabla_{3}(L)$} 
\psfrag{onecd}{$\nabla_{1}$} \psfrag{twocd}{$\nabla_{2}$} \psfrag{thrcd}{$\nabla_{3}$} 
\psfrag{one}{$L_{\nodeone{i}{j}}$}
\psfrag{two}{$L_{\nodetwo{i}{j}}$}
\psfrag{thr}{$L_{\nodethr{i}{j}}$} 
\psfrag{fou}{$L_{\nodefou{i}{j}}$} 
\psfrag{un}{1} \psfrag{deu}{2} \psfrag{tr}{3}
\psfrag{aa}{(a)}
\psfrag{bb}{(b)}
\psfrag{cc}{(c)}
\centering
\includegraphics[width = 0.99750 \linewidth]{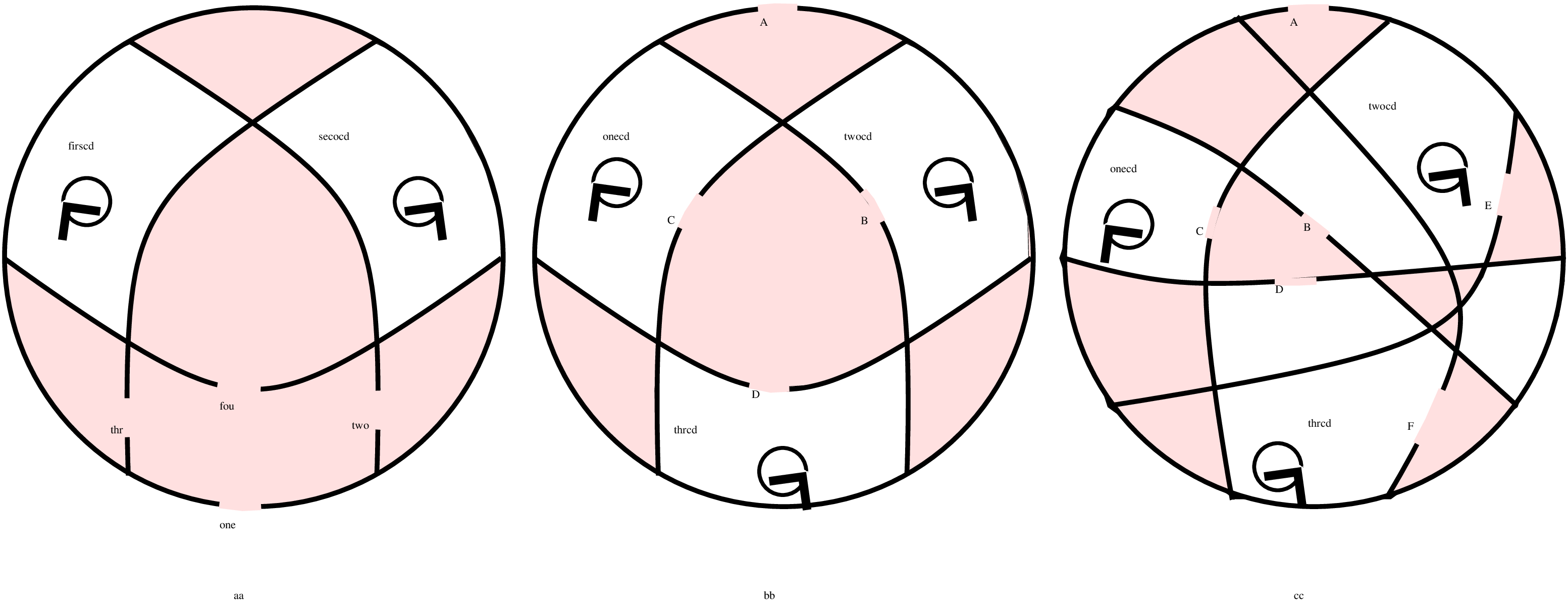}
\caption{
(a) A  raiponce on the indexing set $\{i,j\}$;  
(b) a raiponce on the indexing set $\{1,2,3\}$ composed of $4$ pseudolines: 
$A=L_{\nodeone{1}{2}}=L_{\nodefou{1}{3}}=L_{\nodetwo{2}{3}}$,
$B=L_{\nodetwo{1}{2}}=L_{\nodeone{1}{3}}=L_{\nodefou{2}{3}}$,
$C=L_{\nodethr{1}{2}}=L_{\nodetwo{1}{3}}=L_{\nodethr{2}{3}}$, 
$D=L_{\nodefou{1}{2}}=L_{\nodethr{1}{3}}=L_{\nodeone{2}{3}}$.
(c) a raiponce on the indexing set $\{1,2,3\}$ composed of $6$ pseudolines: 
$A=L_{\nodetwo{2}{3}}$,
$E=L_{\nodeone{1}{2}}$,
$F=L_{\nodefou{1}{3}}$,
$B=L_{\nodetwo{1}{2}}=L_{\nodeone{1}{3}}=L_{\nodefou{2}{3}}$,
$C=L_{\nodethr{1}{2}}=L_{\nodetwo{1}{3}}=L_{\nodethr{2}{3}}$, 
$D=L_{\nodefou{1}{2}}=L_{\nodethr{1}{3}}=L_{\nodeone{2}{3}}$
 \label{gather}} 
\end{figure}
note that $\nabla_{i,j}$ and $\nabla_{j,i}$ are by construction disjoint and that their closures share two vertices but no edge;
\item for any $i\in I$ the subarrangement of $L$ whose pseudolines are the $L_{\alpha}$, $\alpha\in \setnode{i}{j}$, $j\in I\setminus i$, is cyclic and  walking along the boundary of 
one of its oriented central cells we encounter for any $j \in I\setminus i$ the pseudolines $L_{\alpha}$, $\alpha \in \setnode{i}{j}$, 
in the circular order $L_{\vijone},L_{\vijtwo},L_{\vijthr}, L_{\vijfou}$; this oriented central cell, denoted $\nabla_i$,
 is necessarily the intersection of the $\nabla_{i,j}$, $j \in I\setminus i$. The indexed family of $\nabla_i$  is called the family of {\it central cells} of $L$.
\end{enumerate}
Let $L$ be a raiponce, let $\QNodes(L)$ be the quotient of the indexing set of $L$  under the relation ``to index the same pseudoline of $L$''
and let $\CC(L)$ be the indexed family  
of circular sequences of elements of $\QNodes(L)$ encountered when
walking along the (oriented) boundaries of the central cells of $L$, each sequence being indexed with the index of the central cell 
on the boundary of which is done the walk.
An element of  $\CC(L)$ will be called a {\it cycle} of $L$.
The reader will easily check that the families of cycles of the raiponces of Fig.~\ref{gather}a,~\ref{gather}b and~\ref{gather}c  coincide  with the families of cycles of the 
indexed arrangements of oriented double pseudolines of  Fig.~\ref{CodingADP}a,~\ref{CodingADP}b and~\ref{CodingADP}c.

Now let $\Delta$ be an indexed  configuration of oriented convex bodies with the property that its arrangement 
of bitangents is simple and let $\Delta^*$  be its dual (indexed and oriented) arrangement.
 Clearly the indexed family $v(\Delta^*)$ of vertices of $\Delta^*$---see  Section \ref{sec:cycles} for its definition---is a raiponce on the indexing set of $\Delta$, 
called the {\it raiponce of $\Delta$} thereafter.  
The following lemma claims that any raiponce is the raiponce of an indexed configuration of oriented convex bodies and that the map that assigns to an indexed configuration of convex bodies the isomorphism class of its dual arrangement can be factorized 
through the map that assigns to an indexed configuration of oriented convex bodies its raiponce.
The proof is easy.

\begin{lemma}
Let $L$ be  a raiponce on the indexing set $I$, let $\nabla$ be its indexed family of central cells,  let $\cal G$ be a projective plane extension of $L$, and let
${\cal R}(L, {\cal G})$ be the class of indexed configurations of oriented  convex bodies $\Delta$  of $\cal G$ with indexing set $I$ 
 such that for any $i\in I$
\begin{enumerate} 
\item  $\Delta_i$ is inscribed in the central cell $\nabla_i$, and 
\item the orientations of $\Delta_i$ and $\nabla_i$ are coherent.
\end{enumerate}
Then
\begin{enumerate}
\item ${\cal R} (L, {\cal G})$ is nonempty, and
\item for any $\Delta \in {\cal R}(L, {\cal G})$, the raiponce of $\Delta$ is $L$ and the isomorphism class of its dual arrangement $\Delta^*$ depends only on $L.$ 
\end{enumerate}
\end{lemma}
\begin{proof} The first point is clear since by construction the closures of the $\nabla_i$ intersect pairwise in at most two vertices. 
Similarly the second point is clear since by construction ${\cal V}(\Delta^* ) = {\cal V}(L)$ and  $\CC(\Delta^*) = \CC(L)$.  
\end{proof}

A {\it completion} of a raiponce $L$ is an indexed configuration of oriented convex bodies whose raiponce is $L$,
and a {\it primal representation} of an indexed arrangement of oriented double pseudolines $\Gamma$ is a raiponce $L$ with the property that the isomorphism
class of the dual arrangements of its completions is  the isomorphism class of $\Gamma$.
For example the raiponces of Fig.~\ref{gather}a,~\ref{gather}b and ~\ref{gather}c are  primal representations of the indexed arrangements of oriented 
double pseudolines of Fig.~\ref{CodingADP}a,~\ref{CodingADP}b and~\ref{CodingADP}c, respectively.
According to the previous discussion the properties `to be the dual arrangement of 
a family of pairwise disjoint convex bodies' and `to have a primal representation' are equivalent.
The next step is devoted to the proof that this last property is stable under mutations. 

\begin{remark} The dual arrangement of the family of central cells of a primal representation   
of an arrangement of double pseudolines  is, up to homeomorphism,  obtained from the arrangement of double pseudolines  by shrinking its digons into edges. 
(Here the duality is defined with respect to any projective plane extension of the primal representation.) 
\end{remark}
\subsection{Stability under  mutations}
\begin{theorem} \label{hgr}
Let $\Gamma$ and $\Gamma'$ be two indexed arrangements of oriented double pseudolines related by a mutation. Then 
$\Gamma$ has a primal representation if and only if $\Gamma'$ has a primal representation.  \qed
 \end{theorem}

Before embarking on the proof we isolate a simple property of
primal representations. The proof is  easy.

\begin{lemma}\label{dicoone} 
Let $L$ be a primal representation of an indexed arrangement of oriented double pseudolines $\Gamma$, let $\nabla$ be its 
indexed family of central cells, let $\sigma$ be a one-cell of $\Gamma$ supported by the curve $\Gamma_i$, 
let $v_\alpha$ and $v_\beta$ be endpoints of $\sigma$.
Then $L_\alpha$ and $L_\beta$ are consecutive pseudolines of the boundary of $\nabla_i$ and for any index $j \neq i$ of the indexing set of $\Gamma$ one has
\begin{enumerate} 
\item $\sigma$ is contained in  the crosscap side of $\Gamma_j$ if and only if 
the arrangement of pseudolines  $L_\alpha$, $L_\beta$, 
$L_{\nodeone{i}{j}}$, 
$L_{\nodetwo{i}{j}}$, 
$L_{\nodethr{i}{j}}$, and  
$L_{\nodefou{i}{j}}$ is the one depicted in 
Fig.~\ref{basicdicoone}a;
\item $\sigma$ is contained in the disk side of $\Gamma_i$ if and only if 
the arrangement of pseudolines  $L_\alpha$, $L_\beta$, 
$L_{\nodeone{i}{j}}$, 
$L_{\nodetwo{i}{j}}$, 
$L_{\nodethr{i}{j}}$, and  
$L_{\nodefou{i}{j}}$ is the one depicted in 
Fig.~\ref{basicdicoone}b. \qed
\end{enumerate}
\end{lemma}
\begin{figure}[!htb]
\centering

\psfrag{one}{${\nodeone{i}{j}}$}
\psfrag{two}{${\nodetwo{i}{j}}$}
\psfrag{thr}{${\nodethr{i}{j}}$} 
\psfrag{fou}{${\nodefou{i}{j}}$} 
\psfrag{xx}{$\alpha$} 
\psfrag{yy}{$\beta$} 

\psfrag{firs}{$i$} 
\psfrag{seco}{$j$} 
\psfrag{firscd}{$\nabla_{i,j}$}
\psfrag{secocd}{$\nabla_{j,i}$}
\psfrag{aa}{(a)} \psfrag{bb}{(b)}
\psfrag{cc}{(c)} \psfrag{dd}{(d)}
\psfrag{ee}{(e)} \psfrag{ff}{(f)}
\psfrag{gg}{(g)}
\includegraphics[width=0.750\linewidth]{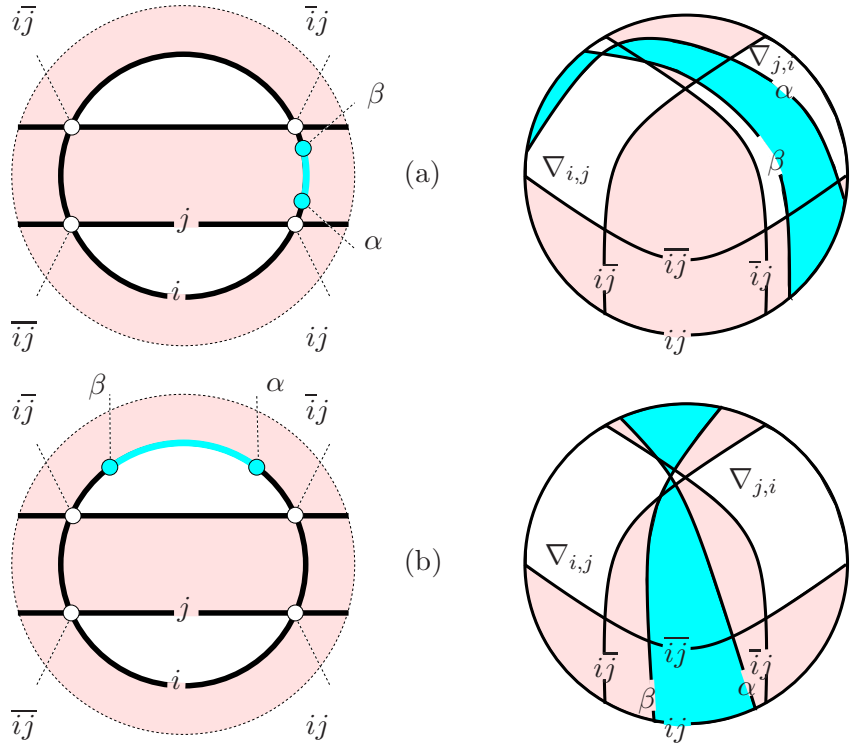}
\caption{\label{basicdicoone}
Dictionnary between 
the relative positions of an edge  $\sigma$ supported by the double pseudoline $\Gamma_j$ and a double pseudoline $\Gamma_i$ of an indexed arrangement $\Gamma$ of oriented double pseudolines 
and the relative positions of the corresponding central cells $\nabla_i$ and $\nabla_j$ of a primal representation of $\Gamma$ 
}
\end{figure}
\begin{proof}[Proof of Theorem~\ref{hgr}]
Let $L$ be a primal representation of an arrangement of oriented double pseudolines $\Gamma$, and consider a mutation connecting $\Gamma$ to an arrangement $\Gamma'$.
Our goal is to show that $\Gamma'$ has a primal representation~$L'$. 
Without loss of generality one can assume that $\Gamma$ is the dual arrangement of a completion $\Delta$ of~$L.$

We first examine the case of a merging mutation. 

Let $\Sigma$ be the  complex of adjacent triangular two-cells of $\Gamma$ involved in the merging
mutation and let $\lift{\Sigma}$ be one of its two lifts in a two-covering of the underlying  cross surface. 
We consider the set of vertices of $\lift{\Sigma}$ as an arrangement
$\Psi$ of oriented pseudolines and we introduce 
the subarrangement $\Psi_0$ composed of the three vertices of the boundary $\partial \lift{\Sigma}$ of $\lift{\Sigma}$ and 
the  one level $\ell$ of $\Psi_0$  with respect to its unique two-cell $\sigma_0$ with cyclic boundary;
note that $\ell$ is by construction a pseudoline and that any pseudoline in $L$ not in  $\Psi$ crosses $\ell$ in at most three
points. Fig.~\ref{dualityone} depicts the complex $\Sigma$ of a merging mutation, the subarrangement $\Psi_0$ with its cyclic two-cell $\sigma_0$ marked, and the one-level $\ell$
of $\Psi_0$ with respect to $\sigma_0.$  
\begin{figure}[!htb]
\centering
\psfrag{sz}{$\sigma_0$}
\psfrag{ell}{$\ell$}
\psfrag{w}{\tiny moving curve}
\psfrag{aa}{(a)}
\psfrag{bb}{(b)}
\psfrag{cc}{(c)}
\includegraphics[width=0.9875\linewidth]{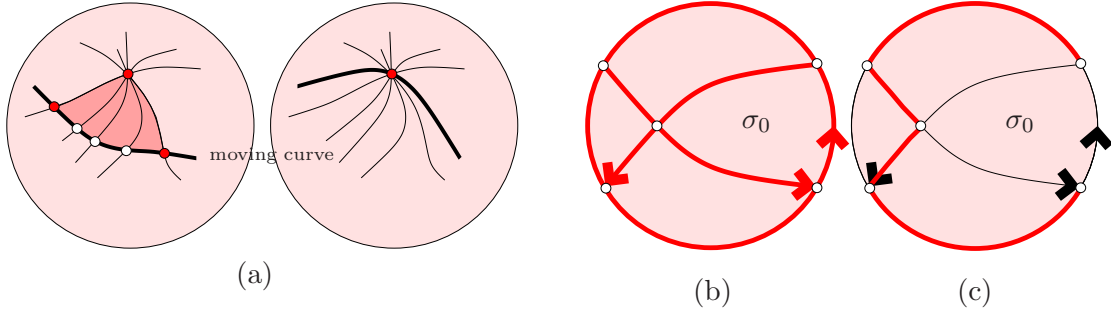}
\caption{ (a)
The complex $\Sigma$ of triangular two-cells involved in the merging mutation connecting $\Gamma$ to $\Gamma'$; 
(b)  the arrangement $\Psi_0$ composed of the vertices of the boundary of the complex $\lift{\Sigma}$;
 and (c) the one-level $\ell$ of the arrangement $\Psi_0$ \label{dualityone}}
\end{figure}

Let $L'$ be the indexed family of pseudolines defined by 
\begin{equation}
L'_{\tau} = \begin{cases}
         \ell & \text{ if $L_{\tau}$ is a vertex of $\Sigma$};\\
         L_{\tau} &  \text{otherwise},
\end{cases}
\end{equation}
where $\tau$ ranges over the indexing set of $L$.
\begin{lemma} We claim that 
\begin{enumerate}
\item $L'$ is a simple arrangement of pseudolines;
\item $L'$ is a primal representation of the arrangement $\Gamma'.$ 
\end{enumerate}
\end{lemma}
\begin{proof} 
Let $K$  be the set of indices of the supporting double pseudolines of the one-cells of $\Sigma$,  
let $K'\subseteq K$ be the set of three indices of the three supporting double pseudolines of the three sides of the boundary of $\Sigma$, 
and let $w \in K'$ be the index of the moving curve of the mutation. 
We denote by $\opposite{w}$ the vertex of the boundary of $\Sigma$ opposite the
side supported by $\Gamma_w$, and for any $t \in K\setminus \{w\}$ we denote by $\opposite{t}$ the vertex of $\Sigma$ where the double pseudolines $\Gamma_t$ and $\Gamma_w$ intersect. 
\begin{figure}[!htb]
\centering
\psfrag{caseone}{$\bigcup \Sigma \subset \TD(\Gamma_t)$}
\psfrag{casetwo}{$\bigcup \Sigma \subset \MS(\Gamma_t)$}
\psfrag{nat}{$\nabla_t$}
\psfrag{itemone}{$t \in I\setminus K$}

\psfrag{ot}{$\ot$}
\psfrag{ow}{$\wo$}
\psfrag{itemtwo}{$t\in K'\setminus\{w\}$}

\psfrag{caseonw}{$\bigcup \Sigma \subset \TD(\Gamma_w)$}
\psfrag{casetww}{$\bigcup \Sigma \subset \MS(\Gamma_w)$}

\psfrag{itemthr}{$t\in K\setminus K'$}

\psfrag{ou1}{$\opposite{u}_1$} \psfrag{ou2}{$\opposite{u}_2$} \psfrag{ou3}{$\opposite{u}_3$} \psfrag{ou4}{$\opposite{u}_4$} \psfrag{ou5}{$\opposite{u}_m$}
\psfrag{naw}{$\nabla_w$}
\psfrag{G0}{$\nabla_w$}
\psfrag{G0s}{$\nabla^*_w$}

\psfrag{itemfou}{}
\psfrag{itemfiv}{$t=w$}
\includegraphics[width=0.9850\linewidth]{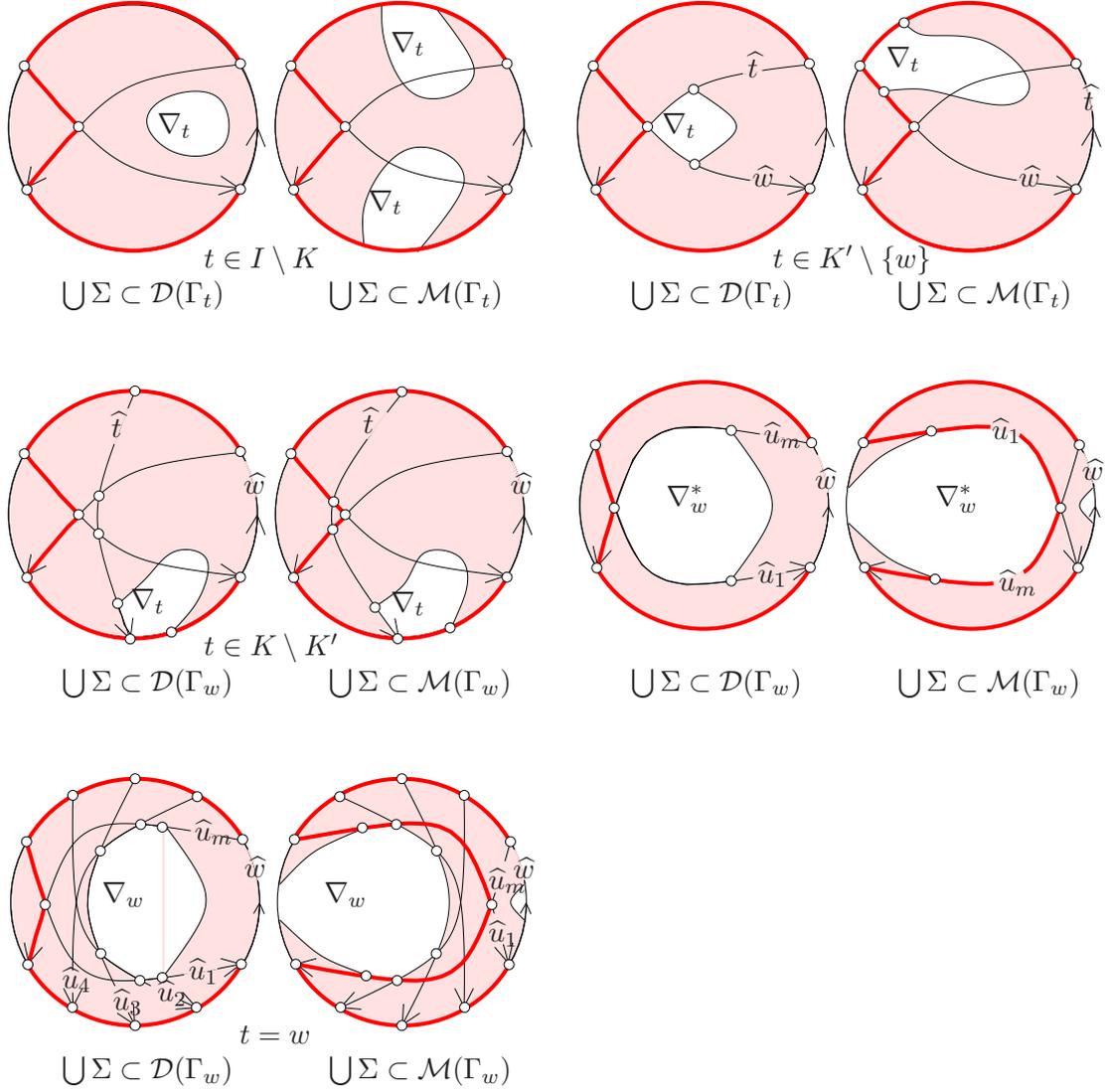}
\caption{Relative position of $\nabla_t$ in the arrangement~${\Psi_0}^+$  
\label{relativepositionglobal}}
\end{figure}
Let $\Psi_0^+$ be the arrangement $\Psi_0$ augmented with the line $\opposite{t}$ if $t\in K\setminus K'$; the arrangement $\Psi$ if $t=w$; the arrangement $\Psi_0$ otherwise.
We denote by $L^*$ the sub-raiponce of $L$ obtained by deleting the $L_{ij}$ with $i \in  K \setminus K'$.  The indexed families of centrals cells of $L$ and $L^*$ 
are denoted $\nabla$ and $\nabla^*$, respectively.
Finally let $\uo_1,\uo_2,\ldots,\uo_m$ be the sequence of vertices $\neq \wo$ of $\Sigma$ ordered along $\Gamma_w.$

Applying Lemma~\ref{dicoone} to the one-cells of $\lift{\Sigma}$ (and using induction on the size of $K\setminus K'$) we see easily that 
the relative position of $\nabla_t$ in the arrangement $\Psi_0^+$ depends only on whether the triangular two-cells of $\Sigma$ 
are contained in the disk side  $\TD(\Gamma_t)$ of $\Gamma_t$ or contained in the crosscap side $\MS(\Gamma_t)$ 
of $\Gamma_{t}$, as indicated in Fig.~\ref{relativepositionglobal}.
Furthermore one can also check that 
\begin{enumerate}
\item for any $t \in K\setminus \{w\}$ the pseudoline $\ell$ is tangent to $\nabla_t$ at the intersection point of $\opposite{w}$ and $\opposite{t}$;
\item the pseudolines in $\Psi\setminus \Psi_0$ cross the pseudoline $\ell$  all in three points or all in one point; 
\item the arrangement $\Psi$ is cyclic;
\item the relative position of $\nabla^*_{w}$ 
in the arrangement $\Psi_0$ depends only on whether the triangular two-cells of $\Sigma$ are contained in $\TD(\Gamma_w)$ or in $\MS(\Gamma_{w})$
as indicated in Fig.~\ref{relativepositionglobal};
in particular we note that $\ell$ is tangent to $\nabla^*_w$ at the intersection
point of $\opposite{u}_1$ and $\opposite{u}_m$.
\end{enumerate}

Pick now a pseudoline $\ell'$ such that  $L \cup \{\ell'\}$ is a simple  arrangement of pseudolines. 
Assume that $\ell'$ and $\ell$ cross three times. Clearly $\ell'$ avoids the cyclic two-cell of $\Psi_0$
and consequently---thanks to our  previous discussion on the position of the $\nabla_t$ in the arrangement $\Psi_0^+$---$\ell'$ is transversal to any $\nabla_t$, 
$t\in I\setminus (K\setminus K')$, not contained in $\sigma_0$. 
It follows that $\ell' \notin L\setminus \Psi$ and, consequently, there is no pseudoline of $L\setminus \Psi$ crossing $\ell$ three times: thus $L'$ is a simple arrangement of pseudolines.

We now prove that $L'$ is a raiponce and a primal representation of $\Gamma'$.  Given a subfamily $S$ of $L$ we define $S'$ to be the  corresponding subfamily of~$L'$. 
For any index  $i\in I$ let $M_i$ be the arrangement of pseudolines composed of the $L_{\alpha}$, $\alpha \in \setnode{i}{j}$, $j \in I\setminus \{i\}$, 
and  let  $N_{ij} = M_i \cap M_j$, $i\neq j \in I.$ Observe that $N_{ij}$
contains at most one element of $\Sigma$ and that  $\ell \notin L$: consequently $N'_{ij}$ is an arrangement of four pseudolines. 
By construction 
\begin{equation}
M'_{t} = \begin{cases}
           M_t \setminus \{\wo,\opposite{t}\}\cup \{\ell\} & \text{if $t \in K\setminus \{w\}$};\\
           M_t \setminus \Sigma\cup \{\ell\} & \text{if $t =w$};\\
           M_t &  \text{otherwise}.
\end{cases}
\end{equation}
Since for any $i \in K \setminus \{w\}$ the pseudoline $\ell$ is tangent to $\nabla_{i}$ at the intersection point of 
$\wo$ and $\ot$, and since 
$\ell$ is tangent to  $\nabla^*_{w}$ at the intersection point of $\uo_1$ and
$\uo_m$
it follows, according to Lemma~\ref{relaxation}, 
that for any $i$ the arrangement $M'_i$ is cyclic, that $\nabla_i$ is contained in one of its central two-cells $\nabla'_i$, and 
that walking along its boundary (oriented according to the orientation of $\nabla_i$) we encounter for any $j \in I\setminus i$
the pseudolines $L'_{\alpha}$, $\alpha \in \setnode{i}{j}$, in the circular order $L'_{\vijone}, L'_{\vijtwo}, L'_{\vijthr}, L'_{\vijfou}$; consequently 
$L'$ is a raiponce  and is (by construction) a primal representation of $\Gamma'.$
\end{proof}
\begin{figure}[!htb]
\centering
\psfrag{sz}{$\sigma_0$}
\psfrag{ell}{$\ell$}
\psfrag{w}{\tiny moving curve}
\psfrag{dt}{$\wo$}
\psfrag{dtstar}{$w^*$}
\psfrag{aone}{$A_1$}
\psfrag{atwo}{$A_2$}
\psfrag{athr}{$A_3$}
\psfrag{afou}{$A_4$}
\psfrag{bone}{$B_1$}
\psfrag{btwo}{$B_2$}
\psfrag{bi}{$B_i$}
\psfrag{bj}{$B_j$}
\psfrag{bfou}{$B_4$}
\psfrag{S}{$S$}
\psfrag{Q}{$Q$}
\psfrag{Qp}{$Q'$}
\psfrag{caseone}{$\wo \in {\cal D}(\Gamma'_w)$}
\psfrag{casetwo}{$\wo \in {\cal M}(\Gamma'_w)$}
\psfrag{aa}{(a)}
\psfrag{bb}{(b)}
\psfrag{cc}{(c)}
\includegraphics[width=0.875\linewidth]{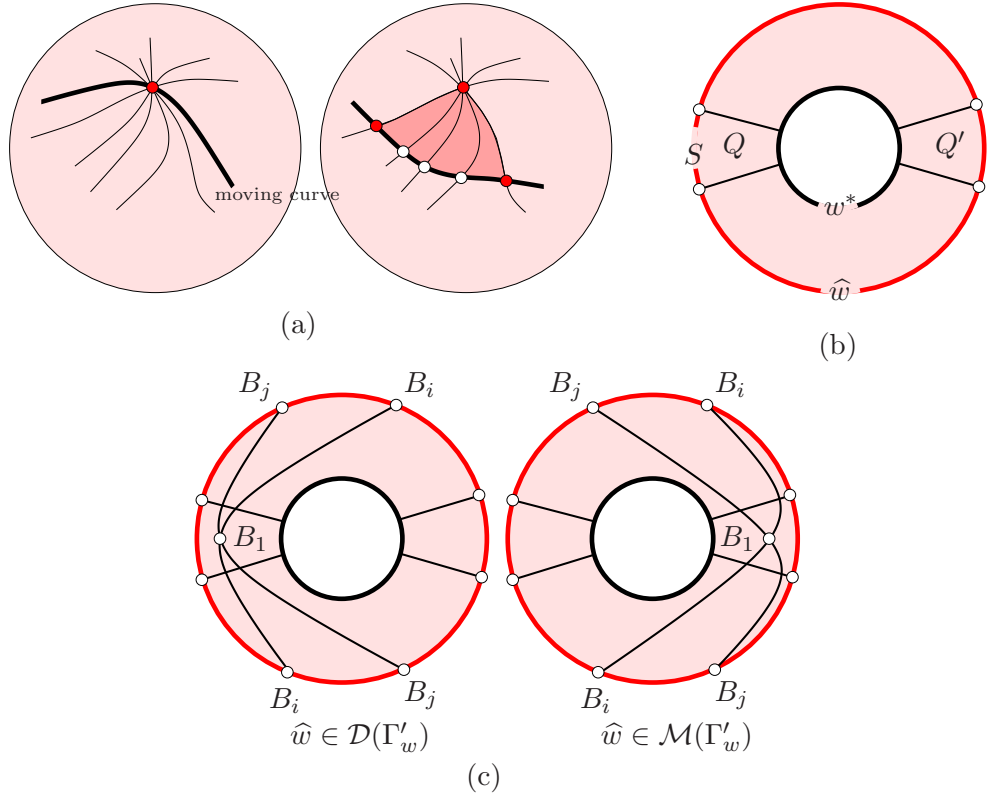}
\caption{Stability under splitting mutations
 \label{dualitytwo}}
\end{figure}
We now examine the case of a splitting mutation.

Let $K$ be the set of indices of the double pseudolines involved in the splitting mutation, let $w$ be the index of the moving double pseudoline,
 let $\wo$ be the vertex involved in the mutation, and for any $v \in K\setminus\{w\}$ let $x(v) \in \setnode{w}{v}$ defined by the condition $L_{x(v)} = \wo.$ 

Let $w^*$ be a double pseudoline containing $\wo$ in its crosscap side such
that any pseudoline of $L$ crosses $w^*$ in exactly two points and such  that no vertex
of the arrangement $L$ belongs to the M{\"o}bius strip $\MS(w^*)$ bounded by $w^*.$  
The pseudolines of $L$ induce a decomposition of $\MS(w^*)$ into
quadrilateral regions.
In particular the trace of the central cell of the raiponce $L$ indexed by $w$ onto $\MS(w^*)$ is one of its 
quadrilateral regions that we shall denote by $Q$. We denote by $S$ and $S^*$ the sides of $Q$
supported by $\wo$ and $w^*$, respectively, and  we denote by $Q'$ the second quadrilateral region of $\MS(w^*)$ bounded by $S$.

Let $B_1$ be a generic point of $Q$ if $\wo \in {\cal D}(\Gamma'_{w})$; otherwise let $B_1$ be a generic point of $Q'.$

For any $i \in K\setminus \{w\}$ we insert a generic point $B_i$ on the interior of the edge of the central cell of $L$ indexed by $i$ supported by  $\wo$, 
and we insert in the underlying pseudoline arrangement of $L$ a pseudoline $\ell_i$ such that 
\begin{enumerate} 
\item $\ell_i$ goes through the points $B_1$ and $B_i$,  and is contained in  $\MS(w^*)$;
\item the vertices of $L \cup \Psi$ are simple,  except $B_1$;
\end{enumerate}
and we perturb the pencil of pseudolines $\ell_i$ in the vicinity of $B_1$ into
a cyclic arrangement ${\ell_i^*}$ with a central cell containing $S^*$ or $S$
depending on whether $\wo \in {\cal D}(\Gamma'_{w})$ or not.

Now let $L'$ be the indexed family of pseudolines defined by 
$$
L'_{\tau} = \begin{cases}
\ell^*_v & \text{if $\tau = x(v)$ with $v\in K\setminus \{w\}$;} \\
L_{\tau} & \text{otherwise,}
\end{cases}
$$ 
where $\tau$ ranges over the indexing set of $L$.
A simple case analysis shows that $L'$ is a well-defined raiponce and is a primal
representation of $\Gamma'.$  Details are left to the reader.
\end{proof}

\begin{remark} Our proof of the Geometric Representation Theorem is constructive. For an alternative construction see~\cite{fpp-nsafd-11}.
\end{remark}


\section{Cycles, cocycles and chirotopes} \label{secfou}
In this section we prove that 
(1) the isomorphism class of an indexed arrangement of oriented double pseudolines
depends only on its family of isomorphism classes of subarrangements of size three, i.e., depends only on what we have called its chirotope;
 and that 
(2) the map that assigns to an indexed configuration of oriented convex bodies the isomorphism class of its dual arrangement is compatible with the isomorphism relations on the set of configurations of convex bodies, and it 
induces a one-to-one correspondence between the set of isomorphism classes of 
indexed configurations of oriented convex bodies and the set of isomorphism classes 
of indexed arrangements of oriented  double pseudolines.  
The main ingredients of our proof are 
\begin{enumerate}
\item the coding of the isomorphism class of an indexed arrangement of oriented double pseudolines by its family of side cycles;
\item the list of martagons on three and four double pseudolines, established in Section~\ref{sec:homotopy};
\item the injectivity of the map that assigns to each cell of the dual arrangement of an indexed configuration of two oriented convex bodies 
the cocycle of the configuration at some (hence any) element of the cell; and 
\item the injectivity of the map that assigns to a bitangent cocycle of an indexed family of at least three oriented convex bodies 
the  sub-cocycles obtained by removing in turn  each of the convex bodies. 
 \end{enumerate}

\subsection{Side cycles} We repeat the definition of side cycles given in the introduction. 
Let $\Gamma$ be an indexed arrangement of oriented double pseudolines and recall that $\Gamma$ is extended to the negative indices by assigning to a negative index 
the reoriented version of the oriented double pseudoline assigned to its positive version.
The {\it side cycle of disk type} assigned to the (signed) indice $i$, denoted~$D_i$, is the circular sequence of indices 
of the double pseudolines crossed by the side wheel of a sidecar rolling on $\Gamma_i$, side wheel on the disk side of $\Gamma_i$, 
that are (locally) oriented away from~$\Gamma_i$.  
Similarly  the {\it side cycle of crosscap type} assigned to the index $i$, denoted $M_i$,  is the circular sequence of indices of the 
double pseudolines crossed by the side wheel of a sidecar rolling on $\Gamma_i$, side wheel on the crosscap side of $\Gamma_i$,
that are (locally) oriented away from~$\Gamma_i$.  
Note that the side cycles of disk (crosscap)  type assigned to an index and its complement are reverse to one another 
and that for simple  arrangements the side cycle of disk type assigned to an index is the complement of its side cycle of crosscap type and vice versa.
 \begin{example} 
The side cycles of disk type of an arrangement $\Gamma$ on two double pseudolines, say indexed by $i,j$, are 
$$
\begin{array}{cl}
i: & \jbar{j}\jbar{j}\jind{j}\jind{j}\\
j: & \jbar{i}\,\jbar{i}\,\jind{i}\,\jind{i}.
\end{array}
$$
This can be easily read in Fig.~\ref{CodingADPBIS} where we have displayed the first barycentric subdivision of the one-skeleton of the arrangement and  
labeled each edge of the subdivision with the index of the  supporting double pseudoline of the edge that is, 
locally on the edge, oriented away from the vertex of the arrangement to which the edge
is incident.     
Observe that each symbol in these cycles corresponds in the natural way to a unique node of the arrangement, namely 
the linear sequence of symbols  $\jbar{j}\jbar{j}\jind{j}\jind{j}$ corresponds to the linear sequence of nodes $\{\nodeone{i}{j}\}\{\nodetwo{i}{j}\}\{\nodethr{i}{j}\}\{\nodefou{i}{j}\}$, as illustrated in 
Fig.~\ref{CodingADPBIS}. The side cycles of crosscap type of $\Gamma$ coincide with its side cycles of disk type but now the linear sequence of symbols  $\jbar{j}\jbar{j}\jind{j}\jind{j}$ 
corresponds to the linear sequence of nodes $\{\nodethr{i}{j}\}\{\nodefou{i}{j}\}\{\nodeone{i}{j}\}\{\nodetwo{i}{j}\}$.  
\end{example}
\begin{figure}[!htb]
\psfrag{fir}{$\jind{i}$}
\psfrag{firb}{$\jbar{i}$}
\psfrag{sec}{$\jind{j}$}
\psfrag{secb}{$\jbar{j}$}
\psfrag{one}{$\nodeone{i}{j}$}
\psfrag{two}{$\nodetwo{i}{j}$}
\psfrag{thr}{$\nodethr{i}{j}$}
\psfrag{fou}{$\nodefou{i}{j}$}
\centering
\includegraphics[width=0.43750000000875\linewidth]{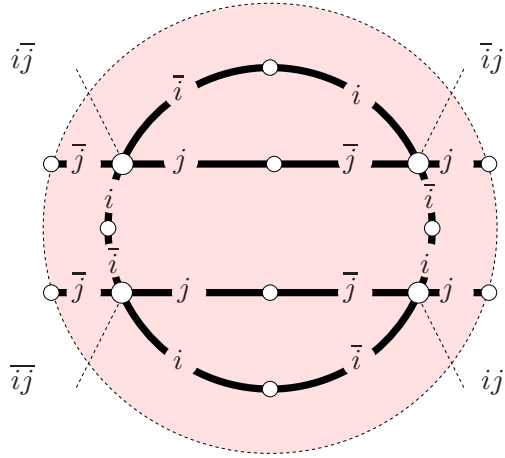}
\caption{The first barycentric subdivision of the one-skeleton of an arrangement of two double pseudolines: each edge of the subdivision is labeled 
with the signed index of its signed supporting curve that is,
locally on the edge, oriented away from the vertex of the arrangement to which the edge
is incident \label{CodingADPBIS}}
\end{figure}
Let $\Gamma$ be an indexed arrangement of oriented double pseudolines with indexing set~$I$,
let $D_i$ be its side cycles of disk type and $M_i$ those of crosscap type. 
Let $S_i$ be the result of replacing in $D_i$ the linear subsequences $\jbar{j}\jbar{j}\jind{j}\jind{j}$, $j \neq i$,  
by the linear sequences  $\{\nodeone{i}{j}\}\{\nodetwo{i}{j}\}\{\nodethr{i}{j}\}\{\nodefou{i}{j}\}$; similarly let  $T_i$ 
be the result of replacing in $M_i$ the linear subsequences $\jind{j}\jind{j}\jbar{j}\jbar{j}$, $j\neq i$, by the linear sequences
 $\{\nodeone{i}{j}\}\{\nodetwo{i}{j}\}\{\nodethr{i}{j}\}\{\nodefou{i}{j}\}$.
Clearly there is a one-to-one correspondence between the vertices of the arrangement lying on the curve indexed by $i$ and the maximal 
 factors  $\{i_1j_1\}\{i_2j_2\}\{i_3j_3\}\ldots \{i_kj_k\}$ of $S_i$ with $j_l \notin\{j_{l '}, \overline{j}_{l'}\}$ for all $1\leq l<l'\leq k$  
that appear in reverse order $\{i_kj_k\}\ldots\{i_3j_3\}\{i_2j_2\}\{i_1j_1\}$  in $T_i$, {\it prime factors} for short.
More precisely: 
\begin{enumerate}
\item the node of vertex $v$ associated with the  prime 
factor $\{i_1j_1\}\{i_2j_2\}\{i_3j_3\}\ldots \{i_kj_k\}$ of $S_i$ is the set of 
$\{i_{l}j_{l}\} \otimes \{i_{l'}j_{l'}\}$, $1\leq l \leq l' \leq k$,
 where $\{i_lj_l\}\otimes  \{i_{l'}j_{l'} \}$ is the element of the $4$-set $\setnode{j_l}{j_{l'}}$ indexing the intersection point of $\Gamma_{j_{l}}$ and $\Gamma_{j_{l'}}$ 
that coincides with~$v$.  As illustrated in Fig.~\ref{finalbouquet} (which depicts implicitly the $4\times 4$ possible nodes involving three curves indexed by $i,j$ and $j_*$ 
where $i\in I$,$j,j_*\in I\cup \overline{I}$ and where the dashed sides stand for the crosscap sides) this
\begin{figure}[!htb]
\psfrag{i}{$\jind{i}$}
\psfrag{j}{$\jind{j}$}
\psfrag{k}{$\jind{j_*}$}
\psfrag{ib}{$\jbar{i}$}
\psfrag{jb}{$\jbar{j}$}
\psfrag{kb}{$\jbar{j_*}$}
\psfrag{A}{$\{\node{i}{j}\}\{\node{i}{j_*}\}$}
\psfrag{B}{$\{\node{i}{j}\}\{\node{\jbar{i}}{j_*}\}$}
\psfrag{E}{$\{\node{\jbar{i}}{j}\}\{\node{i}{j_*}\}$}
\psfrag{F}{$\{\node{\jbar{i}}{j}\}\{\node{\jbar{i}}{j_*}\}$}
\centering
\includegraphics[width=0.987500\linewidth]{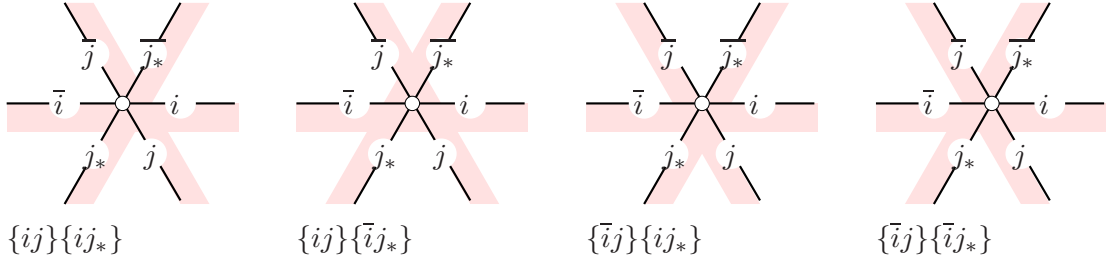}
\caption{\label{finalbouquet} Implicit description of the $4\times 4$ possible nodes involving three curves indexed by $i,j$ and ${j}_{*}$
where $i\in I$, $j,j_*\in I \cup \overline{I}$. The dashed sides stand for the crosscap sides}
\end{figure}
 element
depends solely on the information contained in the (ordered) pair $\{i_lj_l\}\{i_{l'}j_{l'}\}$,
 and the multiplication table of $\otimes$ is the following 
$$
\left\{
\begin{array}{ccc}
\{\node{i}{j}{}\} \otimes \{\node{i}{j_{\phantom{*}}}{}\} & = & \{\node{j}{i_{\phantom{*}}}{} \}\\
\{\node{i}{j}{}\} \otimes \{\node{i}{j_*}{}\} & = & \{\node{j}{\jbar{j_*}}{} \} \\
\{\node{i}{j}{}\} \otimes \{\node{\jbar{i}}{j_*}{} \} & = & \{\node{\jbar{j}}{\jbar{j_*}}{} \}\\
\{\node{\jbar{i}}{j}{}\} \otimes \{\node{i}{j_*}{} \} & = &\{\node{j}{j_*}{} \} \\
\{\node{\jbar{i}}{j}{}\} \otimes \{\node{\jbar{i}}{j_*}{} \} & = & \{\node{\jbar{j}}{j_*}{} \} 
\end{array}\right.
$$
where $i\in I$, $j,j_*\in I \cup \overline{I}$ with $j\notin \{j_*,\overline{j}_*\}$; 
\item Conversely,
the prime factor of $S_i$ corresponding to vertex $v$ is 
the sequence of $\{\eta j\}$, $\eta \in  \{i,\overline{i}\}$, that belong to the node $N$ of $v$, ordered according to the dominance relation  
$\{\eta j\} \prec \{\eta'j'\}$ if $\{\eta j\} \otimes \{\eta'j'\} \in N$. 
\end{enumerate}
Thus node cycles and side cycles carry exactly the same information about the arrangement. 
Since, according to Theorem~\ref{coding},  two indexed and oriented arrangements are isomorphic if and only if they have the same family of node cycles we get  
\begin{theorem}\label{cycle}
Two indexed arrangements of oriented double pseudolines are isomorphic if and only if they have the same family of side cycles.  \qed
\end{theorem}


Let $X$ be an arrangement of double pseudolines, let $X_*$ be an indexed and oriented version of $X$ and
recall that we have extended $X_*$ to the complements of the original indices by assigning to a negative index the reoriented version of the double pseudoline assigned to its complement. 
Let $G$ be the group of permutations of the signed indices which are compatible with the operation of taking the complement,  
 let $G_{X_*}$ 
be the stabilizer of $X_*$, i.e., the subgroup of $G$ whose elements are the permutations $\sigma$ such that $X_*\sigma$ and $X_*$ are isomorphic,  
and let $G_X$ be the group of automorphisms of  $X$, which  we think of as a subgroup of the group of permutations of the signed (or oriented) double pseudolines of the arrangement.
 Clearly the map that assigns to $\sigma \in G_X$ its conjugate $ X_*^{-1}\sigma X_* \in G$
under $X_*$ is a monomorphism of $G_X$ onto  $G_{X_*}$.  Thus we can see $G_X$ as a subgroup $G_{X_*}$ of $G$ and the number of distinct
indexed and oriented versions of $X$ is the index $[G:G_{X^*}]$ of $G_{X_*}$ in $G$.
In the sequel we use the notation $X(\sigma)$ for the arrangement $X_*\sigma$, $\sigma \in G$; hence $X(1) = X_*$, where $1$ is the unit of $G$.

\begin{example} Let $Z$ be the hemi-cube arrangement and let $Z_*= Z(123)$ be one of its indexed and oriented version on the indexing set $\{1,2,3\}$; cf.  Fig.~\ref{finalgroup}.
The group $G_{Z}$ is, as explained in Section~\ref{secone}, $S_4$.
Thus the number of distinct indexed and oriented versions of
$Z$ is $3!2^3/24 = 2$.
The group $G_{Z_*}$ is of order $24$ generated by the permutations
$\jbar{1}\jind{3}\jind{2}$ and
$\jbar{1}\jbar{2}\jind{3}$
and $\jind{2}\jbar{3}\jbar{1}$
(which correspond to the automorphisms $\tau_{1}$, $\tau_{2}$ and $\tau_{3}$ of Section~\ref{secone}), respectively.
Its two cosets are 
$$
G_{Z_*} = 
\left\{
\begin{array}{ccc}
123& 231& 312 \\
\overline{1}\overline{2}3&\overline{2}\overline{3}1&\overline{3}\overline{1}2\\
\overline{1}2\overline{3}&\overline{2}3\overline{1}&\overline{3}1\overline{2}\\
1\overline{2}\overline{3}&2\overline{3}\overline{1}&3\overline{1}\overline{2}\\
21\overline{3}&32\overline{1}&13\overline{2}\\
3\overline{2}1&1\overline{3}2&2\overline{1}3\\
\overline{1}32&\overline{2}13&\overline{3}21\\
\overline{2}\overline{1}\overline{3}& \overline{3}\overline{2}\overline{1}&\overline{1}\overline{3}\overline{2}
 \end{array}
\right\}, \qquad
(213) G_{Z_*} = \left\{
\begin{array}{ccc}
213 & 321 & 132\\
\overline{2}\overline{1}3&\overline{3}\overline{2}1&\overline{1}\overline{3}2\\
\overline{2}1\overline{3}&\overline{3}2\overline{1}&\overline{1}3\overline{2}\\
2\overline{1}\overline{3}&3\overline{2}\overline{1}&1\overline{3}\overline{2}\\
12\overline{3}&23\overline{1}&31\overline{2}\\
3\overline{1}2&1\overline{2}3&2\overline{3}1\\
\overline{2}31&\overline{3}12&\overline{1}23\\
\overline{1}\overline{2}\overline{3}&\overline{2}\overline{3}\overline{1}& \overline{3}\overline{1}\overline{2}
 \end{array}
\right\}.
$$
\begin{figure}[!htb]
\psfrag{un}{$1$}
\psfrag{deu}{$2$}
\psfrag{tr}{$3$}
\psfrag{oga}{$24$}
\psfrag{nrr}{$2$}
\psfrag{name}{$Z$}
\psfrag{namebis}{$Z(123)$}
\psfrag{nameter}{$Z(\overline{1}23)$}
\psfrag{cycles}{$
\begin{array}{cl}
\jind{2}\jbar{2}\jind{3}\jbar{3}\jbar{2}\jind{2}\jbar{3}\jind{3}\\ 
\jind{1}\jbar{1}\jbar{3}\jind{3}\jbar{1}\jind{1}\jind{3}\jbar{3}\\ 
\jbar{1}\jind{1}\jbar{2}\jind{2}\jind{1}\jbar{1}\jind{2}\jbar{2} 
\end{array}
$}
\psfrag{cyclesbis}{$
\begin{array}{cl}
\jind{2}\jbar{2}\jbar{3}\jind{3}\jbar{2}\jind{2}\jind{3}\jbar{3}\\ 
\jind{1}\jbar{1}\jbar{3}\jind{3}\jbar{1}\jind{1}\jind{3}\jbar{3}\\ 
\jbar{1}\jind{1}\jbar{2}\jind{2}\jind{1}\jbar{1}\jind{2}\jbar{2} 
\end{array}
$}
\psfrag{cycles}{}
\psfrag{cyclesbis}{}
\centering
\includegraphics[width=0.875\linewidth]{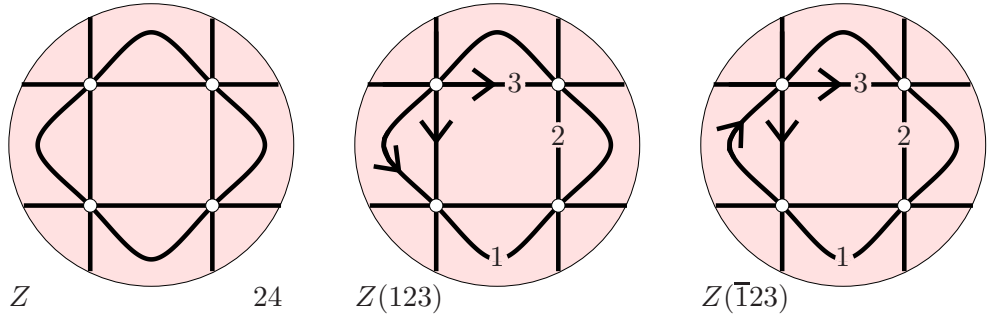}
\caption{The hemi-cube arrangement and two of its indexed and oriented versions on the indexing set $\{1,2,3\}$\label{finalgroup}}
\end{figure}
\end{example}

\begin{example}\label{examplemartagon}
As illustrated in Fig.~\ref{FinalMartagonTer}  
the two families of cycles on the indexing set $\{1,2,3,4\}$ 
$$
\begin{array}{ccl}
\AAA: & & 
\thrcrossR{\AAA}{\BBB} \foucrossR{\AAA}{\BBB} \onecrossR{\AAA}{\BBB} \twocrossR{\AAA}{\BBB}
\thrcrossR{\AAA}{\CCC} \foucrossR{\AAA}{\CCC} \onecrossR{\AAA}{\CCC} \twocrossR{\AAA}{\CCC}
\thrcrossR{\AAA}{\DDD} \foucrossR{\AAA}{\DDD} \onecrossR{\AAA}{\DDD} \twocrossR{\AAA}{\DDD}
\\
\BBB: &  &
\thrcrossR{\BBB}{\DDD} \foucrossR{\BBB}{\DDD} \foucrossR{\BBB}{\AAA} \onecrossR{\BBB}{\AAA}
\onecrossR{\BBB}{\CCC} \twocrossR{\BBB}{\CCC} \onecrossR{\BBB}{\DDD} \twocrossR{\BBB}{\DDD}
\twocrossR{\BBB}{\AAA} \thrcrossR{\BBB}{\AAA} \thrcrossR{\BBB}{\CCC} \foucrossR{\BBB}{\CCC} \\
\CCC: &  & 
\thrcrossR{\CCC}{\BBB} \foucrossR{\CCC}{\BBB} \foucrossR{\CCC}{\AAA} \onecrossR{\CCC}{\AAA}
\onecrossR{\CCC}{\DDD} \twocrossR{\CCC}{\DDD} \onecrossR{\CCC}{\BBB} \twocrossR{\CCC}{\BBB}
\twocrossR{\CCC}{\AAA} \thrcrossR{\CCC}{\AAA} \thrcrossR{\CCC}{\DDD} \foucrossR{\CCC}{\DDD} \\
\DDD: &  & 
\thrcrossR{\DDD}{\CCC} \foucrossR{\DDD}{\CCC} \foucrossR{\DDD}{\AAA} \onecrossR{\DDD}{\AAA}
\onecrossR{\DDD}{\BBB} \twocrossR{\DDD}{\BBB} \onecrossR{\DDD}{\CCC} \twocrossR{\DDD}{\CCC}
\twocrossR{\DDD}{\AAA} \thrcrossR{\DDD}{\AAA} \thrcrossR{\DDD}{\BBB} \foucrossR{\DDD}{\BBB}
\end{array}
\qquad \text{and} \qquad
\begin{array}{ccl}
\AAA: & &
\onecrossR{\AAA}{\BBB} \twocrossR{\AAA}{\BBB} \thrcrossR{\AAA}{\BBB} \foucrossR{\AAA}{\BBB}
\onecrossR{\AAA}{\CCC} \twocrossR{\DDD}{\CCC} \thrcrossR{\AAA}{\CCC} \foucrossR{\AAA}{\CCC}
\twocrossR{\AAA}{\DDD} \thrcrossR{\AAA}{\DDD} \foucrossR{\AAA}{\DDD} \onecrossR{\AAA}{\DDD}
\\
\BBB: &  &
\onecrossR{\BBB}{\CCC} \twocrossR{\BBB}{\CCC} \foucrossR{\BBB}{\AAA} \onecrossR{\BBB}{\AAA}
\twocrossR{\BBB}{\DDD} \thrcrossR{\BBB}{\DDD} \thrcrossR{\BBB}{\CCC} \foucrossR{\BBB}{\CCC}
\twocrossR{\BBB}{\AAA} \thrcrossR{\BBB}{\AAA} \foucrossR{\BBB}{\DDD} \onecrossR{\BBB}{\DDD}
\\
\CCC: &  & 
\onecrossR{\CCC}{\BBB} \twocrossR{\CCC}{\BBB} \foucrossR{\CCC}{\DDD} \onecrossR{\CCC}{\DDD}
\foucrossR{\CCC}{\AAA} \onecrossR{\CCC}{\AAA} \thrcrossR{\CCC}{\BBB} \foucrossR{\CCC}{\BBB}
\twocrossR{\CCC}{\DDD} \thrcrossR{\CCC}{\DDD} \twocrossR{\CCC}{\AAA} \thrcrossR{\CCC}{\AAA}
\\
\DDD: &  & 
\onecrossR{\DDD}{\BBB} \twocrossR{\DDD}{\BBB} \foucrossR{\DDD}{\AAA} \onecrossR{\DDD}{\AAA}
\thrcrossR{\DDD}{\BBB} \foucrossR{\DDD}{\BBB} \thrcrossR{\DDD}{\CCC} \foucrossR{\DDD}{\CCC}
\twocrossR{\DDD}{\AAA} \thrcrossR{\DDD}{\AAA} \onecrossR{\DDD}{\CCC} \twocrossR{\DDD}{\CCC}
\end{array}
$$
are the side cycles of disk type of indexed and oriented versions of the two 
martagons $M_1$ and $M_2$ on four double pseudolines, depicted in Fig.~\ref{FinalMartagon}. 
\begin{figure}[!htb]
\centering
\psfrag{4}{} \psfrag{6}{} \psfrag{3}{} \psfrag{2}{}
\psfrag{O}{$326020$} \psfrag{P}{$228010$}
\psfrag{PNS1}{$\Gamma^{1}(X,Y,Z)$}
\psfrag{PNS10}{$\Gamma^{10}(X,Y,Z)$} \psfrag{PNS2}{$\Gamma^2(X,Y,Z)$}
\psfrag{PNS11}{$\Gamma^{11}(X,Y,Z)$}
\psfrag{U}{$\uu$} \psfrag{V}{$\vv$} \psfrag{W}{$\ww$} \psfrag{OO}{$\OO$}
\psfrag{U}{$X$} \psfrag{V}{$Y$} \psfrag{W}{$Z$}
\psfrag{one}{$1$}
\psfrag{two}{$2$}
\psfrag{thr}{$3$}
\psfrag{fou}{$4$}
\psfrag{Mone}{$M_1(1234)$}
\psfrag{Mtwo}{$M_2(1234)$}
\psfrag{oMone}{$4!2^4/6$}
\psfrag{oMtwo}{$4!2^4/2$}
\includegraphics[width = 0.750 \linewidth]{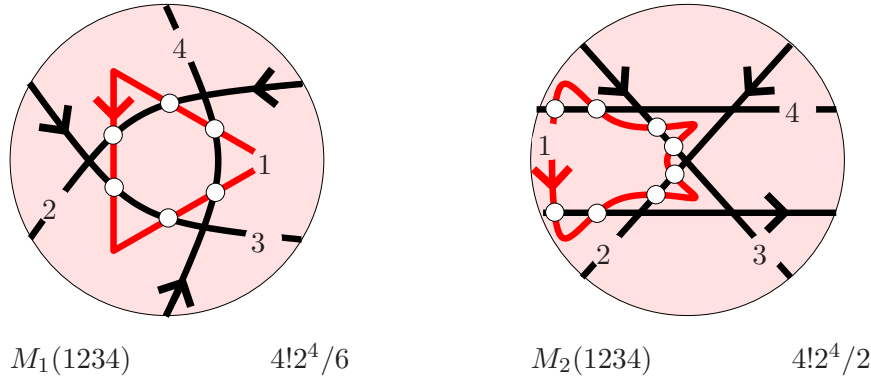}
\caption{Oriented and indexed versions of the martagons on $4$ double pseudolines
\label{FinalMartagonTer}}
\end{figure}
The automorphism group of $M_1$ is $S_3$ (dihedral group $D_3$) generated by the involutions $\overline{1324}$ and $\overline{1243}$, for example.
The automorphism group of $M_2$ is of order 2 generated by the permutation $\overline{1324}$.
\end{example}

We close this section by introducing, in Fig.~\ref{fulllistdecored0443WCY} and~\ref{fulllistdecored2264WCY},
one oriented and indexed version of  each of the thirteen simple arrangements of three double pseudolines of Fig.~\ref{fulllist}.  
We let the reader check that
the four subarrangements of size three of   
$M_1(1234)$ are 
$\bname{22}{1\overline{23}}$, $\bname{22}{1\overline{34}}$, $\bname{22}{1\overline{42}}$ and $\bname{04}{234}$. 
Similarly the four subarrangements of size three of   
the martagon $M_2(1234)$ are  
$\bname{22}{123}$, $\bname{22}{42\overline{3}}$, $\bname{32}{142}$, $\bname{32}{143}.$

\begin{figure}[!htb]
\footnotesize
\def\factor{0.2625325}
\centering
\psfrag{A}{$04\scriptstyle9000$}
\psfrag{B}{$07\scriptstyle3300$}
\psfrag{C}{$18\scriptstyle1030$}
\psfrag{D}{$25\scriptstyle2310$} 
\psfrag{F}{$07\scriptstyle3300$} 
\psfrag{G}{$37\scriptstyle0003$}
\psfrag{H}{$15\scriptstyle4300$}
\psfrag{I}{\textcolor{red}{$252310$}}
\psfrag{J}{$43\scriptstyle3021$}
\psfrag{K}{$25\scriptstyle2310$}
\psfrag{L}{$33\scriptstyle3310$}
\psfrag{M}{$32\scriptstyle6020$}
\psfrag{N}{$25\scriptstyle2310$}
\psfrag{Nstar}{$25\scriptstyle2310$}
\psfrag{O}{$32\scriptstyle6020$}
\psfrag{P}{$22\scriptstyle8010$}
\psfrag{Q}{$25\scriptstyle2310$}
\psfrag{R}{$36\scriptstyle0040$}
\psfrag{Z}{$64{\scriptstyle 00003}$}

\psfrag{A}{}
\psfrag{B}{}
\psfrag{C}{}
\psfrag{D}{} 
\psfrag{F}{} 
\psfrag{G}{}
\psfrag{H}{}
\psfrag{I}{}
\psfrag{J}{}
\psfrag{K}{}
\psfrag{L}{}
\psfrag{M}{}
\psfrag{N}{}
\psfrag{Nstar}{}
\psfrag{O}{}
\psfrag{P}{}
\psfrag{Q}{}
\psfrag{R}{}
\psfrag{Z}{}

\psfrag{one}{$\AAA$}
\psfrag{two}{$\BBB$}
\psfrag{thr}{$\CCC$}
\psfrag{oA}{$2$}
\psfrag{name}{$\name{04}(\AAA\BBB\CCC)$}
\psfrag{cycles}{$\begin{array}{c}
                 \thrcrossR{1}{2}
                 \foucrossR{1}{2}
                 \thrcrossR{1}{3}
                 \foucrossR{1}{3}
                 \onecrossR{1}{2}
                 \twocrossR{1}{2}
                 \onecrossR{1}{3}
                 \twocrossR{1}{3}
\\
                 \thrcrossR{2}{3}
                 \foucrossR{2}{3}
                 \thrcrossR{2}{1}
                 \foucrossR{2}{1}
                 \onecrossR{2}{3}
                 \twocrossR{2}{3}
                 \onecrossR{2}{1}
                 \twocrossR{2}{1}
\\
                 \thrcrossR{3}{1}
                 \foucrossR{3}{1}
                 \thrcrossR{3}{2}
                 \foucrossR{3}{2}
                 \onecrossR{3}{1}
                 \twocrossR{3}{1}
                 \onecrossR{3}{2}
                 \twocrossR{3}{2}
\end{array}$}
\includegraphics[width = \factor\linewidth]{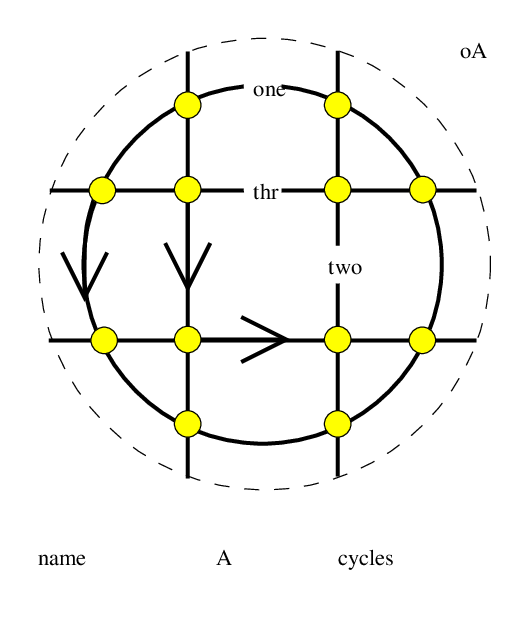}
\psfrag{one}{$\AAA$}
\psfrag{two}{$\BBB$}
\psfrag{thr}{$\CCC$}
\psfrag{cycles}{$\begin{array}{c}
                 \thrcrossR{1}{2}
                 \foucrossR{1}{2}
                 \thrcrossR{1}{3}
                 \foucrossR{1}{3}
                 \onecrossR{1}{2}
                 \onecrossR{1}{3}
                 \twocrossR{1}{2}
                 \twocrossR{1}{3}
\\
                 \thrcrossR{2}{3}
                 \thrcrossR{2}{1}
                 \foucrossR{2}{3}
                 \foucrossR{2}{1}
                 \onecrossR{2}{3}
                 \twocrossR{2}{3}
                 \onecrossR{2}{1}
                 \twocrossR{2}{1}
\\
                 \thrcrossR{3}{1}
                 \foucrossR{3}{1}
                 \thrcrossR{3}{2}
                 \onecrossR{3}{1}
                 \foucrossR{3}{2}
                 \twocrossR{3}{1}
                 \onecrossR{3}{2}
                 \twocrossR{3}{2}
\end{array}$}
\psfrag{name}{$\name{07}(\AAA\BBB\CCC)$}
\psfrag{oB}{$8$}
\includegraphics[width = \factor\linewidth]{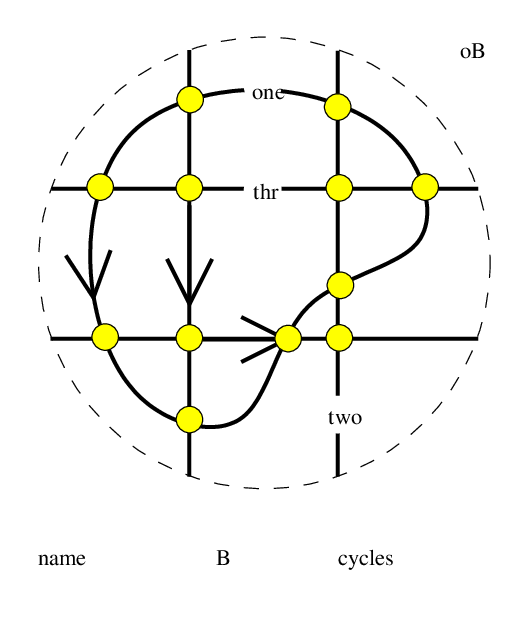}
\psfrag{one}{$\AAA$}
\psfrag{two}{$\BBB$}
\psfrag{thr}{$\CCC$}
\psfrag{cycles}{$\begin{array}{c}
                 \twocrossR{1}{3}
                 \foucrossR{1}{2}
                 \thrcrossR{1}{3}
                 \foucrossR{1}{3}
                 \onecrossR{1}{2}
                 \onecrossR{1}{3}
                 \twocrossR{1}{2}
                 \thrcrossR{1}{2}
\\
                 \twocrossR{2}{1}
                 \thrcrossR{2}{1}
                 \foucrossR{2}{3}
                 \foucrossR{2}{1}
                 \onecrossR{2}{3}
                 \twocrossR{2}{3}
                 \onecrossR{2}{1}
                 \thrcrossR{2}{3}
\\
                 \twocrossR{3}{2}
                 \foucrossR{3}{1}
                 \thrcrossR{3}{2}
                 \onecrossR{3}{1}
                 \foucrossR{3}{2}
                 \twocrossR{3}{1}
                 \onecrossR{3}{2}
                 \thrcrossR{3}{1}
\end{array}$}
\psfrag{name}{$\name{18}(\AAA\BBB\CCC)$}
\psfrag{oC}{$12$}
\includegraphics[width = \factor\linewidth]{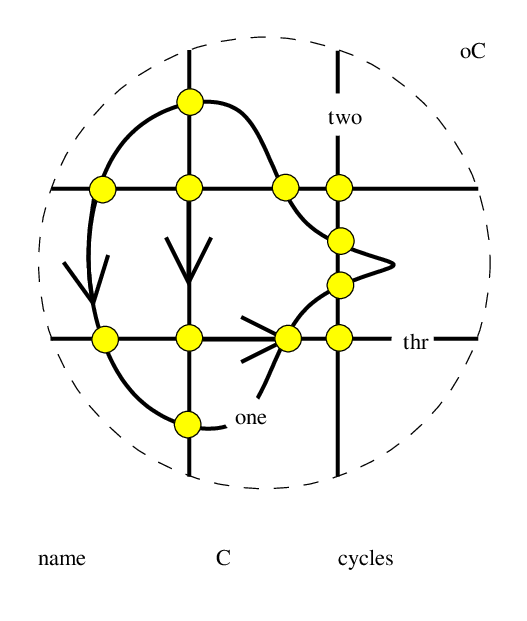}
\psfrag{one}{$\AAA$}
\psfrag{two}{$\BBB$}
\psfrag{thr}{$\CCC$}
\psfrag{name}{$\name{37}(\AAA\BBB\CCC)$}
\psfrag{oG}{$8$}
\psfrag{cycles}{$\begin{array}{c}
                 \twocrossR{1}{3}
                 \foucrossR{1}{2}
                 \thrcrossR{1}{3}
                 \onecrossR{1}{2}
                 \foucrossR{1}{3}
                 \onecrossR{1}{3}
                 \twocrossR{1}{2}
                 \thrcrossR{1}{2}
\\
                 \twocrossR{2}{1}
                 \thrcrossR{2}{1}
                 \foucrossR{2}{3}
                 \foucrossR{2}{1}
                 \onecrossR{2}{3}
                 \onecrossR{2}{1}
                 \twocrossR{2}{3}
                 \thrcrossR{2}{3}
\\
                 \twocrossR{3}{2}
                 \thrcrossR{3}{2}
                 \foucrossR{3}{1}
                 \onecrossR{3}{1}
                 \foucrossR{3}{2}
                 \twocrossR{3}{1}
                 \onecrossR{3}{2}
                 \thrcrossR{3}{1}
\end{array}$}
\includegraphics[width = \factor\linewidth]{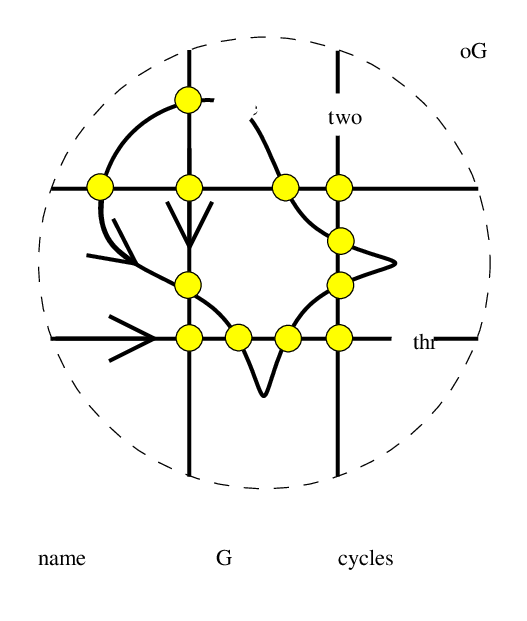}
\psfrag{one}{$\AAA$}
\psfrag{two}{$\BBB$}
\psfrag{thr}{$\CCC$}
\psfrag{name}{$\name{15}(\AAA\BBB\CCC)$}
\psfrag{oH}{$24$}
\psfrag{cycles}{$\begin{array}{c}
                 \thrcrossR{1}{3}
                 \foucrossR{1}{3}
                 \onecrossR{1}{2}
                 \twocrossR{1}{2}
                 \onecrossR{1}{3}
                 \thrcrossR{1}{2}
                 \foucrossR{1}{2}
                 \twocrossR{1}{3}
\\
                 \thrcrossR{2}{3}
                 \twocrossR{2}{1}
                 \foucrossR{2}{3}
                 \thrcrossR{2}{1}
                 \onecrossR{2}{3}
                 \foucrossR{2}{1}
                 \twocrossR{2}{3}
                 \onecrossR{2}{1}
\\
                 \twocrossR{3}{1}
                 \thrcrossR{3}{1}
                 \onecrossR{3}{2}
                 \twocrossR{3}{2}
                 \foucrossR{3}{1}
                 \thrcrossR{3}{2}
                 \foucrossR{3}{2}
                 \onecrossR{3}{1}
\end{array}$}
\includegraphics[width = \factor\linewidth]{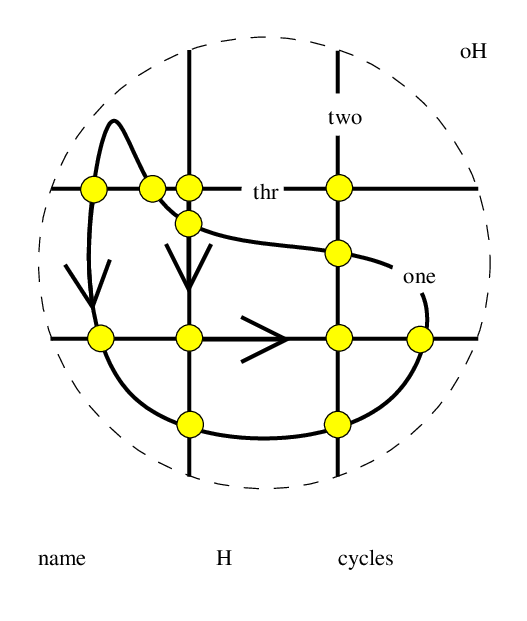}
\psfrag{one}{$\AAA$}
\psfrag{two}{$\BBB$}
\psfrag{thr}{$\CCC$}
\psfrag{name}{$\name{43}(\AAA\BBB\CCC)$}
\psfrag{oJ}{$24$}
\psfrag{cycles}{$\begin{array}{c}
                 \foucrossR{1}{2}
                 \twocrossR{1}{3}
                 \thrcrossR{1}{3}
                 \onecrossR{1}{2}
                 \foucrossR{1}{3}
                 \onecrossR{1}{3}
                 \twocrossR{1}{2}
                 \thrcrossR{1}{2}
\\
                 \twocrossR{2}{1}
                 \thrcrossR{2}{1}
                 \foucrossR{2}{3}
                 \onecrossR{2}{3}
                 \foucrossR{2}{1}
                 \onecrossR{2}{1}
                 \twocrossR{2}{3}
                 \thrcrossR{2}{3}
\\
                 \twocrossR{3}{2}
                 \thrcrossR{3}{2}
                 \foucrossR{3}{1}
                 \onecrossR{3}{1}
                 \foucrossR{3}{2}
                 \twocrossR{3}{1}
                 \thrcrossR{3}{1}
                 \onecrossR{3}{2}
\end{array}$}
\includegraphics[width = \factor\linewidth]{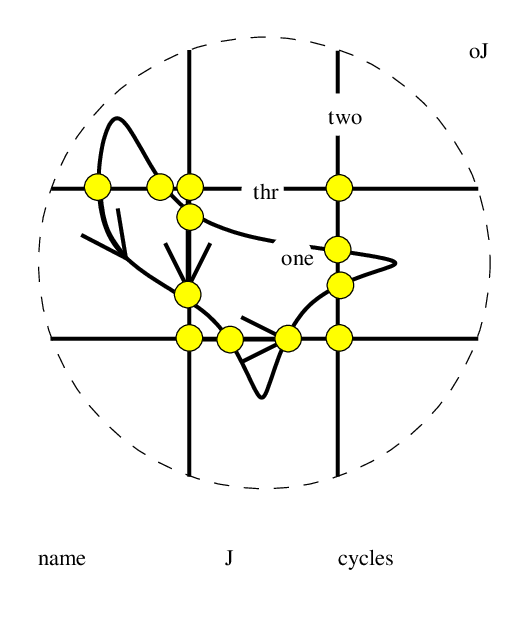}
\caption{Indexed and oriented versions of the arrangements $\name{04},\name{07},\name{18},\name{37},\name{15},\name{43}$.
Each diagram is labeled at its top right  by its number of (distinct) reindexed and reoriented versions and at its bottom right
by its side cycles of disk type 
\label{fulllistdecored0443WCY}}
\end{figure}

\begin{figure}[!htb]
\footnotesize
\def\factor{0.2625325}
\centering
\psfrag{A}{$04\scriptstyle9000$}
\psfrag{B}{$07\scriptstyle3300$}
\psfrag{C}{$18\scriptstyle1030$}
\psfrag{D}{$25\scriptstyle2310$} 
\psfrag{F}{$07\scriptstyle3300$} 
\psfrag{G}{$37\scriptstyle0003$}
\psfrag{H}{$15\scriptstyle4300$}
\psfrag{I}{\textcolor{red}{$252310$}}
\psfrag{J}{$43\scriptstyle3021$}
\psfrag{K}{$25\scriptstyle2310$}
\psfrag{L}{$33\scriptstyle3310$}
\psfrag{M}{$32\scriptstyle6020$}
\psfrag{N}{$25\scriptstyle2310$}
\psfrag{Nstar}{$25\scriptstyle2310$}
\psfrag{O}{$32\scriptstyle6020$}
\psfrag{P}{$22\scriptstyle8010$}
\psfrag{Q}{$25\scriptstyle2310$}
\psfrag{R}{$36\scriptstyle0040$}
\psfrag{Z}{$64{\scriptstyle 00003}$}

\psfrag{A}{}
\psfrag{B}{}
\psfrag{C}{}
\psfrag{D}{} 
\psfrag{F}{} 
\psfrag{G}{}
\psfrag{H}{}
\psfrag{I}{}
\psfrag{J}{}
\psfrag{K}{}
\psfrag{L}{}
\psfrag{M}{}
\psfrag{N}{}
\psfrag{Nstar}{}
\psfrag{O}{}
\psfrag{P}{}
\psfrag{Q}{}
\psfrag{R}{}
\psfrag{Z}{}

\psfrag{one}{$\AAA$}
\psfrag{two}{$\BBB$}
\psfrag{thr}{$\CCC$}
\psfrag{oP}{$12$}
\psfrag{name}{$\name{22}(\AAA\BBB\CCC)$}
\psfrag{cycles}{$\begin{array}{c}
                 \onecrossR{1}{3}
                 \twocrossR{1}{3}
                 \thrcrossR{1}{3}
                 \foucrossR{1}{3}
                 \onecrossR{1}{2}
                 \twocrossR{1}{2}
                 \thrcrossR{1}{2}
                 \foucrossR{1}{2}
\\
                 \thrcrossR{2}{3}
                 \foucrossR{2}{3}
                 \twocrossR{2}{1}
                 \thrcrossR{2}{1}
                 \onecrossR{2}{3}
                 \twocrossR{2}{3}
                 \foucrossR{2}{1}
                 \onecrossR{2}{1}
\\
                 \thrcrossR{3}{2}
                 \foucrossR{3}{2}
                 \twocrossR{3}{1}
                 \thrcrossR{3}{1}
                 \onecrossR{3}{2}
                 \twocrossR{3}{2}
                 \foucrossR{3}{1}
                 \onecrossR{3}{1}
\end{array}$}
\includegraphics[width = \factor\linewidth]{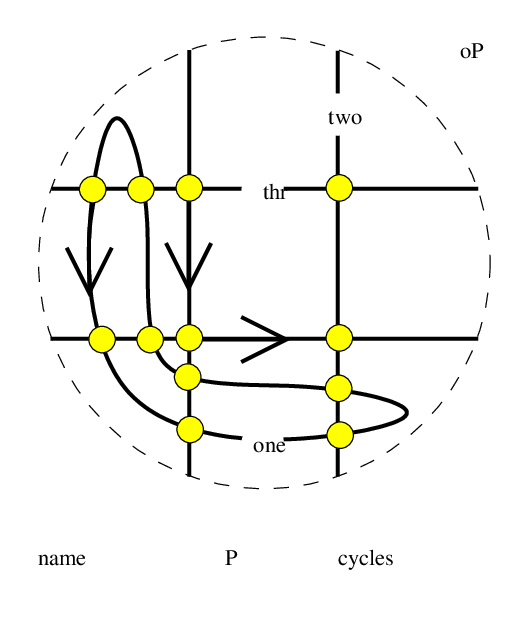}
\psfrag{one}{$\AAA$}
\psfrag{two}{$\BBB$}
\psfrag{thr}{$\CCC$}
\psfrag{oL}{$24$}
\psfrag{name}{$\name{33}(\AAA\BBB\CCC)$}
\psfrag{cycles}{$\begin{array}{c}
                 \foucrossR{1}{2}
                 \twocrossR{1}{3}
                 \thrcrossR{1}{3}
                 \onecrossR{1}{2}
                 \foucrossR{1}{3}
                 \twocrossR{1}{2}
                 \thrcrossR{1}{2}
                 \onecrossR{1}{3}
\\
                 \thrcrossR{2}{3}
                 \foucrossR{2}{3}
                 \twocrossR{2}{1}
                 \thrcrossR{2}{1}
                 \onecrossR{2}{3}
                 \foucrossR{2}{1}
                 \onecrossR{2}{1}
                 \twocrossR{2}{3}
\\
                 \onecrossR{3}{1}
                 \foucrossR{3}{2}
                 \twocrossR{3}{1}
                 \thrcrossR{3}{1}
                 \onecrossR{3}{2}
                 \twocrossR{3}{2}
                 \thrcrossR{3}{2}
                 \foucrossR{3}{1}
\end{array}$}
\includegraphics[width = \factor\linewidth]{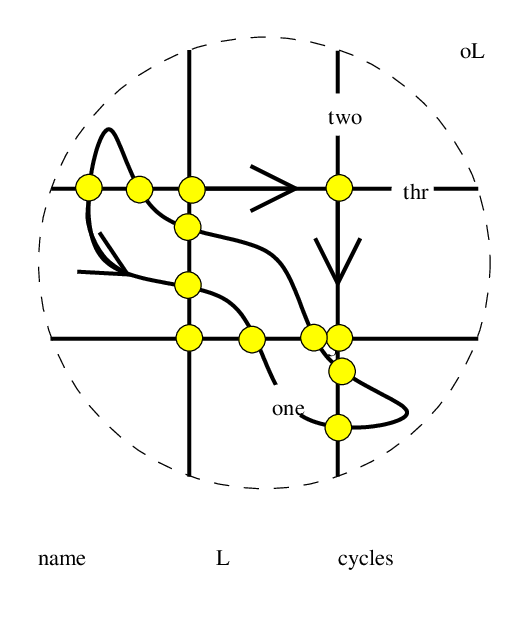}
\psfrag{one}{$\AAA$}
\psfrag{two}{$\BBB$}
\psfrag{thr}{$\CCC$}
\psfrag{name}{$\name{32}(\AAA\BBB\CCC)$}
\psfrag{oO}{$24$}
\psfrag{cycles}{$\begin{array}{c}
                 \twocrossR{1}{3}
                 \thrcrossR{1}{3}
                 \foucrossR{1}{3}
                 \onecrossR{1}{3}
                 \onecrossR{1}{2}
                 \twocrossR{1}{2}
                 \thrcrossR{1}{2}
                 \foucrossR{1}{2}
\\
                 \thrcrossR{2}{3}
                 \twocrossR{2}{1}
                 \thrcrossR{2}{1}
                 \foucrossR{2}{3}
                 \onecrossR{2}{3}
                 \foucrossR{2}{1}
                 \onecrossR{2}{1}
                 \twocrossR{2}{3}
\\
                 \thrcrossR{3}{2}
                 \foucrossR{3}{2}
                 \twocrossR{3}{1}
                 \thrcrossR{3}{1}
                 \onecrossR{3}{2}
                 \twocrossR{3}{2}
                 \foucrossR{3}{1}
                 \onecrossR{3}{1}
\end{array}$}
\includegraphics[width = \factor\linewidth]{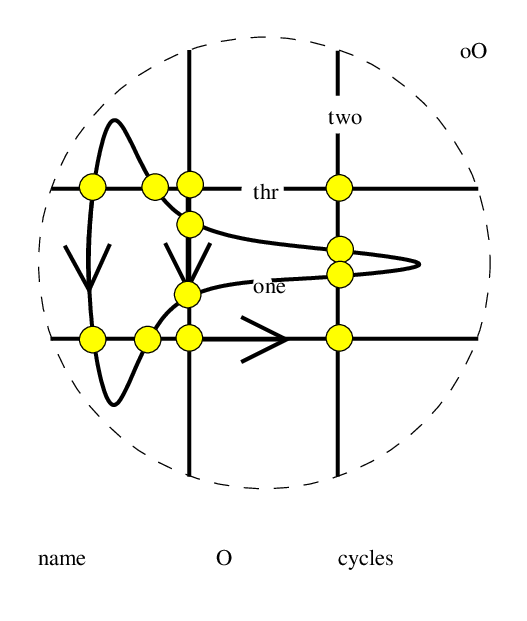}
\psfrag{one}{$\AAA$}
\psfrag{two}{$\BBB$}
\psfrag{thr}{$\CCC$}
\psfrag{oN}{$24$}
\psfrag{name}{$\name{{{25}_2}}(\AAA\BBB\CCC)$}
\psfrag{cycles}{$\begin{array}{c}
                 \foucrossR{1}{2}
                 \twocrossR{1}{3}
                 \thrcrossR{1}{3}
                 \foucrossR{1}{3}
                 \onecrossR{1}{2}
                 \twocrossR{1}{2}
                 \thrcrossR{1}{2}
                 \onecrossR{1}{3}
\\
                 \thrcrossR{2}{3}
                 \foucrossR{2}{3}
                 \twocrossR{2}{1}
                 \thrcrossR{2}{1}
                 \onecrossR{2}{3}
                 \foucrossR{2}{1}
                 \twocrossR{2}{3}
                 \onecrossR{2}{1}
\\
                 \onecrossR{3}{1}
                 \foucrossR{3}{2}
                 \twocrossR{3}{1}
                 \thrcrossR{3}{1}
                 \onecrossR{3}{2}
                 \twocrossR{3}{2}
                 \foucrossR{3}{1}
                 \thrcrossR{3}{2}
\end{array}$}
\includegraphics[width = \factor\linewidth]{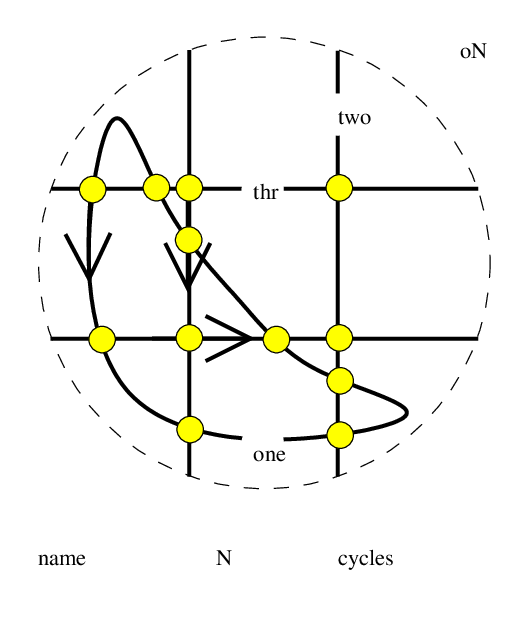}
\psfrag{one}{$\AAA$}
\psfrag{two}{$\BBB$}
\psfrag{thr}{$\CCC$}
\psfrag{oNstar}{$48$}
\psfrag{name}{$\name{{25_1}}(\AAA\BBB\CCC)$}
\psfrag{cycles}{$\begin{array}{c}
                 \foucrossR{1}{2}
                 \thrcrossR{1}{3}
                 \onecrossR{1}{2}
                 \foucrossR{1}{3}
                 \twocrossR{1}{2}
                 \thrcrossR{1}{2}
                 \onecrossR{1}{3}
                 \twocrossR{1}{3}
\\
                 \foucrossR{2}{1}
                 \onecrossR{2}{3}
                 \onecrossR{2}{1}
                 \twocrossR{2}{3}
                 \thrcrossR{2}{3}
                 \foucrossR{2}{3}
                 \twocrossR{2}{1}
                 \thrcrossR{2}{1}
\\
                 \twocrossR{3}{1}
                 \onecrossR{3}{2}
                 \thrcrossR{3}{1}
                 \twocrossR{3}{2}
                 \thrcrossR{3}{2}
                 \foucrossR{3}{1}
                 \onecrossR{3}{1}
                 \foucrossR{3}{2}
\end{array}$}
\includegraphics[width = \factor\linewidth]{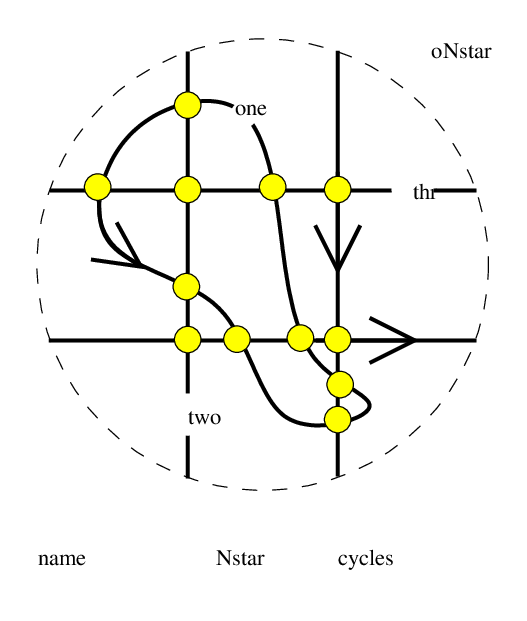}
\psfrag{one}{$\AAA$}
\psfrag{two}{$\BBB$}
\psfrag{thr}{$\CCC$}
\psfrag{oR}{$4$}
\psfrag{name}{$\name{36}(\AAA\BBB\CCC)$}
\psfrag{cycles}{$\begin{array}{c}
                 \thrcrossR{1}{3}
                 \twocrossR{1}{2}
                 \foucrossR{1}{3}
                 \onecrossR{1}{3}
                 \thrcrossR{1}{2}
                 \twocrossR{1}{3}
                 \foucrossR{1}{2}
                 \onecrossR{1}{2}
\\
                 \thrcrossR{2}{3}
                 \thrcrossR{2}{1}
                 \foucrossR{2}{3}
                 \foucrossR{2}{1}
                 \onecrossR{2}{1}
                 \onecrossR{2}{3}
                 \twocrossR{2}{1}
                 \twocrossR{2}{3}
\\
                 \foucrossR{3}{1}
                 \onecrossR{3}{1}
                 \onecrossR{3}{2}
                 \twocrossR{3}{1}
                 \twocrossR{3}{2}
                 \thrcrossR{3}{2}
                 \thrcrossR{3}{1}
                 \foucrossR{3}{2}
\end{array}$}
\includegraphics[width = \factor\linewidth]{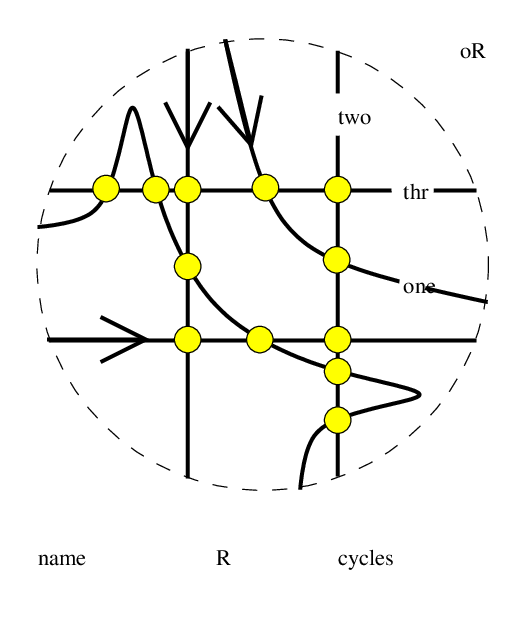}
\psfrag{one}{$\AAA$}
\psfrag{two}{$\BBB$}
\psfrag{thr}{$\CCC$}
\psfrag{cycles}{$\begin{array}{c}
                 \foucrossR{1}{2}
                 \onecrossR{1}{2}
                 \foucrossR{1}{3}
                 \onecrossR{1}{3}
                 \twocrossR{1}{2}
                 \thrcrossR{1}{2}
                 \twocrossR{1}{3}
                 \thrcrossR{1}{3}
\\
                 \foucrossR{2}{3}
                 \onecrossR{2}{3}
                 \foucrossR{2}{1}
                 \onecrossR{2}{1}
                 \twocrossR{2}{3}
                 \thrcrossR{2}{3}
                 \twocrossR{2}{1}
                 \thrcrossR{2}{1}
\\
                 \foucrossR{3}{1}
                 \onecrossR{3}{1}
                 \foucrossR{3}{2}
                 \onecrossR{3}{2}
                 \twocrossR{3}{1}
                 \thrcrossR{3}{1}
                 \twocrossR{3}{2}
                 \thrcrossR{3}{2}
\end{array}$}
\psfrag{name}{$\name{64}(\AAA\BBB\CCC)$}
\psfrag{oZ}{$2$}
\includegraphics[width = \factor\linewidth]{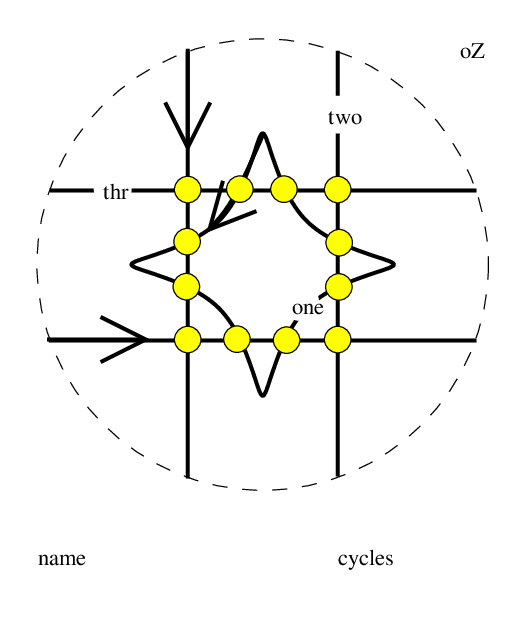}
\caption{Indexed and oriented versions of the arrangements $\name{22},\name{33},\name{32},\name{25_2},\name{25_1},\name{36},\name{64}.$ 
Each diagram is labeled at its top right  by its number of (distinct) reindexed and reoriented versions and at its bottom right
by its side cycles of disk type
\label{fulllistdecored2264WCY}}
\end{figure}


\clearpage
\subsection{Chirotopes}
We are now ready to show that the isomorphism class of an arrangement of oriented double pseudolines 
depends only on its chirotope; cf. first part of Theorem~\ref{theoADP}.  
\begin{theorem} The map that assigns to an isomorphism class of indexed arrangements of oriented double pseudolines 
its chirotope is one-to-one. 
\end{theorem}
\begin{proof}
Let $\Gamma$ be an indexed arrangement of oriented double pseudo\-lines. 
According to Theorem~\ref{coding} the isomorphism class of $\Gamma$ depends only on the family of cycles of $\Gamma$; 
 therefore it is sufficient to show that the family of cycles of $\Gamma$
 depends only on the chirotope of $\Gamma$. 

Clearly the set of nodes of $\Gamma$ depends only on the chirotope of $\Gamma$,  and 
clearly we can restrict our attention to the case where $\Gamma$ has four
elements, say $\GOO,\Guu,\Gvv$ and $\Gww$. 

We now show that the cycle of $\Gamma$ indexed by $\OO$ depends only on the 
chirotope of $\Gamma$. 
We write $\XX{\xx}$, $\xx \in \{\uu,\vv,\ww\}$, for the set of nodes of $\Gamma$ indexing the intersection points of  $\GOO$ and $\Gxx$,
$\XX{\xx,\yy,\ldots}$ for $\XX{\{\xx,\yy,\ldots\}}$,  
and for any $\nC,\nC' \in \XX{\uu,\vv,\ww}$, $\nC \neq \nC'$, we write $\interval{\nC}{\nC'}$  for the set of $\nA \in\XX{\uu,\vv,\ww}$, including $\nC$ and $\nC'$,  
that appear between  $\nC$ and $\nC'$  on the cycle of $\Gamma$ indexed by~$\OO$. 

Let $\nC,\nC' \in \XX{\uu}$, $\nC\neq \nC'$, and let $\nD,\nD'\in \XX{\uu,\vv,\ww}$.
We say that the  pair $(\nC,\nC')$ separates the  pair $(\nD,\nD')$ if  $\nD \in \interval{\nC}{\nC'}$ and 
$\nD'\in \interval{\nC'}{\nC}$.  Clearly one can decide, using only the chirotope of $\Gamma$,  
\begin{enumerate}
\item whether $\nD$ belongs to the interval $\interval{\nC}{\nC'}$ or not; and
\item whether the pair $(\nC,\nC')$ separates the pair $(\nD,\nD')$ or not.
\end{enumerate}
Assume that a pair $(\nC,\nC')$ of distinct elements of $\XX{\uu}$ 
 separates a pair $(\nD,\nD')$ of elements of $\XX{\vv,\ww}$. 
(In particular this happens if one of the intersection points between $\GOO$ and $\Guu$ is not an ordinary vertex of the arrangement $\Gamma$.) In that case 
a pair $\nA,\nB$  of distinct elements of $\XX{\vv,\ww}$ lying in the open interval 
$\interval{\nC}{\nC'}\setminus \{\nC,\nC'\}$ appears in the linear order $\nA\nB$ in the interval 
$\interval{\nC}{\nC'}$ if and only if  $\nD', \nA, \nB$ appear in the cyclic order  $\nD'\nA\nB$ 
on the cycle of the arrangement $\{\GOO,\Gvv,\Gww\}$ indexed by $\OO$. Consequently the cycle of $\Gamma$ indexed by $\OO$ depends only on the 
chirotope of $\Gamma$ and we are done. Similarly we are done if a pair of distinct elements of $\XX{\vv}$ separates 
a pair of elements of $\XX{\uu,\ww}$, or if a pair of distinct elements of $\XX{\ww}$ separates a pair of elements of 
$\XX{\uu,\vv}$. 
Thus it remains to examine the case where for every $k\in\{\uu,\vv,\ww\}$ no pair of distinct elements of $\XX{k}$ 
separates a pair of elements of $\XX{\{\uu,\vv,\ww\}\setminus\{k\}}$,   
i.e., using the terminology introduced in the previous section, the  case where the arrangement $\Gamma$ is a  martagon with respect to~$\GOO$.
According to Lemma~\ref{martagonthr} and the notations introduced in Example~\ref{examplemartagon},
 this means that, up to permutation of the indices $\OO,\uu,\vv,\ww$ and their negatives,  
$\Gamma = M_1(1234)$ or $\Gamma = M_2(1234)$.
The theorem follows. Indeed if the chirotope of $\Gamma$ is the chirotope of $M_1(1234)$ then the family of side cycles  (of disk type) 
of $\Gamma$ is necessarily  
either the family  
$$
\begin{array}{ccl}
\calC{\AAA}: & & 
\thrcrossR{\AAA}{\BBB} \foucrossR{\AAA}{\BBB} \onecrossR{\AAA}{\BBB} \twocrossR{\AAA}{\BBB}
\thrcrossR{\AAA}{\CCC} \foucrossR{\AAA}{\CCC} \onecrossR{\AAA}{\CCC} \twocrossR{\AAA}{\CCC}
\thrcrossR{\AAA}{\DDD} \foucrossR{\AAA}{\DDD} \onecrossR{\AAA}{\DDD} \twocrossR{\AAA}{\DDD}
\\
\calC{\BBB}: &  &
\thrcrossR{\BBB}{\DDD} \foucrossR{\BBB}{\DDD} \foucrossR{\BBB}{\AAA} \onecrossR{\BBB}{\AAA}
\onecrossR{\BBB}{\CCC} \twocrossR{\BBB}{\CCC} \onecrossR{\BBB}{\DDD} \twocrossR{\BBB}{\DDD}
\twocrossR{\BBB}{\AAA} \thrcrossR{\BBB}{\AAA} \thrcrossR{\BBB}{\CCC} \foucrossR{\BBB}{\CCC} \\
\calC{\CCC}: &  & 
\thrcrossR{\CCC}{\BBB} \foucrossR{\CCC}{\BBB} \foucrossR{\CCC}{\AAA} \onecrossR{\CCC}{\AAA}
\onecrossR{\CCC}{\DDD} \twocrossR{\CCC}{\DDD} \onecrossR{\CCC}{\BBB} \twocrossR{\CCC}{\BBB}
\twocrossR{\CCC}{\AAA} \thrcrossR{\CCC}{\AAA} \thrcrossR{\CCC}{\DDD} \foucrossR{\CCC}{\DDD} \\
\calC{\DDD}: &  & 
\thrcrossR{\DDD}{\CCC} \foucrossR{\DDD}{\CCC} \foucrossR{\DDD}{\AAA} \onecrossR{\DDD}{\AAA}
\onecrossR{\DDD}{\BBB} \twocrossR{\DDD}{\BBB} \onecrossR{\DDD}{\CCC} \twocrossR{\DDD}{\CCC}
\twocrossR{\DDD}{\AAA} \thrcrossR{\DDD}{\AAA} \thrcrossR{\DDD}{\BBB} \foucrossR{\DDD}{\BBB}
\end{array}
$$
of side cycles of $M_1(1234)$
or the family
$$
\begin{array}{ccl}
\calCstar{\AAA}: & & 
\thrcrossR{\AAA}{\CCC} \foucrossR{\AAA}{\CCC} \onecrossR{\AAA}{\CCC} \twocrossR{\AAA}{\CCC}
\thrcrossR{\AAA}{\BBB} \foucrossR{\AAA}{\BBB} \onecrossR{\AAA}{\BBB} \twocrossR{\AAA}{\BBB}
\thrcrossR{\AAA}{\DDD} \foucrossR{\AAA}{\DDD} \onecrossR{\AAA}{\DDD} \twocrossR{\AAA}{\DDD} \\
\calCstar{\BBB}: &  &
\thrcrossR{\BBB}{\DDD} \foucrossR{\BBB}{\DDD} \foucrossR{\BBB}{\AAA} \onecrossR{\BBB}{\AAA}
\onecrossR{\BBB}{\CCC} \twocrossR{\BBB}{\CCC} \onecrossR{\BBB}{\DDD} \twocrossR{\BBB}{\DDD}
\twocrossR{\BBB}{\AAA} \thrcrossR{\BBB}{\AAA} \thrcrossR{\BBB}{\CCC} \foucrossR{\BBB}{\CCC} \\
\calCstar{\CCC}: &  & 
\thrcrossR{\CCC}{\BBB} \foucrossR{\CCC}{\BBB} \foucrossR{\CCC}{\AAA} \onecrossR{\CCC}{\AAA}
\onecrossR{\CCC}{\DDD} \twocrossR{\CCC}{\DDD} \onecrossR{\CCC}{\BBB} \twocrossR{\CCC}{\BBB}
\twocrossR{\CCC}{\AAA} \thrcrossR{\CCC}{\AAA} \thrcrossR{\CCC}{\DDD} \foucrossR{\CCC}{\DDD} \\
\calCstar{\DDD}: &  & 
\thrcrossR{\DDD}{\CCC} \foucrossR{\DDD}{\CCC} \foucrossR{\DDD}{\AAA} \onecrossR{\DDD}{\AAA}
\onecrossR{\DDD}{\BBB} \twocrossR{\DDD}{\BBB} \onecrossR{\DDD}{\CCC} \twocrossR{\DDD}{\CCC}
\twocrossR{\DDD}{\AAA} \thrcrossR{\DDD}{\AAA} \thrcrossR{\DDD}{\BBB} \foucrossR{\DDD}{\BBB}
\end{array}
$$
obtained from the family $\cal C$ by switching the blocks $\thrcrossR{\AAA}{\BBB} \foucrossR{\AAA}{\BBB} \onecrossR{\AAA}{\BBB} \twocrossR{\AAA}{\BBB}$ and $\thrcrossR{\AAA}{\CCC} \foucrossR{\AAA}{\CCC} \onecrossR{\AAA}{\CCC} \twocrossR{\AAA}{\CCC}$ in the cycle assigned to $1$ and by leaving the other cycles unchanged.
To rule out ${\cal C}^*$ from the set of families of cycles of double pseudoline arrangements it remains to observe that the  permutations 
that carry ${\cal C}_1$ onto ${\cal C}^*_1$ are exactly the $4$ permutations $1324,1243,1432,\overline{1234}$ and that none of these $4$ permutations 
leaves unchanged the triplet ${\cal C}_1,{\cal C}_2,{\cal C}_3.$ 
Similarly if the chirotope of $\Gamma$ is the chirotope of $M_2(12234)$ then 
the family of side cycles of $\Gamma$  is either the family 
$$
\begin{array}{ccl}
\calC{\AAA}: & &
\onecrossR{\AAA}{\BBB} \twocrossR{\AAA}{\BBB} \thrcrossR{\AAA}{\BBB} \foucrossR{\AAA}{\BBB}
\onecrossR{\AAA}{\CCC} \twocrossR{\DDD}{\CCC} \thrcrossR{\AAA}{\CCC} \foucrossR{\AAA}{\CCC}
\twocrossR{\AAA}{\DDD} \thrcrossR{\AAA}{\DDD} \foucrossR{\AAA}{\DDD} \onecrossR{\AAA}{\DDD}
\\
\calC{\BBB}: &  &
\onecrossR{\BBB}{\CCC} \twocrossR{\BBB}{\CCC} \foucrossR{\BBB}{\AAA} \onecrossR{\BBB}{\AAA}
\twocrossR{\BBB}{\DDD} \thrcrossR{\BBB}{\DDD} \thrcrossR{\BBB}{\CCC} \foucrossR{\BBB}{\CCC}
\twocrossR{\BBB}{\AAA} \thrcrossR{\BBB}{\AAA} \foucrossR{\BBB}{\DDD} \onecrossR{\BBB}{\DDD}
\\
\calC{\CCC}: &  & 
\onecrossR{\CCC}{\BBB} \twocrossR{\CCC}{\BBB} \foucrossR{\CCC}{\DDD} \onecrossR{\CCC}{\DDD}
\foucrossR{\CCC}{\AAA} \onecrossR{\CCC}{\AAA} \thrcrossR{\CCC}{\BBB} \foucrossR{\CCC}{\BBB}
\twocrossR{\CCC}{\DDD} \thrcrossR{\CCC}{\DDD} \twocrossR{\CCC}{\AAA} \thrcrossR{\CCC}{\AAA}
\\
\calC{\DDD}: &  & 
\onecrossR{\DDD}{\BBB} \twocrossR{\DDD}{\BBB} \foucrossR{\DDD}{\AAA} \onecrossR{\DDD}{\AAA}
\thrcrossR{\DDD}{\BBB} \foucrossR{\DDD}{\BBB} \thrcrossR{\DDD}{\CCC} \foucrossR{\DDD}{\CCC}
\twocrossR{\DDD}{\AAA} \thrcrossR{\DDD}{\AAA} \onecrossR{\DDD}{\CCC} \twocrossR{\DDD}{\CCC}
\end{array}
$$
of side cycles of $M_2(1234)$  or the family 
$$\begin{array}{ccl}
\calCstar{\AAA}: & &
\onecrossR{\AAA}{\CCC} \twocrossR{\DDD}{\CCC} \thrcrossR{\AAA}{\CCC} \foucrossR{\AAA}{\CCC}
\onecrossR{\AAA}{\BBB} \twocrossR{\AAA}{\BBB} \thrcrossR{\AAA}{\BBB} \foucrossR{\AAA}{\BBB}
\twocrossR{\AAA}{\DDD} \thrcrossR{\AAA}{\DDD} \foucrossR{\AAA}{\DDD} \onecrossR{\AAA}{\DDD}
\\
\calCstar{\BBB}: &  &
\onecrossR{\BBB}{\CCC} \twocrossR{\BBB}{\CCC} \foucrossR{\BBB}{\AAA} \onecrossR{\BBB}{\AAA}
\twocrossR{\BBB}{\DDD} \thrcrossR{\BBB}{\DDD} \thrcrossR{\BBB}{\CCC} \foucrossR{\BBB}{\CCC}
\twocrossR{\BBB}{\AAA} \thrcrossR{\BBB}{\AAA} \foucrossR{\BBB}{\DDD} \onecrossR{\BBB}{\DDD}
\\
\calCstar{\CCC}: &  & 
\onecrossR{\CCC}{\BBB} \twocrossR{\CCC}{\BBB} \foucrossR{\CCC}{\DDD} \onecrossR{\CCC}{\DDD}
\foucrossR{\CCC}{\AAA} \onecrossR{\CCC}{\AAA} \thrcrossR{\CCC}{\BBB} \foucrossR{\CCC}{\BBB}
\twocrossR{\CCC}{\DDD} \thrcrossR{\CCC}{\DDD} \twocrossR{\CCC}{\AAA} \thrcrossR{\CCC}{\AAA}
\\
\calCstar{\DDD}: &  & 
\onecrossR{\DDD}{\BBB} \twocrossR{\DDD}{\BBB} \foucrossR{\DDD}{\AAA} \onecrossR{\DDD}{\AAA}
\thrcrossR{\DDD}{\BBB} \foucrossR{\DDD}{\BBB} \thrcrossR{\DDD}{\CCC} \foucrossR{\DDD}{\CCC}
\twocrossR{\DDD}{\AAA} \thrcrossR{\DDD}{\AAA} \onecrossR{\DDD}{\CCC} \twocrossR{\DDD}{\CCC}
\end{array}
$$
obtained from the family $\cal C$ by switching the blocks $\thrcrossR{\AAA}{\BBB} \foucrossR{\AAA}{\BBB} \onecrossR{\AAA}{\BBB} \twocrossR{\AAA}{\BBB}$ and $\thrcrossR{\AAA}{\CCC} \foucrossR{\AAA}{\CCC} \onecrossR{\AAA}{\CCC} \twocrossR{\AAA}{\CCC}$ in the cycle assigned to $1$ and by leaving the other cycles invariant.
Again to rule out ${\cal C}^*$ from the set of families of cycles of double pseudoline arrangements it remains 
to observe that the permutations
that carry ${\cal C}_1$ onto ${\cal C}^*_1$ are  exactly the two permutations 
$1324,\overline{123}4,$ and that none of these $2$ permutations leaves invariant the cycle~${\cal C}_4$.
(In Section~\ref{secfiv} we interpret the ${\cal C}^*$ as side cycles of arrangements of double pseudolines living in a triple cross surface.) \end{proof}

\subsection{Cocycles \label{monades}}
Let $\Delta$ be an indexed configuration of oriented convex bodies of a projective plane $(\pp,\lpp)$ and let $\tau$ be a line of $(\pp,\lpp)$.
Recall that we have defined  
\begin{enumerate}
\item the {\it cocycle of $\Delta$ at $\tau$} or the {\it cocycle of $\tau$ with respect to $\Delta$} or the {\it cocycle of the pair $(\Delta,\tau)$} 
as the homeomorphism class of the image of the pair $(\Delta,\tau)$ under the 
quotient map $\Map{\omega_\tau}{\pp}{\pp/\rr_\tau}$ relative to the equivalence relation $\rr_\tau$ on $\pp$ generated by the pairs of points lying on a same line segment  
 of $\Delta\cap \tau$; 
\item a {\it bitangent cocycle} or {\it zero-cocycle} as a cocycle at a bitangent;
\item the {\it isomorphism class of $\Delta$} as the set of configurations that have the same set of bitangent cocycles as $\Delta$; and 
\item the {\it chirotope of $\Delta$} as the map that assigns to each $3$-subset $J$ of the indexing set of $\Delta$ the isomorphism 
class of the subfamily indexed by $J$. 
\end{enumerate}
To these four definitions we add the following one
\begin{enumerate}
\item[(5)] the {\it cocycle map} of $\Delta$ is the map that assigns to each cell $\sigma$ 
of the dual arrangement of $\Delta$ the cocycle of $\Delta$ at some (hence any) element of $\sigma$.  
\end{enumerate}
Fig.~\ref{FinalCocycles} depicts, 
up to reorientation and reindexing of the convex bodies, the cocycles of families of two and three pairwise disjoint convex bodies 
with respective indexing sets $\{\ii,\kk\}$ and $\{\ii,\kk,\jj\}$; in this figure each circular diagram is labeled at its 
bottom right by its number of reoriented and reindexed versions and at its 
bottom left by  its signature, a natural coding of the cocycle introduced in Section~\ref{secone} and that we will not repeat here.
\begin{figure}[!htb]
\psfrag{one}{$\ii$}\psfrag{two}{$\kk$}\psfrag{thr}{$\jj$}
\psfrag{one12}{$\ii$}\psfrag{two12}{$\kk$}\psfrag{thr12}{$\jj$}
\psfrag{one13}{$\ii$}\psfrag{two13}{$\jj$}\psfrag{thr13}{$\kk$}
\psfrag{one23}{$\kk$}\psfrag{two23}{$\jj$}\psfrag{thr23}{$\ii$}
\psfrag{lone}{$\ii\kk\ptg\ptg$}
\psfrag{ltwo}{$\ii\kk\jj\ptg\ptg\ptg$}
\psfrag{lthr}{$\ii\ptg \kk\ptg \jj\ptg$}
\psfrag{lfou}{$\ii \kk\ptg\ptg\ptg \jj$}
\psfrag{lfiv}{$\ii\kk\ptg\ptg,\jj$}
\psfrag{lsix}{$\ii\bjj\kk \ptg \jj \ptg$}
\psfrag{lsev}{$\bjj\ii\kk\jj \ptg \ptg$}
\psfrag{lfiv13}{$\ii\jj\ptg\ptg,\kk$}
\psfrag{lsix13}{$\ii\kk\jj \ptg \bkk \ptg$}
\psfrag{lsev13}{$\ii\jj\kk \ptg \ptg \bkk $}
\psfrag{lfiv23}{$\kk\jj\ptg\ptg,\ii$}
\psfrag{lsix23}{$\kk\ii\jj \ptg \bii \ptg$}
\psfrag{lsev23}{$\kk\jj\ii \ptg \ptg \bii $}
\psfrag{o24}{$24$} \psfrag{o12}{$12$} \psfrag{o8}{$8$}\psfrag{o6}{$6$}\psfrag{o4}{$4$}\psfrag{o4?}{$4$}\psfrag{o2}{$2$}\psfrag{o1}{$1$}
\psfrag{A}{$\ii\ptg,\kk,\jj$}
\psfrag{B}{$\bkk\ii\kk\ptg,\jj$}
\psfrag{C}{$\bkk\bjj\ii\kk\jj\ptg$}
\psfrag{D}{$\ii,\kk,\jj$}
\psfrag{E}{$\ii\bii,\kk,\jj$}
\psfrag{F}{$\ii\kk\bii\bkk,\jj$}
\psfrag{G}{$\ii\kk\jj\bii\bkk\bjj$}
\psfrag{BBBP}{$\bkk\ii\kk\ptg$}
\psfrag{BPvB}{$\ii\ptg,\kk$}
\psfrag{BvB}{$\ii,\kk$}
\psfrag{BBBB}{$\ii\kk\bii\bkk$}
\psfrag{BBvB}{$\ii\bii,\kk$}
\centering
\includegraphics[width=0.875\linewidth]{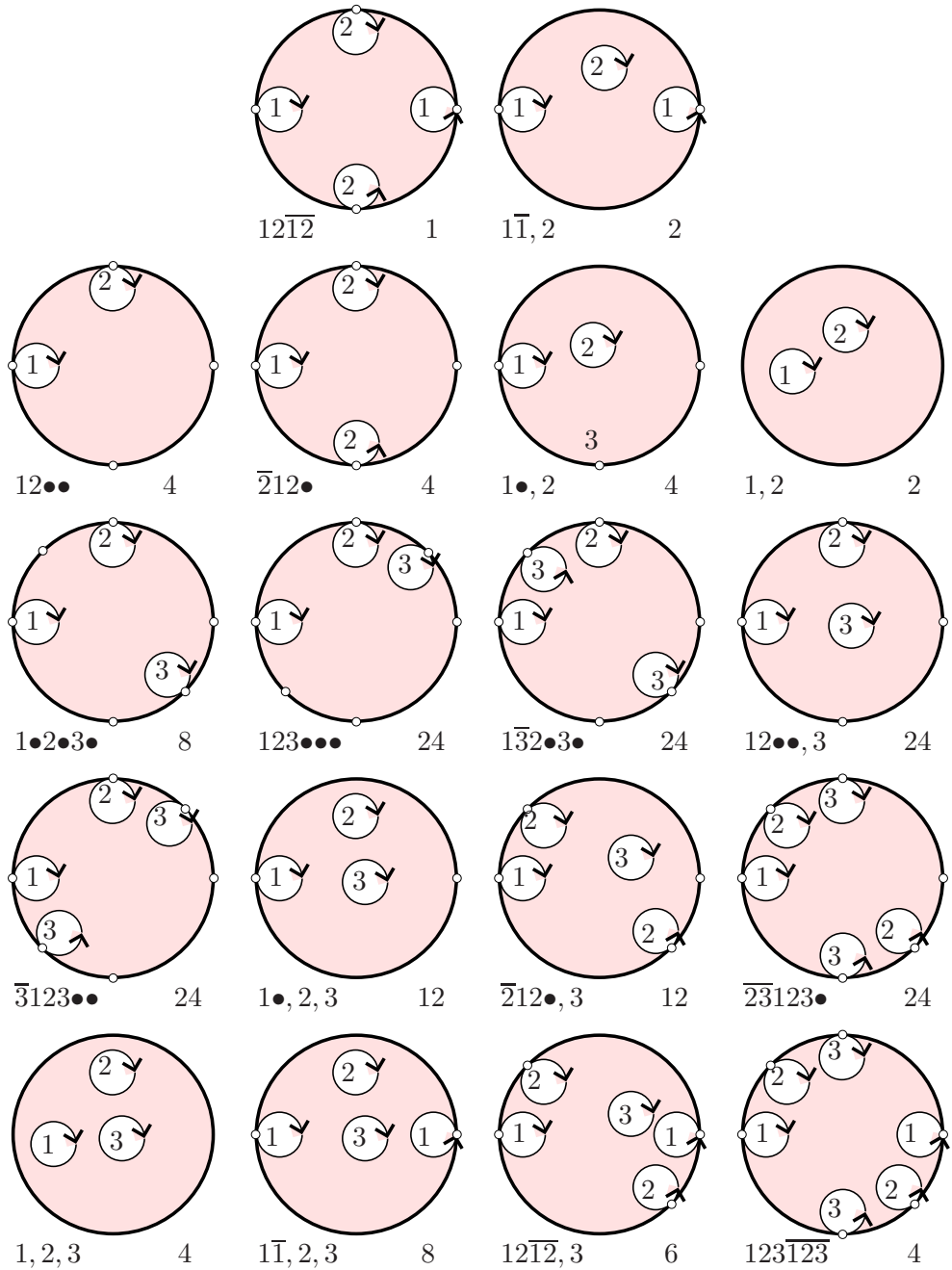}
\caption{Cocycles of indexed configurations of two and three  oriented convex bodies. Each cocycle is labeled at its bottom left with its signature and at its bottom right by its number of reoriented and reindexed versions \label{FinalCocycles}}
\end{figure}
Fig.~\ref{FinalThrCocyclesMap} depicts examples of cocycle maps of families of one, two, and three pairwise disjoint convex bodies 
with respective indexing sets $\{1\}$, $\{1,2\}$, and $\{1,2,3\}.$
\begin{figure}[!htb]
\centering
\psfrag{AAAA4}{$\un\ptg$}
\psfrag{AAA5}{$\un$} 
\psfrag{CCC9}{$\unb\un$} 
\psfrag{CC12}{$\de\ptg,\un$}
\psfrag{CC11}{$\un,\de$} 
\psfrag{CC10}{$\un\ptg,\de$} 
\psfrag{CC9}{$\unb\un,\de$}
\psfrag{CCC6}{$\un\de\ptg\ptg$} 
\psfrag{CC5}{$\un\de\unb\ptg$} 
\psfrag{CC4}{$\un\de\unb\deb$}
\psfrag{CC3}{$\deb\un\de\ptg$}
\psfrag{DDD3}{$\deb\un\ptg\ptg$} 
\psfrag{DD1}{$\un,\deb$}
\psfrag{AA7}{$\un\ptg,\deb$} 
\psfrag{AA6}{$\deb\ptg,\un$} 
\psfrag{AA5}{$\de\deb,\un$}
\psfrag{AAA4}{$\un\deb\ptg\ptg$} 
\psfrag{BBB2}{$\de\un\ptg\ptg$}
\psfrag{BB3}{$\unb\de\un\ptg$}
\psfrag{AA2}{$\de\un\deb\ptg$} 
\psfrag{DD3}{$\un\ptg\tr\ptg\deb\ptg$}
\psfrag{CC6}{$\un\ptg\de\ptg\tr\ptg$} 
\psfrag{BB2}{$\un\ptg\trb\ptg\de\ptg$}
\psfrag{AA4}{$\un\ptg\deb\ptg\trb\ptg$} 
\psfrag{un}{$1$} \psfrag{deu}{$2$} \psfrag{tr}{$3$}
\psfrag{A}{$A$} \psfrag{B}{$B$} \psfrag{C}{$C$} \psfrag{D}{$D$}
\psfrag{onecd}{$1$} \psfrag{twocd}{$2$} \psfrag{thrcd}{$3$}
\centering
\includegraphics[width=0.875\linewidth]{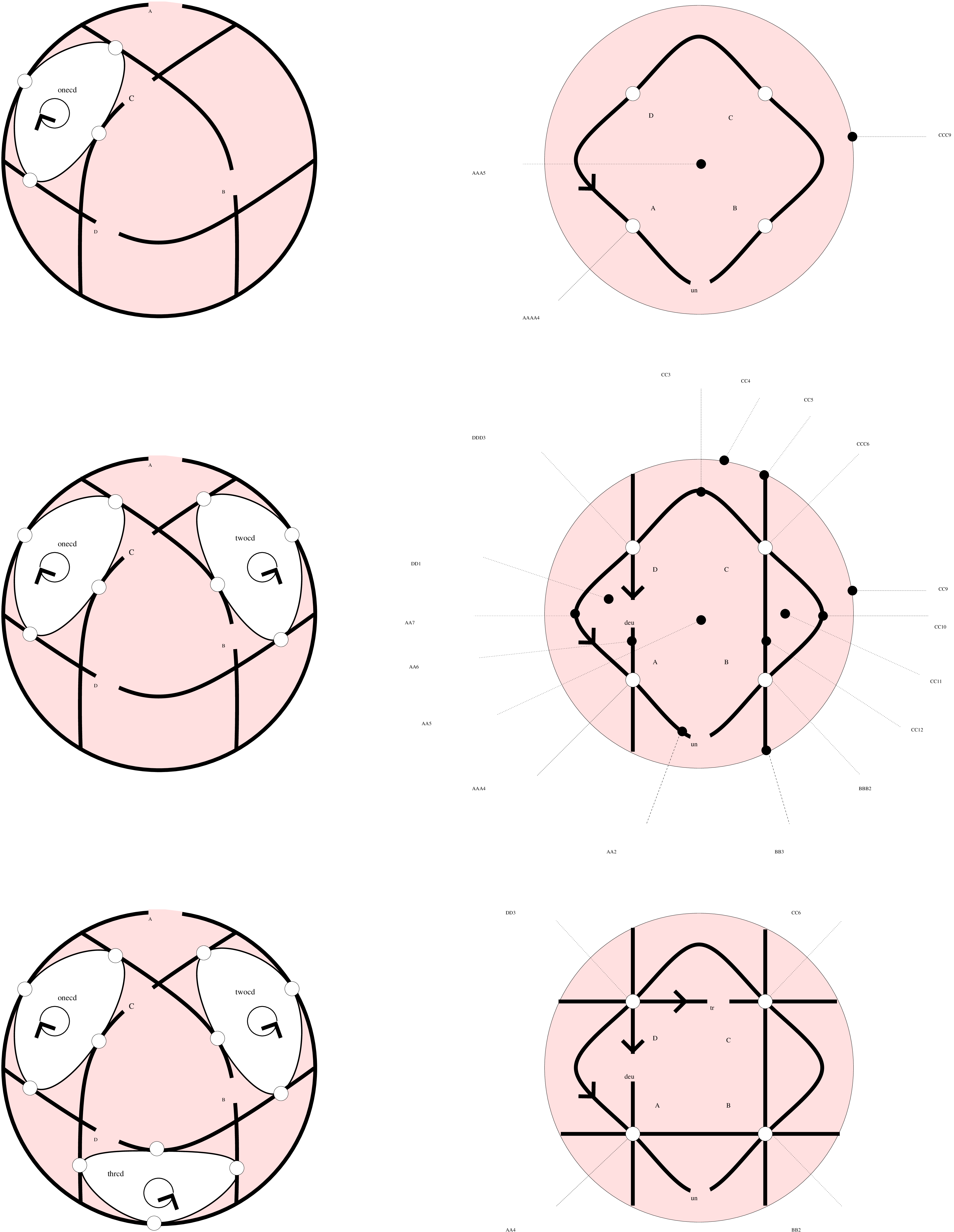}
\caption{Indexed configurations of one, two and three oriented convex bodies and the cocycle labeled versions of their dual arrangements
\label{FinalThrCocyclesMap}}
\end{figure}

In particular one can easily check that the cocycle map of a family of two bodies is one-to-one. 

\begin{lemma} Cocycle maps of families of two disjoint convex bodies are one-to-one. \qed
\end{lemma}

We are now ready to prove that the map that assigns to an indexed configuration of oriented convex bodies the isomorphism class of its dual arrangement is compatible 
with the isomorphism relation on families of convex bodies, and that the quotient map is one-to-one (and onto).  
This means, for example, that the signatures $ \{1 \ptg \overline{2} \ptg \overline{3} \ptg\}$, 
$\{1 \ptg \overline{3} \ptg 2 \ptg\}$,$\{1 \ptg 2\ptg 3 \ptg\}$, $\{1 \ptg 3\ptg \overline{2} \ptg\}$ 
of the bitangent cocycles  of the configuration of three convex bodies depicted at the  bottom left of Fig.~\ref{FinalThrCocyclesMap}
is a coding of the isomorphism class of the dual arrangement of the configuration. 

\begin{theorem}\label{theotwobis}  
Let $\Delta$ and $\Delta'$ be two indexed configurations of oriented convex bodies. 
Then the  following four assertions are equivalent:
\begin{enumerate}
\item $\Delta$ and $\Delta'$ have the same chirotope;
\item $\Delta$ and $\Delta'$ have isomorphic dual arrangements;
\item $\Delta$ and $\Delta'$ have isomorphic cocycle maps; 
\item $\Delta$ and $\Delta'$ have the same set of $0$-cocycles (i.e., are isomorphic).  \qed
\end{enumerate} 
\end{theorem}
\begin{proof} Some implications are clear: 
\begin{enumerate}[(i)]
\item $(4) \Rightarrow (1)$; 
\item $(4),(2) \Rightarrow (3)$,  using a perturbation argument;
\item $(1),(2) \Rightarrow (4)$, because the family of $0$-cocycles of $\Delta$ depends only on 
the family of cocycle-labeled versions of the dual arrangements of subfamilies of three bodies and on the
isomorphism class of the dual arrangement of $\Delta$;
\item $(3) \Rightarrow (4),(2),(1).$
\end{enumerate} 
We now prove that $(1) \Leftrightarrow (2).$

We first prove that $(2) \Rightarrow (1).$  
Let ${\cal V}$ be the (finite) set of signatures $\transmap(\Delta,\tau)$ of the pairs $(\Delta,\tau)$ as  $\Delta$ ranges over the set of families of
$n\geq 3$ convex bodies indexed by $\{1,2,3,\ldots, n\}$ and where $\tau$ ranges over the set of bitangents of $\Delta$.
We leave the verification of the following property of the set ${\cal V}$ to the reader:
the map that assigns to any element $\transmap(\Delta,\tau)$ of ${\cal V}$ the set of $\transmap(\Delta',\tau)$ where $\Delta'$ ranges over the set of subfamilies of size $n-1$ of 
$\Delta$ is one-to-one: see Table~\ref{latable} for the case $n=3$; 
this proves that $(2) \Rightarrow (1).$ 

\begin{table}[!htb]
$$
\begin{array}{c||c|c|c}
\transmap(123,\tau) &\transmap(12,\tau)&\transmap(13, \tau)&\transmap(23,\tau) \\
\hline\hline 
&&&\\
123\ptg\ptg\ptg 
& 12\ptg\ptg
& 13\ptg\ptg
& 23\ptg\ptg  \\
1\ptg 2\ptg 3\ptg 
& 12\ptg\ptg 
& 31\ptg\ptg
& 23\ptg\ptg \\
12\ptg\ptg,3
& 12\ptg\ptg
& 1\ptg, 3
& 2\ptg,3\\
132 \ptg \overline{3} \ptg
& 12 \ptg \ptg
& \overline{3} 13 \ptg 
& 32\overline{3}\ptg  \\
\overline{3}123 \ptg \ptg 
& 12 \ptg \ptg
& \overline{3}13 \ptg 
& \overline{3}23 \ptg 
\end{array}
$$
\caption{The map that assigns to a bitangent cocycle of a family of three convex bodies its sub-cocycles on two bodies is one-to-one \label{latable}}
\end{table}

We now prove that $(1) \Rightarrow (2).$  
It is sufficient to prove it for families of three bodies.
Let $I$ be the indexing set of $\Delta$, let ${\cal I}$ be the set of pairs $(i,J)$ where $i$ ranges over $I$ and where $J$ ranges 
over the set of $3$-subsets of $I$ that contains $i$, and for $(i,J) \in {\cal I}$
let $C_{i,J}$ be the circular ordering of the bitangents $v_{\alpha}$, 
$\alpha \in \setnodebis{i}{j}$, 
along the (oriented) dual curve
$\Delta_i^*$ of $\Delta_i.$
According to Theorem~\ref{coding} 
the isomorphism class of the dual arrangement of $\Delta$ depends only on the family  of 
$C_{i,J}$, $(i,J) \in {\cal I}.$ Thus proving that $(1) \Rightarrow (2)$ comes  
down to proving that the $C_{i,J}$ depend only on the chirotope of $\Delta.$ This
latter statement is a simple consequence of the following two observations:
\begin{enumerate}
\item for any $j \in I\setminus \{i\}$, the  four vertices $v_{\alpha}$, 
$\alpha \in \setnode{i}{j}$, appear by definition 
in the circular order $v_{\vijone}$,$v_{\vijtwo}$,$v_{\vijthr}$,$v_{\vijfou}$ along $\Delta_i^*;$
\item 
for any $j \in I\setminus \{i\}$ and any  
$\alpha \in \setnode{i}{j}$ 
the position of $v_{\alpha}$ with respect to the $v_{\beta}$, $\beta \in \setnode{i}{k}$, $k \in I\setminus \{i,j\}$, depends only on the chirotope 
of $\Delta$ for the cocycle map is one-to-one for families of two bodies.  
\end{enumerate}
\end{proof}

\begin{remark} It is incorrect to say, as we did in~\cite{G-hp-adp-09}, that cocycle maps are one-to-one.   
However one can show that the space of transversals with given cocycle is connected.  
This can be used, as explained in the forthcoming paper~\cite{gp-pvhtt-13}, to extend to the projective setting one of the Wenger's generalizations of the Hadwiger's Transversal Theorem~\cite{h-uegt-57,w-ghtti-90} : 
{\it Let $\Delta_1,\Delta_2,\ldots,\Delta_n$  be a finite indexed family of at least $4$ pairwise disjoint oriented convex bodies of a
projective plane with the property  that for any quadruplet of indices $i<j<k<l$ there is a line  whose  signature
with respect to the subfamily $\Delta_i,\Delta_j,\Delta_k,\Delta_l$ is $ijkl\overline{ijkl}$. Then there is a line  whose signature
with respect to the family $\Delta$ is $123\ldots n \overline{1}\overline{2}\overline{3}\ldots\overline{n}$.} 
\end{remark}

Fig.~\ref{fulllistdecored6464},~\ref{fulllistdecored0407} and~\ref{fulllistdecored2233} depict the zero-cocycle labeled versions 
of the  thirteen indexed and oriented simple arrangements on three double pseudolines of Fig.~\ref{fulllistdecored0443WCY} and~\ref{fulllistdecored2264WCY}.

\begin{figure}[!htb]
\footnotesize
\centering
\def\factor{0.480}
\psfrag{oZ}{$2$}
\psfrag{Z}{$\name{64}(\ii\kk\jj)$}
\psfrag{one}{$\ii$}
\psfrag{two}{$\kk$}
\psfrag{thr}{$\jj$}
\psfrag{ZA}{$\BBBPBPS{\jj}{\ii}{\bkk}$}
\psfrag{ZB}{$\BBBPBPS{\kk}{\ii}{\jj}$}
\psfrag{ZC}{$\BBBPBPS{\bjj}{\ii}{\kk}$}
\psfrag{ZD}{$\BBBPBPS{\bkk}{\ii}{\bjj}$}
\psfrag{ZE}{$\BBBPBPS{\jj}{\kk}{\ii}$}
\psfrag{ZF}{$\BBBPBPS{\ii}{\jj}{\kk}$}
\psfrag{ZG}{$\BBBPBPF{\bjj}{\kk}{\ii}$}
\psfrag{ZH}{$\BBBPBPF{\ii}{\jj}{\bkk}$}
\psfrag{ZI}{$\BBBPBPF{\ii}{\kk}{\jj}$}
\psfrag{ZJ}{$\BBBPBPS{\bkk}{\jj}{\ii}$}
\psfrag{ZK}{$\BBBPBPF{\kk}{\jj}{\ii}$}
\psfrag{ZL}{$\BBBPBPS{\ii}{\kk}{\bjj}$}
\includegraphics[width = 0.6275\linewidth]{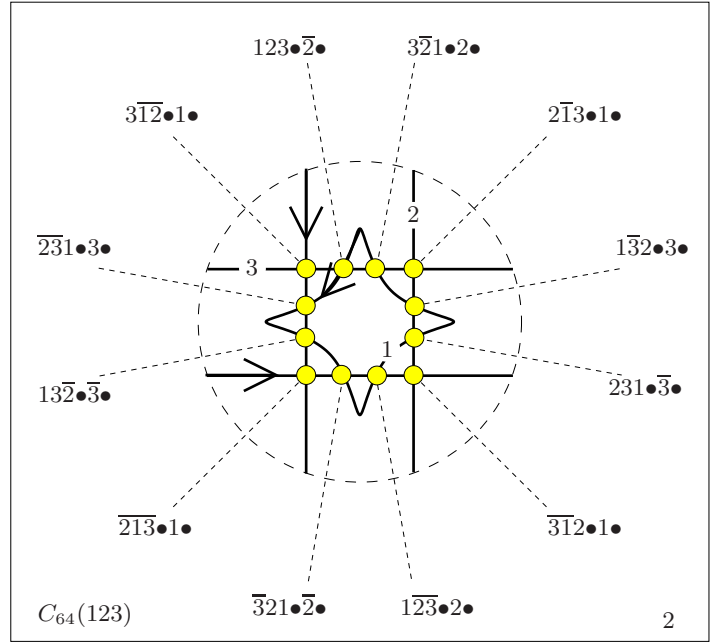}
\caption{The zero-cocycle labeled version of the arrangement $\bname{64}{123}$ \label{fulllistdecored6464}}
\end{figure}

\begin{figure}[!htb]
\footnotesize
\tiny
\def\factor{0.485}
\centering
\psfrag{one}{$\ii$}
\psfrag{two}{$\kk$}
\psfrag{thr}{$\jj$}
\psfrag{name}{$\name{04}(\ii\kk\jj)$}
\psfrag{oA}{$2$}
\psfrag{AA}{$\bbpp{\jj}{\bkk},\ii$}
\psfrag{AB}{$\bbpp{\kk}{\jj},\ii$}
\psfrag{AC}{$\bbpp{\bjj}{\kk},\ii$}
\psfrag{AD}{$\bbpp{\bkk}{\bjj},\ii$}
\psfrag{AF}{$\bbpp{\jj}{\ii},\kk$}
\psfrag{AE}{$\bbpp{\ii}{\kk},\jj$}
\psfrag{AH}{$\bbpp{\bjj}{\ii},\bkk$}
\psfrag{AG}{$\bbpp{\ii}{\bkk},\bjj$}
\psfrag{AJ}{$\bbpp{\ii}{\jj},\bkk$}
\psfrag{AI}{$\bbpp{\bkk}{\ii},\jj$}
\psfrag{AL}{$\bbpp{\kk}{\ii},\bjj$}
\psfrag{AK}{$\bbpp{\ii}{\bjj},\kk$}
\includegraphics[width = \factor\linewidth]{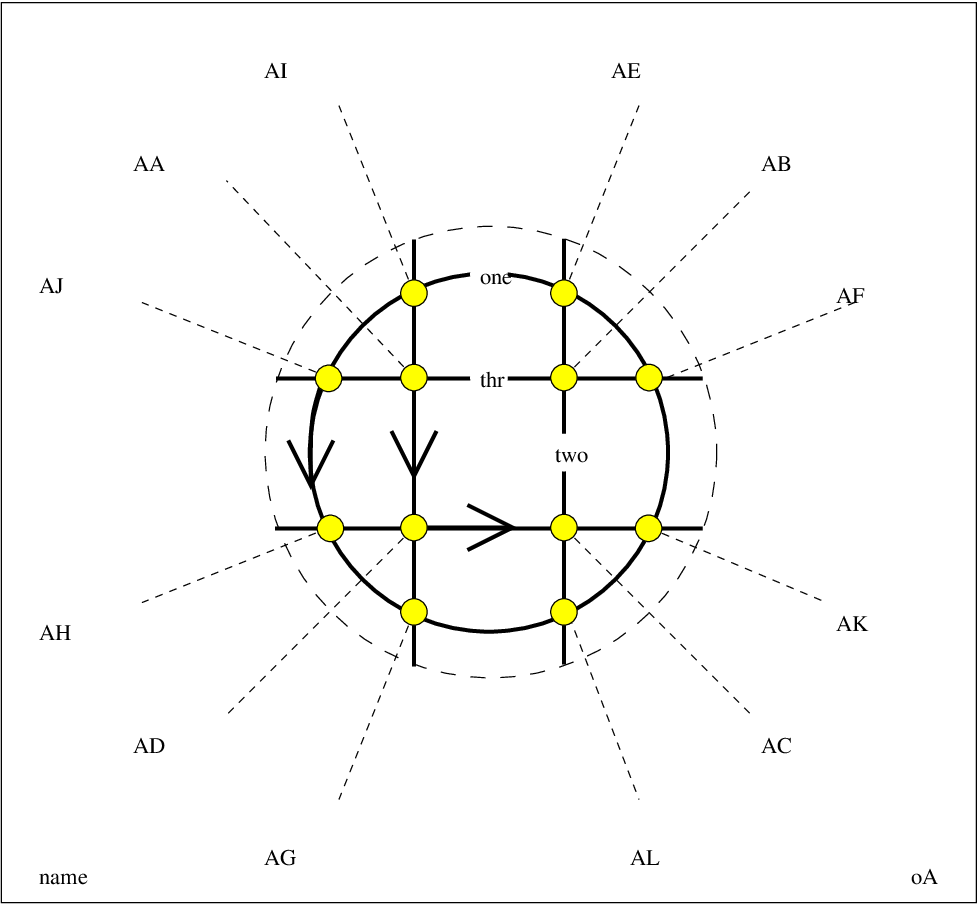}
\psfrag{name}{$\name{07}(\ii\kk\jj)$}
\psfrag{oB}{$8$}
\psfrag{AA}{$\ptg\jj\bkk\ptg,\ii$}
\psfrag{AB}{$\ptg\kk\jj\ptg,\ii$}
\psfrag{AC}{$\ii\ptg\bjj\ii\kk\ptg$}
\psfrag{AF}{$\ii\ptg\ptg\jj,\kk$}
\psfrag{AE}{$\ii\kk\ptg\ptg,\jj$}
\psfrag{AH}{$\ii\ptg\ptg\bjj,\bkk$}
\psfrag{AG}{$\ii\bkk\ptg\ptg,\bjj$}
\psfrag{AJ}{$\ii\jj\ptg\ptg,\bkk$}
\psfrag{AI}{$\ii\ptg\ptg\bkk,\jj$}
\psfrag{AK}{$\ii\ptg\bjj\ptg\kk\jj$}
\psfrag{AL}{$\ii\bkk\bjj\ptg\kk\ptg$}
\psfrag{AA}{$\BBPP{\jj}{\bkk},\ii$}
\psfrag{AB}{$\BBPP{\kk}{\jj},\ii$}
\psfrag{AC}{$\BBPP{\bjj}{\kk},\ii$}
\psfrag{AD}{$\BBPP{\bkk}{\bjj},\ii$}
\psfrag{AF}{$\BBPP{\jj}{\ii},\kk$}
\psfrag{AE}{$\BBPP{\ii}{\kk},\jj$}
\psfrag{AH}{$\BBPP{\bjj}{\ii},\bkk$}
\psfrag{AG}{$\BBPP{\ii}{\bkk},\bjj$}
\psfrag{AJ}{$\BBPP{\ii}{\jj},\bkk$}
\psfrag{AI}{$\BBPP{\bkk}{\ii},\jj$}
\psfrag{AL}{$\BBPP{\kk}{\ii},\bjj$}
\psfrag{AK}{$\BBPP{\ii}{\bjj},\kk$}
\psfrag{ZA}{$\BBBPBPS{\jj}{\ii}{\bkk}$}
\psfrag{ZB}{$\BBBPBPS{\kk}{\ii}{\jj}$}
\psfrag{AC}{$\BBBPBPS{\bjj}{\ii}{\kk}$}
\psfrag{ZD}{$\BBBPBPS{\bkk}{\ii}{\bjj}$}
\psfrag{ZE}{$\BBBPBPS{\jj}{\kk}{\ii}$}
\psfrag{ZF}{$\BBBPBPS{\ii}{\jj}{\kk}$}
\psfrag{ZG}{$\BBBPBPF{\bjj}{\kk}{\ii}$}
\psfrag{ZH}{$\BBBPBPF{\ii}{\jj}{\bkk}$}
\psfrag{ZI}{$\BBBPBPF{\ii}{\kk}{\jj}$}
\psfrag{ZJ}{$\BBBPBPS{\bkk}{\jj}{\ii}$}
\psfrag{AK}{$\BBBPBPF{\kk}{\jj}{\ii}$}
\psfrag{AL}{$\BBBPBPS{\ii}{\kk}{\bjj}$}
\includegraphics[width = \factor\linewidth]{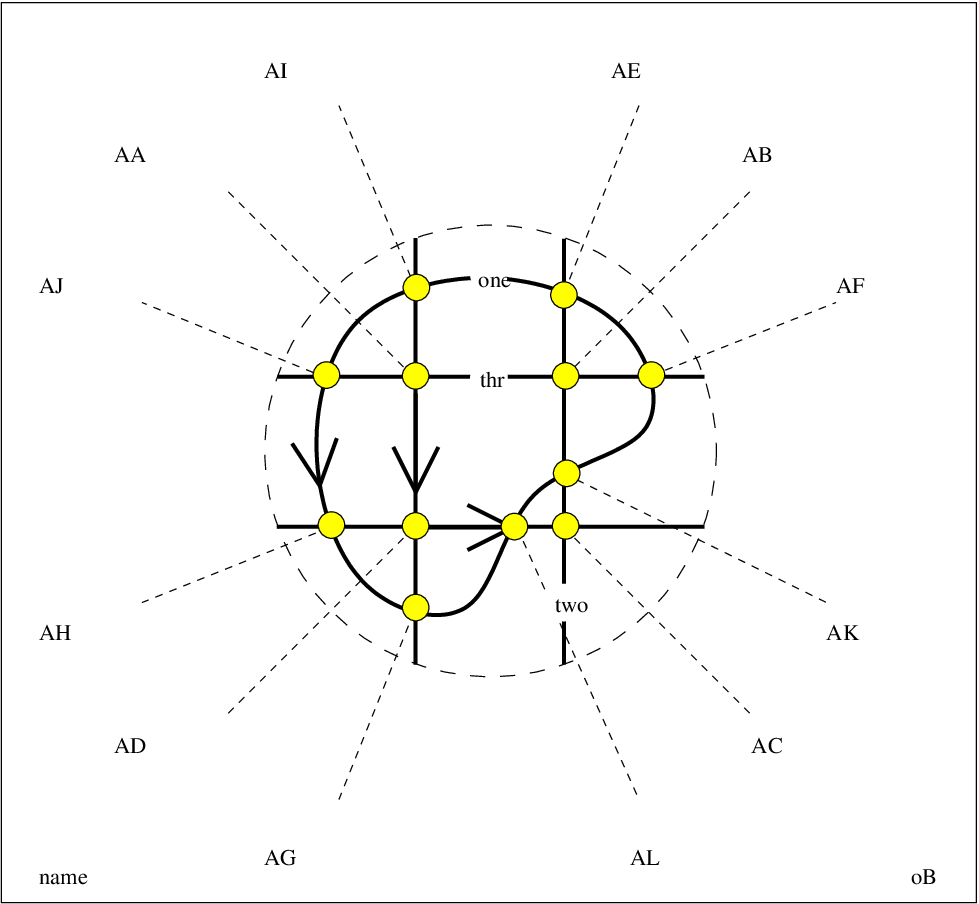}
\psfrag{oC}{$12$}
\psfrag{name}{$\name{18}(\ii\kk\jj)$}
\psfrag{AA}{$\BBPP{\jj}{\bkk},\ii$}
\psfrag{AB}{$\BBBPBPS{\kk}{\ii}{\jj}$}
\psfrag{AC}{$\BBBPBPS{\bjj}{\ii}{\kk}$}
\psfrag{AD}{$\BBPP{\bkk}{\bjj},\ii$}
\psfrag{AE}{$\BBBPBPS{\jj}{\kk}{\ii}$}
\psfrag{AF}{$\BBBPBPS{\ii}{\jj}{\kk}$}
\psfrag{AG}{$\BBPP{\ii}{\bkk},\bjj$}
\psfrag{AH}{$\BBPP{\bjj}{\ii},\bkk$}
\psfrag{AI}{$\BBPP{\bkk}{\ii},\jj$}
\psfrag{AJ}{$\BBPP{\ii}{\jj},\bkk$}
\psfrag{AK}{$\BBBPBPF{\kk}{\jj}{\ii}$}
\psfrag{AL}{$\BBBPBPS{\ii}{\kk}{\bjj}$}
\includegraphics[width = \factor\linewidth]{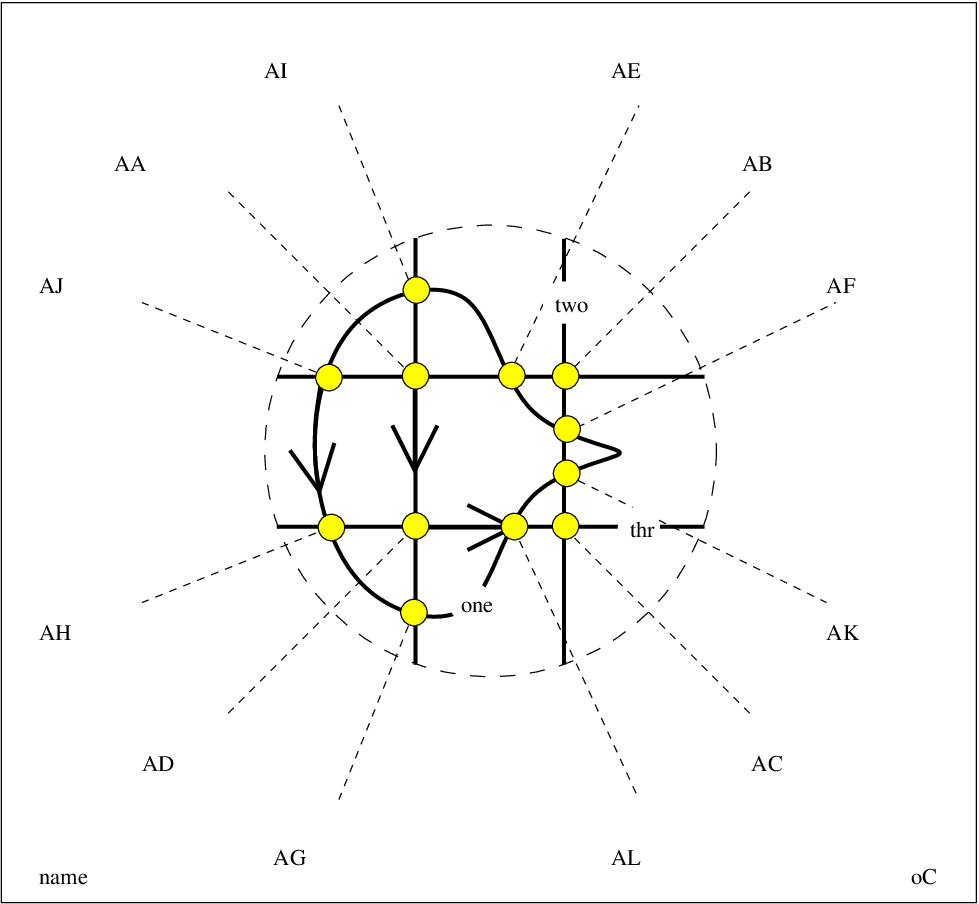}
\psfrag{oG}{$8$}
\psfrag{name}{$\name{37}(\ii\kk\jj)$}
\psfrag{AA}{$\BBPP{\jj}{\bkk},\ii$}
\psfrag{AB}{$\BBPP{\kk}{\jj},\ii$}
\psfrag{AC}{$\BBPP{\bjj}{\kk},\ii$}
\psfrag{AD}{$\BBPP{\bkk}{\bjj},\ii$}
\psfrag{AF}{$\BBPP{\jj}{\ii},\kk$}
\psfrag{AE}{$\BBPP{\ii}{\kk},\jj$}
\psfrag{AH}{$\BBPP{\bjj}{\ii},\bkk$}
\psfrag{AG}{$\BBPP{\ii}{\bkk},\bjj$}
\psfrag{AJ}{$\BBPP{\ii}{\jj},\bkk$}
\psfrag{AI}{$\BBPP{\bkk}{\ii},\jj$}
\psfrag{AL}{$\BBPP{\kk}{\ii},\bjj$}
\psfrag{AK}{$\BBPP{\ii}{\bjj},\kk$}
\psfrag{ZA}{$\BBBPBPS{\jj}{\ii}{\bkk}$}
\psfrag{AB}{$\BBBPBPS{\kk}{\ii}{\jj}$}
\psfrag{AC}{$\BBBPBPS{\bjj}{\ii}{\kk}$}
\psfrag{AD}{$\BBBPBPS{\bkk}{\ii}{\bjj}$}
\psfrag{AE}{$\BBBPBPS{\jj}{\kk}{\ii}$}
\psfrag{AF}{$\BBBPBPS{\ii}{\jj}{\kk}$}
\psfrag{AG}{$\BBBPBPF{\bjj}{\kk}{\ii}$}
\psfrag{AH}{$\BBBPBPF{\ii}{\jj}{\bkk}$}
\psfrag{ZI}{$\BBBPBPF{\ii}{\kk}{\jj}$}
\psfrag{ZJ}{$\BBBPBPS{\bkk}{\jj}{\ii}$}
\psfrag{AK}{$\BBBPBPF{\kk}{\jj}{\ii}$}
\psfrag{AL}{$\BBBPBPS{\ii}{\kk}{\bjj}$}
\includegraphics[width = \factor\linewidth]{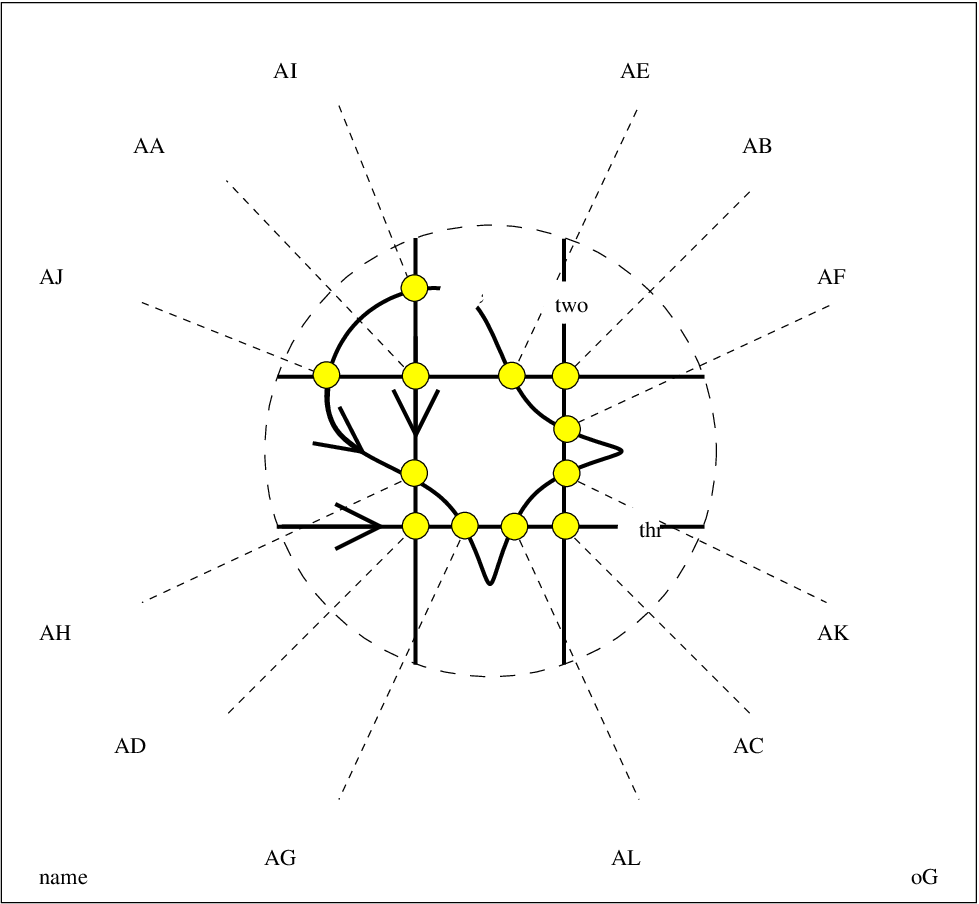}
\psfrag{oH}{$24$}
\psfrag{name}{$\name{15}(\ii\kk\jj)$}
\psfrag{AA}{$\BBBBPPS{\ii}{\jj}{\bkk}$}
\psfrag{AB}{$\BBBPBPS{\kk}{\ii}{\jj}$}
\psfrag{AC}{$\BBPP{\bjj}{\kk},\ii$}
\psfrag{AD}{$\BBPP{\bkk}{\bjj},\ii$}
\psfrag{AE}{$\BBBBPPF{\jj}{\bkk}{\ii}$}
\psfrag{AF}{$\BBBPBPS{\ii}{\jj}{\kk}$}
\psfrag{AH}{$\BBPP{\bjj}{\ii},\bkk$}
\psfrag{AG}{$\BBPP{\ii}{\bkk},\bjj$}
\psfrag{AJ}{$\BBPP{\ii}{\jj},\bkk$}
\psfrag{AI}{$\BBPP{\jj}{\ii},\bkk$}
\psfrag{AL}{$\BBPP{\kk}{\ii},\bjj$}
\psfrag{AK}{$\BBPP{\ii}{\bjj},\kk$}
\includegraphics[width = \factor\linewidth]{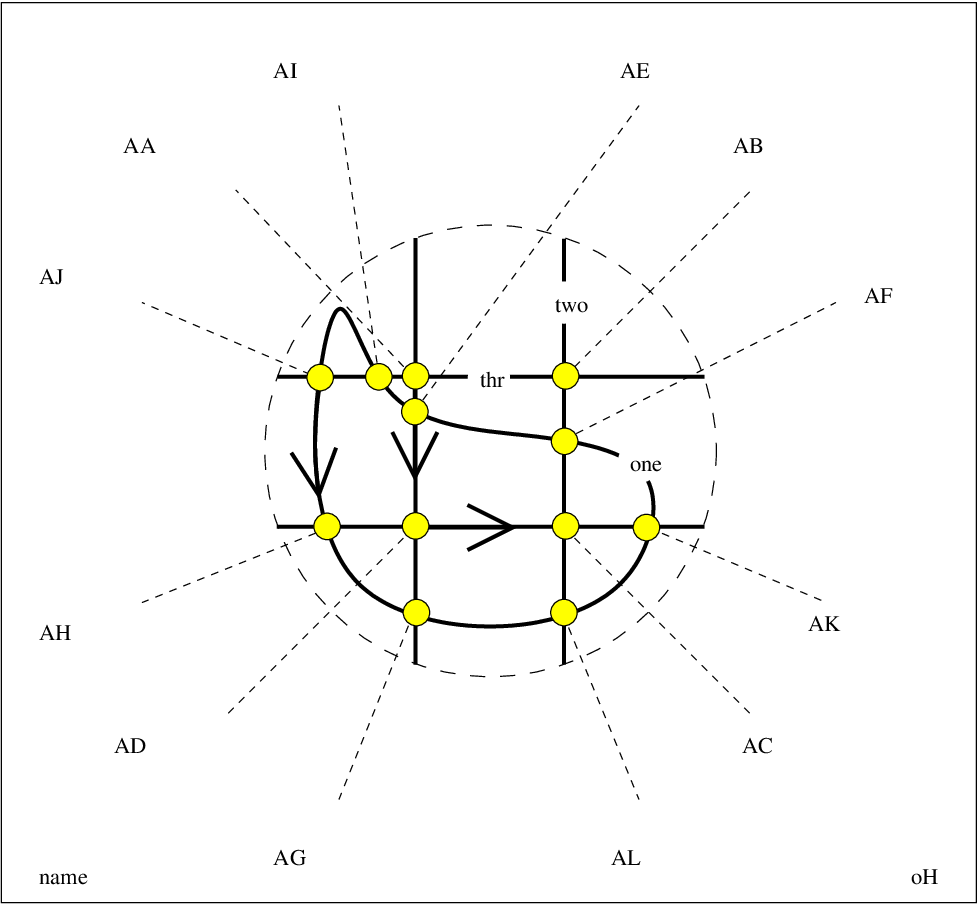}
\psfrag{oJ}{$24$}
\psfrag{name}{$\name{43}(\ii\kk\jj)$}
\psfrag{AA}{$\BBBBPPS{\ii}{\jj}{\bkk}$}
\psfrag{AB}{$\BBBPBPS{\kk}{\ii}{\jj}$}
\psfrag{AC}{$\BBPP{\bjj}{\kk},\ii$}
\psfrag{AD}{$\BBPP{\bkk}{\bjj},\ii$}
\psfrag{AE}{$\BBBBPPF{\jj}{\bkk}{\ii}$}
\psfrag{AF}{$\BBBPBPS{\ii}{\jj}{\kk}$}
\psfrag{AH}{$\BBPP{\bjj}{\ii},\bkk$}
\psfrag{AG}{$\BBPP{\ii}{\bkk},\bjj$}
\psfrag{AJ}{$\BBPP{\ii}{\jj},\bkk$}
\psfrag{AI}{$\BBPP{\jj}{\ii},\bkk$}
\psfrag{AL}{$\BBPP{\kk}{\ii},\bjj$}
\psfrag{AK}{$\BBPP{\ii}{\bjj},\kk$}
\psfrag{ZA}{$\BBBPBPS{\jj}{\ii}{\bkk}$}
\psfrag{ZB}{$\BBBPBPS{\kk}{\ii}{\jj}$}
\psfrag{AC}{$\BBBPBPS{\bjj}{\ii}{\kk}$}
\psfrag{AD}{$\BBBPBPS{\bkk}{\ii}{\bjj}$}
\psfrag{ZE}{$\BBBPBPS{\jj}{\kk}{\ii}$}
\psfrag{ZF}{$\BBBPBPS{\ii}{\jj}{\kk}$}
\psfrag{AG}{$\BBBPBPF{\bjj}{\kk}{\ii}$}
\psfrag{AH}{$\BBBPBPF{\ii}{\jj}{\bkk}$}
\psfrag{ZI}{$\BBBPBPF{\ii}{\kk}{\jj}$}
\psfrag{ZJ}{$\BBBPBPS{\bkk}{\jj}{\ii}$}
\psfrag{AK}{$\BBBPBPF{\kk}{\jj}{\ii}$}
\psfrag{AL}{$\BBBPBPS{\ii}{\kk}{\bjj}$}
\includegraphics[width = \factor\linewidth]{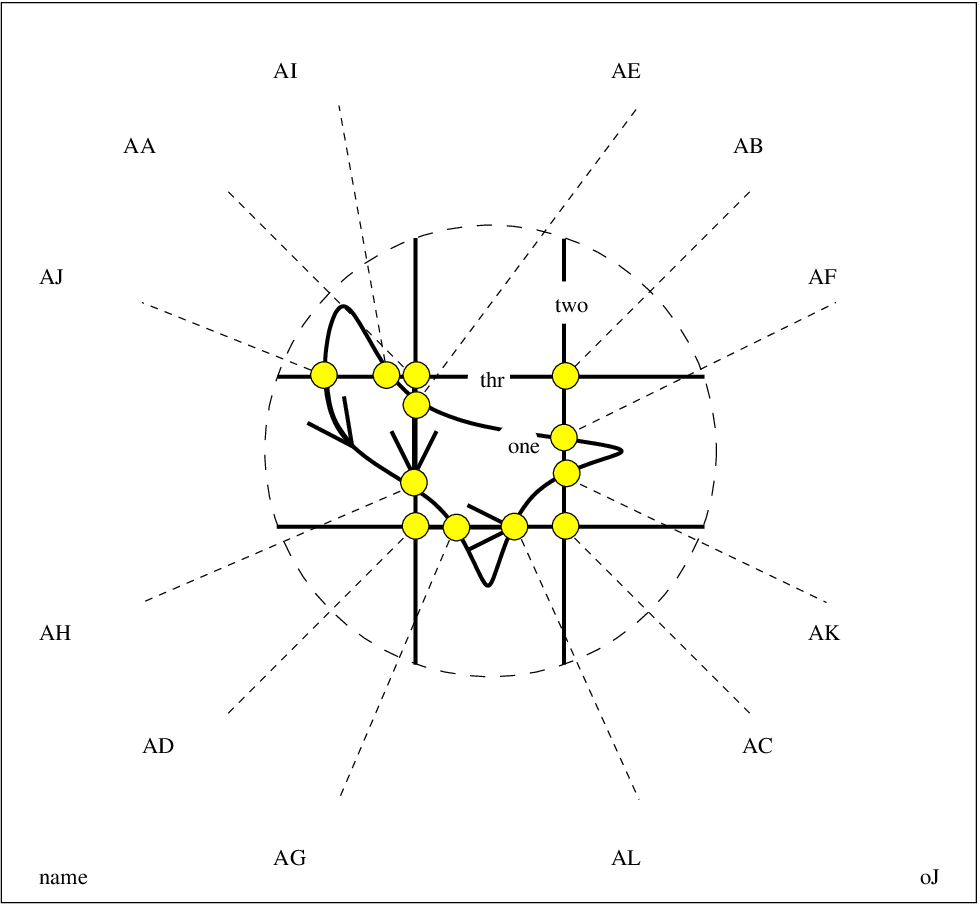}
\caption{Zero-cocycle labeled versions of the arrangements $\name{04}(\ii\kk\jj),\name{07}(\ii\kk\jj),
\name{18}(\ii\kk\jj),\name{37}(\ii\kk\jj),\name{15}(\ii\kk\jj),\name{43}(\ii\kk\jj)$ \label{fulllistdecored0407}}
\end{figure}

\begin{figure}[!htb]
\footnotesize
\tiny
\def\factor{0.495}
\centering
\psfrag{one}{$\ii$}
\psfrag{two}{$\kk$}
\psfrag{thr}{$\jj$}
\psfrag{oP}{$12$}
\psfrag{name}{$\name{22}(\ii\kk\jj)$}
\psfrag{AA}{$\BBBBPPS{\ii}{\jj}{\bkk}$}
\psfrag{AB}{$\BBBPBPS{\kk}{\ii}{\jj}$}
\psfrag{AC}{$\BBBBPPF{\ii}{\bjj}{\kk}$}
\psfrag{AD}{$\BBBPBPF{\bkk}{\ii}{\bjj}$}
\psfrag{AE}{$\BBPP{\ii}{\bjj},\bkk$}
\psfrag{AF}{$\BBPP{\bkk}{\ii}{\bjj}$}
\psfrag{AH}{$\BBPP{\bjj}{\ii},\bkk$}
\psfrag{AG}{$\BBPP{\ii}{\bkk},\bjj$}
\psfrag{AJ}{$\BBPP{\ii}{\jj},\bkk$}
\psfrag{AI}{$\BBPP{\jj}{\ii},\bkk$}
\psfrag{AK}{$\BBPP{\ii}{\kk},\bjj$}
\psfrag{AL}{$\BBPP{\kk}{\ii},\bjj$}
\includegraphics[width = \factor\linewidth]{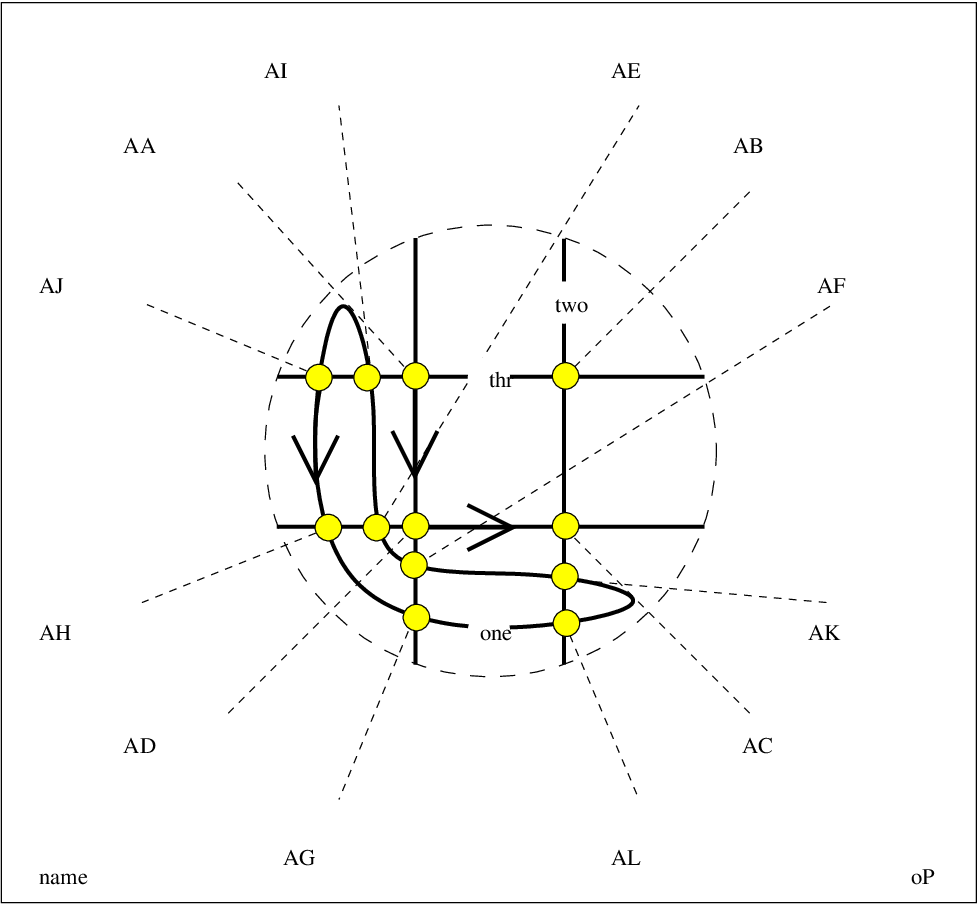}
\psfrag{oL}{$24$}
\psfrag{name}{$\name{33}(\ii\kk\jj)$}
\psfrag{AA}{$\BBBBPPS{\ii}{\jj}{\bkk}$}
\psfrag{AB}{$\BBBPBPS{\kk}{\ii}{\jj}$}
\psfrag{AC}{$\BBBBPPF{\ii}{\bjj}{\kk}$}
\psfrag{AD}{$\BBPP{\bkk}{\bjj},\ii$}
\psfrag{AE}{$\BBBBPPF{\jj}{\bkk}{\ii}$}
\psfrag{AF}{$\BBBBPPS{\kk}{\ii}{\bjj}$}
\psfrag{AH}{$\BBPP{\bjj}{\ii},\bkk$}
\psfrag{AG}{$\BBPP{\ii}{\bkk},\bjj$}
\psfrag{AJ}{$\BBPP{\ii}{\jj},\bkk$}
\psfrag{AI}{$\BBPP{\jj}{\ii},\bkk$}
\psfrag{AK}{$\BBPP{\ii}{\kk},\bjj$}
\psfrag{AL}{$\BBPP{\kk}{\ii},\bjj$}
\psfrag{ZA}{$\BBBPBPS{\jj}{\ii}{\bkk}$}
\psfrag{ZB}{$\BBBPBPS{\kk}{\ii}{\jj}$}
\psfrag{ZC}{$\BBBPBPS{\bjj}{\ii}{\kk}$}
\psfrag{AD}{$\BBBPBPS{\bkk}{\ii}{\bjj}$}
\psfrag{ZE}{$\BBBPBPS{\jj}{\kk}{\ii}$}
\psfrag{ZF}{$\BBBPBPS{\ii}{\jj}{\kk}$}
\psfrag{AG}{$\BBBPBPF{\bjj}{\kk}{\ii}$}
\psfrag{AH}{$\BBBPBPF{\ii}{\jj}{\bkk}$}
\psfrag{ZI}{$\BBBPBPF{\ii}{\kk}{\jj}$}
\psfrag{ZJ}{$\BBBPBPS{\bkk}{\jj}{\ii}$}
\psfrag{ZK}{$\BBBPBPF{\kk}{\jj}{\ii}$}
\psfrag{ZL}{$\BBBPBPS{\ii}{\kk}{\bjj}$}
\includegraphics[width = \factor\linewidth]{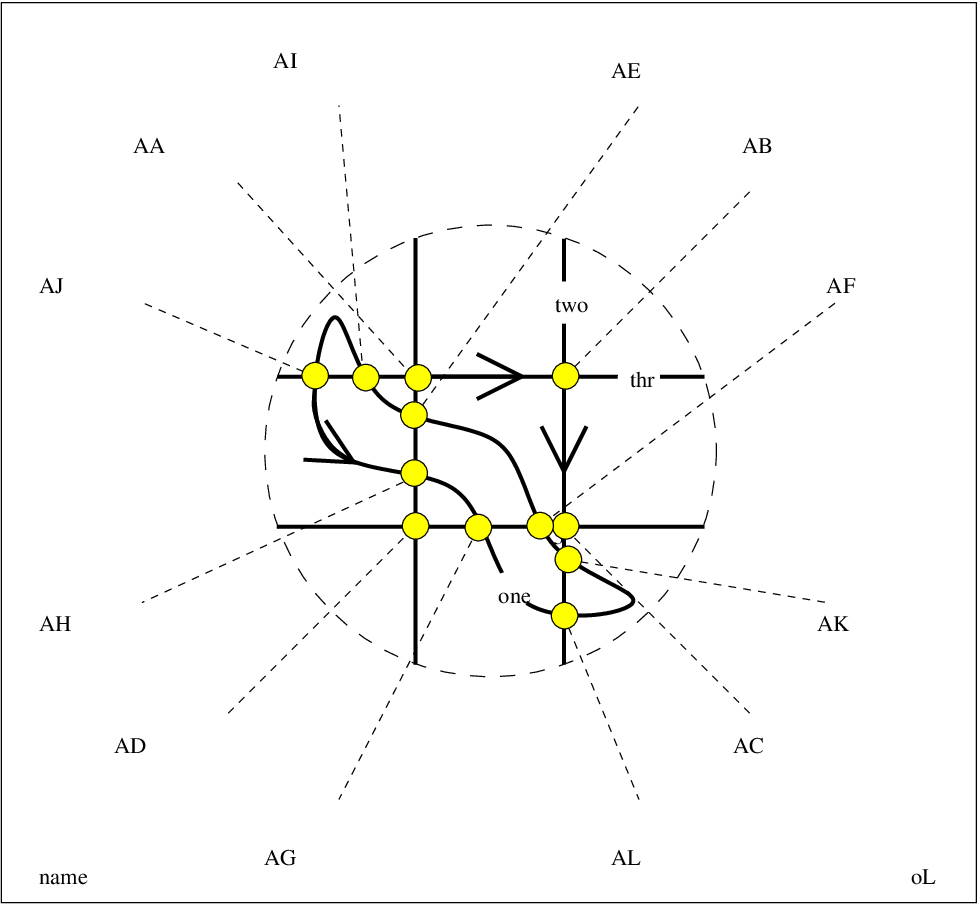}
\psfrag{oO}{$24$}
\psfrag{name}{$\name{32}(\ii\kk\jj)$}
\psfrag{AA}{$\BBBBPPS{\ii}{\jj}{\bkk}$}
\psfrag{AB}{$\BBBPBPS{\kk}{\ii}{\jj}$}
\psfrag{AC}{$\BBBPBPS{\bjj}{\ii}{\kk}$}
\psfrag{AD}{$\BBBBPPF{\ii}{\bkk}{\bjj}$}
\psfrag{AE}{$\BBBBPPF{\jj}{\bkk}{\ii}$}
\psfrag{AF}{$\BBBPBPS{\ii}{\jj}{\kk}$}
\psfrag{AG}{$\BBPP{\ii}{\bjj},\bkk$}
\psfrag{AH}{$\BBPP{\bjj}{\ii},\bkk$}
\psfrag{AI}{$\BBPP{\jj}{\ii},\bkk$}
\psfrag{AJ}{$\BBPP{\ii}{\jj},\bkk$}
\psfrag{AK}{$\BBBPBPF{\kk}{\jj}{\ii}$}
\psfrag{AL}{$\BBBBPPF{\jj}{\ii}{\bkk}$}
\includegraphics[width = \factor\linewidth]{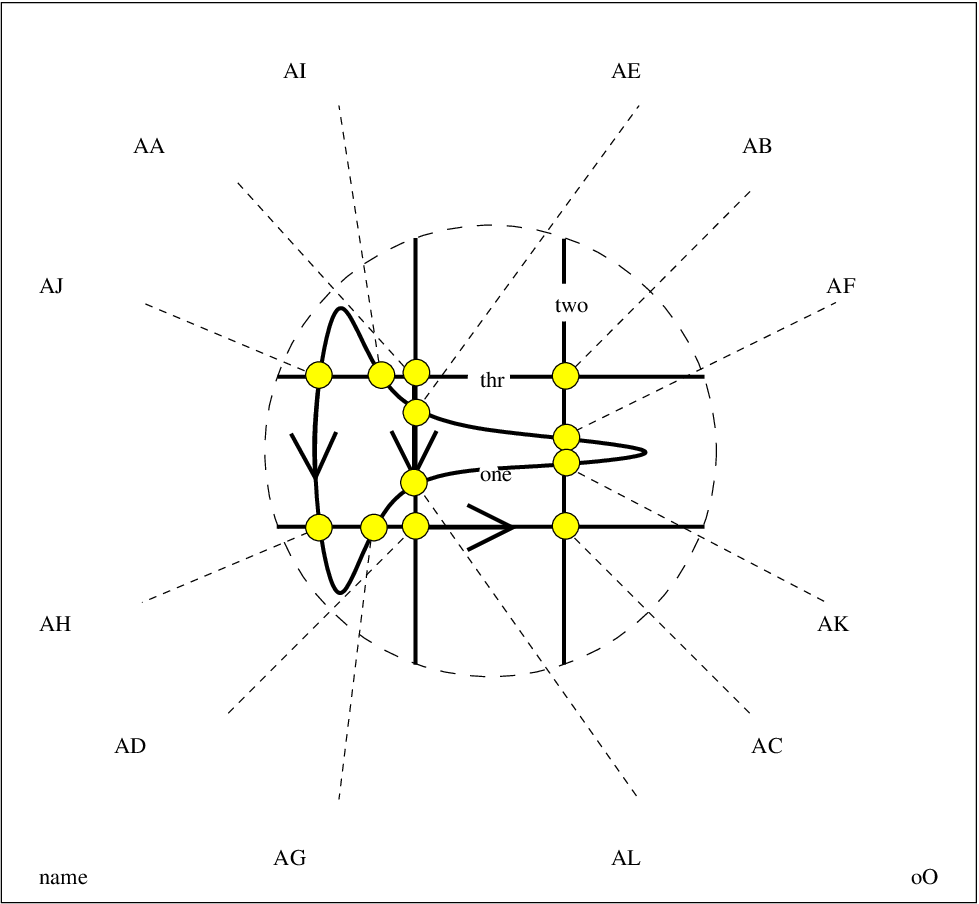}
\psfrag{oN}{$24$}
\psfrag{name}{$\name{{25_2}}(\ii\kk\jj)$}
\psfrag{AA}{$\BBBBPPS{\ii}{\jj}{\bkk}$}
\psfrag{AB}{$\BBBPBPS{\kk}{\ii}{\jj}$}
\psfrag{AC}{$\BBBBPPF{\ii}{\bjj}{\kk}$}
\psfrag{AD}{$\BBPP{\bkk}{\bjj},\ii$}
\psfrag{AE}{$\BBBBPPF{\jj}{\bkk}{\ii}$}
\psfrag{AF}{$\BBBBPPS{\kk}{\ii}{\bjj}$}
\psfrag{AH}{$\BBPP{\bjj}{\ii},\bkk$}
\psfrag{AG}{$\BBPP{\ii}{\bkk},\bjj$}
\psfrag{AJ}{$\BBPP{\ii}{\jj},\bkk$}
\psfrag{AI}{$\BBPP{\jj}{\ii},\bkk$}
\psfrag{AK}{$\BBPP{\ii}{\kk},\bjj$}
\psfrag{AL}{$\BBPP{\kk}{\ii},\bjj$}
\includegraphics[width = \factor\linewidth]{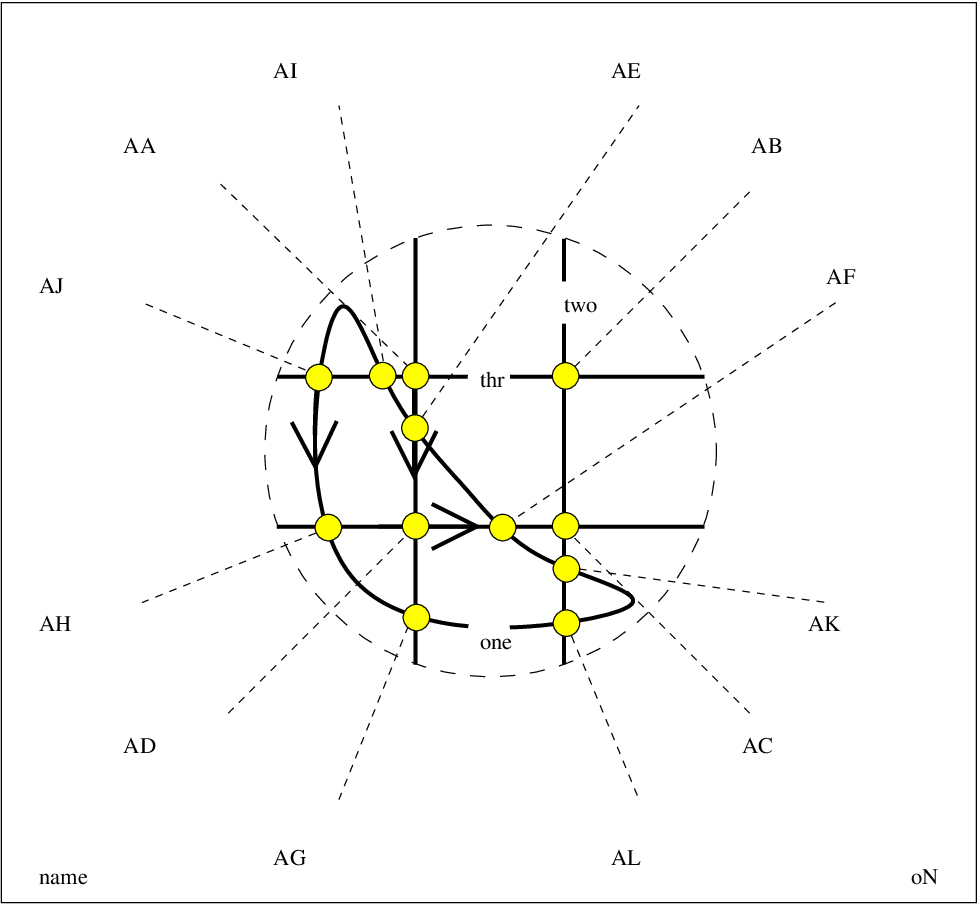}
\psfrag{oNstar}{$48$}
\psfrag{name}{$\name{{25_1}}(\ii\kk\jj)$}
\psfrag{AA}{$\BBPP{\jj}{\bkk},\ii$}
\psfrag{AB}{$\BBPP{\kk}{\jj},\ii$}
\psfrag{AC}{$\BBPP{\bjj}{\kk},\ii$}
\psfrag{AD}{$\BBPP{\bkk}{\bjj},\ii$}
\psfrag{AF}{$\BBPP{\jj}{\ii},\kk$}
\psfrag{AE}{$\BBPP{\ii}{\kk},\jj$}
\psfrag{AH}{$\BBPP{\bjj}{\ii},\bkk$}
\psfrag{AG}{$\BBPP{\ii}{\bkk},\bjj$}
\psfrag{AJ}{$\BBPP{\ii}{\jj},\bkk$}
\psfrag{AI}{$\BBPP{\bkk}{\ii},\jj$}
\psfrag{AL}{$\BBPP{\kk}{\ii},\bjj$}
\psfrag{AK}{$\BBPP{\ii}{\bjj},\kk$}
\psfrag{AD}{$\BBBPBPS{\bkk}{\ii}{\bjj}$}
\psfrag{AG}{$\BBBPBPF{\bjj}{\kk}{\ii}$}
\psfrag{AH}{$\BBBPBPF{\ii}{\jj}{\bkk}$}
\psfrag{AB}{$\BBBPBPS{\kk}{\ii}{\jj}$}
\psfrag{AE}{$\BBBPBPS{\jj}{\kk}{\ii}$}
\psfrag{AC}{$\BBBBPPF{\ii}{\bjj}{\kk}$}
\psfrag{AF}{$\BBBBPPS{\kk}{\ii}{\bjj}$}
\psfrag{AK}{$\BBPP{\ii}{\kk},\bjj$}
\psfrag{AL}{$\BBPP{\kk}{\ii},\bjj$}
\includegraphics[width = \factor\linewidth]{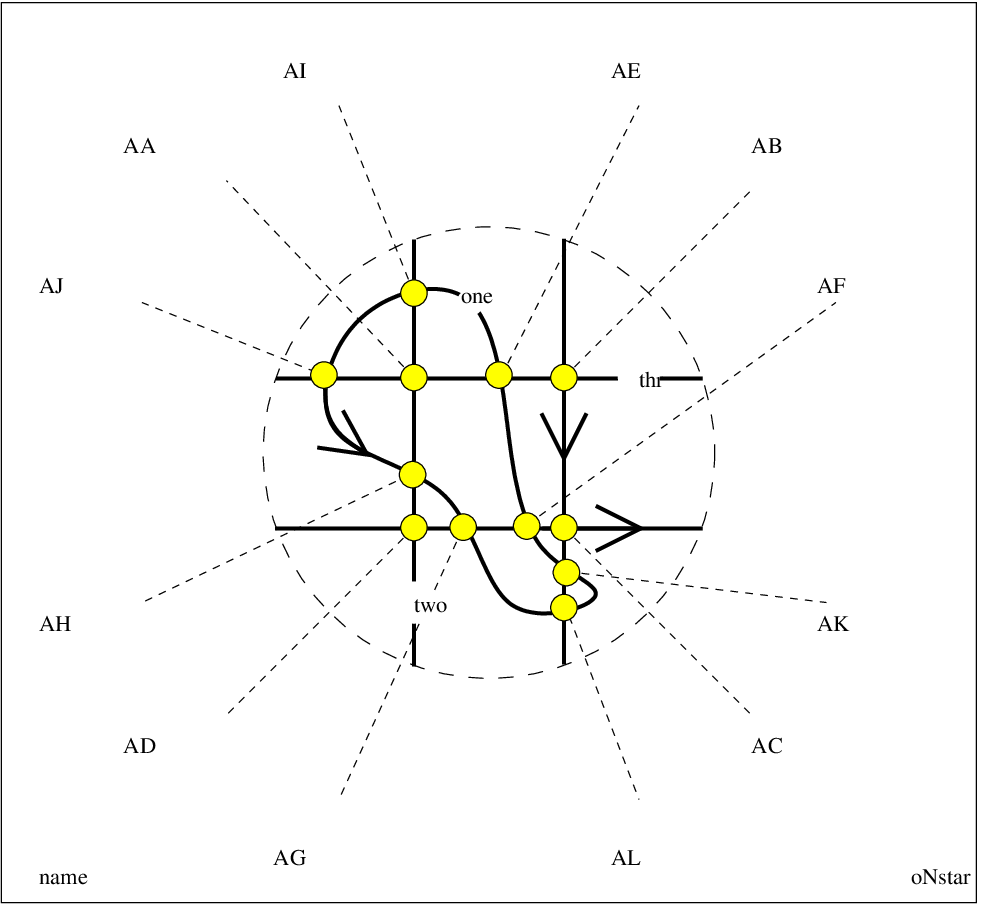}
\psfrag{oR}{$4$}
\psfrag{name}{$\name{36}(\ii\kk\jj)$}
\psfrag{RA}{$\BBBBPPS{\ii}{\jj}{\bkk}$}
\psfrag{RB}{$\BBPP{\kk}{\jj},{\ii}$}
\psfrag{RC}{$\BBBBPPF{\ii}{\bjj}{\kk}$}
\psfrag{RD}{$\BBPP{\bkk}{\bjj},\ii$}
\psfrag{RE}{$\BBBBPPF{\kk}{\ii}{\jj}$}
\psfrag{RF}{$\BBBBPPS{\kk}{\ii}{\jj}$}
\psfrag{RH}{$\BBPP{\bjj}{\ii},\bkk$}
\psfrag{RG}{$\BBPP{\ii}{\bkk},\bjj$}
\psfrag{RJ}{$\BBPP{\ii}{\jj},\bkk$}
\psfrag{RI}{$\BBPP{\jj}{\ii},\bkk$}
\psfrag{RK}{$\BBPP{\ii}{\kk},\bjj$}
\psfrag{RL}{$\BBPP{\kk}{\ii},\bjj$}
\includegraphics[width = \factor\linewidth]{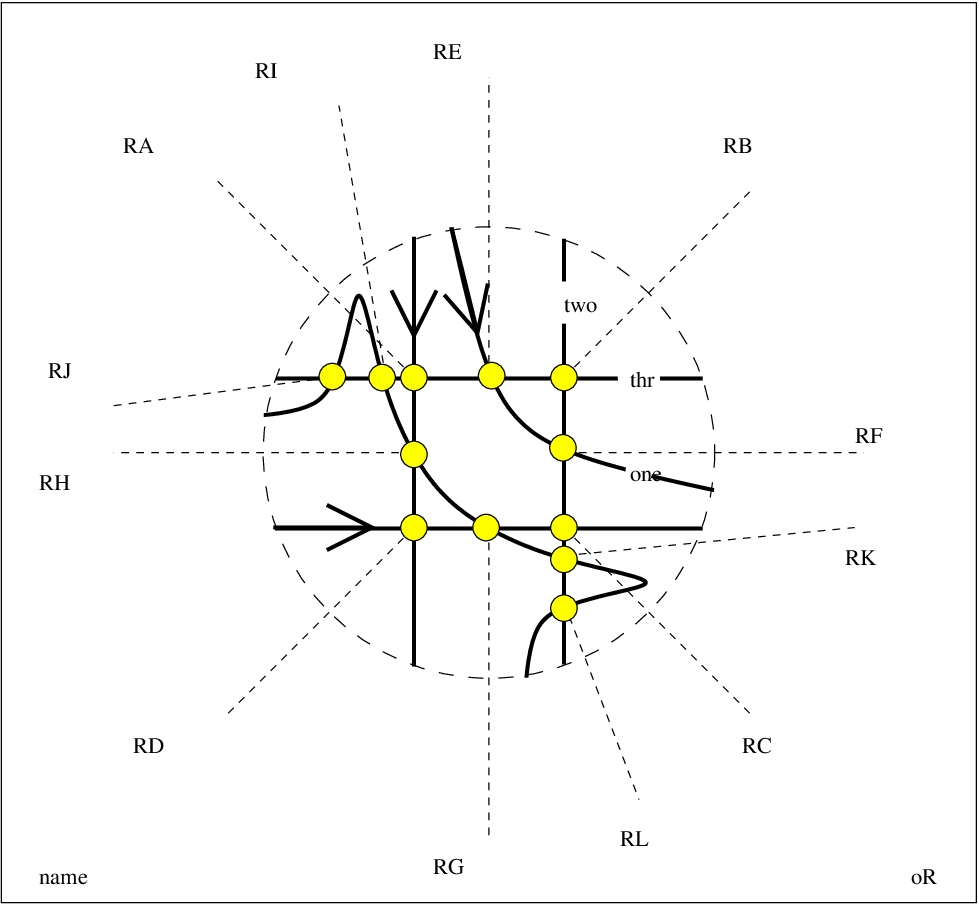}
\caption{Zero-cocycle labeled versions of the arrangements $\name{22}(\ii\kk\jj),\name{33}(\ii\kk\jj),\name{32}(\ii\kk\jj),\name{{25_2}}(\ii\kk\jj),\name{{25_1}}(\ii\kk\jj)$ and $\name{36}(\ii\kk\jj)$
\label{fulllistdecored2233}}
\end{figure}

\phantom{sautdepage}

\clearpage
\section{LR characterization\label{secfiv}}
\label{axioms}
In this section we prove  the LR characterization of chirotopes of double pseudoline arrangements; cf. Theorem~\ref{theoADP}. 
As said in the introduction, the proof goes through the notion of arrangements of double pseudolines 
living in nonorientable surfaces of any genus. In addition, we apply the same proof technique to give two new proofs of the classical 
LR characterization of chirotopes of pseudoline arrangements.

\subsection{Arrangements of genus $1,2,\ldots$ \label{karrang}}
By an arrangement of double pseudolines of genus $g\geq 1$ we mean a finite family $\Gamma$ 
of at least two simple closed curves cellularly embedded in a compact  
nonorientable surface $\SSJ_\Gamma$ of genus $g$ with the property that there exist closed tubular neighborhoods $R_i$ of the $\Gamma_i$ (ribbons for short)
such that  
\begin{enumerate}
\item for any subfamily $\subarrang$ of $\Gamma$ the union of its ribbons, denoted  $R_\subarrang$, is a closed tubular
 neighborhood of the union of its curves; the compact surface 
obtained by attaching topological disks to the boundary curves of $R_\subarrang$, using homeomorphisms for the attaching maps,
 is denoted $\SSJ_\subarrang$;
 \item any subfamily $\subarrang$ of $\Gamma$ of size $2$ considered as embedded not in $\SSJ_\Gamma$ but in $\SSJ_\subarrang$ 
is homeomorphic to the dual arrangement of some  (hence any) configuration of two convex bodies; 
\item for any $\Gamma_i, \Gamma_j\in \Gamma$ the intersection of  the ribbon $R_i$ of $\Gamma_i$ and the disk side of $\Gamma_i$ in the subarrangement $\Gamma_i,\Gamma_j$ 
is independent of $\Gamma_j$. 
\end{enumerate}
Thus arrangements of double pseudolines of genus $1$ are the arrangements of double pseudolines as defined in the previous sections. 
Fig.~\ref{cuffarrang} depicts two embeddings  in $3$-space of the tubular neighborhood of an arrangement of two double pseudolines (thus a union of two ribbons). 
\begin{figure}[!htb]
\footnotesize
\psfrag{one}{\normalsize$1$} \psfrag{two}{\normalsize $2$}
\psfrag{a}{$\nodeone{1}{2}$}
\psfrag{b}{$\nodetwo{1}{2}$}
\psfrag{c}{$\nodethr{1}{2}$}
\psfrag{d}{$\nodefou{1}{2}$}

\psfrag{un}{$1$} \psfrag{de}{$2$} \psfrag{tr}{$3$} \psfrag{qu}{$4$} \psfrag{ci}{$5$}
\centering
\includegraphics[width=0.875\linewidth]{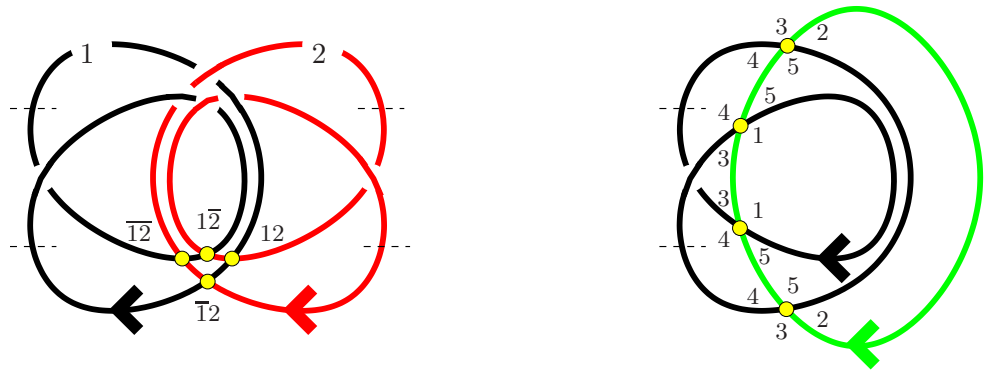}
\caption{\label{cuffarrang}Two embeddings  in $3$-space of the tubular neighborhood of an indexed arrangement of two oriented double pseudolines}
\end{figure}
A horizontal dashed line segment indicates the presence of a half-twist ($180$ degrees) of the ribbon crossed by the line segment and the numbers, in the right diagram, 
label the corners of the polygonal boundary curves  of the neighborhood (the corners of a polygonal boundary curve being labeled by the same number). 
We extend in the natural way to the class of  arrangements of double pseudolines of arbitrary genus  
the notions of thinness,  mutations, isomorphism classes, node cycles (you must {\it not} forget the  binary operation $\otimes$), 
side cycles of disk and crosscap type together with their prime factors, ($\Delta$-)chirotopes, and so one 
associated with the class of arrangements of double pseudolines of  genus $1$. 
As for arrangements of genus $1$, the isomorphism class of an arrangement of any genus depends only on its family of side cycles, with the 
net benefit  that there is now a  very simple characterization of cycles that arose as side cycles of simple arrangements: 
{\it  
a family of circular sequences $D_i$, $i\in I$, is the family of side cycles of disk type of a simple arrangement of oriented double pseudolines indexed by $I$ 
if and only if  the  $D_i$ are shuffles of the elementary circular sequences $\jind{j}\jind{j}\jbar{j}\jbar{j}$, $j\neq i$.} 
The case of any arrangements is hardly more complicated~: only the condition that prime factors occur consistently on side cycles has to be taken in account; the exact formulation 
is postponed to the end of the section. 
In this broader context the (range part of the) LR characterization of chirotopes of arrangements of double pseudolines of genus $1$ is a direct consequence of the two following theorems.
\begin{theorem} \label{LRCanygenus} 
The map which  assigns to an isomorphism class of indexed arrangements of oriented double pseudolines its $4$-chirotope is one-to-one 
and that which assigns its $5$-chirotope is (one-to-one and) onto.\qed
\end{theorem}
\begin{theorem} \label{uptofive}
The class of arrangements of double pseudolines of genus $1$ is the class of arrangements of double pseudolines whose subarrangements 
of size at most $5$ are of genus $1$. \qed
\end{theorem}

In a similar way, we introduce the notion of arrangements of pseudolines of arbitrary genus (ribbons are now crosscaps) and we extend the related terminology~: mutations, 
isomorphism classes, side cycles, ($\Delta$-)chirotopes, and so on. 
Furthermore, exactly as we did for the collection of isomorphism classes of  simple arrangements of pseudolines of genus $1$, 
we embed the collection of isomorphism classes of simple arrangements of pseudolines into the collection 
of isomorphism classes of arrangements of double pseudolines via the support of the isomorphism classes of thin arrangements of double pseudolines. 
Fig.~\ref{cuffarrang} shows an embedding  in $3$-space of the tubular neighborhood of an indexed arrangement of two oriented pseudolines. 
\begin{figure}[!htb]
\psfrag{one}{$1$} \psfrag{two}{$2$}
\psfrag{a}{$\nodeone{1}{2}$}
\psfrag{b}{$\nodetwo{1}{2}$}
\psfrag{c}{$\nodethr{1}{2}$}
\psfrag{d}{$\nodefou{1}{2}$}
\psfrag{un}{$1$} \psfrag{de}{$2$} \psfrag{tr}{$3$} \psfrag{qu}{$4$} \psfrag{ci}{$5$}
\centering
\includegraphics[width=0.4500\linewidth]{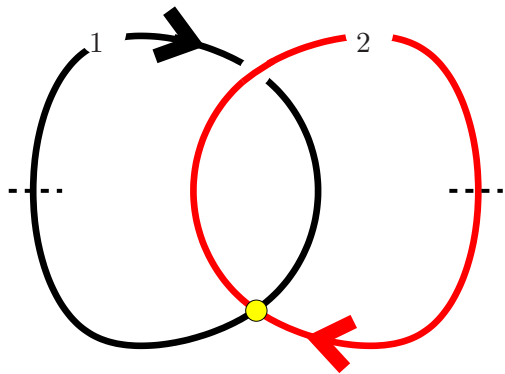}
\caption{\label{cuffarrang} Embedding  in $3$-space of the tubular neighborhood of an indexed arrangement of two oriented pseudolines}
\end{figure}
Again, in this broader context, the classical LR characterization of chirotopes of arrangements of pseudolines of genus $1$ is a direct consequence of the two following theorems.  
\begin{theorem}\label{LRCanygenusPL}
The map which  assigns to an isomorphism class of indexed arrangements of oriented pseudolines its $3$-chirotope is one-to-one
and that which assigns its $5$-chirotope is (one-to-one and) onto. \qed
\end{theorem}
\begin{theorem}\label{uptofour}
The class of arrangements of pseudolines of genus $1$ is the class of arrangements
 pseudolines whose subarrangements of size at most $4$ are of genus $1$. \qed
\end{theorem}

Before proving these theorems (and discuss improved versions of Theorems~\ref{uptofive} and~\ref{uptofour}) we give few examples of arrangements.
\begin{example} Fig.~\ref{noccside}a depicts a family of three curves cellularly embedded in a Klein bottle (decomposed by the curves into $2$ digons, $2$ trigons, $6$ tetragons, $1$ hexagon and $1$ octagon) 
that fulfills condition (2) but not condition (3) of the definition 
of an arrangement of double pseudolines: the disk side of the green curve in the  arrangement composed of the green and red curves 
and 
the disk side of the green curve in the  arrangement composed of the green and black curves intersect the  ribbon of the green curve in two distinct cylinders  
(on the other hand, disk and crosscap sides of the red and black curves are well-defined).
Fig.~\ref{noregcomplex}b depicts an arrangement of three curves,  obtained by adding two twists on the green curve of the configuration  of Fig.~\ref{noccside}a. 
\begin{figure}[!htb]
\footnotesize 
\psfrag{i}{$\jind{i}$} \psfrag{j}{$\jind{j}$} \psfrag{k}{$\jind{k}$}
\psfrag{ib}{$\jbar{i}$} \psfrag{jb}{$\jbar{j}$} \psfrag{kb}{$\jbar{k}$}
\psfrag{un}{$1$} \psfrag{de}{$2$} \psfrag{tr}{$3$}
\psfrag{ontw}{$12$} \psfrag{onth}{$13$} \psfrag{twth}{$23$}
\psfrag{M}{$M$} \psfrag{M}{} \psfrag{C}{$\gamma$}
\psfrag{hei}{$8$}
\psfrag{a}{(a)}
\psfrag{b}{(b)}
\centering
\includegraphics[width=0.95\linewidth]{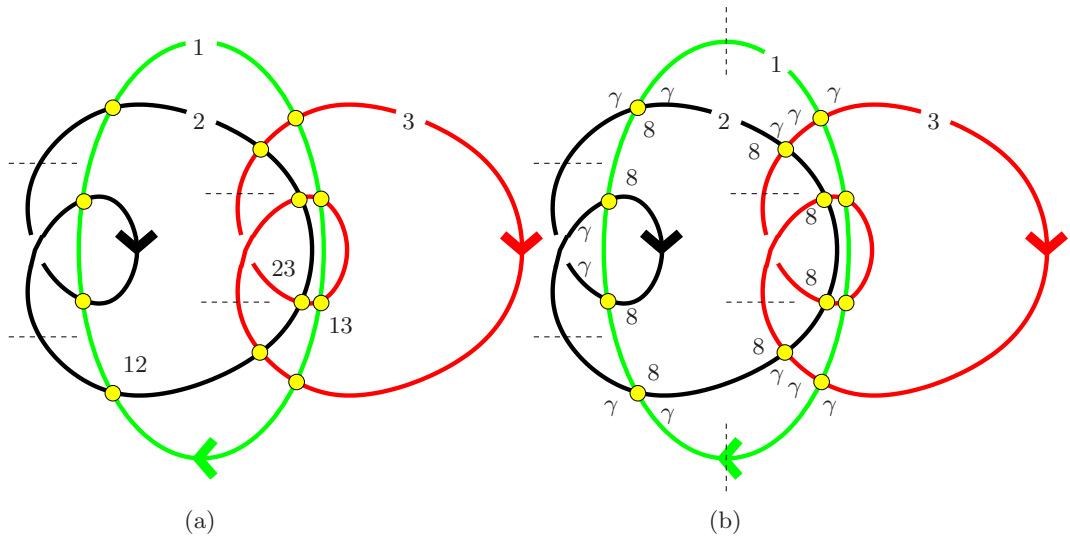}
\caption{\label{noregcomplex}\label{noccside} (a) A family of three curves cellularly embedded in a Klein bottle that fulfills condition (2) but not condition (3) 
of the definition of an arrangement of double pseudolines; (b) An arrangement of three curves living in a double Klein bottle}
\end{figure}
It is composed of 
$2$ digons,  $6$ tetragons, $1$ octagon and $1$ dodecagon (the corners of the octagon are labeled with the numeral $8$ and 
those of the dodecagon by the letter $\gamma$). 
It lives in a double Klein bottle. 
Its node cycles are  the same as those of the previous example, i.e., 
$$\begin{array}{cc}
1 : & \nodeone{1}{2},\nodetwo{1}{2},\nodethr{1}{2},\nodefou{1}{2},\nodethr{1}{3},\nodefou{1}{3},\nodeone{1}{3},\nodetwo{1}{3}\phantom{.}\\
2 : & \nodeone{2}{1},\nodetwo{2}{1},\nodethr{2}{1},\nodefou{2}{1},\nodethr{2}{3},\nodefou{2}{3},\nodeone{2}{3},\nodetwo{2}{3}\phantom{.}\\
3:  & \nodethr{3}{1},\nodethr{3}{2},\nodefou{3}{2},\nodefou{3}{1},\nodeone{3}{1},\nodeone{3}{2},\nodetwo{3}{2},\nodetwo{3}{1}.
\end{array}
$$
\end{example}
\begin{example}
Fig.~\ref{cuffarrangthr} depicts embeddings in $3$-space of tubular neighborhoods of two arrangements on three curves. 
Again the horizontal dashed line segments indicate the presence of half-twists ($180$ degrees) of the ribbons  of the tubular neighborhood. 
\begin{figure}[!htb]
\footnotesize
\psfrag{un}{$8$}
\psfrag{de}{$9$}
\psfrag{u}{\normalsize $1$}
\psfrag{d}{\normalsize $2$}
\psfrag{tr}{\normalsize $3$}
\centering
\includegraphics[width=0.875\linewidth]{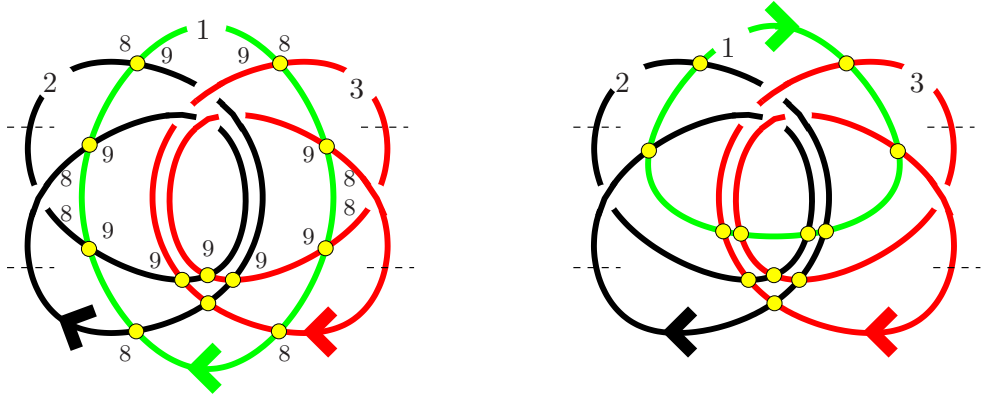}
\caption{\label{cuffarrangthr} Two arrangements on three curves living in a double Klein bottle}
\end{figure}
Both live in a sphere with 4 crosscaps (a double Klein bottle) 
 decomposed by the curves into 1  trigon, 7  tetragons, 1  octagon and  1  nonagon. (In the left diagram the corners of the octagon are labeled by the numeral $8$ and those of the nonagon by the numeral $9$.)
If we orient clockwise the curves and use, respectively, the indices 1, 2 and 3 for the green, blue and red curves, then the side cycles of disk type  of the arrangements are, respectively, 
$$\begin{array}{cc}
1: &\onecrossR{1}{2} \twocrossR{1}{2} \thrcrossR{1}{2} \foucrossR{1}{2} \onecrossR{1}{3} \twocrossR{1}{3} \thrcrossR{1}{3} \foucrossR{2}{3}\\
2: &\onecrossR{2}{3} \twocrossR{2}{3} \onecrossR{2}{1} \twocrossR{2}{1} \thrcrossR{2}{3} \foucrossR{2}{3} \thrcrossR{2}{1} \foucrossR{2}{1}\\
3: &\onecrossR{3}{2} \twocrossR{3}{2} \thrcrossR{3}{1} \foucrossR{3}{1} \thrcrossR{3}{2} \foucrossR{3}{2} \onecrossR{3}{1} \twocrossR{3}{1}
\end{array}
\qquad \text{and} \qquad  
\begin{array}{cc}
1: &\onecrossR{1}{2} \twocrossR{1}{2} \thrcrossR{1}{3} \foucrossR{1}{3} \thrcrossR{1}{2} \foucrossR{1}{2} \onecrossR{1}{3} \twocrossR{1}{3}\phantom{.} \\
2: &\onecrossR{2}{3} \twocrossR{2}{3} \twocrossR{2}{1} \thrcrossR{2}{1} \thrcrossR{2}{3} \foucrossR{2}{3} \foucrossR{2}{1} \onecrossR{2}{1}\phantom{.} \\
3: &\onecrossR{3}{2} \twocrossR{3}{2} \twocrossR{3}{1} \thrcrossR{3}{1} \thrcrossR{3}{2} \foucrossR{3}{2} \foucrossR{3}{1} \onecrossR{3}{1}.
\end{array}
$$
Observe that the first arrangement is a martagon (with respect to and only to the green curve) but the second one is not and 
that these two arrangements are connected by a sequence of four mutations.  
\end{example}

\begin{example}
Fig.~\ref{thinarrangthr} depicts embeddings in $3$-space of tubular neighborhoods of two thin arrangements on three curves. 
\begin{figure}[!htb]
\psfrag{one}{$1$}
\psfrag{two}{$2$}
\psfrag{thr}{$3$}
\centering
\includegraphics[width=0.875\linewidth]{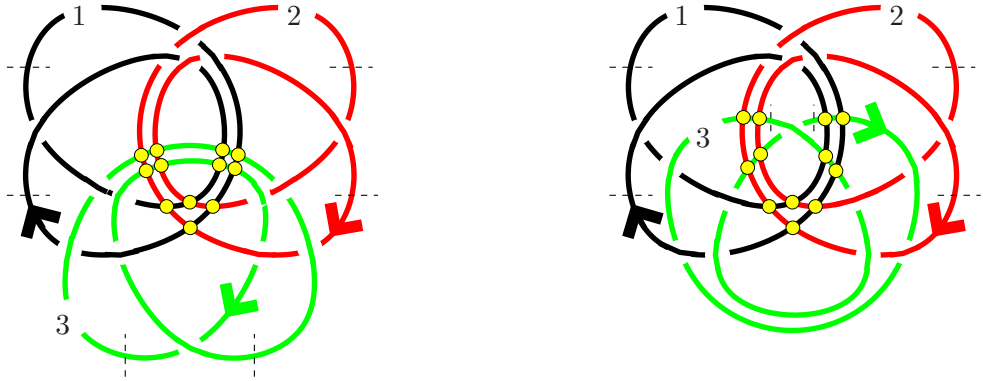}
\caption{\label{thinarrangthr} Two thin arrangements of three double pseudolines}
\end{figure}
The first arrangement lives in a cross surface decomposed by the curves into $4$ trigons and $9$ tetragons, 
and the second one in a surface with  $3$ crosscaps  decomposed by the curves into $2$ hexagons and $9$ tetragons.  Their side cycles of disk type are, respectively,
$$\begin{array}{cc}
1: & \thrcrossR{1}{2} \foucrossR{1}{2} \thrcrossR{1}{3} \foucrossR{2}{3} \onecrossR{1}{2} \twocrossR{1}{2} \onecrossR{1}{3} \twocrossR{1}{3} \\
2: & \thrcrossR{2}{3} \foucrossR{2}{3} \thrcrossR{2}{1} \foucrossR{2}{1} \onecrossR{2}{3} \twocrossR{2}{3} \onecrossR{2}{1} \twocrossR{2}{1} \\
3: & \thrcrossR{3}{1} \foucrossR{3}{1} \thrcrossR{3}{2} \foucrossR{3}{2} \onecrossR{3}{1} \twocrossR{3}{1} \onecrossR{3}{2} \twocrossR{3}{2} 
\end{array}
\qquad \text{and} \qquad  
\begin{array}{cc}
1: & \thrcrossR{1}{2} \foucrossR{1}{2} \thrcrossR{1}{3} \foucrossR{2}{3} \onecrossR{1}{2} \twocrossR{1}{2} \onecrossR{1}{3} \twocrossR{1}{3}\phantom{.}\\
2: & \thrcrossR{2}{3} \foucrossR{2}{3} \thrcrossR{2}{1} \foucrossR{2}{1} \onecrossR{2}{3} \twocrossR{2}{3} \onecrossR{2}{1} \twocrossR{2}{1}\phantom{.} \\
3: & \thrcrossR{3}{1} \foucrossR{3}{1} \onecrossR{3}{2} \twocrossR{3}{2} \onecrossR{3}{1} \twocrossR{3}{1} \thrcrossR{3}{2} \foucrossR{3}{2}. 
\end{array}
$$
These two arrangements are doubles of those of Fig~\ref{arrangPLthr}. 
Note that a family of circular sequences $D_i$, $i\in I$, is the family of side cycles of disk type of a thin arrangement of oriented double pseudolines indexed by $I$ 
if and only if  the  $D_i$ are the images under the morphism $\varphi(x) = xx$ of the side cycles of a simple arrangement of oriented pseudolines  indexed by $I$. 
\end{example}

\begin{example} 
Fig.~\ref{arrangPLthr} depicts embeddings in $3$-space of tubular neighborhoods of two arrangements of three  pseudolines. 
\begin{figure}[!htb]
\psfrag{one}{$1$}
\psfrag{two}{$2$}
\psfrag{thr}{$3$}
\centering
\includegraphics[width=0.875\linewidth]{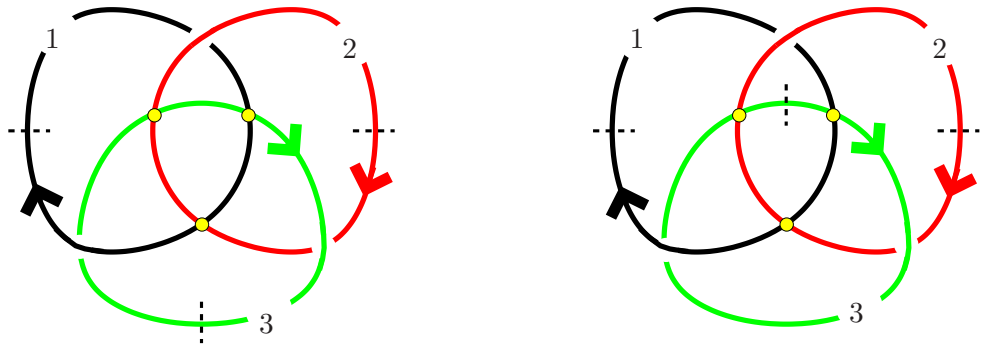}
\caption{\label{arrangPLthr} Two arrangements of three pseudolines}
\end{figure}
The first arrangement lives in a cross surface decomposed by the curves into $4$ trigons, 
and the second one in a surface with  $3$ crosscaps  decomposed by the curves into $2$ hexagons.  Their side cycles are, respectively,
$$\begin{array}{cc}
1: & \foucrossR{1}{2} \foucrossR{2}{3} \twocrossR{1}{2} \twocrossR{1}{3} \\
2: & \foucrossR{2}{3} \foucrossR{2}{1} \twocrossR{2}{3} \twocrossR{2}{1} \\
3: & \foucrossR{3}{1} \foucrossR{3}{2} \twocrossR{3}{1} \twocrossR{3}{2} 
\end{array}
\qquad \text{and} \qquad  
\begin{array}{cc}
1: & \foucrossR{1}{2} \foucrossR{2}{3} \twocrossR{1}{2} \twocrossR{1}{3}\phantom{.} \\
2: & \foucrossR{2}{3} \foucrossR{2}{1} \twocrossR{2}{3} \twocrossR{2}{1}\phantom{.} \\
3: & \foucrossR{3}{1} \twocrossR{3}{2} \twocrossR{3}{1}  \foucrossR{3}{2}. 
\end{array}
$$
These two arrangements are core arrangements of those of Fig~\ref{thinarrangthr}.
Note that a family of circular sequences $D_i$, $i\in I$, is the family of side cycles of a simple arrangement of oriented pseudolines indexed by $I$ 
if and only if  the  $D_i$ are {\it antipodal} shuffles of the elementary circular sequences $\jind{j}\jbar{j}$, $j\neq i$. 
(Here antipodal means that $\jind{j}$ and $\jbar{j}$ occur at positions that differ by the maximum amount, i.e., the cardinality of $I$ minus $1$.) 
\end{example}

\begin{example} 
Fig.~\ref{celldecompici} depicts the cell complex of a simple arrangement of three octagonal curves (colored red, green and purple in colored pdf)  
$$
\begin{array}{ccl}
\tau_1 & = & 
\ansour\anormal\bnsour\bnormal\cnsour\cnormal\dnsour\dnormal\ensour\enormal\fnsour\fnormal\gnsour\gnormal\hnsour\hnormal \\ 
\tau_2 & = & 
\acsour\achapeau\bcsour\bchapeau\ccsour\cchapeau\dcsour\dchapeau\ecsour\echapeau\fcsour\fchapeau\gcsour\gchapeau\hcsour\hchapeau, \\  
\tau_3 & = & 
\atsour\atilde\btsour\btilde\ctsour\ctilde\dtsour\dtilde\etsour\etilde\ftsour\ftilde\gtsour\gtilde\htsour\htilde,
\end{array}
$$
living in a triple cross surface as one can check by calculating the Euler characteristic of the surface. 
In the figure the cell complex is augmented with  its dual graph (oriented arbitrarily at our convenience).
\begin{figure}[!htb]
\psfrag{an}{$\anormal$} \psfrag{ac}{$\achapeau$} \psfrag{at}{$\atilde$}
\psfrag{bn}{$\bnormal$} \psfrag{bc}{$\bchapeau$} \psfrag{bt}{$\btilde$}
\psfrag{cn}{$\cnormal$} \psfrag{cc}{$\cchapeau$} \psfrag{ct}{$\ctilde$}
\psfrag{dn}{$\dnormal$} \psfrag{dc}{$\dchapeau$} \psfrag{dt}{$\dtilde$}
\psfrag{en}{$\enormal$} \psfrag{ec}{$\echapeau$} \psfrag{et}{$\etilde$}
\psfrag{fn}{$\fnormal$} \psfrag{fc}{$\fchapeau$} \psfrag{ft}{$\ftilde$}
\psfrag{gn}{$\gnormal$} \psfrag{gc}{$\gchapeau$} \psfrag{gt}{$\gtilde$}
\psfrag{hn}{$\hnormal$} \psfrag{hc}{$\hchapeau$} \psfrag{ht}{$\htilde$}

\psfrag{one}{\footnotesize \tiny $1$}
\psfrag{two}{\footnotesize \tiny $2$}
\psfrag{thr}{\footnotesize \tiny$3$}
\psfrag{fou}{\footnotesize \tiny$4$}
\psfrag{fiv}{\footnotesize \tiny$5$}
\psfrag{six}{\footnotesize \tiny$6$}
\psfrag{sev}{\footnotesize \tiny$7$}
\psfrag{hei}{\footnotesize \tiny$8$}
\psfrag{nin}{\footnotesize \tiny$9$}
\psfrag{ten}{\footnotesize \tiny$10$}
\psfrag{ele}{\footnotesize \tiny$11$}

\psfrag{vun}{\footnotesize \tiny$\dnsour,\ecsour$}
\psfrag{vde}{\footnotesize \tiny$\dtsour,\ensour$}
\psfrag{vtr}{\footnotesize \tiny$\dcsour,\etsour$}
\psfrag{vqu}{\footnotesize \tiny$\ansour,\hcsour$}
\psfrag{vci}{\footnotesize \tiny$\atsour,\hnsour$}
\psfrag{vsi}{\footnotesize \tiny$\acsour,\htsour$}
\psfrag{vse}{\footnotesize \tiny$\bnsour,\fcsour$}
\psfrag{vhu}{\footnotesize \tiny$\cnsour,\gcsour$}
\psfrag{vne}{\footnotesize \tiny$\btsour,\fnsour$}
\psfrag{vdi}{\footnotesize \tiny$\ctsour,\gnsour$}
\psfrag{von}{\footnotesize \tiny$\bcsour,\ftsour$}
\psfrag{vdo}{\footnotesize \tiny$\ccsour,\gtsour$}
\centering
\includegraphics[width = 0.875\linewidth]{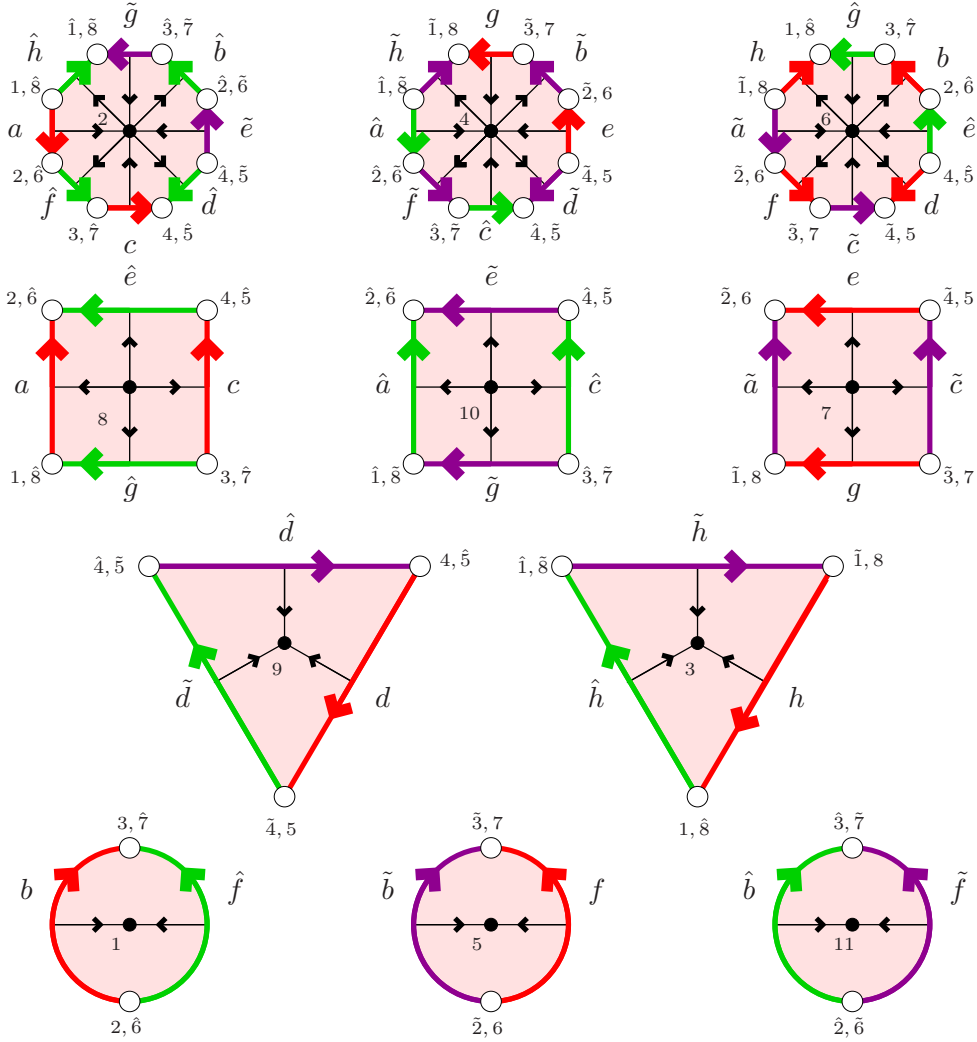}
\caption{An arrangement of three double pseudolines living in a triple cross surface. The double pseudolines are drawn red, green and purple in colored pdf
\label{celldecompici}}
\end{figure}
Using the symbol of an edge of the cell complex to denote its dual
we get a dual  presentation
composed of a system of $12$ equations in $24$ symbols
$$
\begin{array}{rrrr}
\anormal\fchapeau = \echapeau\bnormal & \anormal\hchapeau = \gchapeau\hnormal & \cnormal\fchapeau = \gchapeau \bnormal  & \cnormal\dchapeau = \echapeau\dnormal \\
\atilde \fnormal  = \enormal\btilde   & \atilde\hnormal  = \gnormal\htilde    & \ctilde\fnormal  = \gnormal\btilde      & \ctilde\dnormal  = \enormal \dtilde \\
\achapeau\ftilde  = \etilde\bchapeau  & \achapeau\htilde  = \gtilde\hchapeau  & \cchapeau\ftilde  = \gtilde \bchapeau& \cchapeau \dtilde = \etilde \dchapeau 
\end{array}
$$
providing evidence that this system of curves is a well-defined arrangement, 
as illustrated in Fig.~\ref{twoarrangbis} where the shaded regions denote the crosscap sides of the curves. 
\begin{figure}[!htb]
\psfrag{a}{$\gchapeau$} \psfrag{f}{$\echapeau$} \psfrag{p}{$\ctilde$} \psfrag{u}{$\atilde$} \psfrag{g}{$\cnormal$}
\psfrag{o}{$\anormal$} \psfrag{s}{$\gtilde$} \psfrag{r}{$\etilde$} \psfrag{d}{$\cchapeau$} \psfrag{c}{$\achapeau$}
\psfrag{m}{$\gnormal$} \psfrag{l}{$\enormal$} \psfrag{b}{$\hchapeau$} \psfrag{e}{$\dchapeau$} \psfrag{t}{$\htilde$}
\psfrag{q}{$\dtilde$} \psfrag{h}{$\dnormal$} \psfrag{n}{$\hnormal$} \psfrag{x}{$\fnormal$} \psfrag{y}{$\btilde$}
\psfrag{z}{$\fchapeau$} \psfrag{v}{$\bchapeau$} \psfrag{w}{$\ftilde$} \psfrag{oo}{$\bnormal$}
\psfrag{ap}{$\gchapeau$} \psfrag{fp}{$\echapeau$} \psfrag{pp}{$\ctilde$} \psfrag{up}{$\atilde$} \psfrag{gp}{$\cnormal$}
\psfrag{op}{$\anormal$} \psfrag{sp}{$\gtilde$} \psfrag{rp}{$\etilde$} \psfrag{dp}{$\cchapeau$} \psfrag{cp}{$\achapeau$}
\psfrag{mp}{$\gnormal$} \psfrag{lp}{$\enormal$} \psfrag{bp}{$\hchapeau$} \psfrag{ep}{$\dchapeau$} \psfrag{tp}{$\htilde$}
\psfrag{qp}{$\dtilde$} \psfrag{hp}{$\dnormal$} \psfrag{np}{$\hnormal$} \psfrag{xp}{$\fnormal$} \psfrag{yp}{$\btilde$}
\psfrag{zp}{$\fchapeau$} \psfrag{vp}{$\bchapeau$} \psfrag{wp}{$\ftilde$} \psfrag{oop}{$\bnormal$}

\psfrag{soan}{\footnotesize $\ansour$} \psfrag{sobn}{\footnotesize $\bnsour$} \psfrag{socn}{\footnotesize $\cnsour$} \psfrag{sodn}{\footnotesize $\dnsour$}
\psfrag{soen}{\footnotesize $\ensour$} \psfrag{sofn}{\footnotesize $\fnsour$} \psfrag{sogn}{\footnotesize $\gnsour$} \psfrag{sohn}{\footnotesize $\hnsour$}
\psfrag{soat}{\footnotesize $\atsour$} \psfrag{sobt}{\footnotesize $\btsour$} \psfrag{soct}{\footnotesize $\ctsour$} \psfrag{sodt}{\footnotesize $\dtsour$}
\psfrag{soet}{\footnotesize $\etsour$} \psfrag{soft}{\footnotesize $\ftsour$} \psfrag{sogt}{\footnotesize $\gtsour$} \psfrag{soht}{\footnotesize $\htsour$}
\psfrag{soac}{\footnotesize $\acsour$} \psfrag{sobc}{\footnotesize $\bcsour$} \psfrag{socc}{\footnotesize $\ccsour$} \psfrag{sodc}{\footnotesize $\dcsour$}
\psfrag{soec}{\footnotesize $\ecsour$} \psfrag{sofc}{\footnotesize $\fcsour$} \psfrag{sogc}{\footnotesize $\gcsour$} \psfrag{sohc}{\footnotesize $\hcsour$}

\psfrag{one}{\footnotesize \tiny $1$} \psfrag{two}{\footnotesize \tiny $2$} \psfrag{thr}{\footnotesize \tiny$3$} \psfrag{fou}{\footnotesize \tiny$4$}
\psfrag{fiv}{\footnotesize \tiny$5$} \psfrag{six}{\footnotesize \tiny$6$} \psfrag{sev}{\footnotesize \tiny$7$} \psfrag{hei}{\footnotesize \tiny$8$}
\psfrag{nin}{\footnotesize \tiny$9$} \psfrag{ten}{\footnotesize \tiny$10$} \psfrag{ele}{\footnotesize \tiny$11$} 
\centering
\includegraphics[width = 0.875\linewidth]{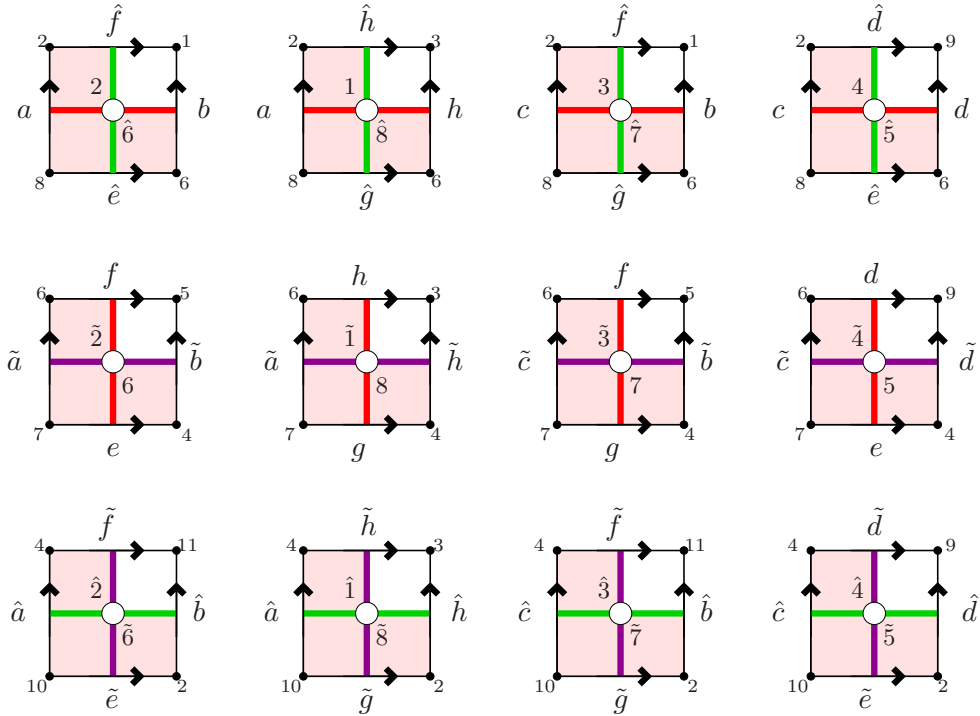}
\caption{An  arrangement of three double pseudolines living in a triple cross surface\label{twoarrangbis}}
\end{figure}
We built it as the simple arrangement with side cycles (of disk type)
$$
\begin{array}{cc}
1: & \thrcrossR{1}{2}\foucrossR{1}{2}\onecrossR{1}{2}\twocrossR{1}{2}\thrcrossR{1}{3}\foucrossR{1}{3}\onecrossR{1}{3}\twocrossR{1}{3}\\
2: & \thrcrossR{2}{3} \foucrossR{2}{3} \onecrossR{2}{3} \twocrossR{2}{3} \thrcrossR{2}{1} \foucrossR{2}{1} \onecrossR{2}{1} \twocrossR{2}{1}\\
3: & \thrcrossR{3}{1} \foucrossR{3}{1} \onecrossR{3}{1} \twocrossR{3}{1} \thrcrossR{3}{2} \foucrossR{3}{2} \onecrossR{3}{2} \twocrossR{3}{2}
\end{array}
$$
(this can be read easily on the dual presentation). Observe that it is a martagon with respect to each of its three curves. 
\end{example}

\clearpage
\begin{example} The thin $3$-chirotope $\chi$ on the indexing set $\{1,2,3,4,5\}$ with entries 
$$
\begin{array}{ccccc}
\bname{04}{123} & \bname{04}{124} & \bname{04}{125} & \bname{04}{134} & \bname{04}{145} \\ 
\bname{04}{234} & \bname{04}{245} & \bname{04}{345} & \bname{04}{153} & \bname{04}{253} 
\end{array}
$$
admits a $4$-extension (i.e., $\chi$ is the restriction of a $4$-chirotope), depicted in Fig.~\ref{fourchi}, 
\begin{figure}[!htb]
\psfrag{one}{$1$} \psfrag{two}{$2$} \psfrag{thr}{$3$} \psfrag{fou}{$4$} \psfrag{fiv}{$5$} \psfrag{one}{$1$}
\psfrag{AA}{$1234$} \psfrag{BB}{$1245$} \psfrag{CC}{$1235$} \psfrag{DD}{$1345$} \psfrag{EE}{$2345$}
\centering
\includegraphics[width=0.85\linewidth]{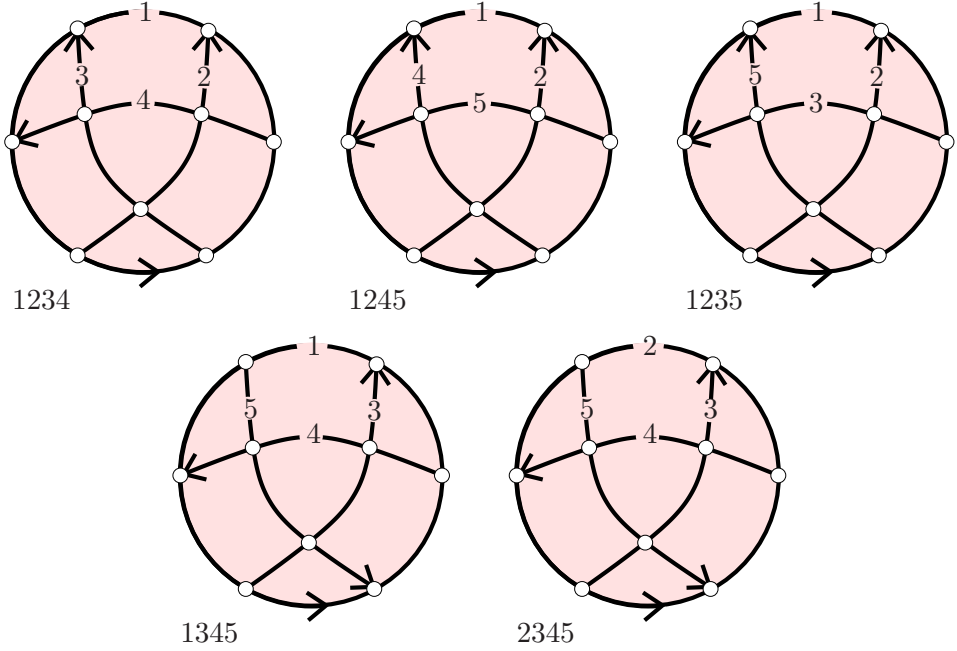}
\caption{A $4$-chirotope on the indexing set $\{1,2,3,4,5\}$  that is not a $5$-chirotope \label{fourchi}}
\end{figure}
but no $5$-extension because 
there is no cycle involving the indices $2,3,4,5$ and their negatives exactly twice in which  
the side cycles of disk type 
$2233\overline{2233}$,
$2244\overline{2244}$,
$2255\overline{2255}$,
$3344\overline{3344}$,
$4455\overline{4455}$, $5533\overline{5533}$ assigned to the index 
$1$ of the entries $\bname{04}{123}$, $\bname{04}{124}$, $\bname{04}{125}$, $\bname{04}{134}$, $\bname{04}{145}$ and $\bname{04}{153}$ of~$\chi$
are subcycles.  The same conclusion holds if we interpret the entries as entries of a $3$-chirotope of pseudoline arrangements.
\end{example}
\begin{example}\label{examplefour}
The simple $3$-chirotope on the indexing set  $\{1,2,3,4\}$ with entries 
$ \bname{64}{123}$, $\bname{64}{124}$, $\bname{64}{134}$, $\bname{64}{234}$ 
is the chirotope of a unique simple arrangement $\Upsilon$ on four curves
whose side cycles of disk type 
are
$$\begin{array}{cc}
1: & \cyclefouR{1}{2}{3}{4}\phantom{.}\\
2: & \cyclefouR{2}{3}{4}{1}\phantom{.}\\
3: & \cyclefouR{3}{4}{1}{2}\phantom{.}\\
4: & \cyclefouR{4}{1}{2}{3}.
\end{array}
$$
The surface is a sphere with 7 crosscaps decomposed by the curves into $19$ two-cells ($12$ digons, $3$ octagons, and $4$ dodecagons) put together according to the following presentation:
$$
\begin{array}{rrcrrc}
1:   &  \aone\ratwo & = 1  &      2:   &  \gone\gtwo  & = 1\\       
3:   &  \itwo\rfthr & = 1  &      4:   &  \ctwo\rlthr & = 1\\       
5:   &  \bthr\rffou & = 1  &      6:   &  \hthr\lfou  & = 1\\       
7:   &  \eone\rbfou & = 1  &      8:   &  \kone\hfou  & = 1\\        
9:   &  \cone\rdthr & = 1  &      10:  &  \ione\jthr  & = 1\\       
11:  &  \ktwo\rdfou & = 1  &      12:  &  \etwo\jfou & = 1\\       
13 : & \rbtwo\bone\rcthr\gfou\lone\rltwo\efou\rathr  & = 1  \\
14 : & \afou\gthr\rjtwo\cfou\rfone\rhtwo\ethr\rdone  & = 1 & 15 : & \ftwo\hone\kthr\dtwo\kfou\ithr\jone\ifou      & = 1\\
16 : & \aone\btwo\rlthr\dtwo\rjfou\ftwo\rgone\htwo\fthr\jtwo\dfou\ltwo & = 1 &  17 : & \atwo\bone\dthr\done\bfou\fone\rgtwo\hone\rjthr\jone\rhfou\lone & = 1 \\
18 : & \ethr \itwo \gthr \rlfou \ithr \rione \kthr \rctwo \athr \ffou \cthr \cone & = 1 & 19 : & \ctwo \ktwo \efou \bthr \gfou \rkone \ifou \retwo \kfou \rhthr \afou \eone
& = 1
\end{array}
$$
where 
$$\begin{array}{ccc}
\Upsilon_1 & = & \aone\bone\cone\done\eone\fone\gone\hone\ione\jone\kone\lone\phantom{.}\\
\Upsilon_2 & = & \atwo\btwo\ctwo\dtwo\etwo\ftwo\gtwo\htwo\itwo\jtwo\ktwo\ltwo\phantom{.}\\
\Upsilon_3 & = & \athr\bthr\cthr\dthr\ethr\fthr\gthr\hthr\ithr\jthr\kthr\lthr\phantom{.}\\
\Upsilon_4 & = & \afou\bfou\cfou\dfou\efou\ffou\gfou\hfou\ifou\jfou\kfou\lfou.
\end{array}
$$
A dual presentation  is given by the following system of 24 equations in 48 symbols
$$\begin{array}{cccc}
\aone\btwo = \atwo\bone & \gtwo\fone = \gone \htwo & \gone\ftwo = \gtwo \fone & \aone\ltwo = \atwo \lone  \\
\cone\ethr = \dthr\done & \jthr\hone = \ione \kthr & \ione\ithr = \jthr \jone & \dthr\bone = \cone \cthr  \\
\eone\cfou = \bfou\fone & \hfou\jone = \kone \ifou & \kone\gfou = \hfou \lone & \bfou\done = \eone \afou  \\
\itwo\gthr = \fthr\jtwo & \ctwo\athr = \lthr \btwo & \ctwo\kthr = \lthr \dtwo & \fthr\htwo = \itwo \ethr  \\
\ktwo\efou = \dfou\ltwo & \jfou\dtwo = \etwo \kfou & \etwo\ifou = \jfou \ftwo & \dfou\jtwo = \ktwo \cfou  \\
\bthr\gfou = \ffou\cthr & \lfou\gthr = \hthr \afou & \hthr\kfou = \lfou \ithr & \ffou\athr = \bthr \efou  
\end{array}
$$
where we use the same symbol to denote an edge and its dual; the dual presentation is also depicted in Fig.~\ref{c64fouhere}.

\begin{figure}[!htb]
\psfrag{A}{} \psfrag{B}{} \psfrag{C}{} \psfrag{D}{}
\psfrag{one}{$ij_1$} \psfrag{two}{$ij_2$} \psfrag{thr}{$ij_3$} \psfrag{fou}{$ij_4$}
\psfrag{one}{} \psfrag{two}{} \psfrag{thr}{} \psfrag{fou}{}
\psfrag{mm}{$--$} \psfrag{mp}{$-+$} \psfrag{pm}{$+-$} \psfrag{pp}{$++$}
\psfrag{mm}{} \psfrag{mp}{} \psfrag{pm}{} \psfrag{pp}{}
\psfrag{ii}{}

\psfrag{A}{$1$} \psfrag{B}{$2$} \psfrag{C}{$3$} \psfrag{D}{$4$}
\psfrag{E}{$5$} \psfrag{F}{$6$} \psfrag{G}{$7$} \psfrag{H}{$8$}
\psfrag{I}{$9$} \psfrag{J}{$10$} \psfrag{K}{$11$} \psfrag{L}{$12$}
\psfrag{M}{$13$} \psfrag{N}{$14$} \psfrag{O}{$15$} \psfrag{P}{$16$}
\psfrag{Q}{$17$} \psfrag{R}{$18$} \psfrag{S}{$19$}

\psfrag{aone}{$\aone$} \psfrag{bone}{$\bone$} \psfrag{cone}{$\cone$} \psfrag{done}{$\done$}
\psfrag{eone}{$\eone$} \psfrag{fone}{$\fone$} \psfrag{gone}{$\gone$} \psfrag{hone}{$\hone$}
\psfrag{ione}{$\ione$} \psfrag{jone}{$\jone$} \psfrag{kone}{$\kone$} \psfrag{lone}{$\lone$}

\psfrag{atwo}{$\atwo$} \psfrag{btwo}{$\btwo$} \psfrag{ctwo}{$\ctwo$} \psfrag{dtwo}{$\dtwo$}
\psfrag{etwo}{$\etwo$} \psfrag{ftwo}{$\ftwo$} \psfrag{gtwo}{$\gtwo$} \psfrag{htwo}{$\htwo$}
\psfrag{itwo}{$\itwo$} \psfrag{jtwo}{$\jtwo$} \psfrag{ktwo}{$\ktwo$} \psfrag{ltwo}{$\ltwo$}

\psfrag{athr}{$\athr$} \psfrag{bthr}{$\bthr$} \psfrag{cthr}{$\cthr$} \psfrag{dthr}{$\dthr$}
\psfrag{ethr}{$\ethr$} \psfrag{fthr}{$\fthr$} \psfrag{gthr}{$\gthr$} \psfrag{hthr}{$\hthr$}
\psfrag{ithr}{$\ithr$} \psfrag{jthr}{$\jthr$} \psfrag{kthr}{$\kthr$} \psfrag{lthr}{$\lthr$}

\psfrag{afou}{$\afou$} \psfrag{bfou}{$\bfou$} \psfrag{cfou}{$\cfou$} \psfrag{dfou}{$\dfou$}
\psfrag{efou}{$\efou$} \psfrag{ffou}{$\ffou$} \psfrag{gfou}{$\gfou$} \psfrag{hfou}{$\hfou$}
\psfrag{ifou}{$\ifou$} \psfrag{jfou}{$\jfou$} \psfrag{kfou}{$\kfou$} \psfrag{lfou}{$\lfou$}

\centering
\includegraphics[width=0.975\linewidth]{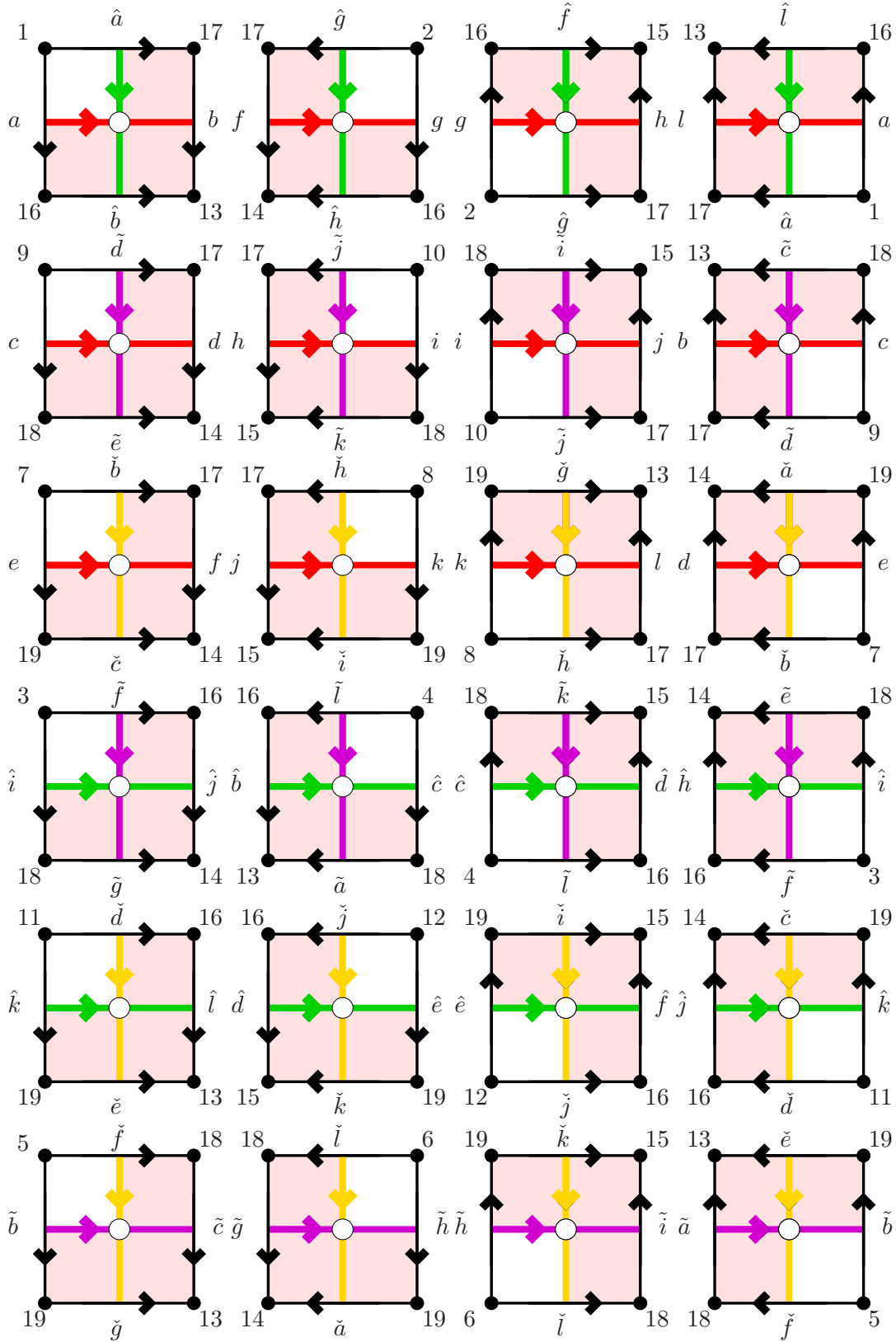}
\caption{The dual presentation of the unique arrangement on four curves whose chirotope is the one with entries $\bname{64}{123}, \bname{64}{124}, \bname{64}{134}, \bname{64}{234}$. 
The underlying surface of this arrangement is a sphere with $7$ crosscaps
 \label{c64fouhere}}
\end{figure}
\end{example}

\clearpage
\begin{example} 
The chirotope of the martagon $M_1(1234)$ is the chirotope of a second martagon $M^*_1(1234)$ defined by the (side) cycles (of disk type)  
$$
\begin{array}{cc}
\AAA: &  
\onecrossR{\AAA}{\BBB} \twocrossR{\AAA}{\BBB} \thrcrossR{\AAA}{\BBB} \foucrossR{\AAA}{\BBB}
\onecrossR{\AAA}{\CCC} \twocrossR{\AAA}{\CCC} \thrcrossR{\AAA}{\CCC} \foucrossR{\AAA}{\CCC}
\onecrossR{\AAA}{\DDD} \twocrossR{\AAA}{\DDD} \thrcrossR{\AAA}{\DDD} \foucrossR{\AAA}{\DDD}\phantom{.} \\
\BBB: & 
\onecrossR{\BBB}{\CCC} \twocrossR{\BBB}{\CCC} \onecrossR{\BBB}{\DDD} \twocrossR{\BBB}{\DDD}
\foucrossR{\BBB}{\AAA} \onecrossR{\BBB}{\AAA} \thrcrossR{\BBB}{\CCC} \foucrossR{\BBB}{\CCC}
\thrcrossR{\BBB}{\DDD} \foucrossR{\BBB}{\DDD} \twocrossR{\BBB}{\AAA} \thrcrossR{\BBB}{\AAA}\phantom{.} \\
\CCC:  &
\onecrossR{\CCC}{\DDD} \twocrossR{\CCC}{\DDD} \onecrossR{\CCC}{\BBB} \twocrossR{\CCC}{\BBB}
\foucrossR{\CCC}{\AAA} \onecrossR{\CCC}{\AAA} \thrcrossR{\CCC}{\DDD} \foucrossR{\CCC}{\DDD}
\thrcrossR{\CCC}{\BBB} \foucrossR{\CCC}{\BBB} \twocrossR{\CCC}{\AAA} \thrcrossR{\CCC}{\AAA}\phantom{.} \\
\DDD: &
\onecrossR{\DDD}{\BBB} \twocrossR{\DDD}{\BBB} \onecrossR{\DDD}{\CCC} \twocrossR{\DDD}{\CCC}
\foucrossR{\DDD}{\AAA} \onecrossR{\DDD}{\AAA} \thrcrossR{\DDD}{\BBB} \foucrossR{\DDD}{\BBB}
\thrcrossR{\DDD}{\CCC} \foucrossR{\DDD}{\CCC} \twocrossR{\DDD}{\AAA} \thrcrossR{\DDD}{\AAA}\phantom{.}
\end{array}
$$
which are obtained from the cycles of $M_1(1234)$ by simply changing the order of the blocks 
$\onecrossR{\AAA}{\BBB} \twocrossR{\AAA}{\BBB} \thrcrossR{\AAA}{\BBB} \foucrossR{\AAA}{\BBB}$,
$\onecrossR{\AAA}{\CCC} \twocrossR{\AAA}{\CCC} \thrcrossR{\AAA}{\CCC} \foucrossR{\AAA}{\CCC}$, and 
$\onecrossR{\AAA}{\DDD} \twocrossR{\AAA}{\DDD} \thrcrossR{\AAA}{\DDD} \foucrossR{\AAA}{\DDD}$ in the cycle indexed by $1$. 
This arrangement lives in a triple cross surface that is decomposed by the curves into 3 digons, 15 tetragons, 3 pentagons, 1 hexagon and 1 nonagon.  
Note that this arrangement has no triangular faces.
Similarly the chirotope of the martagon $M_2(1234)$ is the chirotope of a second martagon $M_2^*(1234)$ defined by the cycles 
$$
\begin{array}{cc}
\AAA: &  
\onecrossR{\AAA}{\BBB} \twocrossR{\AAA}{\BBB} \thrcrossR{\AAA}{\BBB} \foucrossR{\AAA}{\BBB}
\twocrossR{\AAA}{\DDD} \thrcrossR{\AAA}{\DDD} \foucrossR{\AAA}{\DDD} \onecrossR{\AAA}{\DDD}
\onecrossR{\AAA}{\CCC} \twocrossR{\DDD}{\CCC} \thrcrossR{\AAA}{\CCC} \foucrossR{\AAA}{\CCC}\phantom{.}
\\
\BBB: & 
\onecrossR{\BBB}{\CCC} \twocrossR{\BBB}{\CCC} \foucrossR{\BBB}{\AAA} \onecrossR{\BBB}{\AAA}
\twocrossR{\BBB}{\DDD} \thrcrossR{\BBB}{\DDD} \thrcrossR{\BBB}{\CCC} \foucrossR{\BBB}{\CCC}
\twocrossR{\BBB}{\AAA} \thrcrossR{\BBB}{\AAA} \foucrossR{\BBB}{\DDD} \onecrossR{\BBB}{\DDD}\phantom{.}
\\
\CCC: & 
\onecrossR{\CCC}{\BBB} \twocrossR{\CCC}{\BBB} \foucrossR{\CCC}{\DDD} \onecrossR{\CCC}{\DDD}
\foucrossR{\CCC}{\AAA} \onecrossR{\CCC}{\AAA} \thrcrossR{\CCC}{\BBB} \foucrossR{\CCC}{\BBB}
\twocrossR{\CCC}{\DDD} \thrcrossR{\CCC}{\DDD} \twocrossR{\CCC}{\AAA} \thrcrossR{\CCC}{\AAA}\phantom{.}
\\
\DDD: & 
\onecrossR{\DDD}{\BBB} \twocrossR{\DDD}{\BBB} \foucrossR{\DDD}{\AAA} \onecrossR{\DDD}{\AAA}
\thrcrossR{\DDD}{\BBB} \foucrossR{\DDD}{\BBB} \thrcrossR{\DDD}{\CCC} \foucrossR{\DDD}{\CCC}
\twocrossR{\DDD}{\AAA} \thrcrossR{\DDD}{\AAA} \onecrossR{\DDD}{\CCC} \twocrossR{\DDD}{\CCC}. \end{array}
$$
This arrangement lives in a triple cross surface,  decomposed by the curves into 4 digons, 14 tetragons, 3 pentagons, 1 octagon and 1 nonagon.  
Graphical representations of (tubular neighborhoods of) these arrangements are given in Fig.~\ref{FinalMartagonQuater}.
\end{example}

\begin{figure}[!htb]
\centering
\psfrag{4}{} \psfrag{6}{} \psfrag{3}{} \psfrag{2}{}
\psfrag{O}{$326020$} \psfrag{P}{$228010$}
\psfrag{PNS1}{$\Gamma^{1}(X,Y,Z)$}
\psfrag{PNS10}{$\Gamma^{10}(X,Y,Z)$} \psfrag{PNS2}{$\Gamma^2(X,Y,Z)$}
\psfrag{PNS11}{$\Gamma^{11}(X,Y,Z)$}
\psfrag{one}{$1$}
\psfrag{two}{$2$}
\psfrag{thr}{$3$}
\psfrag{fou}{$4$}
\psfrag{Mone}{$M_1$}
\psfrag{Mtwo}{$M_2$}
\psfrag{Monestar}{$M_1^*$}
\psfrag{Mtwostar}{$M_2^*$}
\includegraphics[width = 0.750 \linewidth]{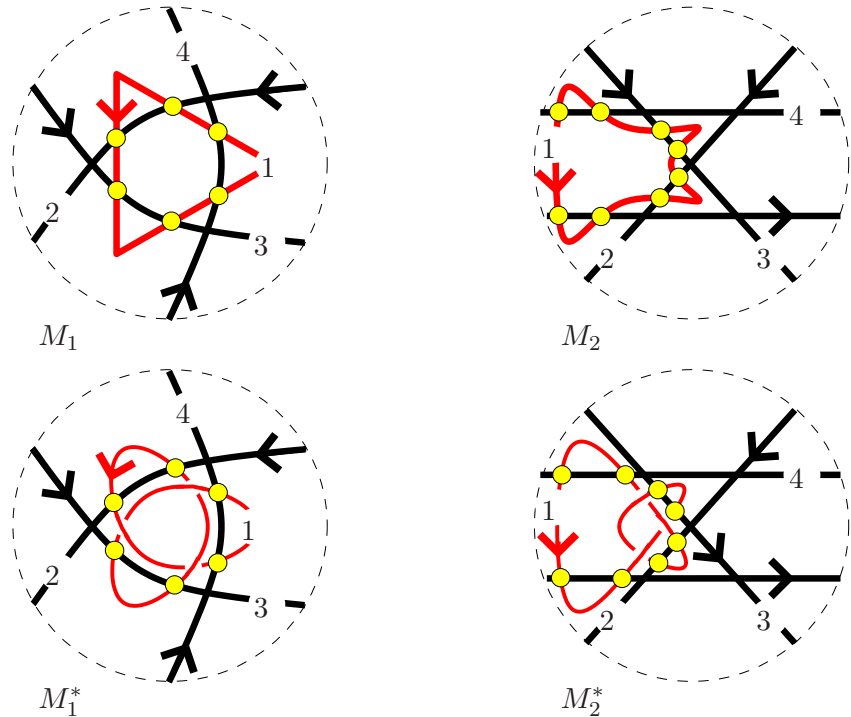}
\caption{The four martagons on four double pseudolines. Two live in a cross surface and two in a triple cross surface
\label{FinalMartagonQuater}}
\end{figure}

\begin{example} 
The $3$-chirotope on the indexing set $\{1,2,3,4\}$  with entries  
$ \bname{32}{123}$,
$\bname{32}{124}$,
$\bname{32}{134}$,
$\bname{32}{234}$
is not the chirotope of an arrangement because the $3$ side cycles indexed by $3$ of the entries 
$\bname{32}{123}$,
$\bname{32}{134}$,
$\bname{32}{234}$, namely
$\cyclethrTHRTWOthrR{1}{2}{3}$,
$\cyclethrTHRTWOtwoR{1}{3}{4}$, and 
$\cyclethrTHRTWOtwoR{2}{3}{4}$, understood as partial circular orders on the indices $1,2,4$ and their negatives, 
are incompatible.
Similarly for the 
$3$-chirotope on $\{1,2,3,4\}$ with entries $\bname{22}{123}$,$\bname{22}{423}$, $\bname{32}{124},\bname{32}{134}$. 
\end{example}

\begin{example} 
Example \ref{examplefour}  generalizes to any number of indices, i.e., the $\bname{64}{ijk}$, where $1\leq i< j< k \leq n$,
are the entries of the chirotope of an arrangement on $n$ curves.  For $n=5,6,7,8,9$ 
we get surfaces of genus $14,21,33,43,58$ decomposed by the curves into 
twenty  $2$-gons, one  $5$-gon, one  $10$-gon, five  $16$-gons, one  $25$-gon for $n=5$; 
thirty  2-gons, five  12-gons and six  20-gons for $n=6$; forty-two  2-gons, one  7-gon, two  14-gons, seven  24-gons, one  49-gon for $n =7$; 
fifty-six  2-gons, seven  16-gons, eight  28-gons, for $n= 8$; and  seventy-two  2-gons, one  9-gon, three  18-gons, three  27-gons, nine  32-gons for $n=9$.
\end{example}

Back to the proof of Theorems~\ref{LRCanygenus},~\ref{uptofive},~\ref{LRCanygenusPL}, and~\ref{uptofour}.

The proof needs some preparations.
Given a sequence $B = \{i_1j_1\}\{i_2j_2\}\{i_3j_3\}\ldots \{i_kj_k\}$, $i_l\in \{i,\overline{i}\}$, we denote by 
$\roll{B}{j_p}$ the sequence  $\alpha_{p+1}\alpha_{p+2}\ldots \alpha_{k}\alpha_p\alpha_1\alpha_2\ldots \alpha_{p-1}$ 
or its reverse $\alpha_{p-1}\alpha_{p-2}\ldots \alpha_2\alpha_{1}\alpha_p\alpha_{k_m}\ldots \alpha_{p+2}\alpha_{p+1}$, depending on whether $i_p = \jbar{i}$ or $i_p = i$,  
where $$\alpha_q = \begin{cases} 
             \{i_pj_{p}\}\otimes \{i_qj_{q}\}  & \text{if $p< q \leq k$}\\
             \{i_pj_{p}\}  & \text{if $q =p$}\\
             \{i_qj_{q}\}\otimes \{i_pj_{p}\} & \text{if $1\leq q<p$.}
\end{cases}
$$
We leave the verification of the following property to the reader: if $B$ is a prime factor of the side cycle of disk type indexed by $i$ 
 of an arrangement of double pseudolines then $\roll{B}{j_p}$ is the corresponding prime factor of the side cycle of disk type indexed by $j_{p}$. 
Routine considerations that we leave to the reader yield the following characterizations of families of cycles that arose as families of 
side cycles of arrangements of double pseudolines and those that arose as families of
side cycles of arrangements of pseudolines.  
\begin{lemma}\label{netbenefit}
 Let  $I$ be a finite set of (at least 2) indices,  let $D_i$ and $M_i$, $i\in I$, 
be two families of circular words on the signed version $\hat{I}$ of $I$, with  the property that  
$D_i$ and $M_i$ are shuffles of the elementary cycles $\jind{j}\jind{j}\jbar{j}\jbar{j}$, $j\neq i$, 
let $S_i$ be the result of  replacing in $D_i$ the linear subsequences $\jbar{j}\jbar{j}\jind{j}\jind{j}$, $j \neq i$,
by the linear sequences  $\{\nodeone{i}{j}\}\{\nodetwo{i}{j}\}\{\nodethr{i}{j}\}\{\nodefou{i}{j}\}$, let  $T_i$ be 
the result of  replacing in $M_i$  linear subsequences $\jbar{j}\jbar{j}\jind{j}\jind{j}$, $j\neq i$, by the linear sequences
 $\{\nodeone{i}{j}\}\{\nodetwo{i}{j}\}\{\nodethr{i}{j}\}\{\nodefou{i}{j}\}$, and let $S_{\jbar{i}}$ and $T_{\jbar{i}}$ be the reversal of $S_{i}$ and $T_{i}$, respectively.
Then the $D_i$ and $M_i$ are the side cycles of disk type and crosscap type of an arrangement $\Gamma$ of oriented double pseudolines indexed by $I$ 
if and only if there exist block decompositions $$B_{i1}B_{i2}\ldots B_{in_i}$$ of the $S_i$, $i\in I$,
where 
$B_{im} = \{i_1j_{m1}\}\{i_2j_{m2}\}\{i_3j_{m3}\}\ldots \{i_kj_{mk_m}\}$, $1\leq m\leq n_i$, 
with $j_{m_l} \notin \{j_{m_{l'}},\overline{j}_{m_{l'}}\}$ for all $1\leq l<l'\leq k_m$,
 such that 
\begin{enumerate}
\item $T_i = B^*_{i1}B^*_{i2}\ldots B^*_{in_i}$ where 
$B^*_{im}$ is the reversal of $B_{im}$ (note that the pair $S_i,T_i$  determine their block decompositions);
\item $\roll{{(B_{im})}}{j_{mp}}$ 
is one of the blocks of the block decomposition  of $S_{j_{mp}}.$ \qed
\end{enumerate}
\end{lemma}

\begin{lemma}\label{netbenefitPL}
Let  $I$ be a finite set of (at least 2) indices,  let $C_i$, $i\in I$, 
be a family of circular words on the signed version $\hat{I}$ of $I$, with  the property that  
the $C_i$ are shuffles of the elementary cycles $\jind{j}\jbar{j}$, $j\neq i$, and let $C_{\overline{i}}$ be the reversal of $C_i$.
Then the $C_i$ are the side cycles of an arrangement $\Gamma$ of oriented pseudolines indexed by $I$ 
if and only if there exist block decompositions 
$$B_{i1}B_{i2}\ldots B_{in_i}B'_{i1}B'_{i2}\ldots B'_{in_i}$$ 
of the $C_i$, $i\in I$,
where 
$B_{im} = j_{m1}j_{m2}j_{m3}\ldots j_{mk_m}$, $1\leq m\leq n_i$, 
with $j_{m_l} \notin \{j_{m_{l'}},\overline{j}_{m_{l'}}\}$ for all $1\leq l<l'\leq k_m$,
 such that 
\begin{enumerate}
\item $B'_{im}$ is the complement of the reversal of $B_{im}$ (note that this condition determines the block decomposition of $C_i$);
\item $j_{m2}j_{m3}\ldots j_{mk_m} \overline{i}$ 
is one of the blocks of the block decomposition  of $C_{j_{m1}}.$ \qed
\end{enumerate}
\end{lemma}

\begin{remark} According to Lemma~\ref{netbenefit}, the number $b_n$ of simple indexed arrangements of oriented double pseudolines on a given set of $n$ indices  is  the $n$-th power of the number 
of shuffles of the $n-1$ circular sequences $\thrcrossR{1}{j}\foucrossR{1}{j}\onecrossR{1}{j}\twocrossR{1}{j}$, $1\leq j\leq n$, $j\neq 1$, or, equivalently, the $n$-th power of the product 
of the number of permutations of a multiset of $4n-5$ elements of multiplicities $3,4,4,\ldots,4$ and the number of cyclic shifts of the $n-2$ linear sequences 
$\thrcrossR{1}{j}\foucrossR{1}{j}\onecrossR{1}{j}\twocrossR{1}{j}$, $2 <j\leq n$. Hence, using the standard notation 
for multinomial coefficients, 
\begin{eqnarray*}
b_n & = & \left\{4^{n-2} \binom{4n-5}{3,4,4,\ldots,4}\right\}^n.
\end{eqnarray*}
The first values are :  $b_2 = 1^2$, $b_3 = 140^3$, $b_4 = 184800^4$ and $b_5 = 10090080005^5$. Similarly, according to Lemma~\ref{netbenefitPL},
 the number $c_n$ of simple indexed arrangements of oriented pseudolines on a 
given set of $n$ indices  is  the $n$-th power of the number of antipodal shuffles of the $n-1$ circular sequences $\onecrossR{1}{j}\foucrossR{1}{j}$, $1\leq j\leq n$, $j\neq 1$, or, equivalently, the   $n$-th power of the number of signed permutations on a set of $n-2$ elements. 
Hence
\begin{eqnarray*}
c_n & = & \left\{2^{n-2} (n-2)!\right\}^n.
\end{eqnarray*}
The first values are : $c_2 = 1^2$, $c_3 = 2^3$, $c_4 = 8^4 = 4096$ and $c_5 = 48^5 = 254 803 968$. 
It will be interesting to have closed formulae also for nonsimple arrangements. 
\end{remark}

\begin{proof}[Proof of Theorem~\ref{LRCanygenus}]
Let $I$ be finite indexing set, let $\Ufive$ be the complex of subsets of size at most $5$ of $I$, let $\chi$ be a $5$-chirotope on the indexing set $I$, 
 and for $J \in \Ufive$, let $D_i(J)$ and $M_i(J)$ be the families of side cycles
 of the entry $\chi(J)$ of  $\chi$.
Proving the theorem boils down to prove that for any index $i\in I$ there exists 
\begin{enumerate}
\item a unique shuffle $D_i$ of the elementary 
cycles $\jind{j}\jind{j}\jbar{j}\jbar{j}$, 
$j\neq i$,  of which the $D_i(J)$, $i\in J \in \Ufive$, are subcycles; 
\item a unique shuffle $M_i$ of the elementary 
cycles $\jind{j}\jind{j}\jbar{j}\jbar{j}$, 
$j\neq i$,  of which the $M_i(J)$, $i\in J \in \Ufive$, are subcycles;  and that
\item 
 the two families of cycles $D_i$ and $M_i$ 
are the families of side cycles of an arrangement of oriented double pseudolines indexed by $I$ whose  $5$-chirotope is~ $\chi$.
\end{enumerate}
For $X \in \{D,M\}$ and $i \in I$, let ${\nRR}^X_i$ be the ternary relation defined on distinct elements $\alpha = i_\alpha j_\alpha \in \setnodebis{i}{j}$,
by $(\alpha,\alpha',\alpha'') \in {\nRR}^X_i$ if ${\alpha}$, $\alpha'$ and $\alpha''$ 
appear in this order on the cycle $X_i(\{i,j_{\alpha},j_{\alpha'},j_{\alpha''}\})$ and let ${\nBR}^X_i$, $i \in I$, be the binary relation defined on 
distinct elements $\alpha = i_\alpha j_\alpha \in \setnodebis{i}{j}$,
by $(\alpha,\alpha') \in {\nBR}^X_i$ if ${\alpha}$ and $\alpha'$
appear in this order in the same prime factor of the prime factor decomposition of $X_i(\{i,j_{\alpha},j_{\alpha'}\})$.
 Clearly 
\begin{enumerate}
\item ${\nRR}^X_i$ is well-defined;
\item for every triple $(\alpha,\alpha',\alpha'')$ one has $(\alpha,\alpha',\alpha'') \in {\nRR}^X_i$ or (exclusive) $(\alpha,\alpha'',\alpha') \in {\nRR}^X_i$;  
\item if $(\alpha,\alpha',\alpha'') \in {\nRR}^X_i$ then $(\alpha',\alpha'',\alpha) \in {\nRR}^X_i$;  
\item ${\nRR}^X_i$ is transitive, i.e., if $(\alpha,\alpha',\alpha'') \in {\nRR}^X_i$ and $(\alpha,\alpha'',\alpha''') \in {\nRR}^X_i$ then $(\alpha,\alpha',\alpha''') \in {\nRR}^X_i$
(because $X_i(i, j_{\alpha}, j_{\alpha'},j_{\alpha''})$, $X_i(i, j_{\alpha}, j_{\alpha''},j_{\alpha'''})$, and  $X_i(i, j_{\alpha}, j_{\alpha'},j_{\alpha''})$ are subcycles of 
$X_i(i, j_{\alpha}, j_{\alpha'},j_{\alpha''},j_{\alpha'''})$);
\item ${\nBR}^X_i$ is well-defined;
\item if $(\alpha,\alpha') \in {\nBR}^X_i$ and $(\alpha',\alpha'') \in {\nBR}^X_i$ then $(\alpha,\alpha'') \in {\nBR}^{X}_i$;
\item if $(\alpha,\alpha'') \in {\nBR}^X_i$ and $(\alpha,\alpha',\alpha'') \in {\nRR}^X_i$ then $(\alpha,\alpha') \in {\nBR}^X_i$ and $(\alpha',\alpha'') \in {\nBR}^X_i$;
\item if $(\alpha,\alpha') \in {\nBR}^X_i$ then $(\alpha',\alpha) \in {\nBR}^{\overline{X}}_i$.
\end{enumerate}
This proves that the shuffle $X_i$ of the elementary
cycles $\jind{j}\jind{j}\jbar{j}\jbar{j}$,
$j\neq i$, given by the ternary relation $\nRR^X_i$, 
is the unique shuffle of the elementary
cycles $\jind{j}\jind{j}\jbar{j}\jbar{j}$,
$j\neq i$, of which the $X_i(J)$, $i\in J \in \Ufive$, are subcycles 
and that there is a unique block decomposition $B_{i1}B_{i2}\ldots B_{in_i}$ of $D_i$, 
where $B_{im} = \{i_1j_{m1}\}\{i_2j_{m2}\}\{i_3j_{m3}\}\ldots \{i_kj_{mk_m}\}$, $1\leq m\leq n_i$,
with $j_{m_l} \notin\{ j_{m_{l'}},\overline{j}_{m_{l'}}\}$ for all $1\leq l<l'\leq k_m$, (the one given by the binary relation $\nBR^X_i$)
 such that
$T_i = B^*_{i1}B^*_{i2}\ldots B^*_{in_i}$ 
where
$B^*_{im}$ is the reversal of $B_{im}$. 
Since by construction  $\roll{B}{j_p}$
is one of the blocks of the block decomposition  of $S_{j_{mp}}$ we are done, thanks to Lemma~\ref{netbenefit}.  
\end{proof}
\begin{proof}[Proof of Theorem~\ref{LRCanygenusPL}] Similar to the proof of Theorem~\ref{LRCanygenus}.
\end{proof}

We come now to the proof of Theorems~\ref{uptofive} and~\ref{uptofour}. 

As said in the introduction, to prove Theorem~\ref{uptofive}, it can be argued  
that the mutation graph on the space of arrangements of double pseudolines of given size whose subarrangements of size at most $5$ are of genus $1$ is connected, or 
that for any pair of distinct faces of an arrangement of double pseudolines of genus $1$ there exists a subarrangement of size at most $3$ whose corresponding faces are distinct. 
Similarly, to prove Theorem~\ref{uptofour}, 
it can be argued that the mutation graph on the space of arrangements of pseudolines of given size whose subarrangements of size at most $4$ are of genus $1$ is connected, or 
that for any pair of distinct faces of an arrangement of pseudolines of genus $1$ there exists a subarrangement of size at most $2$ whose corresponding faces are distinct. 
We set out independently these two arguments  in the two next sections.

\subsection{Pumping lemma and mutations} 
We prove 
the connectedness of the  mutation graph on the space of arrangements of double pseudolines of given size whose subarrangements of size at most $5$ are of genus $1$ and 
the connectedness of the  mutation graph on the space of arrangements of pseudolines of given size whose subarrangements of size at most $4$ are of genus $1.$ 
The proof is based on the following abstractions of the pumping lemmas  of
Section~\ref{sec:homotopy}.
\begin{lemma}
\label{newmainresult}
Let $\Gamma$ be a simple arrangement of double pseudolines whose subarrangements of size at most $5$ are of genus $1$ and let $\gamma \in \Gamma$. 
Assume that there exists a vertex $v$ of the arrangement $\Gamma$  contained in the crosscap side of $\gamma$ in the subarrangement of size three
composed of $\gamma$ and the two double pseudolines crossing at $v$.  
Then there exists a triangular $2$-cell of the arrangement $\Gamma$ with a side supported by $\gamma$ and a vertex $w$ supported  
by the crosscap side of $\gamma$ in the subarrangement composed of $\gamma$ and the two  double pseudolines crossing at $w$. \qed
\end{lemma} 
\begin{lemma}
\label{newmainresultbis}
Let $\Gamma$ be a simple arrangement of pseudolines whose subarrangements of size at most $4$ are of genus $1$, let $\gamma, \gamma' \in \Gamma$, $\gamma \neq \gamma'$, and let 
$M(\gamma,\gamma')$ be one of two $2$-cells of size $2$ of the subarrangement $\{\gamma,\gamma'\}$. 
Assume that there exists a vertex $v$ of the arrangement $\Gamma$  contained in $M(\gamma,\gamma')$ in the subarrangement of size four  
composed of $\gamma,\gamma'$ and the two pseudolines crossing at $v$.  
Then there exists a triangular $2$-cell of the arrangement $\Gamma$ contained in $M(\gamma,\gamma')$ with a side supported by $\gamma$ in the subarrangement 
composed of $\gamma,\gamma'$ and the two  pseudolines crossing at the vertex $w$ opposite the side supported by~$\gamma$. 
\qed
\end{lemma} 
\begin{proof}[Proof of  Lemma~\ref{newmainresult}] 
Let $\Map{p_\Gamma}{\lift{\SSJ}_\Gamma}{\SSJ_\Gamma}$ be a $2$-sheeted unbranched covering of $\SSJ_\Gamma$ which is closed and orientable. 
For example the two relations 
$$ \left\{\begin{array}{c}
c_1c_1'c_2c_2'\ldots c_g c_g' =1\\
c_1'c_1c_2'c_2\ldots c_g' c_g =1
\end{array} \right.
$$
 define a closed and orientable $2$-sheeted unbranched covering of the nonorientable surface of genus $g$ defined by the relation 
$c_1c_1c_2c_2\ldots c_g c_g =1.$  
For any subarrangement $\subarrang$ of $\Gamma$ of size at least $2$ the restriction of $p_\Gamma$ to the pair $p_\Gamma^{-1}(R_\zeta)$, $R_\zeta$ extends naturally to  
a closed and orientable $2$-sheeted unbranched covering $\Map{p_\subarrang}{\lift{\SSJ}_\subarrang}{\SSJ_\subarrang}$ of the nonorientable surface $\SSJ_\subarrang$. 
Without loss of generality  we assume that the surfaces $\SSJ_\subarrang$ intersect pairwise only along their common ribbons, i.e., 
$\SSJ_\subarrang \cap \SSJ_{\subarrang'} = R_{\subarrang''}$ where $\subarrang'' = \subarrang \cap \subarrang'$.  Similarly we assume that the surfaces $\lift{\SSJ}_\subarrang$ 
intersect pairwise only along their common ribbons.
\begin{figure}[!htb]
\centering
\psfrag{UU}{$\UU$}\psfrag{VV}{$\VV$}
\psfrag{alpha}{$\alpha$}
\psfrag{alphaone}{$\alpha_1$}
\psfrag{alphatwo}{$\alpha_2$}
\psfrag{B}{$B$}\psfrag{Bs}{$B_*$}
\psfrag{Ud}{$\tau$}\psfrag{Vd}{$\tau$}
\psfrag{Ud'}{$\tau'$}\psfrag{Vd'}{$\tau'$}
\psfrag{U}{$\tau_+$}\psfrag{V}{$\tau_-$}
\psfrag{U'}{$\tau'_+$}\psfrag{V'}{$\tau'_-$}
\psfrag{Up}{$B$} \psfrag{Upp}{$B'$} 
\psfrag{aa}{(a)} \psfrag{bb}{(b)} \psfrag{cc}{} \psfrag{dd}{(c)} \psfrag{ee}{}
\includegraphics[width=0.8595\linewidth]{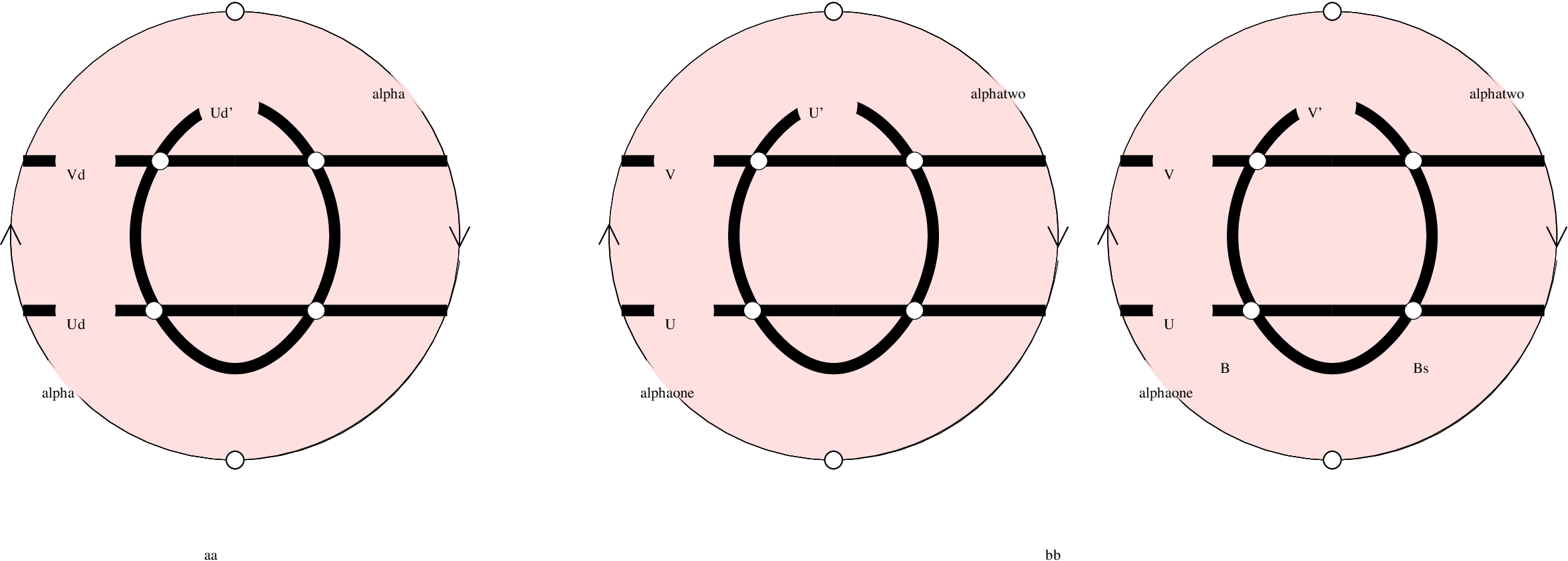}
\caption{ 
(a) A subarrangement of two double pseudolines; (b) its $2$-sheeted unbranched covering
\label{TwoCoveringAbstract}}
\end{figure}
The two lifts under $p_{\Gamma}$ of a curve $\tau$ of $\Gamma$ are denoted $\tau_+$ and~$\tau_-$, and the set of lifts of the curves of $\Gamma$ is denoted $\lift{\Gamma}$.  
Fig.~\ref{TwoCoveringAbstract}a shows a subarrangement of two double pseudolines and Fig.~\ref{TwoCoveringAbstract}b shows its $2$-sheeted unbranched covering. 
We note that two curves of $\lift{\Gamma}$ have exactly  $0$ or $2$ intersection 
points depending on whether they are the lifts of the same 
curve in $\Gamma$, or not. By convention if $B$ is one of the two intersection points of 
two crossing curves of $\lift{\Gamma}$  then the other one is denoted~$B_*$, as illustrated in Fig.~\ref{TwoCoveringAbstract}b. 
For $\zeta$ subarrangement of $\Gamma$ of size $2,3,4$ or $5$ containing $\gamma$ we denote by $C_{\zeta}$ the cylinder of the sphere ${\lift{\SSJ}}_\zeta$ bounded by $\UU$ and $\VV$.  
We introduce  the following terminology.
\begin{enumerate}
\item A {\it \tracecurve\ supported by} $\gamma' \in \Gamma$, $\gamma' \neq \gamma$, is a maximal subcurve of $\gamma'_+$ or $\gamma'_-$ contained in the 
cylinder $C_{\zeta}$ where $\zeta = \{\gamma,\gamma'\}$.  
Observe that there are four \tracecurves\ supported by  $\gamma'$ (two per lift of $\gamma'$)  and that a \tracecurve\ has an endpoint on $\UU$ and 
the other one on $\VV.$
The \tracecurve\ with endpoint $B$ on $\UU$ is denoted $\gcurve(B).$  
\item An {\it arrangement of \tracecurves}  is  a set of at most four \tracecurves\ embedded in the cylinder $C_\zeta$ 
where $\zeta$ is the set of supporting curves of the at most four \tracecurves\ augmented with $\gamma$.  
The cell complex of an arrangement of two \tracecurves\ depends only on the number of intersection points, as depicted in Fig.~\ref{TwoCoveringAbstractYY}.
\begin{figure}[!htb]
\centering
\psfrag{UU}{$\UU$}\psfrag{VV}{$\VV$}
\psfrag{alpha}{$\alpha$}
\psfrag{alphaone}{$\alpha_1$}
\psfrag{alphatwo}{$\alpha_2$}
\psfrag{B}{$B$}\psfrag{Bs}{$B_*$}
\psfrag{Ud}{$\tau$}\psfrag{Vd}{$\tau$}
\psfrag{Ud'}{$\tau'$}\psfrag{Vd'}{$\tau'$}
\psfrag{U}{$\tau_+$}\psfrag{V}{$\tau_-$}
\psfrag{U'}{$\tau'_+$}\psfrag{V'}{$\tau'_-$}
\psfrag{Up}{$B$} \psfrag{Upp}{$B'$} 
\psfrag{aa}{(a)} \psfrag{bb}{(b)} \psfrag{cc}{} \psfrag{dd}{(c)} \psfrag{ee}{}
\includegraphics[width=0.8595\linewidth]{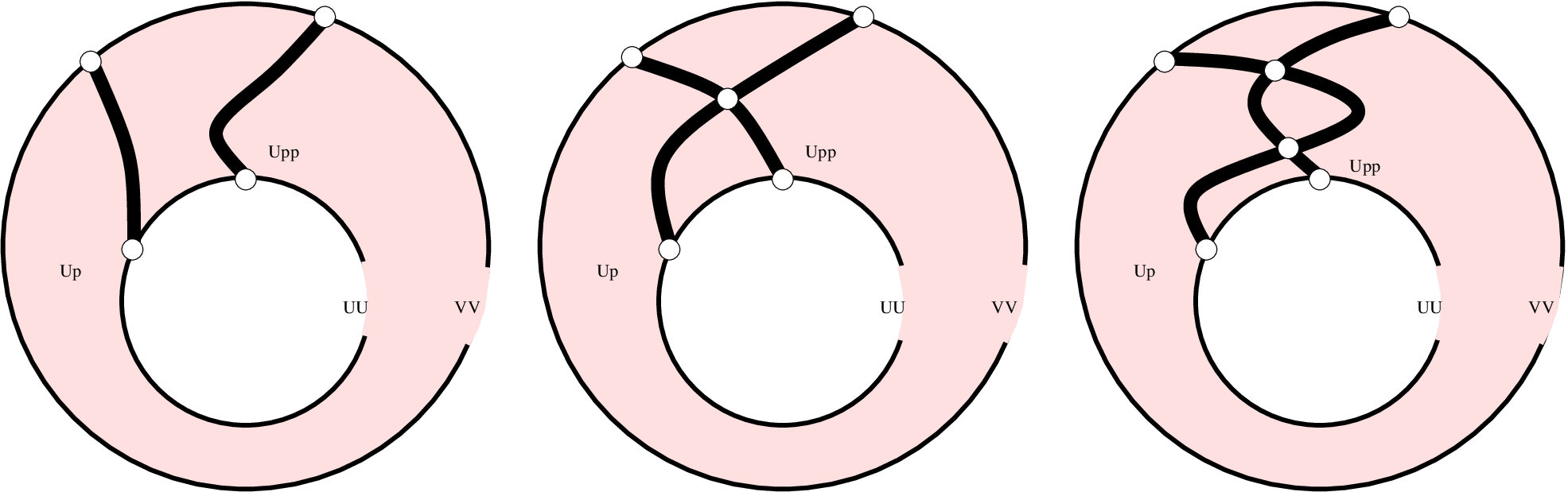}
\caption{The 3 possible arrangements of two  $\gamma$-curves
\label{TwoCoveringAbstractYY}}
\end{figure}
\item A {\it \gtriangle} is  a triangular face of the 
arrangement of  two crossing \tracecurves\ 
with a side supported by $\UU$;
the vertex of a \gtriangle\ not on $\UU$ is called its {\it apex} and the side
of a \gtriangle\ supported by $\UU$ is called its {\it base side}. 
The interior and the exterior of the base side of a  \gtriangle\ $T$, considered as a subset of $\UU$, are denoted $\base{T}$ and $\cobase{T}$, respectively. 
\item A  \gtriangle\ is {\it admissible} if one of its two sides
with the apex as an endpoint is an edge of $\lift{\Gamma}$. 
\begin{figure}[!htb]
\centering
\psfrag{aa}{\normalsize (a)} \psfrag{bb}{\normalsize (b)} \psfrag{cc}{\normalsize (c)}
\tiny
\footnotesize
\psfrag{Delta}{$\Delta$}
\psfrag{CYZ}{\normalsize $\stripbis{Y,Y'}$}
\psfrag{CBBY}{\normalsize $\stripbis{B,B',Y}$}
\psfrag{CBBZ}{\normalsize $\stripbis{B,B',Y'}$}
\psfrag{CBBB}{\normalsize $\stripbis{B,B',B''}$}
\psfrag{x}{$X$}
\psfrag{y1}{$Y'$}
\psfrag{y0}{$Y$}
\psfrag{t0}{$T$} \psfrag{t1}{$T'$} \psfrag{t2}{$T''$} \psfrag{a0}{$A$}
\psfrag{a1}{$A'$} \psfrag{a2}{$A''$} \psfrag{a3}{$A'''$} \psfrag{a4}{$A^{(4)}$}
\psfrag{b0}{$B$} \psfrag{b1}{$B'$} \psfrag{b2}{$B''$} \psfrag{b3}{$B'''$}
\psfrag{b4}{$B^{(4)}$} \psfrag{b5}{$B^{(5)}$} \psfrag{b1s}{$\twin{B}'$}
\psfrag{b2s}{$\twin{B}''$} \psfrag{bp3}{$\twin{B}'''$} \psfrag{T1}{$T'$}
\psfrag{a1s}{$\twin{A'}$}
\psfrag{T3}{$T'''$}
\psfrag{casetwo}{$B'' \in \cobase{T}$}
\psfrag{caseone}{$B'' \in \base{T}$}
\psfrag{aa}{(a)}
\psfrag{bb}{(b)}
\psfrag{cc}{(c)}
\psfrag{dd}{(d)}
\psfrag{ee}{(e)}
\psfrag{ff}{(f)}
\psfrag{gg}{(g)}
\includegraphics[width = 0.985 \linewidth]{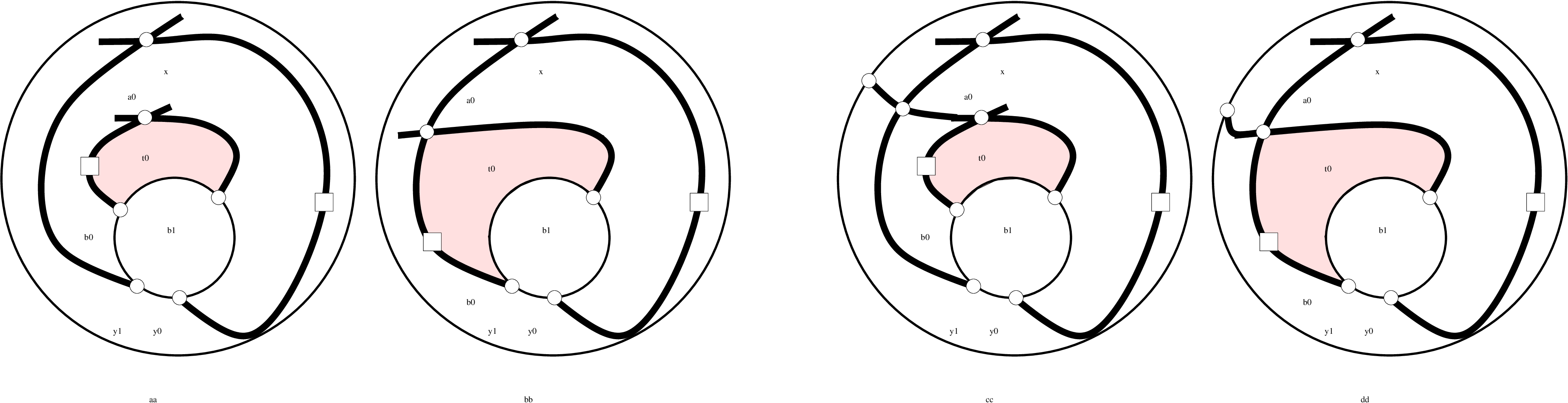}
\caption{The admissible \gtriangle\ $\Delta$ encloses the admissible \gtriangle\ $T$ \label{finalenclose}}
\end{figure}
\item 
An admissible \gtriangle\ $\Delta = XYY'$ with apex $X$ and edge side $XY$ is said to enclose an admissible \gtriangle\ $T = ABB'$ with apex $A$ and edge side $AB$ if $T$ is included in $\Delta$ and walking along the base side of $\Delta$ from $Y$ to $Y'$ we encounter $B'$ before $B$.
 Thus the arrangement of the four \tracecurves\ $\gcurve(Y)$, $\gcurve(Y')$, $\gcurve(B)$, $\gcurve(B')$ is, up to homeomorphism, 
one of those implicitly depicted in Fig.~\ref{finalenclose}a ($B\neq Y'$) or Fig.~\ref{finalenclose}b ($B=Y'$); and, consequently,  
one of those implicitly depicted in Fig.~\ref{finalenclose}c or Fig.~\ref{finalenclose}d since one can easily prove that $\gcurve(B')$ crosses the side $XY'$ only once.
\end{enumerate}
\begin{lemma} 
There is at least one admissible \gtriangle. 
\end{lemma}
\begin{proof}
Since by  assumption there is a vertex of $\Gamma$ in the crosscap side of  the double pseudoline 
$\gamma$ in the subarrangement composed of $\gamma$ and the two double pseudolines meeting at the vertex, there is  a \gtriangle, say $T=ABB'$ with apex $A$. Let 
$A'$ be the vertex of $\lift{\Gamma}$ that follows 
$B'$ on the side $B'A$ of $T$. Then $A'$ is the apex 
of an  admissible \gtriangle\ $T' = A'B'B''$ with edge side $A'B'$. 
This proves that there is at least one admissible \gtriangle.
\end{proof}
Let $T=ABB'$ be an  admissible \gtriangle\ with apex $A$ and edge side $AB$,
and   let $A'$ be the vertex of $\lift{\Gamma}$ that follows 
$B'$ on the side $B'A$ of $T$. Then $A'$ is the apex 
of an admissible \gtriangle\ $T' = A'B'B''$ with edge side $A'B'$.
A simple use of the Jordan curve theorem leads to the following four lemmas that control 
the relative positions of the base sides of $T$ and $T'$, possibly in the presence of a third admissible 
\gtriangle\ $\Delta = XYY'$ with apex $X$ and edge side $XY$ enclosing $T$. 
Fig.~\ref{caseoneBisAbstract}a, ~\ref{caseoneBisAbstract}b, ~\ref{caseoneBisAbstract}c, and Fig.~\ref{enclosedold}.

\begin{figure}[!htb]
\centering
\psfrag{aa}{\normalsize (a)} \psfrag{bb}{\normalsize (b)} \psfrag{cc}{\normalsize (c)}
\tiny
\footnotesize
\psfrag{CBB}{\normalsize $\stripbis{B,B'}$}
\psfrag{CBBB}{\normalsize $\stripbis{B,B',B''}$}
\psfrag{t0t1}{$T,T'$} 
\psfrag{a0a1}{$A,A'$}
\psfrag{t0}{$T$} \psfrag{t1}{$T'$} \psfrag{t2}{$T''$} \psfrag{a0}{$A$}
\psfrag{a1}{$A'$} \psfrag{a2}{$A''$} \psfrag{a3}{$A'''$} \psfrag{a4}{$A^{(4)}$}
\psfrag{b0}{$B$} \psfrag{b1}{$B'$} \psfrag{b2}{$B''$} \psfrag{b3}{$B'''$}
\psfrag{b4}{$B^{(4)}$} \psfrag{b5}{$B^{(5)}$} \psfrag{b1s}{$\twin{B}'$}
\psfrag{b2s}{$\twin{B}''$} \psfrag{bp3}{$\twin{B}'''$} \psfrag{T1}{$T'$}
\psfrag{a1s}{$\twin{A'}$}
\psfrag{T3}{$T'''$}
\includegraphics[width = 0.8750 \linewidth]{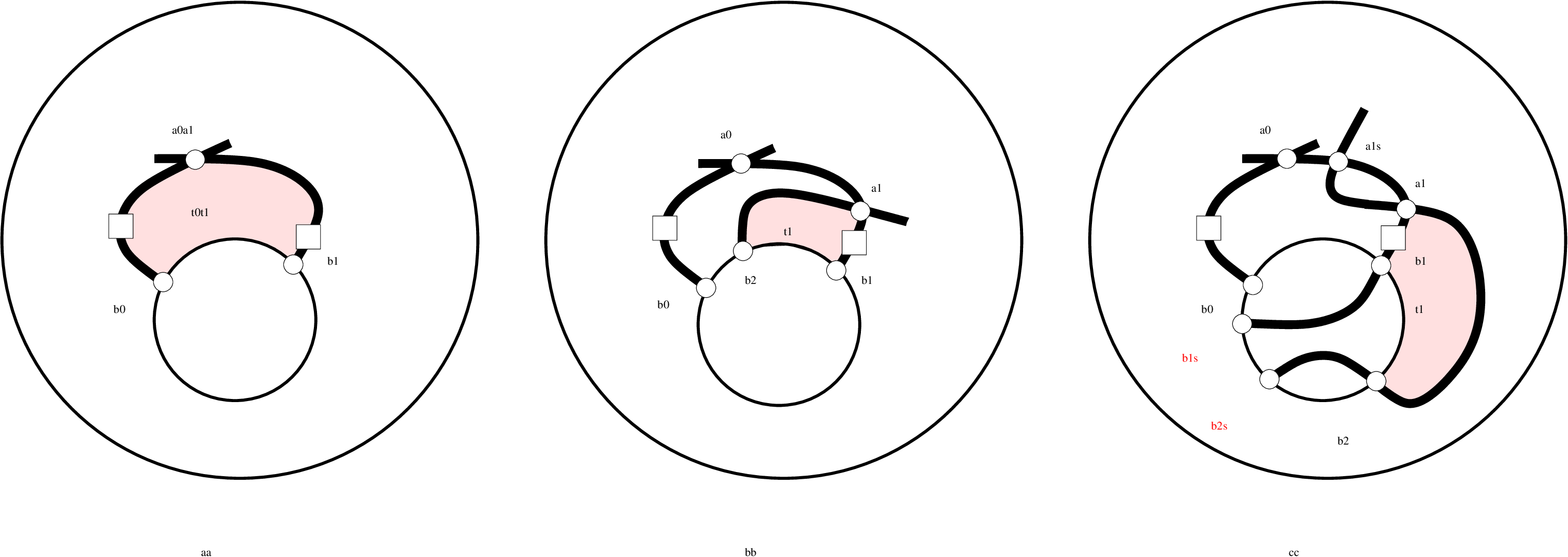}
\caption{Relative positions of an admissible \gtriangle\ $T$ and its derived  admissible \gtriangle~$T'$\label{caseoneBisAbstract} :
(a) $A=A', T = T'$; 
(b) $B'' \in \base{T}$;
(c) $B'' \in \cobase{T}$}
\end{figure}

\begin{lemma} \label{positionthrAbstract} 
Assume that $T=T'$. Then $T$ is a triangular two-cell of $\widetilde{\Gamma}$. \qed
\end{lemma}

\begin{lemma}\label{positiontwoAbstract}
Assume that $T\neq T'$ and  that $B'' \in \base{T}$.  Then:
(1) $\gcurve(B'')$ crosses the side $B'A$ of $T$ exactly once (at $A'$) and 
(2)
$\base{T'}$ is contained in $\base{T}$. \qed
\end{lemma}

\begin{lemma}\label{positiononeAbstract}
Assume that $T\neq T'$ and that $B'' \in \cobase{T}$.  Then:
\begin{enumerate}
\item
$\gcurve(B')$ and $\gcurve(B'')$ cross twice (at $A'$ and $\twin{A'}$) on the side $B'A$ of $T$, 
\item
$\base{T}$ and $\base{T'}$ are interior disjoint, 
\item 
$\twin{B}'$ and $\twin{B}'' \in \cobase{T} \cap \cobase{T'}$, and 
\item 
walking along $\cobase{T} \cap \cobase{T'}$  
from $B''$ to $B$ we encounter successively the points $\twin{B}''$ and $\twin{B}'$. 
\end{enumerate}
Furthermore if $\Delta$  encloses $T$, then $\Delta$ encloses $T'$.
\end{lemma}
\begin{figure}[!htb]
\centering
\psfrag{aa}{\normalsize (a)} \psfrag{bb}{\normalsize (b)} \psfrag{cc}{\normalsize (c)}
\tiny
\footnotesize
\psfrag{Delta}{$\Delta$}
\psfrag{CYZ}{\normalsize $\stripbis{Y,Y'}$}
\psfrag{CBBY}{\normalsize $\stripbis{B,B',Y}$}
\psfrag{CBBZ}{\normalsize $\stripbis{B,B',Y'}$}
\psfrag{CBBB}{\normalsize $\stripbis{B,B',B''}$}
\psfrag{x}{$X$}
\psfrag{y1}{$Y'$}
\psfrag{y0}{$Y$}
\psfrag{t0}{$T$} \psfrag{t1}{$T'$} \psfrag{t2}{$T''$} \psfrag{a0}{$A$}
\psfrag{a1}{$A'$} \psfrag{a2}{$A''$} \psfrag{a3}{$A'''$} \psfrag{a4}{$A^{(4)}$}
\psfrag{b0}{$B$} \psfrag{b1}{$B'$} \psfrag{b2}{$B''$} \psfrag{b3}{$B'''$}
\psfrag{b4}{$B^{(4)}$} \psfrag{b5}{$B^{(5)}$} \psfrag{b1s}{$\twin{B}'$}
\psfrag{b2s}{$\twin{B}''$} \psfrag{bp3}{$\twin{B}'''$} \psfrag{T1}{$T'$}
\psfrag{a1s}{$\twin{A'}$}
\psfrag{T3}{$T'''$}
\psfrag{casetwo}{$B'' \in \cobase{T}$}
\psfrag{caseone}{$B'' \in \base{T}$}
\psfrag{aa}{(a)}
\psfrag{bb}{(b)}
\psfrag{cc}{(c)}
\psfrag{dd}{(d)}
\psfrag{ee}{(a)}
\psfrag{ff}{(b)}
\psfrag{gg}{(c)}
\includegraphics[width = 0.850 \linewidth]{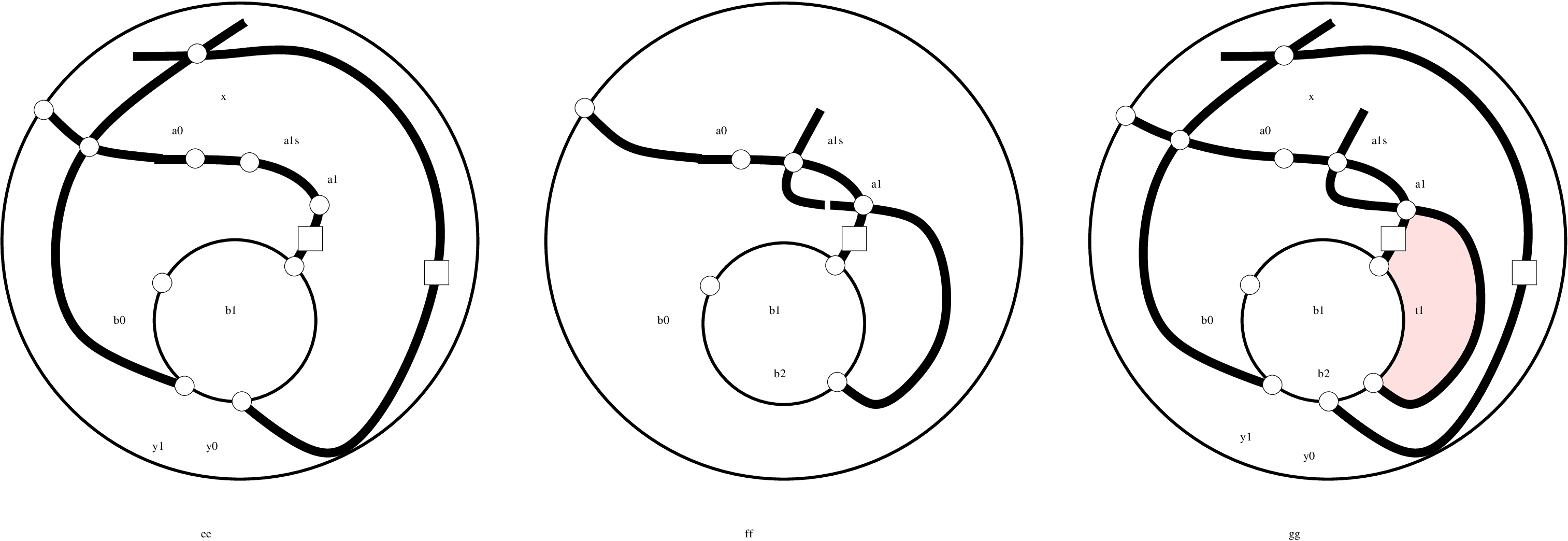}
\caption{(a) $\Delta$ encloses  $T$; 
(b) $B'' \in \cobase{T}$; (c) $\Delta$ encloses $T'$\label{enclosedold}}
\end{figure}
\begin{proof}  We only comment the furthermore part. 
Let $\Upsilon$ be the arrangement of the three \tracecurves\ $\gcurve(Y), \gcurve(Y'),\gcurve(B')$ and let $\Upsilon'$ be the arrangement of the two \tracecurves\ 
$\gcurve(B')$ and $\gcurve(B'')$, both arrangements being augmented with the points $A,B,A'$ and $\twin{A'}$.  
According to the previous discussion $\Upsilon$ is, up to homeorphism, one of those implicitly depicted in Fig.~\ref{enclosedold}a (we omit the case $Y'=B$).  
Similarly $\Upsilon'$ is, up to homeorphism, one of those implicitly depicted in Fig.~\ref{enclosedold}b. Again using the Jordan curve theorem 
we see easily that the only compatible superpositions of these two arrangements are, up to homeorphism, those implicitly depicted in Fig.~\ref{enclosedold}c. The lemma follows.  \end{proof} 
Consider now the sequence of admissible \gtriangles\ 
$T_0,T_1,T_2,\ldots $  
defined inductively by $T_0=T$ and $T_{k+1} = T'_k$  for $k\geq 0$. 
A simple combination of Lemmas~\ref{positiononeAbstract} and~\ref{positiontwoAbstract} leads to the conclusion that the sequence $T_k$ is stationary.
According to Lemma~\ref{positionthrAbstract} the lemma follows. 
\end{proof}
\begin{proof}[Proof of Lemma~\ref{newmainresultbis}.] As the  proof uses  similar ideas to the proof of the previous lemma with a much simpler case analysis,
 we omit it. \end{proof}

\begin{theorem} 
 \label{abstractconnectivitypl}
The mutation graph on the space of pseudoline arrangements of given size whose subarrangements of size at most $4$ are of genus $1$ is connected.  
\end{theorem}
\begin{proof} A {\it good} arrangement is an arrangement of pseudolines whose subarrangements of size at most $4$ are of genus $1$.
Clearly any good arrangement is connected, via a finite sequence of splitting mutations, to a  good simple arrangement. 
Then by a repeated application of Lemma~\ref{newmainresultbis} we see
that any good simple arrangement of size $n+1$ is connected, via a finite sequence of mutations,
to a good simple arrangement of pseudolines obtained from a good simple arrangement of size $n$ by adding a copy of one of its pseudolines 
as indicated in Fig.~\ref{addingacopy}. The result follows by induction.
\begin{figure}[!htb]
\centering
\psfrag{a}{$\alpha$}
\includegraphics[width = 0.8750 \linewidth]{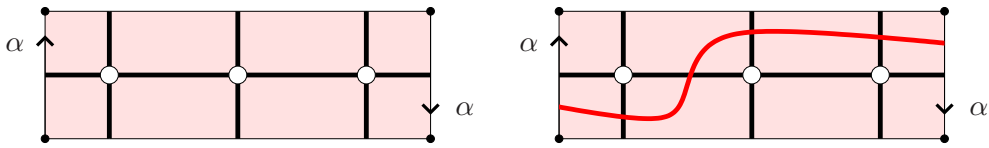}
\caption{\label{addingacopy} Adding a copy of a pseudoline in a pseudoline arrangement is carried out in the vicinity of the pseudoline. The choice of the position of 
intersection point between the pseudoline and its copy is arbitrary }
\end{figure}
\end{proof}
\begin{theorem} 
 \label{abstractconnectivity}
The mutation graph on the space of double pseudoline arrangements of given size whose subarrangements of size at most $5$ are of genus $1$ is connected.  
\end{theorem}
\begin{proof} A {\it good} arrangement is an  arrangement  of double pseudolines whose subarrangements of size at most $5$ are of genus $1$.
Clearly any good arrangement  is connected, via a finite sequence of splitting mutations, to a good simple arrangement. 
By a repeated application of Lemma~\ref{newmainresult} we see
that any good simple arrangement is connected, via a finite sequence of mutations,
to a good thin arrangement.  The results follows thanks to Theorem~\ref{abstractconnectivitypl} and the one-to-one correspondence between isomorphism classes of 
 simple pseudoline arrangements and isomorphism classes of thin double pseudoline arrangements.
\end{proof}

\begin{proof}[Proof of Theorem~\ref{uptofive}]
According to Theorem~\ref{abstractconnectivity} the mutation graph on the space of double pseudoline arrangements of given size whose subarrangements of size at most $5$ 
are of genus $1$ is connected. 
Since a mutation does not change the Euler characteristic of the underlying surface of an arrangement 
this prove that a double pseudoline arrangement whose subarrangements of size at most $5$ are of genus $1$ is of genus $1$.
\end{proof}
\begin{proof}[Proof of Theorem~\ref{uptofour}]
According to Theorem~\ref{abstractconnectivitypl} the mutation graph on the space of  pseudoline arrangements 
of given size  whose subarrangements of size at most $4$ are of genus $1$ is connected.
Since a mutation does not change the Euler characteristic of the underlying surface of an arrangement this prove that a pseudoline 
arrangement whose subarrangements of size at most $4$ are of genus $1$ is  
of genus $1$.
\end{proof}

\subsection{Separation lemma}
We come now to the announced alternative proofs of Theorems~\ref{uptofive} and~\ref{uptofour} before discussing improved versions of both. 
In particular we offer  results in strong support of 
the conjecture that an arrangement of double pseudolines whose subarrangements of size at most $4$ are of genus $1$ is of genus $1$. 
The alternative proofs are based on the following related two observations 
\begin{lemma}[Separation Lemma] \label{separationPL}
Let $F$ and $F'$ be two distinct faces of an arrangement of pseudolines $\Gamma$ of genus $1$. Then there
exists a subarrangement of $\Gamma$ of size  $2$ whose faces $A$ and $A'$ containing $F$ and $F'$ are distinct.
\end{lemma}
\begin{proof} Let $[F]$ and $[F']$  be the subarrangements 
of $\Gamma$ composed of the supporting curves of the sides of $F$ and  $F'$, respectively. Clearly the faces of the subarrangement $[F]$ containing $F$ and $F'$ are distinct. 
Therefore one can assume that $\Gamma = [F] = [F']$ is a cyclic arrangement with at least two central cells.  Then the result follows  
from the classification of cyclic pseudoline arrangements of genus $1$ that was recalled in Section~\ref{raiponces}. \end{proof}
\begin{lemma}[Separation Lemma] \label{separation}
Let $F$ and $F'$ be two distinct faces of an arrangement of double pseudolines $\Gamma$ of genus $1$. Then there
exists a subarrangement of $\Gamma$ of size at most $3$ whose faces $A$ and $A'$ containing $F$ and $F'$ are distinct.
\end{lemma}
\begin{proof} 
For the restricted class of thin arrangements the result follows from 
Lemma~\ref{separationPL} and the one-to-one correspondence between the class of isomorphism classes of simple pseudoline arrangements 
and the class of isomorphism classes of thin double pseudoline arrangements. 
The general case follows easily  from the connectedness of mutation graphs.\end{proof}
We are now ready for the alternative proofs. 
Let $\Gamma$ be an arrangement, let  $\gamma \in \Gamma$, and set $\Gamma' = \Gamma\setminus \{\gamma\}$.
Endow $\gamma$ with an orientation and introduce the set $\caseone{\Gamma}{\gamma}$ 
of pairs of consecutive nodes $NN'$ of the node cycle of $\gamma$ in the arrangement $\Gamma$ such the face $F$ of $\Gamma'$ that 
 $\gamma$ enters at $N$ and the face of $\Gamma'$ that $\gamma$ leaves at $N'$ are dictinct, as illustrated in Fig.~\ref{crosscase}a, and  
the set $\casetwo{\Gamma}{\gamma}$ of pairs of pairs of consecutive nodes $NN'$ and $MM'$ of the node cycle of $\gamma$ in the arrangement $\Gamma$ such that 
(1) $\gamma$ enters at $N$ and $M$ and leaves at $N'$ and $M'$ the same face $F$ of $\Gamma'$, 
(2) the pair $NN'$ separates the pair $MM'$ on the boundary of $F$, as illustrated in Fig.~\ref{crosscase}b.  
Clearly the genus of $\Gamma'$ is less than the genus of $\Gamma$ with equality if and only if $\caseone{\Gamma}{\gamma}$ and $\casetwo{\Gamma}{\gamma}$ are both empty. 
\begin{figure}[!htb]
\centering
\psfrag{face}{$F$}
\psfrag{facep}{$F'$}
\psfrag{N}{$N$}
\psfrag{Np}{$N'$}
\psfrag{M}{$M$}
\psfrag{Mp}{$M'$}
\psfrag{a}{(b)}
\psfrag{b}{(a)}
\includegraphics[width=0.95\linewidth]{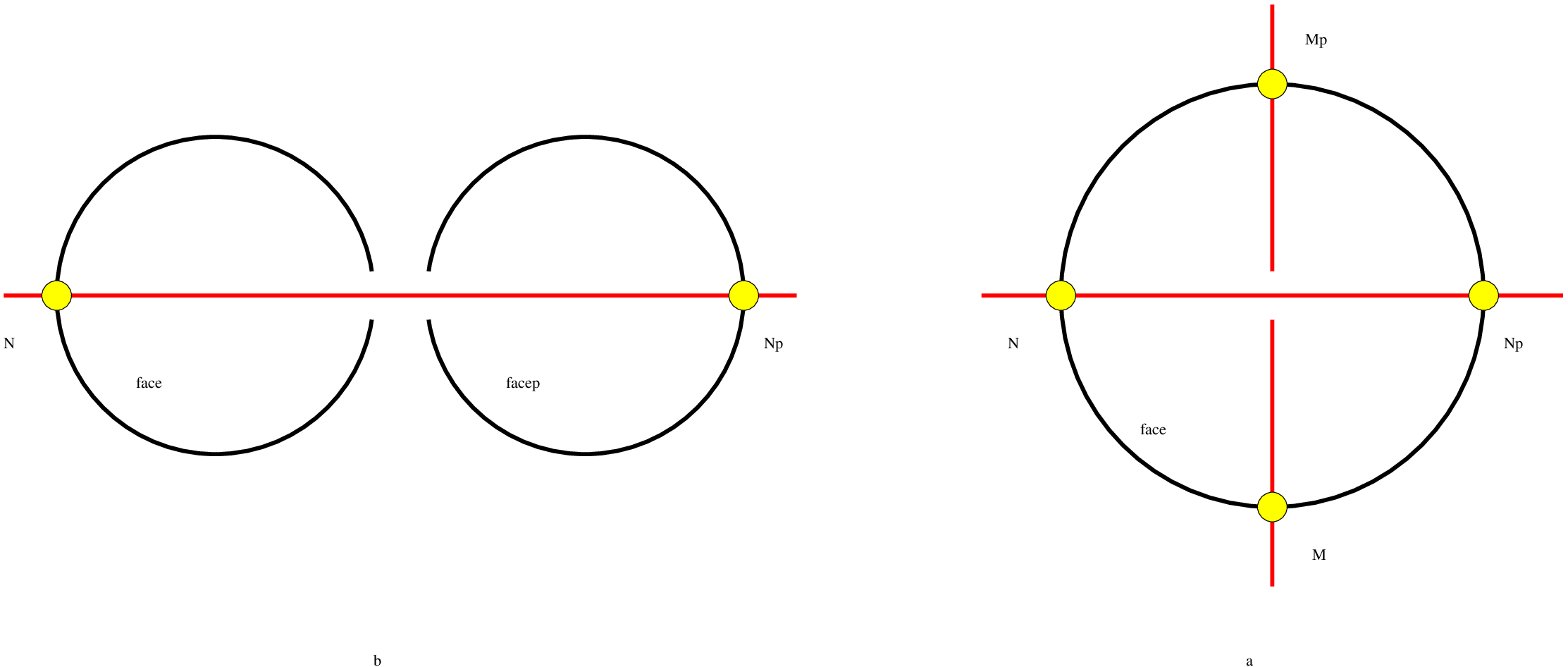}
\caption{\protect \small  (a) $NN'\in \caseone{\Gamma}{\gamma}$ ; (b) $NN'MM'\in \casetwo{\Gamma}{\gamma}$ \label{crosscase}}
\end{figure}
\begin{proof}[Second proof of Theorem~\ref{uptofive}] 
Let $\Gamma$ be an arrangement of double pseudolines whose subarrangements of size at most $5$ are of genus $1$.  Our goal 
is to prove that $\Gamma$ is of genus~$1$.  We proceed by induction on the size $n$ of $\Gamma$. The base  case $n\leq 5$ is clear. 
Assume $n \geq 6$, let $\gamma \in \Gamma$ and 
assume that $\Gamma\setminus \{\gamma\}$ is of genus $1$. 
We  show that $\caseone{\Gamma}{\gamma}$ and $\casetwo{\Gamma}{\gamma}$ are both empty.  Without loss of generality we can assume that $\Gamma$ is a simple arrangement.
We first show that $\caseone{\Gamma}{\gamma}$ is empty. Let $NN'$ be a pair of consecutive nodes of the node cycle of $\gamma$ in the arrangement $\Gamma$.
Let $\Gamma''$ be a subarrangement of $\Gamma'$ of size at most $3$ and let $A$ be a face of $\Gamma".$
We let the reader check that 
the following four claims are equivalent
\begin{enumerate}
\item $N$ is contained in $A$ or $\gamma$ enters $A$ at $N$;
\item $N'$ is contained in $A$ or $\gamma$ leaves $A$ at $N'$;
\item $F$ is contained  in  $A$;
\item $F'$ is contained in $A$.
\end{enumerate}
According to Lemma~\ref{separation} it follows that  $F = F'$. Hence $\caseone{\Gamma}{\gamma}$ is empty.
It remains to observe that if $NN'MM' \in \casetwo{\Gamma}{\gamma}$ then $NN'MM' \in \casetwo{\Gamma''}{\gamma}$ where $\Gamma''$ is the subarrangement of $\Gamma$ composed of $\gamma$ and 
the (at most 4) curves of $\Gamma$ that $\gamma$ crossed at $N,N',M$ and $M'$ to complete the proof. 
\end{proof}
\begin{proof}[Second proof of Theorem~\ref{uptofour}] 
Let $\Gamma$ be an arrangement of pseudolines whose subarrangements of size at most $4$ are of genus $1$.  Our goal 
is to prove that $\Gamma$ is of genus~$1$.  We proceed by induction on the size $n$ of $\Gamma$. The base case $n\leq 4$ is clear. Assume $n \geq 5$, let $\gamma \in \Gamma$ and 
assume that $\Gamma\setminus \{\gamma\}$ is of genus $1$. 
We  show that $\caseone{\Gamma}{\gamma}$ and $\casetwo{\Gamma}{\gamma}$ are both empty.  
Without loss of generality we can assume that $\Gamma$ is a simple arrangement.
We first show that $\caseone{\Gamma}{\gamma}$ is empty. Let $NN'$ be a pair of consecutive nodes of the node cycle of $\gamma$ in the arrangement $\Gamma$.
Let $\Gamma''$ be a subarrangement of $\Gamma'$ of size at most $2$ and let $A$ be a face of $\Gamma".$
We let the reader check that 
the following four claims are equivalent
\begin{enumerate}
\item $N$ is contained in $A$ or $\gamma$ enters $A$ at $N$;
\item $N'$ is contained in $A$ or $\gamma$ leaves $A$ at $N'$;
\item $F$ is contained  in  $A$;
\item $F'$ is contained in $A$.
\end{enumerate}
According to Lemma~\ref{separationPL} it follows that  $F = F'$. Hence $\caseone{\Gamma}{\gamma}$ is empty.
It remains to observe that if $NN'MM' \in \casetwo{\Gamma}{\gamma}$ then $M$ is the initial node of a pair of nodes in $\caseone{\Gamma''}{\gamma}$ 
where $\Gamma''$ is the subarrangement of $\Gamma$ composed of $\gamma$ and the $3$ curves of $\Gamma$ that $\gamma$ crossed at $N,N'$ and $M$ to complete the proof. 
\end{proof}

We now discuss the improved versions of Theorems~\ref{uptofive} and~\ref{uptofour}. They are  consequences of improved versions of the separation lemmas, obtained 
by looking at two-coverings of arrangements. The case of pseudoline arrangements is particularly simple.
\begin{lemma}[Separation Lemma]
\label{separationtwoPL} 
Let $GG'$ be a pair of distinct faces of an unbranched  $2$-covering $\widetilde{\Gamma}$ of an arrangement of pseudolines $\Gamma$ of size $2$ and genus 1. 
Then there exists a subarrangement $\Gamma'$ of $\Gamma$ of size  $1$ whose faces $A$ and $A'$ in $\widetilde{\Gamma}'$ containing $G$ and $G'$ are distinct.
\end{lemma}
\begin{proof} Obvious; cf. Fig.~\ref{TwoCoveringPL}. \end{proof}

\begin{figure}[!htb]
\centering
\psfrag{G}{$G$}
\psfrag{GP}{$G'$}
\psfrag{alpha}{$\alpha$}
\psfrag{alphaone}{$\alpha_1$}
\psfrag{alphatwo}{$\alpha_2$}
\psfrag{aa}{(a)} \psfrag{bb}{(b)} \psfrag{cc}{} \psfrag{dd}{(c)} \psfrag{ee}{}
\includegraphics[width=0.8595\linewidth]{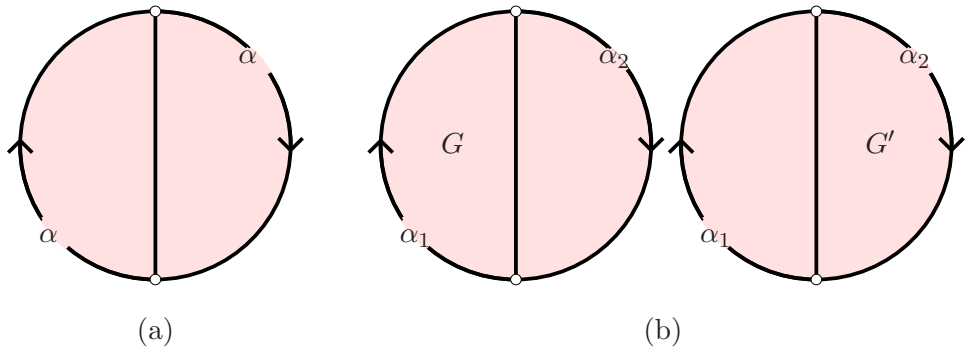}
\caption{ 
(a) An arrangement of two pseudolines; (b) its $2$-sheeted unbranched covering
\label{TwoCoveringPL}}
\end{figure}

\begin{theorem}\label{uptothree}
Let $\Gamma$ be an arrangement pseudolines whose subarrangements of size $3$ are of genus $1$. Then 
$\Gamma$ is of genus $1$. 
\end{theorem}
\begin{proof}
 We argue exactly as in the proof of Theorem~\ref{uptofive} except that we work in a $2$-covering of $\Gamma$ in order to use Lemma~\ref{separationtwoPL} instead 
of Lemma~\ref{separationPL}.
\end{proof}

The case of arrangements of double pseudolines requires a little  more work. 

Let $\Gamma$ be an arrangement of double pseudolines of genus $1$ and let $FF'$ be a pair of distinct faces of $\Gamma$. A {\it S-witness} of $FF'$ is a subarrangement of 
$\Gamma$ whose faces $A$ and $A'$ containing $F$ and $F'$ are distinct. The {\it S-number} of $FF'$ is the minimum of the sizes of its S-witnesses.
Thus our Separation Lemma asserts that the S-number of $FF'$ is $1,2$ or $3$. In case the $S$-number of $FF'$ is $3$ we say that the pair $FF'$ is a {\it critical pair} of $\Gamma$.  
In case $FF'$ is a critical pair and $\Gamma$ is of size $3$ we define the {\it critical graph} of $FF'$ as the graph with two vertices and three edges embedded in the cross surface of the arrangement 
whose vertices are two arbitrary points $a$ and $a'$ chosen in $F$ and $F'$  and whose edges are three paths joining $a$ to $a'$, 
each path avoiding two of the three curves (and crossing necessarily the third one) of the arrangement. 
It is no hard to see that the critical graph is unique up to ambient isotopy.  
Fig.~\ref{onecriticalpairs} and~\ref{twocriticalpairs} show the critical pairs of the arrangements of size $3$ together with their associated critical graphs. 
\begin{figure}[!htb]
\def\factor{0.200019315015000023}
\centering
\psfrag{8}{\small 8} \psfrag{7}{\small 7} \psfrag{6}{\small 6} \psfrag{5}{\small 5} \psfrag{4}{\small 4} \psfrag{3}{\small 3} \psfrag{2}{\small 2}
\psfrag{8}{} \psfrag{7}{} \psfrag{6}{} \psfrag{5}{} \psfrag{4}{} \psfrag{3}{} \psfrag{2}{}
\psfrag{o24}{24} \psfrag{o12}{12} \psfrag{o2}{2} \psfrag{o4}{4} \psfrag{o6}{6} \psfrag{o1}{1}
\psfrag{A}{$\name{04}$}
\psfrag{B}{$\name{07}$}
\psfrag{C}{$\name{18}$}
\psfrag{D}{$\name{25}$} 
\psfrag{F}{$\name{07}$} 
\psfrag{G}{$\name{37}$}
\psfrag{H}{$\name{15}$}
\psfrag{I}{\textcolor{red}{$\name{252310}$}}
\psfrag{J}{$\name{43}$}
\psfrag{K}{$\name{25}$}
\psfrag{L}{$\name{33}$}
\psfrag{M}{$\name{32}$}
\psfrag{N}{$\name{25_2}$}
\psfrag{Nstar}{$\name{25_1}$}
\psfrag{O}{$\name{32}$}
\psfrag{P}{$\name{22}$}
\psfrag{Q}{$\name{25}$}
\psfrag{R}{$\name{36}$}
\psfrag{Z}{$\name{64}$}
\includegraphics[width = \factor\linewidth]{P049000}
\includegraphics[width = \factor\linewidth]{P073300}
\includegraphics[width = \factor\linewidth]{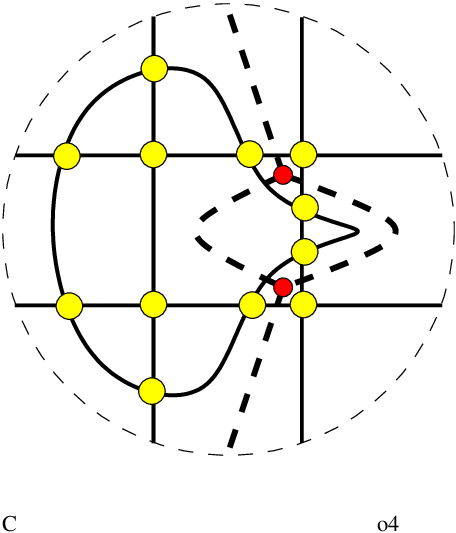}
\includegraphics[width = \factor\linewidth]{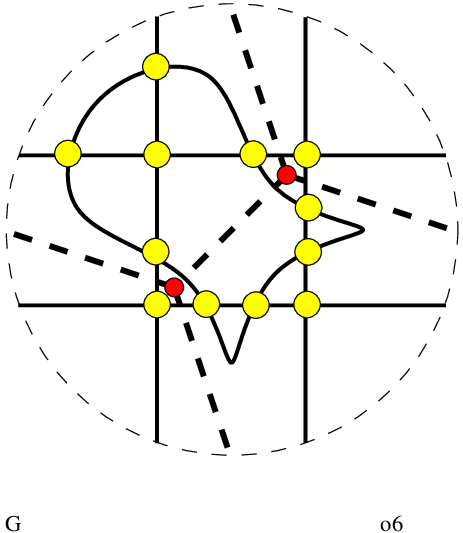}
\includegraphics[width = \factor\linewidth]{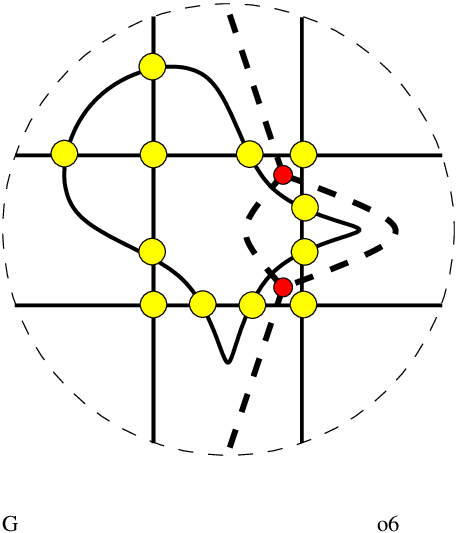}
\includegraphics[width = \factor\linewidth]{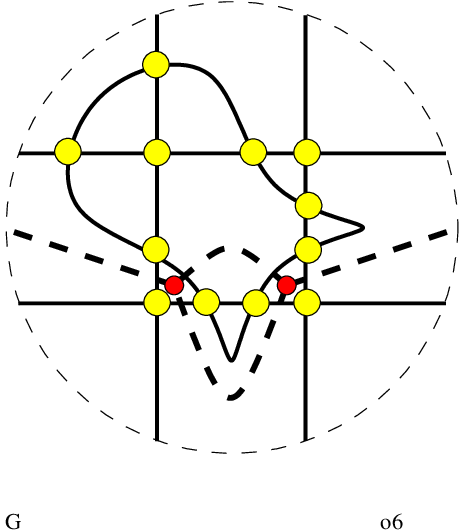}
\includegraphics[width = \factor\linewidth]{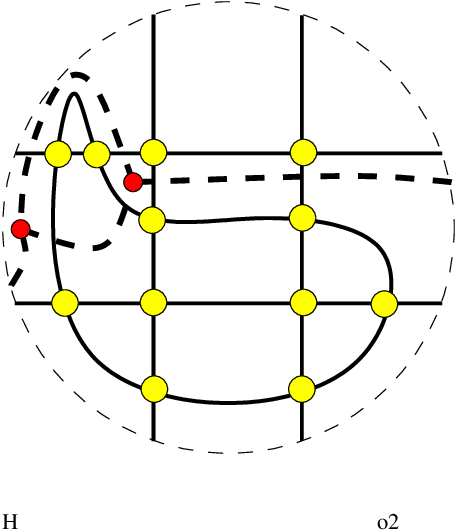}
\includegraphics[width = \factor\linewidth]{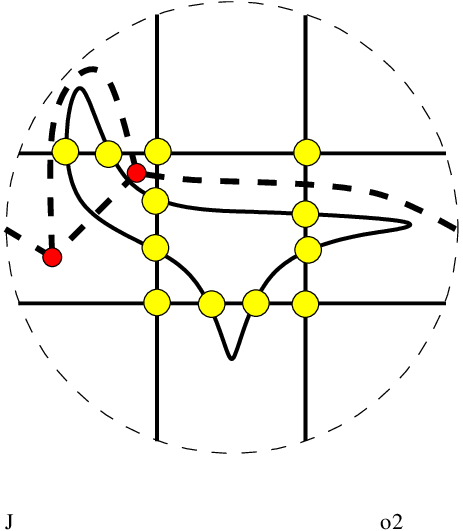}
\includegraphics[width = \factor\linewidth]{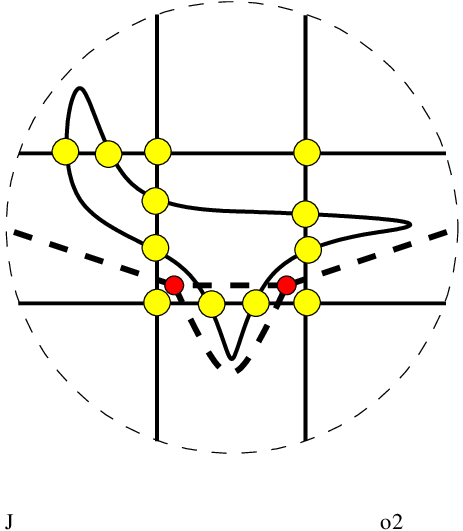}
\includegraphics[width = \factor\linewidth]{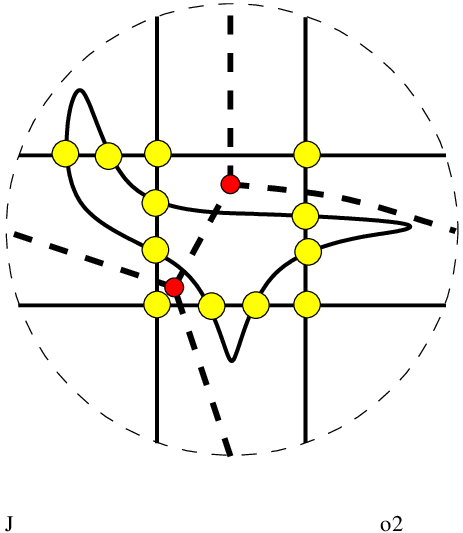}
\includegraphics[width = \factor\linewidth]{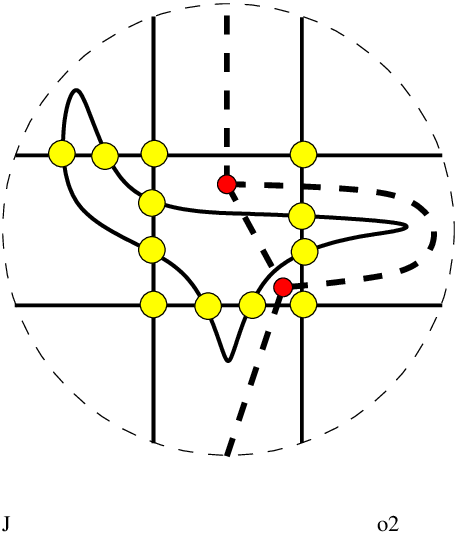}
\includegraphics[width = \factor\linewidth]{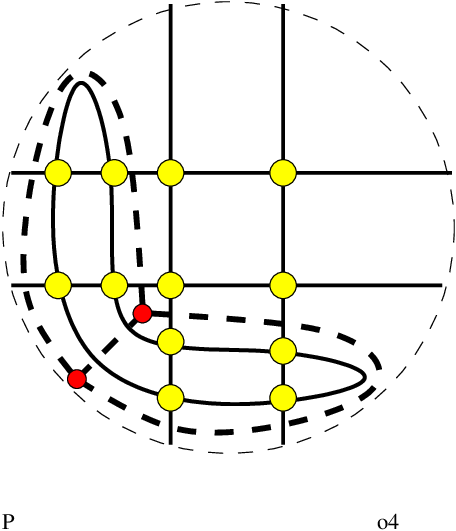}
\includegraphics[width = \factor\linewidth]{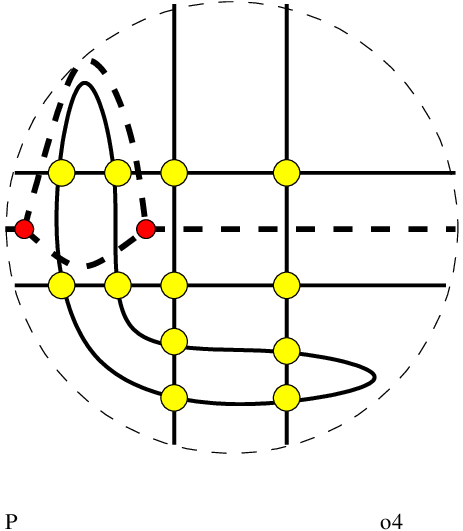}
\includegraphics[width = \factor\linewidth]{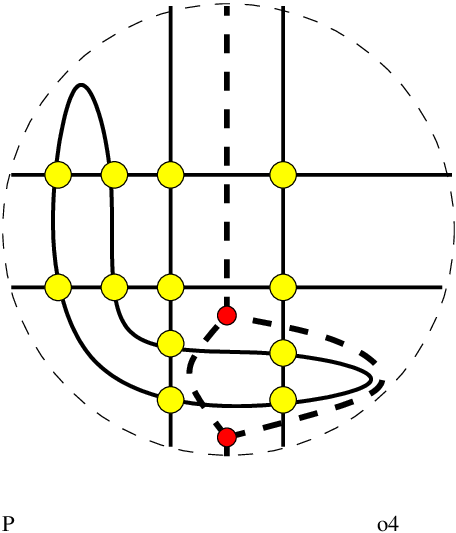}
\caption{Critical pairs of the arrangements $\name{18}$, $\name{37}$, $\name{15}$, $\name{43}$, and $\name{22}$ together with their associated critical graphs. 
The arrangements $\name{04}$ and $\name{07}$ have no critical pairs
\label{onecriticalpairs}}
\end{figure}
\begin{figure}[!htb]
\def\factor{0.200019315015000023}
\centering
\psfrag{8}{\small 8} \psfrag{7}{\small 7} \psfrag{6}{\small 6} \psfrag{5}{\small 5} \psfrag{4}{\small 4} \psfrag{3}{\small 3} \psfrag{2}{\small 2}
\psfrag{8}{} \psfrag{7}{} \psfrag{6}{} \psfrag{5}{} \psfrag{4}{} \psfrag{3}{} \psfrag{2}{}
\psfrag{o24}{24} \psfrag{o12}{12} \psfrag{o2}{2} \psfrag{o4}{4} \psfrag{o6}{6} \psfrag{o1}{1}
\psfrag{A}{$\name{04}$}
\psfrag{B}{$\name{07}$}
\psfrag{C}{$\name{18}$}
\psfrag{D}{$\name{25}$} 
\psfrag{F}{$\name{07}$} 
\psfrag{G}{$\name{37}$}
\psfrag{H}{$\name{15}$}
\psfrag{I}{\textcolor{red}{$\name{252310}$}}
\psfrag{J}{$\name{43}$}
\psfrag{K}{$\name{25}$}
\psfrag{L}{$\name{33}$}
\psfrag{M}{$\name{32}$}
\psfrag{N}{$\name{25_2}$}
\psfrag{Nstar}{$\name{25_1}$}
\psfrag{O}{$\name{32}$}
\psfrag{P}{$\name{22}$}
\psfrag{Q}{$\name{25}$}
\psfrag{R}{$\name{36}$}
\psfrag{Z}{$\name{64}$}
\includegraphics[width = \factor\linewidth]{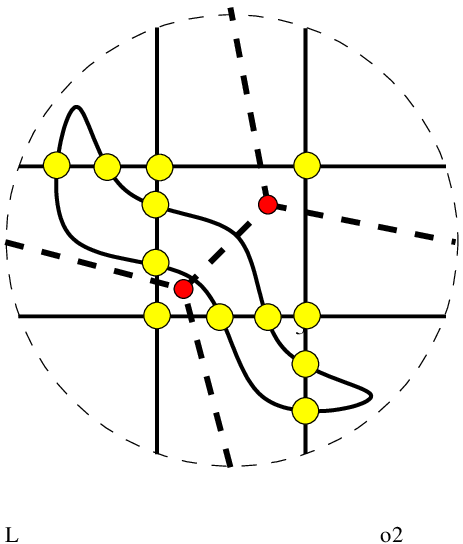}
\includegraphics[width = \factor\linewidth]{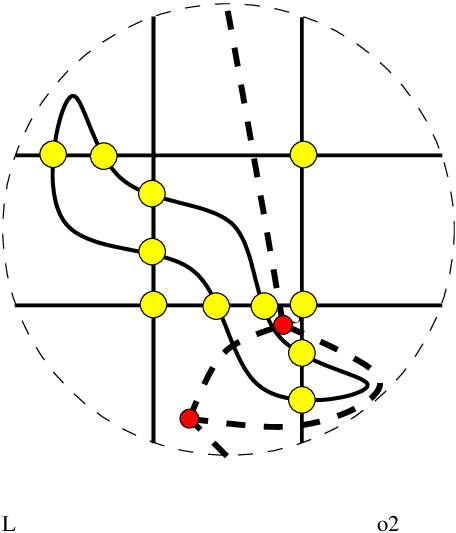}
\includegraphics[width = \factor\linewidth]{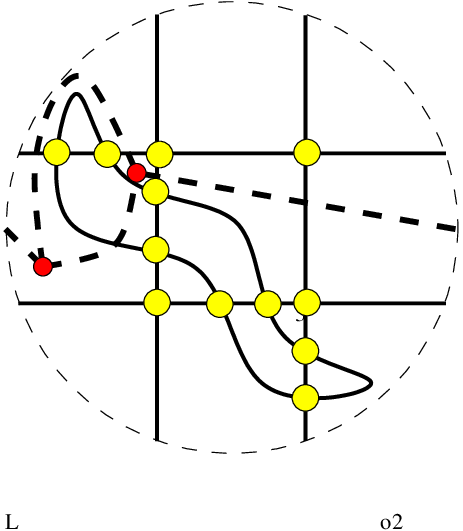}
\includegraphics[width = \factor\linewidth]{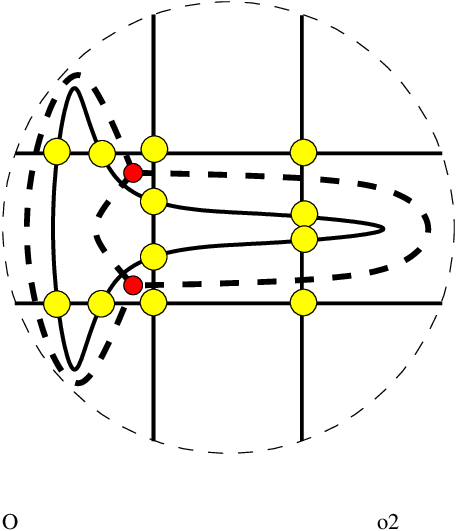}
\includegraphics[width = \factor\linewidth]{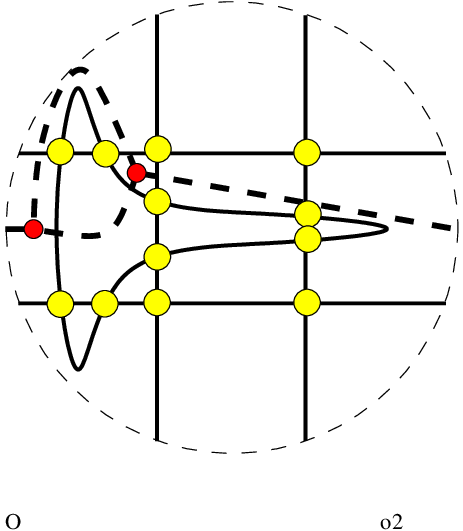}
\includegraphics[width = \factor\linewidth]{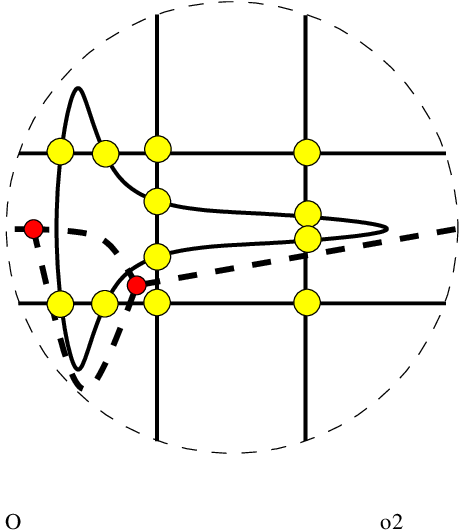}
\includegraphics[width = \factor\linewidth]{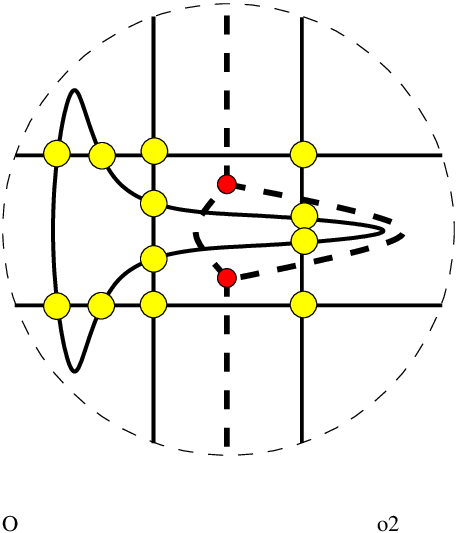}
\includegraphics[width = \factor\linewidth]{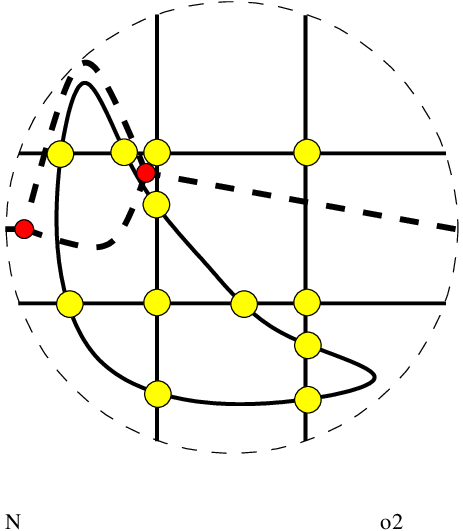}
\includegraphics[width = \factor\linewidth]{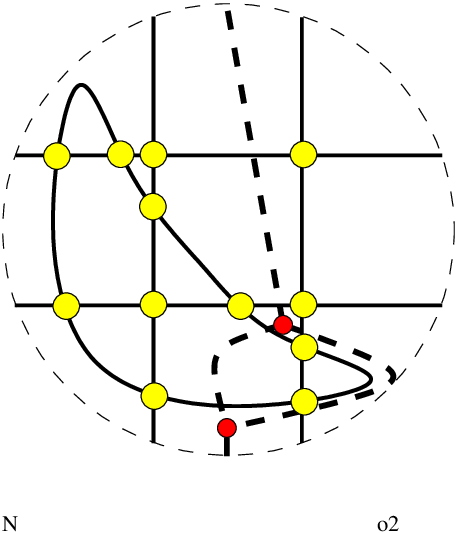}
\includegraphics[width = \factor\linewidth]{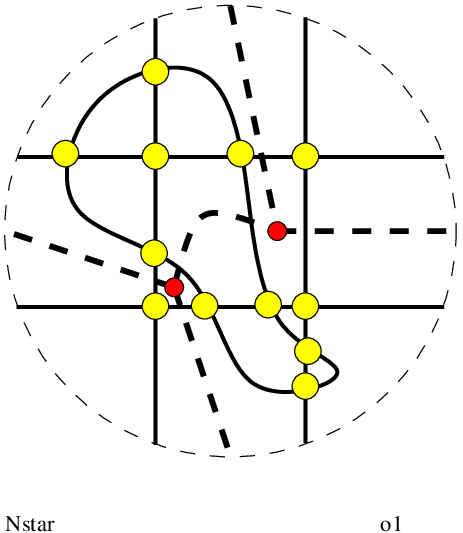}
\includegraphics[width = \factor\linewidth]{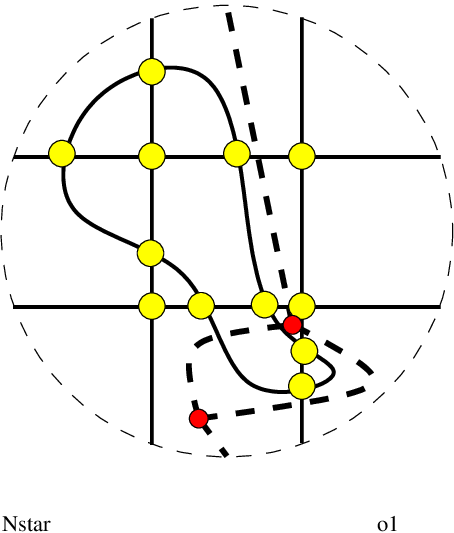}
\includegraphics[width = \factor\linewidth]{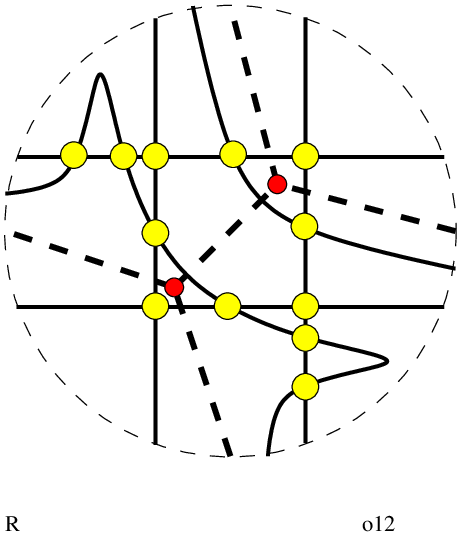}
\includegraphics[width = \factor\linewidth]{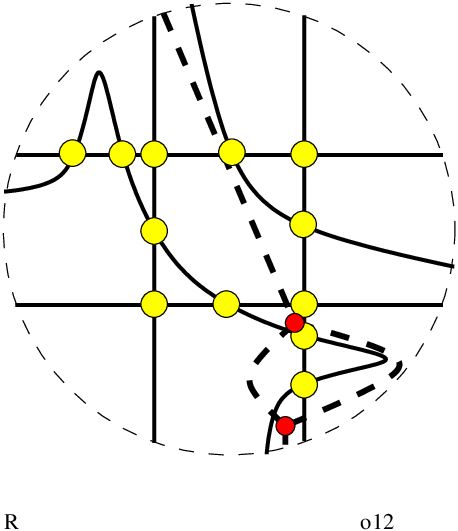}
\includegraphics[width = \factor\linewidth]{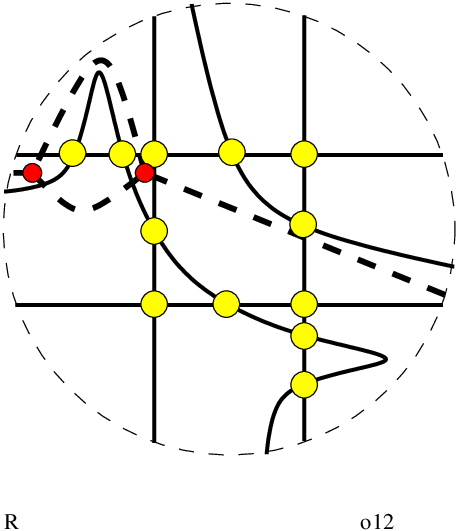}
\includegraphics[width = \factor\linewidth]{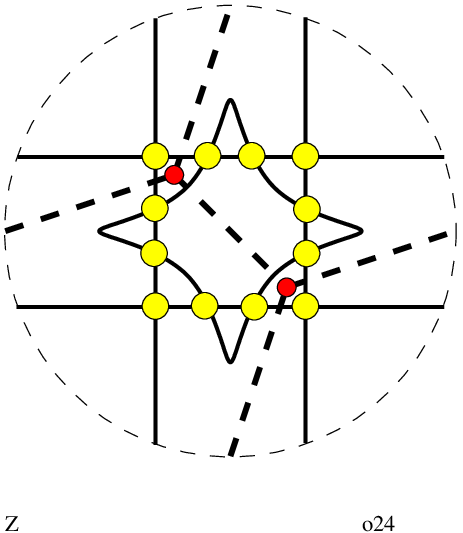}
\includegraphics[width = \factor\linewidth]{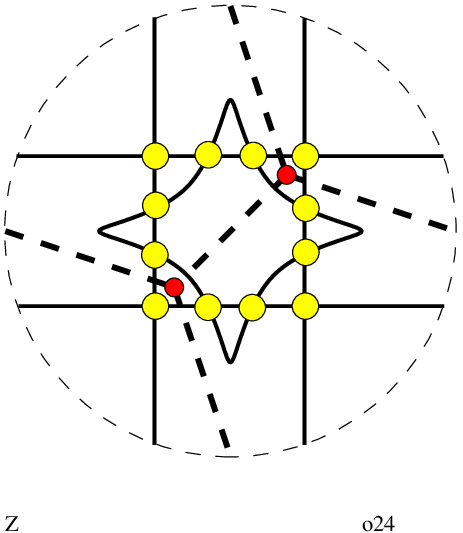}
\includegraphics[width = \factor\linewidth]{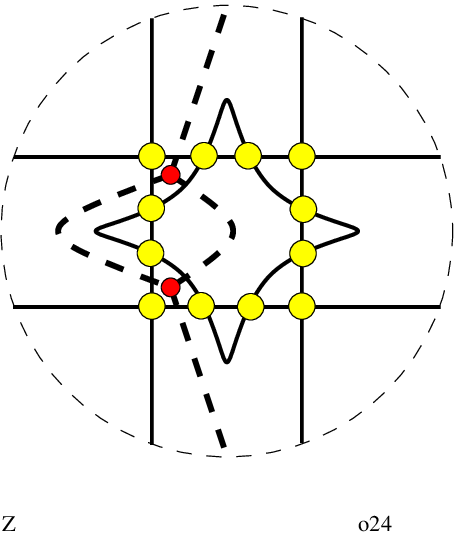}
\includegraphics[width = \factor\linewidth]{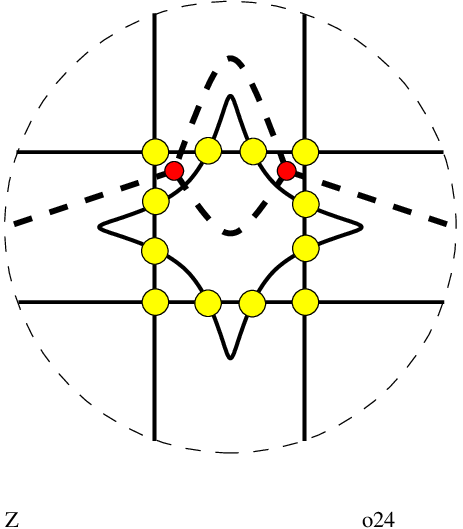}
\includegraphics[width = \factor\linewidth]{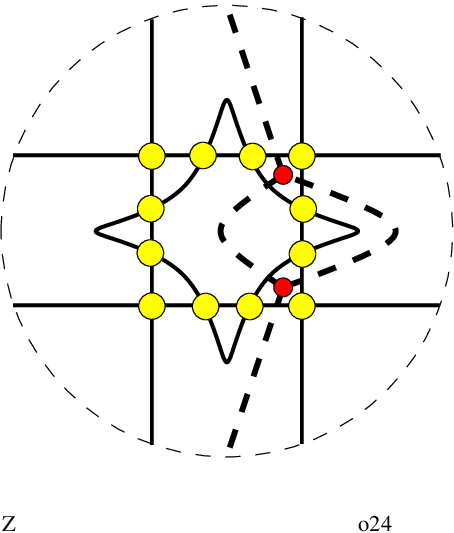}
\includegraphics[width = \factor\linewidth]{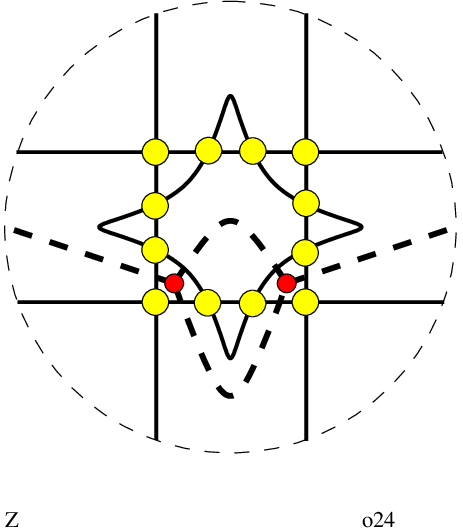}
\caption{Critical pairs of the arrangements $\name{33}$, $\name{32}$, $\name{25_2}$, $\name{25_1}$, $\name{32}$, $\name{22}$, $\name{25}$, $\name{36}$, 
and $\name{64}$ together with their associated critical graphs\label{twocriticalpairs}}
\end{figure}
(The arrangements $\name{04}$ and $\name{07}$ have no critical pairs.) 
A crucial observation is that critical graphs contain pseudolines with two exceptions:  one of the critical graphs the three critical pairs of $\name{22}$ and one of the critical graphs of 
the  four 
critical pairs of $\name{32}$ are free of pseudolines. 
Let us call {\it T-critical} a critical pair whose critical graph is free of pseudolines and {\it F-critical} a critical pair whose critical graph contains pseudolines (T is 
for truly and F is for falsely).
 Hence our Separation Lemma can be completed as follows  
\begin{lemma}[Separation Lemma]
\label{separationtwo} 
Let $GG'$ be a pair of distinct faces of an unbranched  2-covering $\widetilde{\Gamma}$ of an arrangement of double pseudolines $\Gamma$ of size 3 and genus 1. 
Then there exists a subarrangement $\Gamma'$ of $\Gamma$ of size at most $2$ whose faces $A$ and $A'$ in $\widetilde{\Gamma}'$ containing $G$ and $G'$ are distinct unless 
the pair $GG'$ is one of the two lifts of a T-critical pair $FF'$ of $\Gamma$ with the property that their faces are connected by the lift of the corresponding critical graph.  \qed
\end{lemma}
\begin{lemma}
Let $\Gamma$ be an arrangement of size $5$ whose subarrangements of size $4$ are of genus $1$, let $\gamma \in \Gamma$ and assume that $\casetwo{\Gamma}{\gamma}$ is nonempty. 
Then  $\caseone{\Gamma}{\gamma}$ is nonempty.
\end{lemma} 
\begin{proof}
Let $(NN',MM')\in \casetwo{\Gamma}{\gamma}$ with the property that there is no $(NN',XX') \in \casetwo{\Gamma}{\gamma}$ such that 
$NN',MM',XX'$ appear is this order on the node cycle of $\gamma$ in the arrangement $\Gamma$. 
Let $\nu'$ be the curve crossed by $\gamma$ at $N$.  
Let $M''$ be successor of $M'$ in the node cycle of $\gamma$. 
Let $L'$ be the first node of the node cycle of $\gamma$, not supported by $\nu'$, that follows $M'$, and let $L$ be its predecessor.  
A simple case analysis shows that either $LL'\in \caseone{\Gamma}{\gamma}$ or $M'M''\in \caseone{\Gamma}{\gamma}$. Hence $\caseone{\Gamma}{\gamma}$ is nonempty and we are done. \end{proof}
\begin{theorem}\label{preweakuptofour}
Let $\Gamma$ be an arrangement of size $5$  whose subarrangements of size $4$ are of genus $1$. Then
$\Gamma$ is of genus $1$ or $\Gamma$ contains at least one of the two subarrangements $\name{22}$ and $\name{32}$. 
\end{theorem}
\begin{proof} We argue exactly as in the proof of Theorem~\ref{uptofive} except that we work in a $2$-covering of $\Gamma$ in order to use Lemma~\ref{separationtwo} instead 
of Lemma~\ref{separation}.
\end{proof}

We arrive at the result announced in the introduction.
A {\it marked critical arrangement} is an arrangement of size $4$ and genus $1$ together with a pair $FF'$ of distinct faces (the mark) such that 
\begin{enumerate}
\item the S-number of $FF'$ is 3;
\item for any S-witness of $FF'$ of size $3$ the critical pair $AA'$ containing $FF'$ is T-critical;
\item the S-witnesses  of $FF'$ of size $3$ are  two in number.  
\end{enumerate}
Fig.~\ref{WeakuptoFour} shows three marked critical arrangements~:
the two S-witnesses of size $3$ of the mark are obtained by removing the curves labeled $\tau$ and $\tau'$.
Observe that the two last have the same underlying critical arrangement.
It is no hard to see that any marked critical arrangement is connected to one of these three by a sequence of mutations respecting the mark.  
It follows that marked critical arrangements are few dozens~: this number should be compared to the number ($6570$) of simple arrangements of size $4$ and genus $1$; cf.~\cite{fpp-nsafd-11}. 
\begin{figure}[!htb]
\centering
\psfrag{tau}{$\tau$}
\psfrag{taup}{$\tau'$}
\includegraphics[width = 0.975\linewidth]{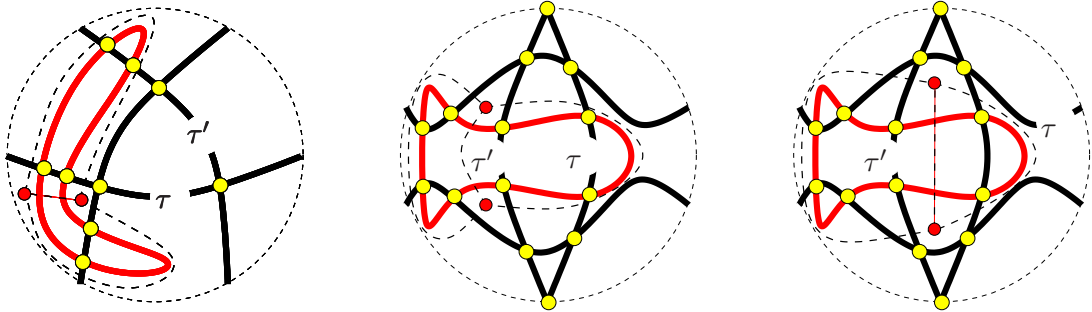}
\caption{Three marked  critical arrangements. In these diagrams the double pseudolines whose crosscap sides are free of vertices are simply represented by one of their core pseudolines 
 \label{WeakuptoFour}}
\end{figure}

\begin{theorem}\label{weakuptofour}
Let $\Gamma$ be an arrangement of double pseudolines of size $5$  whose subarrangements of size $4$ are of genus $1$. Then
$\Gamma$ is of genus $1$ or its subarrangements of size $4$ are critical arrangements.
\end{theorem}
\begin{proof} Let $\gamma \in \Gamma$ and assume that $\caseone{\Gamma}{\gamma}$ is nonempty. Let $NN'\in \caseone{\Gamma}{\gamma}.$
The curves crossing $\gamma$ at $N$ and $N'$ are denoted $\tau$ and $\tau'$, the face that $\gamma$ enters at $N$ is denoted $F$, and the face that $\gamma$ leaves 
at $N'$ is denoted $F'$. 
We let the reader check the following four  claims
\begin{enumerate}
\item The S-number of $FF'$ is 3;
\item for any S-witness of size $3$ the critical pair $AA'$ containing $FF'$ is T-critical;
\item $\tau\neq \tau'$; 
\item the S-witnesses  of $FF'$ are $\Gamma \setminus\{\gamma, \tau\}$ and $\Gamma\setminus \{\gamma,\tau'\}$;  
\end{enumerate}
from which it follows that $\Gamma'$ marked at $FF'$ is a marked critical arrangement. 
\end{proof}

As said in the introduction, Theorem~\ref{weakuptofour} shows that a computer check of the conjecture
that the arrangements of double pseudolines living in cross surfaces are those whose subarrangements of size at most $4$ live in cross surfaces, 
is doable with modest computing ressources. This computer check will the subject of another paper.

\clearpage
\section{An extension and a refinement}
\label{secsix} 
In this sixth and penultimate section 
we discuss {\it arrangements of pseudocircles},
{\it  crosscap or M{\"o}bius arrangements} and the {\it fibrations} of the
latter. The material on fibrations was partially motivated by the question raised by J.~E.~Goodman and R.~Pollack in~\cite{gp-cedcs-08}
about the realizability of their so-called double  permutation sequences by  
families  of pairwise disjoint convex bodies of affine topological  planes. 

Define an {\it oval} as the boundary of a convex body of a projective plane and the {\it dual of an oval} as the set of lines touching the oval but not its disk side.

An {\it arrangement of pseudocircles} is  a finite family of 
pseudocircles embedded in a cross surface, with the property that its subfamilies of size two are homeomorphic to 
  the dual arrangement of two points, two disjoint ovals, or one point and one  oval which are not incident. Observe that arrangements of pseudocircles extend both arrangements of pseudolines and arrangements of double pseudolines.
Fig.~\ref{addp} depicts  representatives of the isomorphism classes of arrangements of two pseudocircles.
\begin{figure}[!htb]
\centering
\psfrag{C}{$\alpha$}
\psfrag{G}{$\gamma$}
\psfrag{MSG}{$\MS(\gamma)$}
\psfrag{PP}{$\PP$}\psfrag{infty}{$\infty$}
\psfrag{1}{$1$-cell}\psfrag{2}{$2$-cell}\psfrag{0}{$0$-cell}\psfrag{3}{$\cal M$}\psfrag{-1}{$\emptyset$}
\psfrag{infty}{$\infty$}
\includegraphics[width=0.875\linewidth]{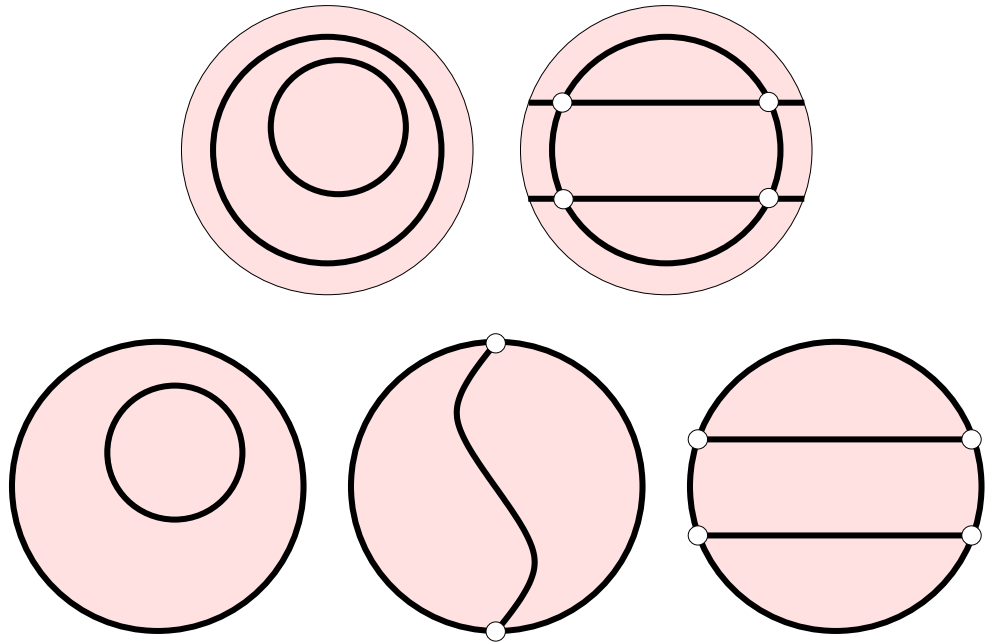}
\caption{Arrangements on two pseudocircles \label{addp}}
\end{figure}
The {\it order} of an arrangement of pseudocircles is defined as the isomorphism class of the poset of the bicolored curves of the arrangement ordered 
by reverse inclusion of their disk sides where by convention a pseudoline is colored in blue  ($\blu$) and a double pseudoline is colored in red ($\red$).
For example the orders of the arrangements depicted in the above figure are 
$\{\red \rightarrow \red\},\{\red,\red\},\{\blu \rightarrow \red)\},\{\blu,\blu\}$, and $\{\blu,\red\}$.
According  to Theorem~\ref{theoone}  examples of arrangements of pseudocircles are given by the dual arrangements of finite families of pairwise disjoint 
ovals and points of projective planes.
Minor adaptations  in our proof of the pumping lemma for arrangements of double pseudolines yield  the following pumping lemma for arrangements of pseudocircles.
We denote by $\TD(\gamma)$ the disk side of a double pseudoline $\gamma$.
\begin{lemma}[Pumping Lemma for Arrangements of Pseudocircles] \label{mainresultgen}
Let $\Gamma$ be a simple arrangement of pseudocircles, let $\gamma \in \Gamma$ be a double pseudoline of $\Gamma$, let 
$\Gamma'$ be the set $\gamma'\in \Gamma$ such that $\TD(\gamma')\supset \TD(\gamma)$, and let $M$ be the trace on the crosscap side of $\gamma$ 
of the  2-cell of the arrangement $\Gamma'$ that contains 
$\TD(\gamma)$. 
Assume that there is a vertex
of the arrangement $\Gamma$ lying in $M$. 
Then there is a triangular two-cell of the arrangement $\Gamma$
 contained in $M$  with 
a side supported by~$\gamma$.\qed
\end{lemma}
Minor variations in our proofs of Theorems~\ref{theoHT},~\ref{theoPM}, and~\ref{theoADP} based on  the above  pumping lemma 
lead to direct extensions of these theorems modulo the following elementary 
dictionnary:

\begin{tabular}{ccc}
&&\\
arrangements of double pseudolines  & $\longleftrightarrow$  & arrangements of pseudocircles \\
convex bodies               & $\longleftrightarrow$   &  ovals and points\\
size                        & $\longleftrightarrow$   &  order\\
&&
\end{tabular}

\noindent
In particular
\begin{theorem}
\label{fmixedpmr}  
Any arrangement  of pseudocircles
is isomorphic to the dual family of a finite family of pairwise disjoint ovals and points  of a projective plane.  \qed
\end{theorem}

For the numbers of isomorphism classes of simple arrangements  with trivial order ($\{\blu,\ldots,\blu,\red,\ldots,\red\}$)
on at most five curves we refer to~\cite{fpp-nsafd-11}, where these arrangements are called mixed arrangements. 


We now discuss a refinement of Theorem~\ref{fmixedpmr} in 
the context of {\it crosscap or M{\"o}bius arrangements} and their {\it fibrations.}  
Let $\MS$ be a M{\"o}bius strip, let $\OPC= \MS \cup \{\infty \}$ be its one-point compactification, and recall that $\OPC$ is a 
cross surface.
Define a {\it arrangement of pseudocircles in $\MS$} as a arrangement of pseudocircles in $\OPC$ with the property that the intersection of 
the disk sides  of  its  pseudolines  and double pseudolines is nonempty and contains the point at infinity; 
define a {\it fibration} of an  arrangement of pseudocircles $\Gamma$ 
in $\MS$ as a sup-arrangement $\Gamma'$ of $\Gamma$ in $\OPC$ composed of
the pseudocircles of $\Gamma$ and of the pseudolines of a pencil of pseudolines through the point at infinity with the property that any pseudoline of the pencil goes through a vertex of $\Gamma$ 
and any vertex of $\Gamma$ is incident to a pseudoline of the pencil. 
Fig.~\ref{fibration} shows an arrangement of two double pseudolines in a M{\"o}bius strip and representatives of its three possible isomorphism classes of fibrations.   
\begin{figure}[!htb]
\centering
\footnotesize
\psfrag{\infty}{$\infty$}
\includegraphics[width=0.9875\linewidth]{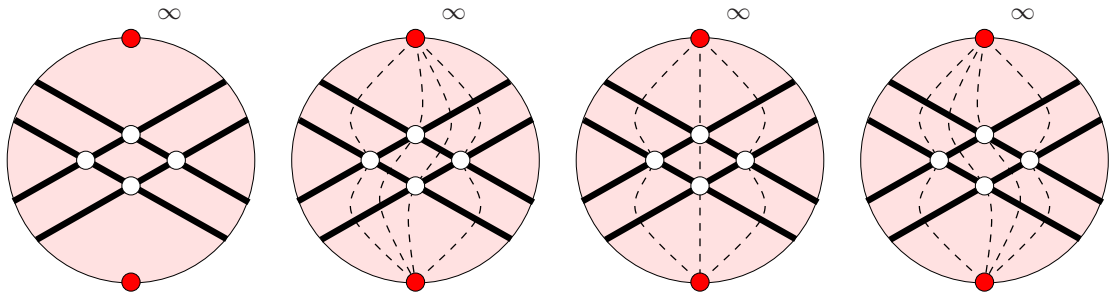}
\caption{An arrangement of two double pseudolines in a M{\"o}bius strip with representatives of its three isomorphism classes of fibrations \label{fibration}}
\end{figure}
According to Theorem~\ref{fmixedpmr} we see easily that any arrangement of pseudocircles $\Gamma$ in  $\MS$
is the dual arrangement of a family of pairwise disjoint ovals and points of  a projective plane $\GPP$ with line space $\OPC$,
with the property that the line at infinity $\infty$ avoids the ovals and the points of the family.  
Consequently the family $\Gamma'$  composed of the pseudocircles of $\Gamma$ and of the dual pseudolines of the intersection points of the line at infinity with the vertices 
of $\Gamma$ is a fibration of $\Gamma$. This fibration is denoted $\Gamma^{\GPP}$.
Conversely let $\Gamma'$ be a fibration of $\Gamma$. Does there exist a projective plane $\GPP$ such that $\Gamma'=\Gamma^{\GPP}$?  
Applying Theorem~\ref{fmixedpmr} to $\Gamma'$ we obtain a positive answer to that question.   

\begin{theorem} Let $\MS$ be a M{\"o}bius strip and  let $\Gamma'$ be a fibration of an arrangement of pseudocircles $\Gamma$ in $\MS$. 
Then $\Gamma$ is the dual family of a finite family of pairwise disjoint ovals and points  of a projective plane $\GPP$ with line space $\OPC$ such that $\Gamma' = \Gamma^{\GPP}$. \qed
\end{theorem}
In particular the above theorem answers positively the question raised by J.~E.~Goodman and R.~Pollack in~\cite{gp-cedcs-08} and~\cite[Problem 20]{dghp-iscep-05}
about the realizability of their so-called double permutation sequences and, more generally,  allowable interval sequences by  
families  of pairwise disjoint convex bodies of real two-dimensio\-nal affine topological planes. 
Indeed double permutation sequences and allowable interval sequences are simply a coding of isomorphism classes of our fibrations.  
For example the allowable interval sequences coding the  three isomorphism 
classes of fibrations of an arrangement of two double pseudolines indexed by $\{1,2\}$ are the following
\begin{equation}
\begin{bmatrix}
2 & 2 & 2 & 1 & 1 \\
2 & 1 & 1 & 2 & 1 \\
1 & 2 & 1 & 1 & 2 \\
1 & 1 & 2 & 2 & 2
\end{bmatrix},
\begin{bmatrix}
2 & 2 & 1 & 1 \\
2 & 1 & 2 & 1 \\
1 & 2 & 1 & 2 \\
1 & 1 & 2 & 2
\end{bmatrix},
\begin{bmatrix}
2 & 2 & 1 & 1 & 1 \\
2 & 1 & 2 & 2 & 1 \\
1 & 2 & 2 & 1 & 2 \\
1 & 1 & 1 & 2 & 2
\end{bmatrix}.
\end{equation}
We refer to~\cite{dghp-iscep-05,gp-cedcs-08} for the precise definition of double permutation sequences and that of allowable interval sequences. 
For recent applications of these notions see~\cite{n-aissc-12,n-aislt-12}. 

We conclude with  the statements of the counterparts of Theorems~\ref{theoHT} and~\ref{theoADP} in the context 
of arrangements in M{\"o}bius strips. 
\begin{theorem}
\label{HPHT} 
Let $\MS$ be a M{\"o}bius strip.
Any two arrangements of pseudocircles in $\MS$ of the same order are homotopic in $\MS$
via a finite sequence of mutations followed by an isotopy. \qed
\end{theorem}
An arrangement of oriented pseudocircles is termed  {\it acyclic} if the orientations of the pseudocircles are coherent, 
in the sense that the pseudocircles are oriented according to the choice of a generator of the (infinite cyclic) 
fundamental group of the underlying M{\"o}bius strip.

\begin{theorem}
\label{tmmr}
\label{tmmrone}
The map that assigns to an isomorphism class of M{\"o}bius arrangements of pseudocircles its chirotope is  one-to-one 
and its range is the set of maps $\chi$ defined on the set of $3$-subsets of a finite set $I$ such that  
for every  $3$-, $4$-, and $5$-subset $J$ of $I$  
the restriction of $\chi$ to the set of $3$-subsets of $J$ is a chirotope of  M{\"o}bius arrangements of pseudocircles.
Furthermore the same result holds for the class of acyclic M{\"o}bius arrangements of pseudocircles.\qed
\end{theorem}
Note that our homotopy theorem  for M{\"o}bius arrangements provides an algorithm to enumerate 
the isomorphism classes of M{\"o}bius arrangements
by traversing the associated mutation graphs. We have implemented this
algorithm to enumerate the isomorphism classes of the simple M{\"o}bius arrangements of double
pseudolines.   
Preliminary counting results (confirmed by two independent implementations) are reported in the following table
$$
\begin{array}{c|cccc}
n  & 1 & 2 & 3 & 4 \\
\hline
a_n & 1 & 1 & 118 & 541820\\
b_n & 1 & 1 & 22  & 22620 \\
c_n & 1 & 1 & 16  & 11502 \\
d_n & 1 & 1 & 12  &  5955 
\end{array}
$$
Here the index $n$ 
 refers to the number of double pseudolines, $a_n$ is the number of 
isotopy classes of simple indexed arrangements of double pseudolines, $b_n$ is the number of  
isotopy classes of simple arrangements of double pseudolines, $c_n$ is the number of isomorphism classes of simple arrangements of double pseudolines, and 
$d_n$ is the number of isomorphism classes of simple arrangements of double pseudolines considered as projective arrangements. 
Fig.~\ref{mobiuslist} depicts representatives of the $22$ isomorphism classes of non indexed arrangements of three double pseudolines: 
\begin{figure}[!htbp]
\footnotesize
\def\factor{0.21315015000023}
\centering
\psfrag{8}{\small 8} \psfrag{7}{\small 7} \psfrag{6}{\small 6} \psfrag{5}{\small 5} \psfrag{4}{\small 4} \psfrag{3}{\small 3} \psfrag{2}{\small 2}
\psfrag{8}{} \psfrag{7}{} \psfrag{6}{} \psfrag{5}{} \psfrag{4}{} \psfrag{3}{} \psfrag{2}{}
\psfrag{AA}{}
\psfrag{A}{$\nameM_{04}$}
\psfrag{B}{$\nameM_{07}$}
\psfrag{C}{$\nameM_{18}$}
\psfrag{G}{$\nameM_{37}$}
\psfrag{H2}{$\nameM_{15(2)}$}
\psfrag{H}{$\nameM_{15(3)}$}
\psfrag{Hm}{$\nameM_{15(3)}^\star$}
\psfrag{J}{$\nameM_{43}$}
\psfrag{L}{$\nameM_{33}$}
\psfrag{Lm}{$\nameM_{33}^\star$}
\psfrag{N2}{$\nameM_{25_2(2)}$}
\psfrag{N3}{$\nameM_{25_2(3)}$}
\psfrag{N2m}{$\nameM_{25(2)}^{\star}$}
\psfrag{Nstar3}{$\nameM_{25_1(3)}$}
\psfrag{Nstar2}{$\nameM_{25_1(2)}$}
\psfrag{Nstar3m}{$\nameM_{25_1(3)}^\star$}
\psfrag{Nstar2m}{$\nameM_{25_1(2)}^\star$}
\psfrag{O}{$\nameM_{32}$}
\psfrag{Om}{$\nameM_{32}^{\star}$}
\psfrag{P4}{$\nameM_{22(4)}$}
\psfrag{P2}{$\nameM_{22(2)}$}
\psfrag{R}{$\nameM_{36}$}
\psfrag{i1}{$A$}
\psfrag{i2}{$B$}
\psfrag{i3}{$C$}
\psfrag{tr}{$3$}
\psfrag{un}{$1$}
\psfrag{de}{$2$}
\includegraphics[width = \factor\linewidth]{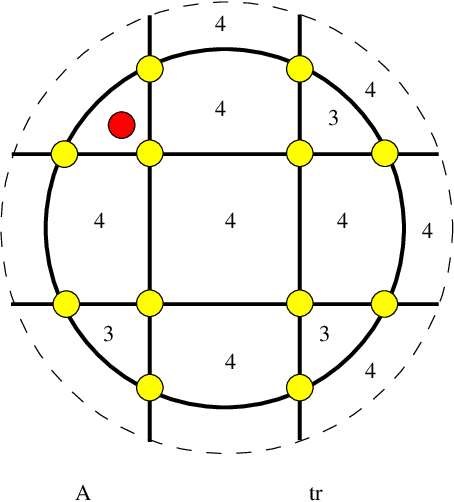}
\includegraphics[width = \factor\linewidth]{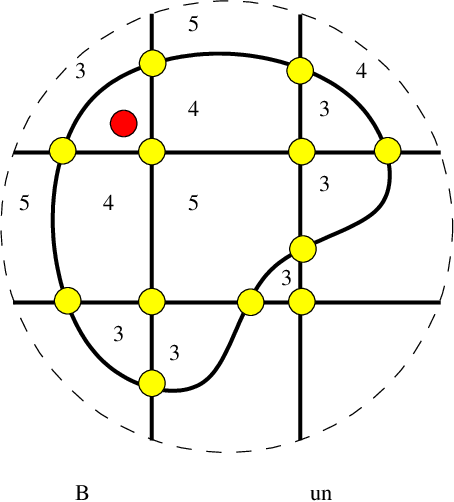}
\includegraphics[width = \factor\linewidth]{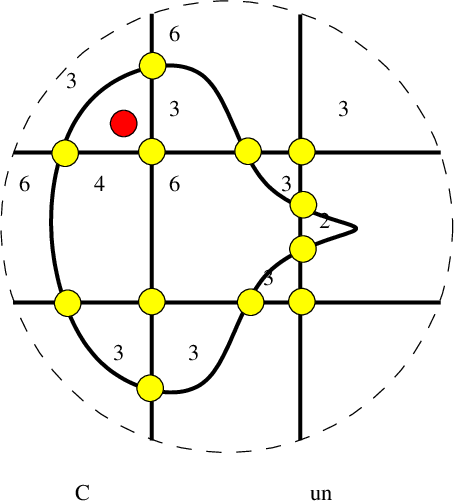}
\includegraphics[width = \factor\linewidth]{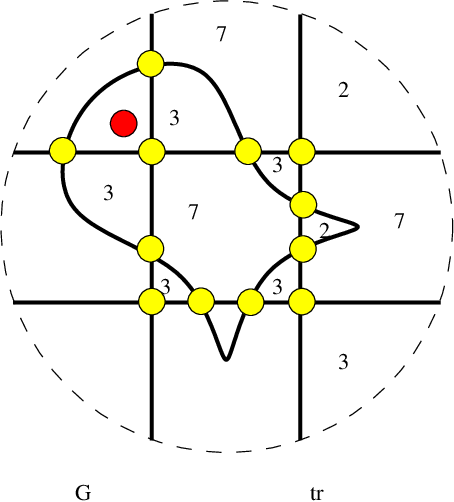}
\includegraphics[width = \factor\linewidth]{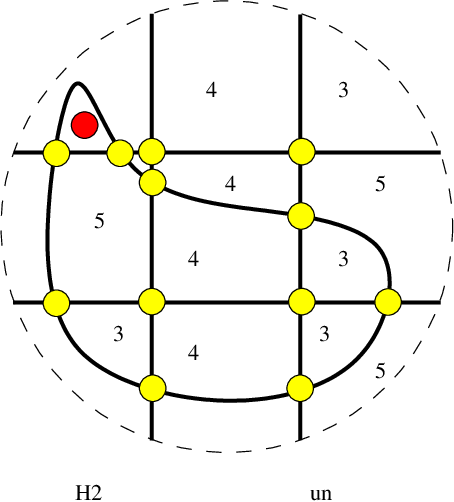}
\includegraphics[width = \factor\linewidth]{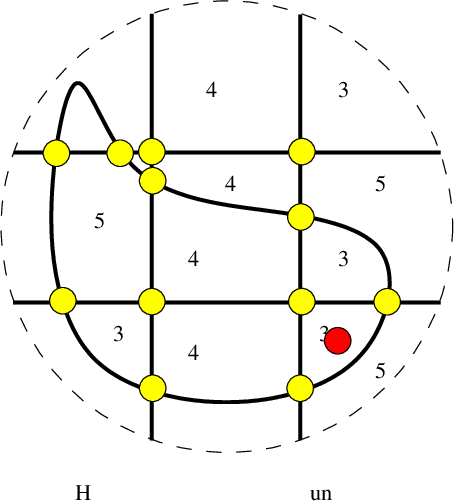}
\includegraphics[width = \factor\linewidth]{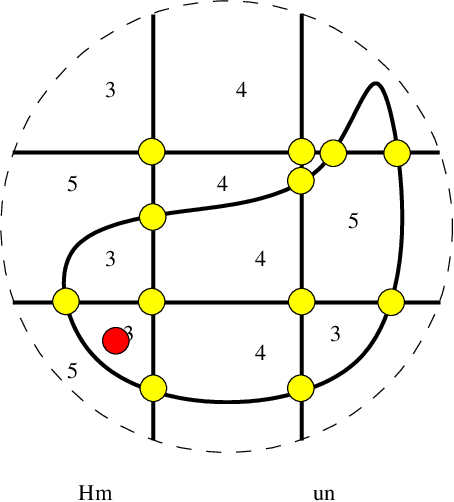}
\includegraphics[width = \factor\linewidth]{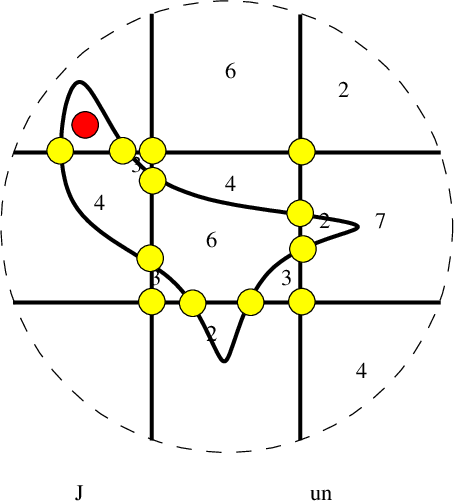}
\includegraphics[width = \factor\linewidth]{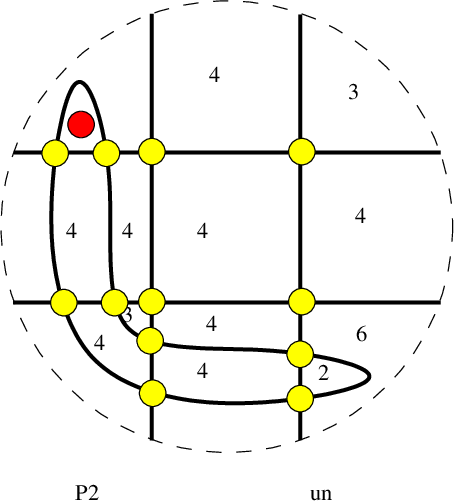}
\includegraphics[width = \factor\linewidth]{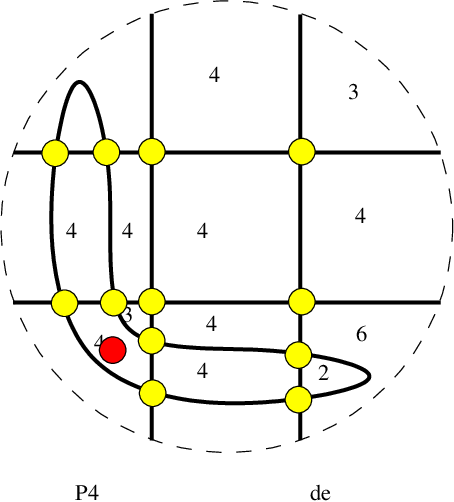}
\includegraphics[width = \factor\linewidth]{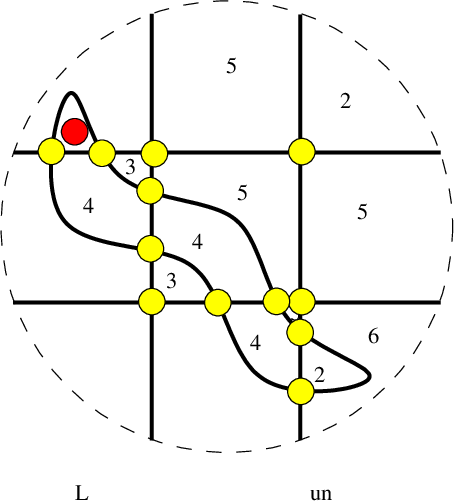}
\includegraphics[width = \factor\linewidth]{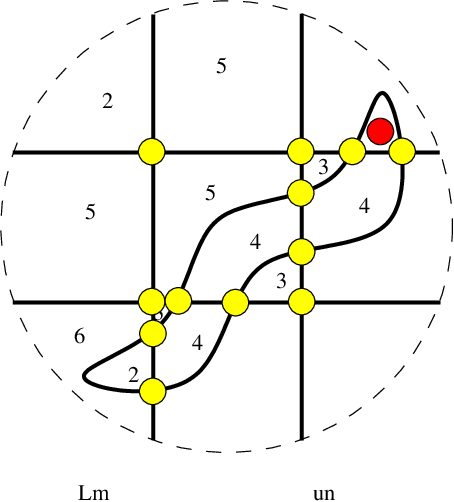}
\includegraphics[width = \factor\linewidth]{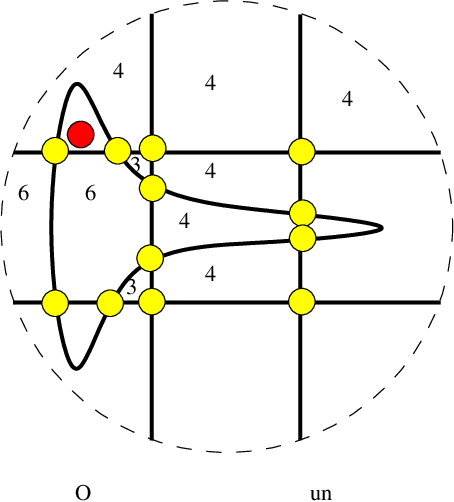}
\includegraphics[width = \factor\linewidth]{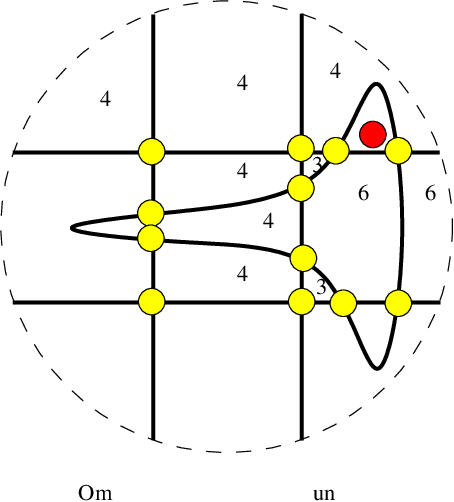}
\includegraphics[width = \factor\linewidth]{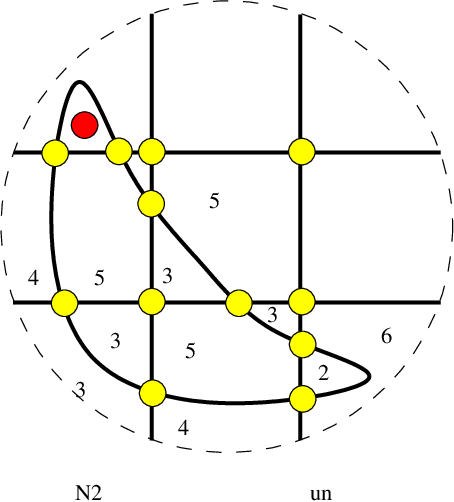}
\includegraphics[width = \factor\linewidth]{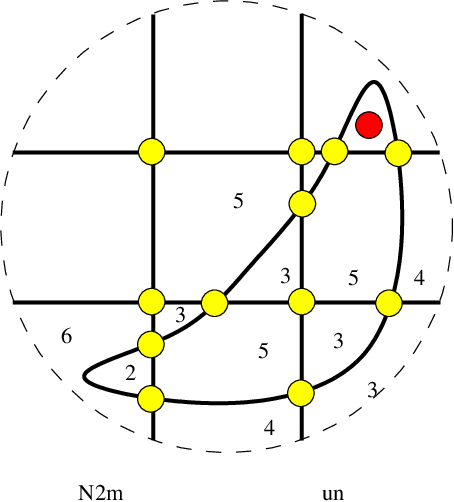}
\includegraphics[width = \factor\linewidth]{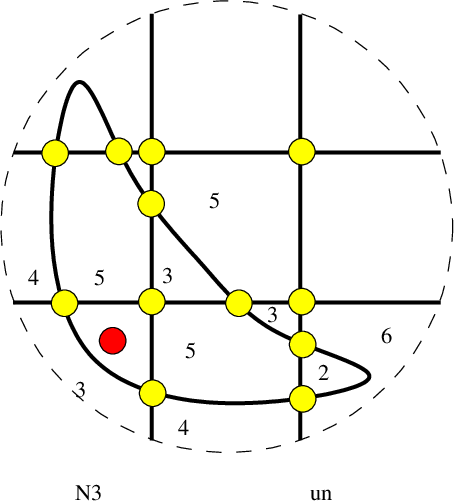}
\includegraphics[width = \factor\linewidth]{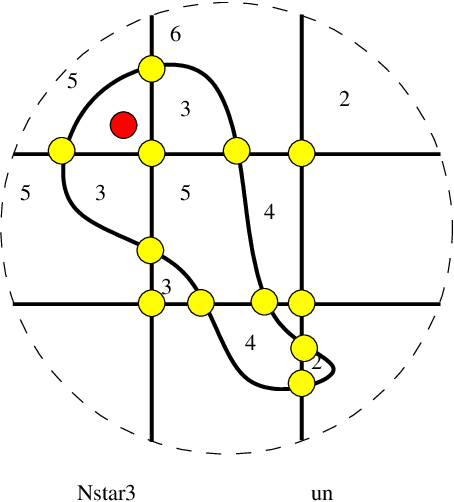}
\includegraphics[width = \factor\linewidth]{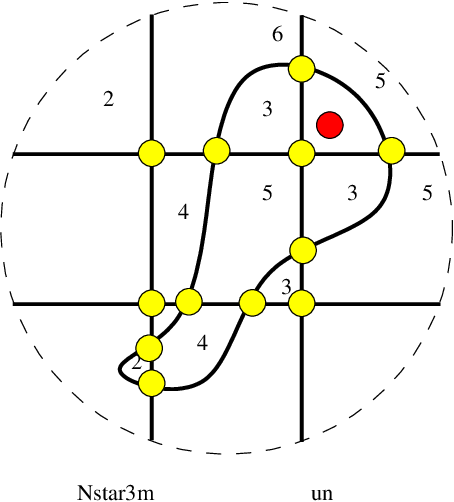}
\includegraphics[width = \factor\linewidth]{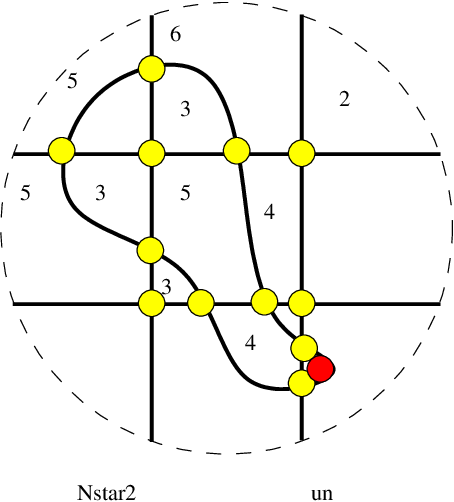}
\includegraphics[width = \factor\linewidth]{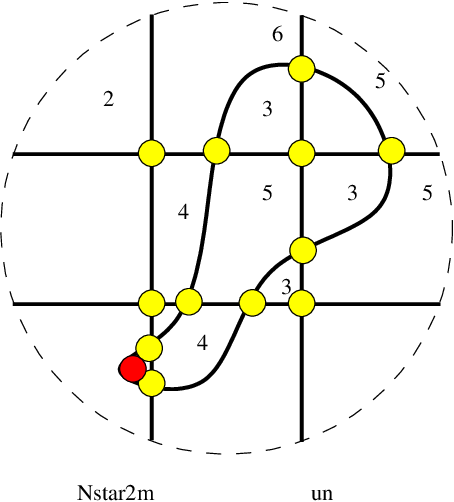}
\includegraphics[width = \factor\linewidth]{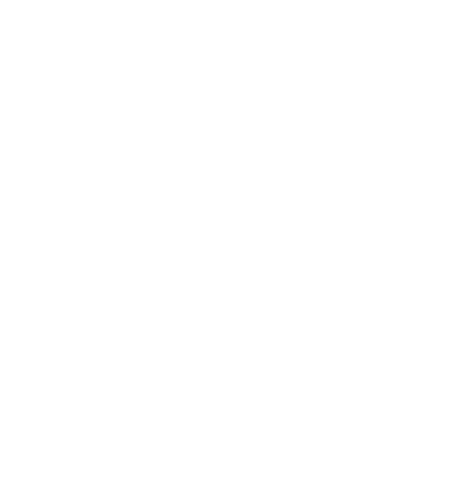}
\includegraphics[width = \factor\linewidth]{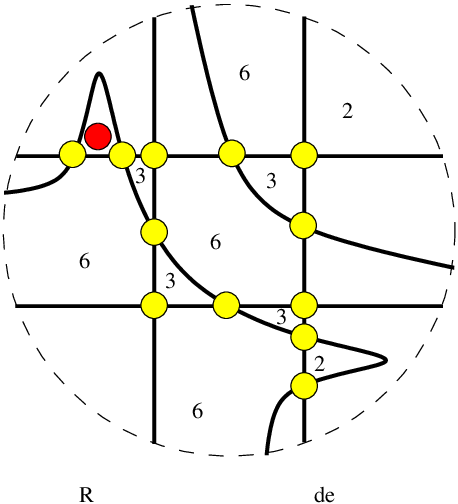}
\includegraphics[width = \factor\linewidth]{000000}
\caption{Representatives of the 22 isomorphism classes of simple non indexed arrangements on three double pseudolines  \label{mobiuslist}}
\end{figure}
Each diagram is labeled at its bottom left with a symbol to name it 
(of type $\nameM_\alpha$ where $\alpha$ is the $2$-sequence of its numbers of $2$-cells of size $2$ and $3$ possibly followed, between brackets, 
with the size of the unbounded $2$-cell of the arrangement in the case where there are several arrangements with the same $2$-sequence; $\nameM_{\alpha}$ and $\nameM_{\alpha}^\star$ 
are mirror images of one another) and is labeled at its 
bottom right with the size of its automorphism group; thus the number ($118$) of simple chirotopes of families of three pairwise disjoint convex bodies 
on a given indexing set of size $3$  can be computed as the sum 
$$\sum_{k\geq 1} \frac{3!}{k} g_k = 
\frac{6}{1} \times 18
+
\frac{6}{2} \times 2 
+ 
 \frac{6}{3}\times 2 
$$
where $g_k$ is the number of arrangements of Fig.~\ref{mobiuslist}  with group of automorphisms  of order~$k$.  
Finally we mention that there are $531$ (simple and non simple) chirotopes on a given indexing set of size $3$; cf.~\cite{fpp-nsafd-11}.

It is interesting to mention that 
the canonical embedding of Lemma~\ref{EmbeddingSimple} can be extended, in the case of M{\"o}bius arrangements, 
to the whole class of simple and non simple arrangements as explained below. 

Define a {\it pencil} of double pseudolines as a M{\"o}bius arrangement of double pseudolines
with the property that any of its subarrangements has only two {\it external} vertices, i.e., only two vertices in the
boundary of the two-cell that contains the point at infinity of the one-point
compactification of the underlying M{\"o}bius strip. 
Fig.~\ref{pencildouble} shows pencils of two, three, four and five double pseudolines; it is
\begin{figure}[!htb]
\centering
\psfrag{\infty}{$\infty$}
\includegraphics[width=0.9875\linewidth]{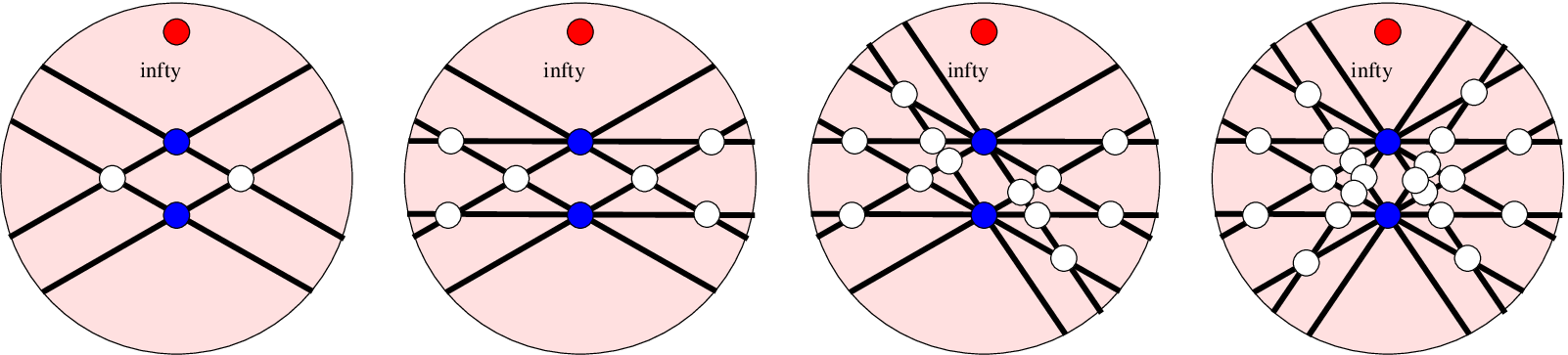}
\caption{Pencils of $2$, $3$, $4$, and $5$ double pseudolines
\label{pencildouble}}
\end{figure}
not hard to see that the isotopy class of a pencil of double pseudolines depends
only on its number of double pseudolines.
Now a M{\"o}bius arrangement of double pseudolines is termed {\it thin} if any of its 
subarrangements whose double pseudolines have associated M{\"o}bius strips with
nonempty intersection is a pencil of double pseudolines---to put it differently,  a M{\"o}bius arrangement of double pseudolines is thin if the crosscap sides of its double pseudolines are free of 
 external vertices---and  a M{\"o}bius arrangement of double pseudolines $\Gamma^*$ is termed 
a {\it double} of a M{\"o}bius pseudoline arrangement $\Gamma$ if there exists a one-to-one correspondence between $\Gamma$ and $\Gamma^*$ such that 
\begin{enumerate}
\item any pseudoline of $\Gamma$ is contained in the crosscap side of its corresponding double pseudoline in $\Gamma^*$, and 
\item any subarrangement of $\Gamma^*$ is a pencil of double pseudolines if and only if the  corresponding subarrangement of $\Gamma$ is a pencil of 
pseudolines.
\end{enumerate} 
In this M{\"o}bius setting  Lemma~\ref{EmbeddingSimple} can be read as follows. 
\begin{lemma}\label{embedding}
Let $\MS$ be a M{\"o}bius strip.
The map that assigns to an arrangement of pseudolines in $\MS$ its set of double versions induces a one-to-one 
and onto 
correspondence between the set of isotopy classes 
of pseudoline arrangements in $\MS$ and the set of isotopy  classes of thin 
double pseudoline arrangements in $\MS$.\qed 
\end{lemma}


\section{Conclusion and open problems}
We have introduced the notion of arrangements of double pseudolines as a combinatorial abstraction 
of families of pairwise disjoint convex bodies of projective planes and we have extended to that setting 
well-known fundamental properties of arrangements of pseudolines.  

Several open questions are raised by our work. We mention five of them below.

The cell poset of a simple arrangement of pseudolines presented  by its chirotope is 
computable in optimal quadratic time and linear working space; cf.~\cite{eg-tsa-89}.  
Can we achieve similar bounds for arrangements of double pseudolines presented by their chirotopes? 
Progress in this direction using the notion of pseudotriangulation 
is reported in our companion paper~\cite{hp-cpvbc-12}. 

There is a closed formula, due to R.~Stanley, counting the number $r_n$  of wiring representations of simple M{\"o}bius arrangements of $n$ pseudolines, 
namely 
$$r_n = \frac{\binom{n}{2}!}{(2n-3)(2n-5)^2(2n-7)^3\cdots 5^{n-3}3^{n-2}}.$$ 
Most of the proofs of this formula, if not all,  are based on  connections between standard Young Tableaux, reduced words and arrangements of pseudolines;
cf.~\cite{s-nrdec-84,eg-bt-87,ls-shaca-82,h-deaic-92,f-srwce-01}. 
Are there similar connections  for M{\"o}bius arrangements of double pseudolines? 
Is there any similar formula counting the number of wiring representations of simple M{\"o}bius arrangements of double pseudolines?

Say that an arrangement of $n$ double pseudolines is realizable if it is the dual of a family of $n$ 
disjoint disks of the standard two-dimensional projective plane  $\spp$.  In that case one can think of the arrangement 
 as the trace on the unit sphere of $\mathbb{R}^3$  of a centrally symmetric affine arrangement of $2n$ planes 
with the property that the distance to the origin of any line defined as the intersection of $d$ of these planes is less than~$1$. 
The open question is then the following : Is  any arrangement of double pseudolines the trace on the unit sphere 
of a centrally symmetric affine arrangement of pseudoplanes in $\mathbb{R}^3$?

Arrangements of double pseudolines are dual families of families of pairwise disjoint convex bodies of projective planes.   
What is the smallest example which is not realizable as the dual of a family of pairwise disjoint disks of the standard projective plane 
 $\spp$? (By a disk we mean  a convex body whose boundary is a circle, i.e., the intersection of 
the unit sphere of $\mathbb{R}^3$ with an affine plane.)  

Arrangements of double pseudolines generalize arrangements of pseudolines.  
What are the similar generalizations  for arrangements of pseudohyperplanes of dimensions 4, 5, etc.?
A generalization that comes naturally to mind  defines  
(1) a double pseudohyperplane as the image of the hypersurface $x_{1}= \pm 1/\sqrt{2}$ of 
the projective space $\mathbb{RP}^{d}$ (defined as the quotient of the unit sphere of  $\mathbb{R}^{d+1}$ under the antipodal map) 
under a self-homeomorphism of $\mathbb{RP}^{d}$, and 
 (2) an arrangement of double pseudohyperplanes 
as a finite family of double pseudohyperplanes with the property that its subfamilies of size $d$ are the images of the arrangement composed 
of the $d$ hypersurfaces $x_{i}= \pm 1/\sqrt{2d}$, $i=1,2,\ldots, d$, 
 under a self-homeomorphism of $\mathbb{RP}^{d}$.  Two questions arise  naturally : 
(1) Does the isomorphism class of a (indexed and oriented) double pseudohyperplane arrangement depend only on its chirotope, i.e., the family of 
isomorphism classes of its subarrangements of size $d+1$? 
(2) Does the class of chirotopes of double pseudohyperplane arrangements coincide with the class of maps that assigns to each $(d+1)$-subset of indices 
an isomorphism class of arrangements of double pseudohyperplanes indexed by that $(d+1)$-subset and whose restrictions to the sets of 
$(d+1)$-subsets of $(d+3)$-subsets of indices are chirotopes of double pseudohyperplane arrangements?

\section*{Acknowledgments.} 
We thank G{\"u}nter Rote for the fruitful discussions we have had on the pumping lemma  during its visit at the Ecole normale sup{\'e}rieure in October 2005, 
the Mezzenile research group for its constant support and for providing us an independent implementation 
of the enumeration algorithm for crosscap arrangements of double pseudolines described in this paper, 
Ricky Pollack for his constructive criticism, his patience and generosity, 
the anonymous referees for their very helpful critical comments on the first drafts, and 
an anonymous referee for his insightful feeback on the pumping lemma.

\clearpage

\bibliographystyle{abbrv}
\bibliography{$HOME/BIB/total,$HOME/BIB/geom,$HOME/BIB/geomplus,$HOME/BIB/livre,$HOME/BIB/divers,$HOME/BIB/algo}

\clearpage
\appendix

\clearpage
\section{Arrangements of pseudolines}
\label{appendix:apl}
We review the basics of  arrangements of pseudolines that fall within the general scope of the paper.  
\subsection{LR characterization}
An {\it arrangement of pseudolines} is a finite set of pseudolines living in the same cross surface with the property  that any two pseudolines intersect in
exactly one point. Two arrangements of pseudolines are {\it isomorphic} if one is the image of the other by an homeomorphism of their underlying cross surfaces.
Fig.~\ref{arrlines} depicts representatives of the isomorphism classes of arrangements of at most five pseudolines.
\begin{figure}[!htb]
\centering
\psfrag{Aaaa}{}
\psfrag{Azer}{}
\psfrag{Aone}{$A$}
\psfrag{Atwo}{$B$}
\psfrag{Athr}{$C$}
\psfrag{Afou}{$D$}
\psfrag{Afiv}{$E$}
\psfrag{Asix}{$F$}
\psfrag{Asev}{$G$}
\psfrag{Ahei}{$H$}
\psfrag{Anin}{$I$}
\psfrag{Aten}{$J$}
\psfrag{Aaaaor}{$2$}
\psfrag{Azeror}{$8$}
\psfrag{Aoneor}{$24$}
\psfrag{Atwoor}{$12$}
\psfrag{Athror}{$24$}
\psfrag{Afouor}{$12$}
\psfrag{Afivor}{$16$}
\psfrag{Asixor}{$10$}
\psfrag{Asevor}{$4$}
\psfrag{Aheior}{$8$}
\psfrag{Aninor}{$16$}
\psfrag{Atenor}{$20$}

\includegraphics[width=0.9875\linewidth]{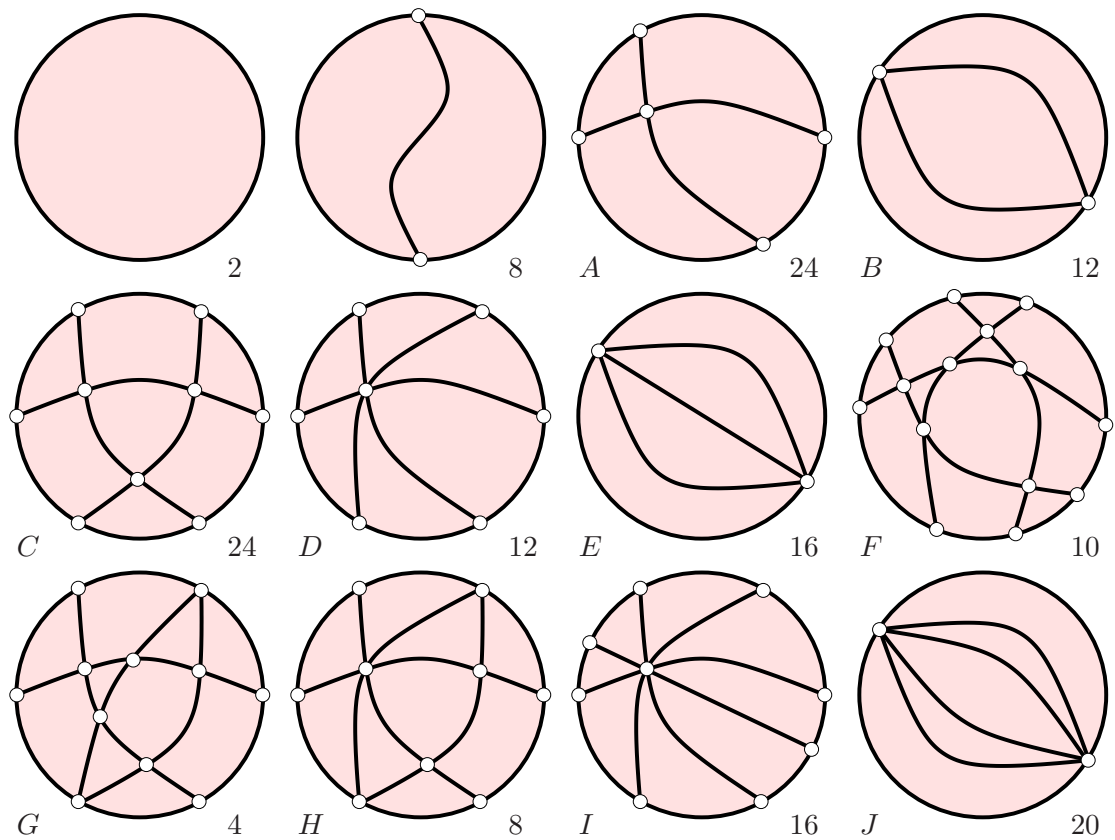}
\caption{Representatives of the isomorphism classes of arrangements of one, two, three, four and five pseudolines. Each diagram is labeled at its bottom right by the size of its automorphism group and at its bottom left by a symbol to name it \label{arrlines}}
\end{figure}
The question of understanding the isomorphism relation between arrangements of pseudolines was addressed and solved
by Ringel~\cite{r-tegtg-56,r-ugal-57} for simple arrangements and by Folkman and Lawrence~\cite{fl-om-78} for any arrangements---moreover, in the broader context of arrangements of pseudo\-hyperplanes---essentially as indicated in the following theorem  
where the term {\it chirotope} applied to an indexed arrangement of oriented pseudolines means the map that assigns to each $3$-subset of indices of the arrangement  
the isomorphism class of the subarrangement indexed by this $3$-subset.

\begin{theorem}[{\bf \cite{r-tegtg-56,r-ugal-57,fl-om-78}}]
\label{pretheofou} The map that assigns to an isomorphism class of indexed arrangements of oriented pseudolines
its chirotope is one-to-one and its range is the set of maps $\chi$ defined on the set of $3$-subsets of a finite set~$I$ 
such  that for every $3$-, $4$-, and $5$-subset $J$ of $I$
the restriction of $\chi$ to the set of $3$-subsets of $J$ is a chirotope of indexed arrangements of oriented pseudolines. 
\qed
\end{theorem}
\subsection{Chirotopes of small size}
The above theorem can be complemented by a comprehensive description of the indexed and oriented arrangements on 3, 4 and 5 pseudolines as we now explain. 
We use the idea of signed indices of an indexing set, namely the original indices $i_1,i_2,\ldots,i_n$ and their complements
$\overline{i}_1, \ldots, \overline{i}_n$. The original indices are said to be positive, their complements are said to be negative, and we 
define the complement $\overline{\eta}$ of a negative index $\eta$ as its positive version.
Let $X$ be an arrangement of pseudolines, let $X_*$ be an indexed and oriented version of $X$ and 
extend $X_*$ to the complements of the original indices by assigning to a negative index the reoriented version of the pseudoline assigned to its complement.  
Let $G$ be the group of permutations of the signed indices which are compatible with the complement operation 
and let $G_X$ be the group of automorphisms of $X$.
Clearly the map that assigns to $\sigma \in G_X$ its conjugate $ X_*^{-1}\sigma X_* \in G$  
under $X_*$ is a monomorphism of $G_X$ into $G$.  Thus we can see $G_X$ as a subgroup $G_{X_*}$ of $G$ and the number of distinct 
indexed and oriented versions of $X$ is the index $[G:G_{X^*}]$ of $G_{X_*}$ in $G$. 
In the sequel we use the notation $X(\sigma)$ for the arrangement $X_*\sigma$, $\sigma \in G$; hence $X(1) = X_*$, where $1$ is the unit of $G$. 

\begin{example}
Fig.~\ref{ioaph} depicts an arrangement $H$ on 5 pseudolines, its first barycentric subdivision, and one of its indexed and oriented versions $H_*$ on the indexing set $\{1,2,3,4,5\}$. 
The group $G_{H}$ is $D_4$  generated, for  example, by the automorphism $\sigma_{12}$ that exchanges the flags numbered $1$ and $2$ in the figure and 
the automorphism $\sigma_{18}$ that exchanges the flags numbered $1$ and $8$ in the figure.
Thus the number of distinct indexed (on a given set of indices) and oriented versions of 
$H$ is $5!2^5/8 = 480$.
The group $G_{H_*}$ is generated by the permutations 
$1\jbar{5}\jbar{4}\jind{2}3$ and 
$\jbar{1}\jind{3}\jind{2}\jbar{4}\jbar{5}$
which correspond to the automorphisms $\sigma_{12}$ and $\sigma_{18}$, respectively. 

\begin{figure}[!htb]
\centering
\psfrag{Aone}{$A(123)$}
\psfrag{Atwo}{$B(123)$}
\psfrag{Athr}{$C(1234)$}
\psfrag{Afou}{$D(1234)$}
\psfrag{Afiv}{$E(1234)$}
\psfrag{Asix}{$F(12345)$}
\psfrag{Asev}{$G(12345)$}
\psfrag{Ahei}{$H$}
\psfrag{IOAhei}{$H(12345)$}
\psfrag{Anin}{$I(12345)$}
\psfrag{Aten}{$J(12345)$}
\psfrag{Aoneor}{$2$}
\psfrag{Atwoor}{$4$}
\psfrag{Athror}{$16$}
\psfrag{Afouor}{$32$}
\psfrag{Afivor}{$24$}
\psfrag{Asixor}{$384$}
\psfrag{Asevor}{$960$}
\psfrag{Aheior}{$8$}
\psfrag{IOAheior}{$480$}
\psfrag{Aninor}{$240$}
\psfrag{Atenor}{$192$}
\psfrag{one}{$1$}
\psfrag{two}{$2$}
\psfrag{thr}{$3$}
\psfrag{fou}{$4$}
\psfrag{fiv}{$5$}
\psfrag{un}{$1$}
\psfrag{de}{$2$}
\psfrag{tr}{$3$}
\psfrag{qu}{$4$}
\psfrag{ci}{$5$}
\psfrag{si}{$6$}
\psfrag{se}{$7$}
\psfrag{hu}{$8$}
\includegraphics[width=0.975\linewidth]{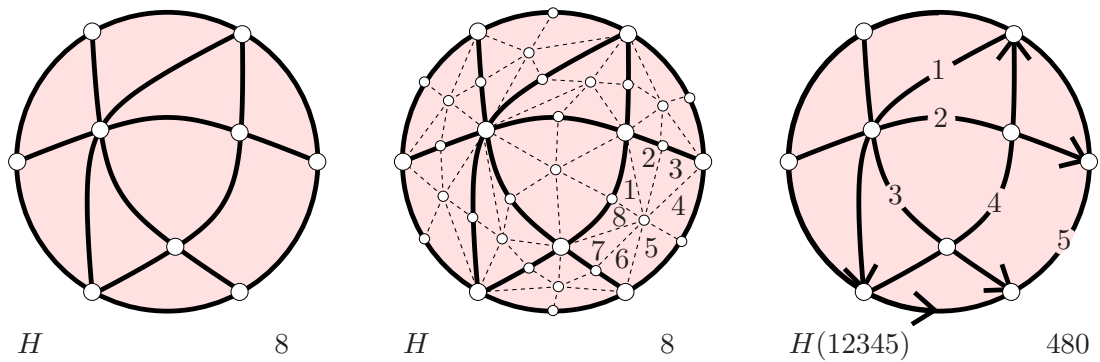}
\caption{\label{ioaph} An arrangement on $5$ pseudolines, its first barycentric subdivision, and one of its indexed and oriented versions} 
\end{figure}
\end{example}

Fig.~\ref{ttcbis} depicts one indexed and oriented version of each of the isomorphism classes of arrangements of three and four pseudolines; 
each diagram is labeled at its bottom right by its number of distinct reindexings (on a given set of indices) and reorientations. 
\begin{figure}[!htb]
\psfrag{one}{$1$} \psfrag{two}{$2$} \psfrag{thr}{$3$} \psfrag{fou}{$4$} \psfrag{fiv}{$5$}
\psfrag{AA}{$123$}
\psfrag{BB}{$132$}
\psfrag{CC}{$1253$}
\psfrag{DD}{$13\overline{5}4$}
\psfrag{EE}{$23\overline{5}4$}
\psfrag{one}{$1$}
\psfrag{two}{$2$}
\psfrag{thr}{$3$}
\psfrag{fou}{$4$}
\psfrag{otwo}{$2$}
\psfrag{ofou}{$4$}
\psfrag{oseize}{$16$}
\psfrag{othrtwo}{$32$}
\psfrag{otwofou}{$24$}
\psfrag{Aone}{$A(123) $}
\psfrag{Atwo}{$B(123) $}
\psfrag{Athr}{$C(1234)$}
\psfrag{Afou}{$D(1234)$}
\psfrag{Afiv}{$E(1234)$}
\centering
\includegraphics[width=0.9875\linewidth]{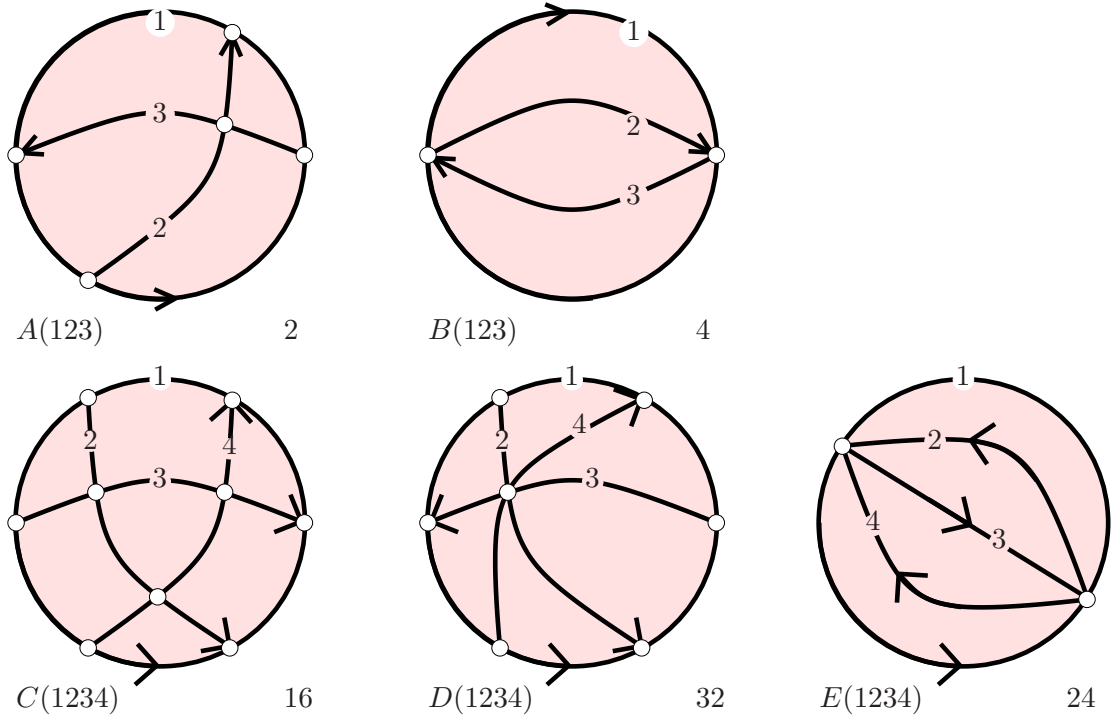}
\caption{\label{ttcbis}
Isomorphism classes of arrangements of three and four oriented pseudolines indexed by $1,2,3,4$}
\end{figure}
Thus the arrangement $A$ has 2 distinct indexed and oriented versions. The group $G_{A_*}$ is $S_4$ generated, for example,
by the permutations $\jbar{1}\jind{3}\jind{2}$, 
$\jbar{1}\jbar{2}\jind{3}$
and $\jind{2}\jbar{3}\jbar{1}$ and its 2 cosets  are 
$$
G_{A_*} = \left\{
\begin{array}{ccc}
123& 231& 312 \\
\overline{1}\overline{2}3&\overline{2}\overline{3}1&\overline{3}\overline{1}2\\
\overline{1}2\overline{3}&\overline{2}3\overline{1}&\overline{3}1\overline{2}\\
1\overline{2}\overline{3}&2\overline{3}\overline{1}&3\overline{1}\overline{2}\\
21\overline{3}&32\overline{1}&13\overline{2}\\
3\overline{2}1&1\overline{3}2&2\overline{1}3\\
\overline{1}32&\overline{2}13&\overline{3}21\\
\overline{2}\overline{1}\overline{3}& \overline{3}\overline{2}\overline{1}&\overline{1}\overline{3}\overline{2}
 \end{array}
\right\}, \qquad
(213)  G_{A_*} = \left\{
\begin{array}{ccc}
213 & 321 & 132\\
\overline{2}\overline{1}3&\overline{3}\overline{2}1&\overline{1}\overline{3}2\\
\overline{2}1\overline{3}&\overline{3}2\overline{1}&\overline{1}3\overline{2}\\
2\overline{1}\overline{3}&3\overline{2}\overline{1}&1\overline{3}\overline{2}\\
12\overline{3}&23\overline{1}&31\overline{2}\\
3\overline{1}2&1\overline{2}3&2\overline{3}1\\
\overline{2}31&\overline{3}12&\overline{1}23\\
\overline{1}\overline{2}\overline{3}&\overline{2}\overline{3}\overline{1}& \overline{3}\overline{1}\overline{2}
 \end{array}
\right\}.
$$
Similarly the number of distinct indexed and oriented versions of the arrangement $B$ is~$4$.   
The group $G_{B_*}$ is $S_3\times \mathbb{Z}_2$, generated for example by  the permutations $231,213,\overline{1}\overline{2}\overline{3}$ 
and its 4 cosets are 
$$
\phantom{(\overline{1}23)}G_{B_*} = \left\{
\begin{array}{ccc}
123&231&312 \\
213&321&132 \\
\overline{1}\overline{2}\overline{3} & \overline{2}\overline{3}\overline{1} &\overline{3}\overline{1}\overline{2} \\
\overline{2}\overline{1}\overline{3}&\overline{3}\overline{2}\overline{1}&\overline{1}\overline{3}\overline{2} \\
\end{array}
\right\},
(\overline{1}23)G_{B_*} 
= \left\{
\begin{array}{ccc}
\overline{1}23&23\overline{1}&3\overline{1}2
\\
2\overline{1}3&32\overline{1}&\overline{1}32
\\
1\overline{2}\overline{3}&\overline{2}\overline{3}1&\overline{3}1\overline{2} 
\\
\overline{2}1\overline{3}&\overline{3}\overline{2}1&1\overline{3}\overline{2} 
\end{array}
\right\},$$
$$
(1\overline{2}3)G_{B_*} = 
\left\{
\begin{array}{ccc}
1\overline{2}3&\overline{2}31&31\overline{2}
\\
\overline{2}13&3\overline{2}1&13\overline{2}
\\
\overline{1}2\overline{3}&2\overline{3}\overline{1}&\overline{3}\overline{1}2 
\\
2\overline{1}\overline{3}&\overline{3}2\overline{1}&\overline{1}\overline{3}2 
\end{array}
\right\},
(12\overline{3})G_{B_*} =
\left\{
\begin{array}{ccc}
12\overline{3}&2\overline{3}1&\overline{3}12
\\
21\overline{3}&\overline{3}21&1\overline{3}2
\\
\overline{1}\overline{2}3&\overline{2}3\overline{1}&3\overline{1}\overline{2}
\\  
\overline{2}\overline{1}3&3\overline{2}\overline{1}&\overline{1}3\overline{2} 
\end{array}
\right\}.
$$

Fig.~\ref{ttc} depicts these $2+4$ distinct indexed and oriented versions of $A$ and $B$.
\begin{figure}[!htb]
\centering
\psfrag{one}{1} \psfrag{two}{2} \psfrag{thr}{3} \psfrag{fou}{$4$}
\psfrag{onetwothr}{$123$}
\psfrag{onethrtwo}{$132$}
\psfrag{onetwo}{$12$}
\psfrag{onethr}{$13$}
\psfrag{twothr}{$23$}
\psfrag{onetwovonethrvtwothr}{$12,13,23$}
\psfrag{otwo}{$24$}
\psfrag{ofou}{$12$}
\psfrag{oseize}{$16$}
\psfrag{othrtwo}{$32$}
\psfrag{otwofou}{$24$}
\psfrag{Aone}{$A(123) $}
\psfrag{Aonebis}{$A(\overline{1}23) $}
\psfrag{Atwo}{$B(123) $}
\psfrag{Atwoone}{$B(\overline{1}23)$}
\psfrag{Atwotwo}{$B(1\overline{2}3)$}
\psfrag{Atwothr}{$B(12\overline{3})$}
\includegraphics[width=0.9875\linewidth]{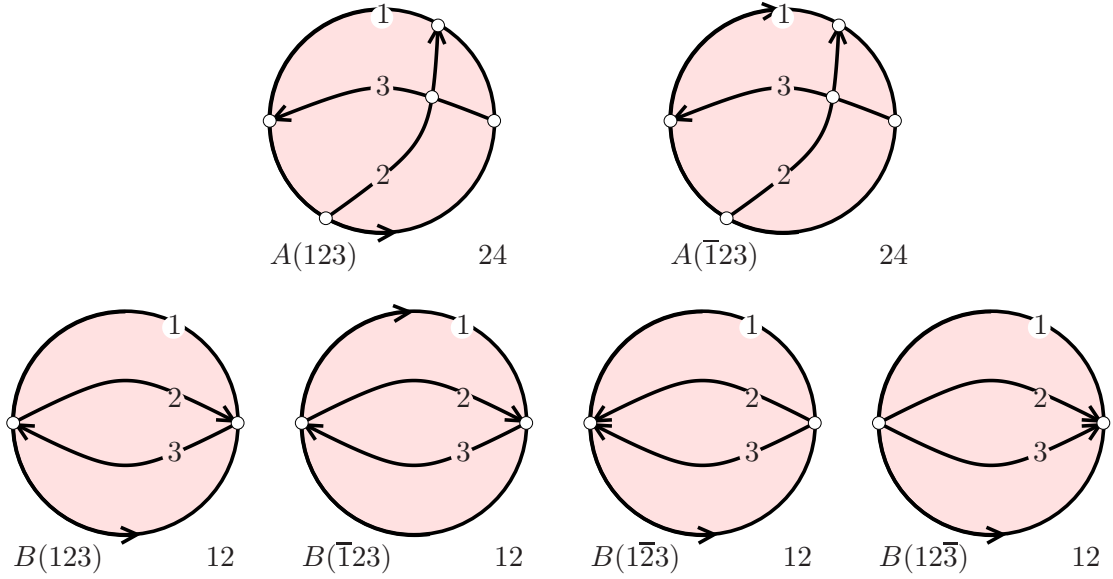}
\caption{\label{ttc} The possible entries of a chirotope of arrangement of pseudolines (of genus 1)}
\end{figure}

Using these notations one can describe the set of chirotopes 
on the indexing set $\{1,2,3,4\}$  as the set of 
$$\chi = \{\chi(123),\chi(124),\chi(134),\chi(234)\}$$ 
such that, up to a signed permutation of the  indices,
\begin{itemize}
\item[$(A_1)$:] if $A(213),A(314),A(412) \in \chi$ then $A(234)\in \chi$ or $A(243)\in \chi$ or  $B(234) \in \chi$;
\item[$(A_2)$:] if $B(12\overline{3}), B(12\overline{4}), B(13\overline{4}) \in \chi$ then $B(2\overline{3}4) \in \chi.$ 
\end{itemize}

To describe the set of chirotopes on $5$ indices we use the natural coding of an indexed arrangement of  
oriented pseudolines by its {\it side cycles} : 
there is exactly one side cycle per pseudoline $\gamma$ of the arrangement and the latter is defined as the circular sequence of signed indices obtained 
by writing down the indices of the pseudolines encountered when walking along the side of the pseudoline $\gamma$, 
each index being signed positively  or negatively depending on whether the encountered pseudoline is (locally) oriented away from or towards the pseudoline $\gamma$. 
For example
the side cycles of the arrangement $H(12345)$ of Fig.~\ref{ioaph} are
$$
\begin{array}{cl}
1 : &  \jind{2}\jind{3}\jind{4}\jind{5}\jbar{3}\jbar{2}\jbar{5}\jbar{4}\\
2 : &  \jind{1}\jind{3}\jbar{4}\jbar{5}\jbar{3}\jbar{1}\jind{4}\jind{5}\\
3 : &  \jbar{1}\jind{2}\jind{4}\jind{5}\jbar{2}\jind{1}\jbar{4}\jbar{5}\\
4 : &  \jind{3}\jind{2}\jbar{5} \jbar{1}\jbar{3}\jbar{2}\jind{1}\jind{5}\\
5 : &  \jind{3}\jind{2}\jbar{1}\jind{4}\jbar{3}\jbar{2}\jbar{4} \jind{1}.
\end{array}
$$
Similarly the three cycles of  $A(123)$ 
are $\cyclePAsimple{2}{3}, \cyclePAsimple{3}{1}$ and $\cyclePAsimple{1}{2}$
and those of $B(123)$ 
are $\cyclePAnonsim{2}{3},\cyclePAnonsim{3}{1}$ and $\cyclePAnonsim{1}{2}.$
The set of chirotopes on a set of $5$ indices, say $\{1,2,3,4,5\}$,  can then  be described as the set of 
\begin{eqnarray*}
\chi  & = & \{\chi(J) : J \in \binom{\{1,2,3,4,5\}}{3}\}\\
      & = & \{\chi(123),\chi(124),\chi(125),\chi(134),\chi(135), \\
      &   & \phantom{\{} \chi(145),\chi(234),\chi(235),\chi(245),\chi(345)\}
\end{eqnarray*}
such that the restrictions of $\chi$ to the sets of $3$-subsets of the $5$ subsets of $4$ indices of $\{1,2,3,4,5\}$ 
are chirotopes on $4$ indices, i.e., satisfy the axioms $(A_1)$ and $(A_2)$ mentioned above, which satisfy a single additional axiom $(A_3)$ 
saying that for any index $i$ the $4$ cycles indexed by $i$ of these $5$ chirotopes on $4$ indices are mergeable.

\subsection{Enlargement theorem} 
The {\it enlargement theorem for pseudoline arrangements}, due to Goodman, Pollack, Wenger and Zamfirescu~\cite{gpwz-atp-94}, 
proving a conjecture of B. Gr{\"u}nbaum~\cite[Conjecture 4.10, page 90]{g-as-72}, is the following. 

\begin{theorem}[{\bf \cite{gpwz-atp-94,k-ortmp-00}}] Any arrangement of pseudolines is an arrangement of lines of a projective plane.\qed
\end{theorem}

Combining the enlargement theorem of pseudoline arrangements with the duality principle for projective planes  we get  the following theorem.
 
\begin{theorem} \label{pretheotwo} Any arrangement of pseudolines is isomorphic to the dual
arrangement of a finite set of points of a projective plane. 
\end{theorem}
\begin{proof}
Indeed any pseudoline arrangement $\cal A$ is isomorphic to the dual of the point set $\cal A$
of the dual projective plane of any projective plane extension of $\cal A$---here we implicitly use the fact that a projective plane is isomorphic to its bidual.
\end{proof}

\clearpage
\section{Chirotopes of finite planar families of points \label{appendix:CFPFP}}
We now review the ``classical'' characterization of chirotopes of finite planar indexed families of oriented points mentioned in the abstract. 
(An oriented point in a projective plane is a point together with an orientation of its neighborhood, indicated in our drawings by an oriented circle 
surrounding the point.)
Our account takes  
advantage  of the relatively recent positive answer of Goodman, Pollack, Wenger and Zamiferescu~\cite{gpwz-atp-94} to the question of Gr{\"u}baum~\cite[Conjecture 4.10, page 90]{g-as-72}
about the embeddability of any arrangement of pseudolines in the line space of a projective plane.
 
Let $\Delta$ be a finite indexed family of oriented points of a projective plane $(\pp,\lpp)$, and let $\tau$ be a line of $(\pp, \lpp)$. 
We define 
\begin{enumerate}
\item the {\it cocycle of $\Delta$ at $\tau$} or the {\it cocycle of $\tau$ with respect to $\Delta$}
  or the {\it cocycle of the pair $(\Delta,\tau)$} as the  homeomorphism class of the pair $(\Delta,\tau)$, i.e., the  
set of $(\varphi \Delta, \varphi\tau)$ as $\varphi$ ranges over the set of homeomorphisms of surfaces
with domain $\pp$;  in other words two pairs $(\Delta,\tau)$ and $(\Delta',\tau')$ define the same cocycle if there exists a homeomorphism $\varphi$ 
of $\pp$ onto $\pp'$ such that  $\Delta' = \varphi \Delta$ and $\tau' = \varphi \tau$;
\item the {\it cocycle-map} as the map that assigns to each cell $\sigma$ of the dual arrangement of $\Delta$ the cocycle of $\Delta$ at an element (hence any) of $\sigma$;
\item a {\it $0$-, $1$-, $2$-cocycle} of $\Delta$ as a cocycle of $\Delta$ at a $0$-, $1$-, $2$-cell of its  dual arrangement; 
\item the {\it isomorphism class of $\Delta$} as the set of configurations $\Delta'$ that have the same set of $0$-cocycles as $\Delta$ (hence,using a simple perturbation argument,
 the same set of cocycles as $\Delta$); and  
\item the  {\it chirotope of $\Delta$} as the map that assigns to each $3$-subset of its indexing set the isomorphism class of the subconfiguration  indexed by this $3$-subset. 
\end{enumerate}
Fig.~\ref{PositionPoints} depicts the cocycles of configurations of three points: each circular diagram is  
\begin{figure}[!htb]
\psfrag{one}{$\ii$}
\psfrag{two}{$\kk$}
\psfrag{thr}{$\jj$}
\psfrag{lone}{$A$} \psfrag{ltwo}{$B$} \psfrag{lthr}{$C$} \psfrag{lfou}{$D$} \psfrag{lfiv}{$E$} \psfrag{lsix}{$F$} \psfrag{lsev}{$G$}
\psfrag{lone}{$\ii\jj\kk \bii\bjj\bkk$} 
\psfrag{ltwo}{$\ii\kk\bii\bkk,\bjj$} 
\psfrag{lthr}{$\ii\bii,\bkk,\bjj$} 
\psfrag{lfou}{$\ii,\kk,\jj$} 
\psfrag{oone}{$4$}
\psfrag{otwo}{$6$}
\psfrag{othr}{$6$}
\psfrag{ofou}{$4$}
\centering
\includegraphics[width=0.9875\linewidth]{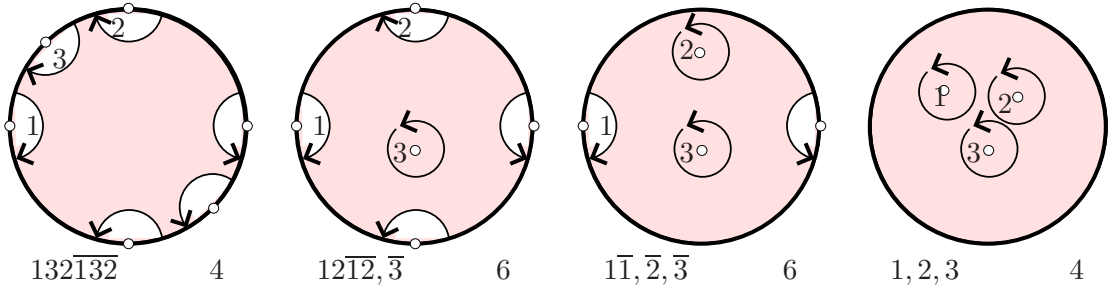}
\caption{The cocycles, up to reindexing and reorientation, of configurations of three points with indexing set $\{\ii,\kk,\jj\}$ 
\label{PositionPoints}}
\end{figure}
labeled at its bottom right  with its number of reindexings and reorientations
 and at its  bottom left with its {\it signature} $\mylabel(\Delta, \tau)$ which is defined as follows.  
Let $\cut{\tau}$ be the closed 2-cell obtained by cutting the cross surface $\pp$ along the line $\tau$, let $\Map{\nu_\tau}{\cut{\tau}}{\pp}$ be the canonical projection, 
 and let $\Sigma_\tau$ be the pre-image of the set of  $\Delta_i$s under $\nu_\tau$;
each element of  $\Sigma_\tau$ is endowed
with the orientation of the corresponding point of $\Delta$ and is labeled with the signed index of the corresponding signed point of $\Delta$.
Choose an orientation $\epsilon$ of  $\tau$, orient $\cut{\tau}$ accordingly, and define the {\it signature} of the triple $(\Delta,\tau, \epsilon)$
as the set of labels of the elements of $\Sigma$ whose corresponding signed points are contained in the interior of $\cut{\tau}$ and oriented consistently with the orientation of 
$\cut{\tau}$ plus the circular sequence of labels of the elements of $\Sigma_\tau$ oriented consistently with the orientation of $\cut{\tau}$
encountered when walking along the boundary of $\bcut{\pp}{\tau}$ according to its orientation.
The {\it signature  of the pair $(\Delta,\tau)$} or the {\it signature of $\Delta$ at $\tau$} is then defined as the unordered pair of signatures of the triples $(\Delta, \tau,\epsilon)$
 and $(\Delta, \tau,-\epsilon)$; it can be represented by either of its two elements since 
the signature of the triple  $(\Delta,\tau,-\epsilon)$ 
is obtained from the signature of the triple $(\Delta,\tau ,\epsilon)$ by replacing each of its elements  with  the reversal of its complement.
Clearly the  cocycle of a pair $(\Delta,\tau)$ depends only on its signature and vice versa.

A simple case analysis shows that the map that assigns to the isomorphism class of an indexed configuration of oriented points its chirotope is well-defined and one-to-one, that there 
are exactly six isomorphism classes of indexed configurations of three oriented points on the indexing set $\{1,2,3\}$, namely, in signature terms, 

$$
\begin{array}{l}
\{\one\two\thr\oneb\twob\thrb\}\\
\{\one\thr\two\oneb\thrb\twob\}\\
\{\two\one\thr\twob\oneb\thrb\}\\
\{\one\thrb\two\oneb\thr\twob\}\\
\{\one\two\oneb\twob,\thr\},\{\two\thr\twob\thrb,\one\},\{\thr\one\thrb\oneb,\two\} \\
\{\one\two\oneb\twob,\thrb\},\{\two\thr\twob\thrb,\oneb\},\{\thr\one\thrb\oneb,\twob\}
\end{array}
$$
and, finally, that 
the map that assigns to an indexed configuration of three oriented points the isomorphism class of its dual arrangement
is compatible with the isomorphism relation on indexed configurations of three oriented points and that the induced (one-to-one and onto) quotient map is the following 
$$
\begin{array}{lcr}
\{\one\two\thr\oneb\twob\thrb\} & \longrightarrow & B(\one\twob\thr) \\ 
\{\one\thr\two\oneb\thrb\twob\} & \longrightarrow & B(\one\two\thrb) \\
\{\two\one\thr\twob\oneb\thrb\} & \longrightarrow & B(\oneb\two\thr) \\
\{\one\twob\thr\oneb\two\thrb\} & \longrightarrow & B(\one\two\thr) \\
\{\one\two\oneb\twob,\thr\},\{\two\thr\twob\thrb,\one\},\{\thr\one\thrb\oneb,\two\}  & \longrightarrow & A(\one\two\thr) \\
\{\one\two\oneb\twob,\thrb\},\{\two\thr\twob\thrb,\oneb\},\{\thr\one\thrb\oneb,\twob\} & \longrightarrow & A(\oneb\two\thr)
\end{array}
$$
Since the map  that assigns to an isomorphism class of indexed arrangement of oriented pseudolines  its chirotope is one-to-one  and since any arrangement of pseudolines is isomorphic 
to the dual arrangement of a family of points 
it follows that the above considerations  concerning indexed configurations of three oriented points and indexed arrangements of three oriented pseudolines extend to configurations 
of any number of points and arrangements of any number of pseudolines.  We summarize:
\begin{theorem}
\label{theoPMversionpoint}
The map that assigns to an indexed configuration of oriented points the isomorphism class of its dual arrangement
is compatible with the isomorphism relation on indexed configurations of oriented points; furthermore the induced quotient map is one-to-one and onto.  \qed
\end{theorem}

Therefore there are also six isomorphism classes of cocycle-maps on the indexing set $\{1,2,3\}$; they are depicted 
in Fig.~\ref{cocyclemaps}.
\begin{figure}[!htb]
\centering
\footnotesize
\psfrag{one}{} \psfrag{two}{} \psfrag{thr}{}
\psfrag{ones}{} \psfrag{twos}{} \psfrag{thrs}{}
\psfrag{one}{1} \psfrag{two}{2} \psfrag{thr}{3}
\psfrag{ones}{$1$} \psfrag{twos}{$2$} \psfrag{thrs}{$3$}
\psfrag{duality}{$\stackrel{\mathrm{duality}}{\longrightarrow}$}
\psfrag{uxdxt}{$1,2,3$}
\psfrag{ubxdxt}{$\overline{1},2,3$}
\psfrag{ubxdbxt}{$\overline{1},\overline{2},3$}
\psfrag{uxdxtb}{$1,2,\overline{3}$}
\psfrag{uxdbxt}{$1,\overline{2},3$}
\psfrag{uubxdxt}{$1\overline{1},2,3$}
\psfrag{uubxdxtb}{$1\overline{1},2,\overline{3}$}
\psfrag{ddbxuxt}{$2\overline{2},1,3$}
\psfrag{ddbxuxtb}{$2\overline{2},1,\overline{3}$}
\psfrag{ttbxuxd}{$3\overline{3},1,2$}
\psfrag{ttbxuxdb}{$3\overline{3},1,\overline{2}$}
\psfrag{udubdbxt}{$12\overline{1}\overline{2},3$}
\psfrag{dtdbtbxu}{$23\overline{2}\overline{3},1$}
\psfrag{udubdbxt}{$12\overline{1}\overline{2},3$}
\psfrag{utbubtxd}{$1\overline{3}\overline{1}3,2$}
\psfrag{udtubdbtb}{$123\overline{123}$}
\includegraphics[width=0.875\linewidth]{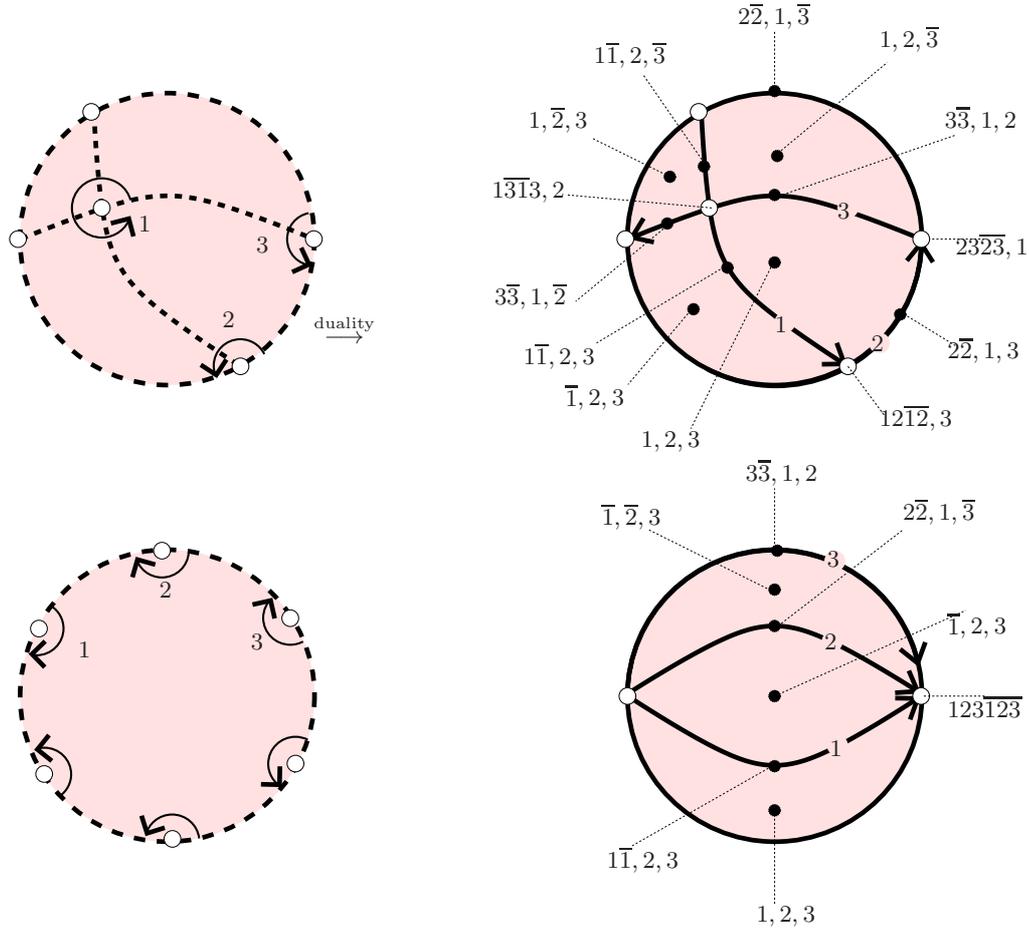}
\caption{The cocycle maps for families of three points indexed by $1,2$ and $3$ \label{cocyclemaps} }
\end{figure}

The well-informed 
reader will have recognized here a reformulation,  taking advantage of the embeddability of any arrangement of pseudolines in the line space of a projective plane,
of the existence of an adjoint---or Type II representation---for every oriented matroid of rank
three~\cite{g-pcbgs-80,c-mor3a-82,gp-sccca-84},~\cite[page 263]{blswz-om-99}.

Combining Theorems~\ref{pretheotwo},~\ref{theoPMversionpoint} and \ref{pretheofou} we get the characterization of chirotopes of planar families of points 
mentioned in the abstract.

\begin{theorem} \label{pretheofiv} 
Let $\chi$ be a map on the set of $3$-subsets of a finite set~$I$.
Then $\chi$ is a chirotope of finite planar families of points  if and only if for every $3$-, $4$-, and $5$-subset $J$ of $I$  
the restriction of $\chi$ to the set of $3$-subsets of $J$ is a chirotope of finite planar families of  points. \qed
 \end{theorem}

\clearpage
\section{Basics of convexity in projective planes}  \label{sectwo}
\newcommand{\NEI}{R}
\newcommand{\SI}{{\cal S}}
\newcommand{\XI}{X}
\newcommand{\YI}{Y}
\newcommand{\ZI}{Z}
\newcommand{\RI}{R}
\newcommand{\QI}{Q}

In this section we establish the basics of convexity in projective planes that we have taken for granded in the paper, namely   
\setcounter{theorem}{0}
\begin{theorem}\label{weakdefinition}
A convex body of a projective plane is a closed topological disk, its polar is a convex body of the dual projective plane, and its dual is the boundary of its polar 
(hence a double pseudoline). Furthermore, up to homeomorphism, the dual arrangement of a pair of disjoint convex bodies of a projective plane is the unique arrangement of
two double pseudolines
that intersect in four transversal intersection points and induce a cellular decomposition of their underlying cross surface.\qed
\end{theorem}
We proceed in two steps, establishing the basics first for neutral (hence affine) planes, which we shall briefly recall, and second for projective planes by reduction to the first step.

\setcounter{theorem}{44}

\subsection{Background material on neutral and affine planes} 
A  {\it neutral plane} is a topological point-line incidence geometry $(\TopPlane,\Lines)$ whose point space~$\TopPlane$ is 
homeomorphic to~$\mathbb{R}^2$, whose line space $\Lines$ is a subspace of the space of pseudolines
of the point space,{}\footnote{ The space of pseudolines of $\mathbb{R}^2$ is the quotient of the space of embeddings 
of $\mathbb{R}$ into $\mathbb{R}^2$ with closed images (i.e., the set of continuous one-to-one maps $\Map{\varphi}{\mathbb{R}}{\mathbb{R}^2}$, 
with the property that
$\varphi(\mathbb{R})$ is closed in $\mathbb{R}^2$, endowed with the compact-open
topology) under the natural action of the group of homeomorphisms of $\mathbb{R}$. 
As usual we identify a pseudoline $\varphi$ with its image $\varphi(\mathbb{R})$.
Similarly the space of oriented pseudolines of $\mathbb{R}^2$ is the quotient of the space of embeddings 
of $\mathbb{R}$ into $\mathbb{R}^2$ with closed images 
under the natural action of the group of direct homeomorphisms of $\mathbb{R}$.}  and whose axiom system is reduced to the following single axiom: 
any two distinct points belong to exactly one line, called their {\it joining line}, which depends continuously on the two points.

The {\it join map}, denoted  $\vee$, assigns to
any ordered pair of distinct points  of $\TopPlane$ their joining line in $\Lines$;  the {\it intersection map}, denoted $\wedge$,
 assigns to any ordered pair of distinct intersecting lines of $\Lines$ their common intersection point in~$\TopPlane$. The join and intersection maps are
continuous and open. 

\begin{theorem}[{\cite[page 220]{bs-mcip-65}}] Let $\Lines$ be the line space of a  neutral plane and let $\OLines$ be the space of oriented versions of the lines of $\Lines$.
Then 
$\Lines$ is  an open  crosscap, the natural projection  $\MapLight{}{\OLines}{\Lines}$ that assigns to an oriented line its unoriented version is a two-covering map, 
and  the pencil of lines through a point is a pseudoline in $\Lines$, i.e., a nonseparating simple closed curve embedded in $\Lines$.  \qed
\end{theorem}

An {\it affine plane} is a neutral plane $(\TopPlane,\Lines)$ with the property that for every point-line pair $(\apoint,\aline)$ 
there exists a unique line $\bline$ incident to $\apoint$ such that either $\bline=\aline$ or $\bline$ and $\aline$ are disjoint; the line $\bline$ is called
 the {\it parallel} to $\aline$ through $\apoint$, and the lines $\aline$ and $\bline$ are said to be {\it parallel}. 
The parallelism relation is an equivalence relation,  the parallel class of $\aline$ is denoted~$[\aline]$,
 and the set of parallel classes is denoted $[\Lines]$.  
The {\it projective completion} of an affine plane $(\TopPlane,\Lines)$ is the topological point-line incidence geometry whose line space  $\pclines$ and 
point space $\pcplane$ are, respectively,
\begin{enumerate}
\item  the set $\{ \aline \cup\{[\aline]\} \mid \aline \in \Lines\} \cup \{[\Lines]\}$, 
endowed with the topology of the
one-point compactification $\Lines \cup \{\infty\}$ of $\Lines$ via the map  
that assigns  $\aline \in \Lines$ to  $\aline \cup \{[\aline]\} \in \pclines$ and $\infty$ to  $[\Lines]$;
 and 
\item the set $\TopPlane \cup [\Lines]$  endowed with 
the topology  with subbase  the  $J_\apoint \wedge J_\apointbis$ where  $\apoint$ and  $\apointbis$ are two
points of  $\pcplane$ and where  $J_\apoint $ and  $J_\apointbis$ are disjoint open
intervals of the pencils of lines through $\apoint$ and $\apointbis$, respectively. 
\end{enumerate}
The {\it affine parts} of a projective plane $(\pp,\lpp)$ are the topological point-line incidence geometries
 $(\pp\setminus \aline, \lpp\setminus \{\aline\})$ where $\aline$ ranges over $\Lines$. 

\begin{theorem}\label{projectivecompletion}
The projective completion  of an affine plane is a projective plane and 
the affine parts of a projective plane are affine planes with the property that their projective completions are isomorphic 
to the initial projective plane.   \qed
\end{theorem}

We refer to the monograph of Salzmann et al~\cite[Chap. 3]{sbghl-cpp-95} for supplementary background material on 
neutral planes, where they are called $\mathbb{R}^2$-planes.

\subsection{Convexity and duality in neutral planes}
We work in a  neutral plane $(\TopPlane, \Lines)$.  
As in the Euclidean plane,
a subset of points is called {\it convex} if it includes the line segments joining its points. 
A {\it convex body} is a compact convex subset of points with nonempty interior, 
its {\it polar}  is the set of lines missing its interior, and 
its {\it dual} is its set of tangent lines or supporting lines (i.e., the set of lines that intersect the body but not its interior or, equivalently, 
the set of lines that intersect the  convex body and that include the body in one of their 
two closed sides). 
A {\it double pseudoline} of an open crosscap is a double pseudoline of the one-point compactification of the open crosscap with the property that 
the point at infinity belongs to the disc side of the double pseudoline.

In this section we establish the following transcription of Theorem~\ref{theoone} for neutral planes.  

\begin{theorem} \label{theooneneutral} A convex body of a neutral plane is a closed topological disk,  its polar is a  
closed topological disk with an interior point deleted, which is closed in the line space and whose intersection  with the pencil of lines through any point is a closed line segment, and 
its dual  is the boundary of its polar, hence a  double pseudoline of the line space. 
Furthermore, up to homeomorphism, the dual arrangement of a pair of disjoint convex bodies of a neutral plane 
is the unique arrangement of two double pseudolines in an open crosscap 
that intersect in four transversal intersection points and induce a cellular decomposition of the one-point compactification of the open crosscap. \qed
\end{theorem}

The proof proceeds by a sequence of auxiliary results.

\subsubsection{Boundary of a convex body and tangents}

The proofs of the two following lemmas are adapted from~\cite[Chap. 11.3]{b-gcppr-78}.

\begin{lemma} 
The boundary of a convex body is a simple closed curve. 
\end{lemma}
\begin{proof} 
Let $U$ be a convex body, let $A$ be one of its interior points, and let $L_A
\approx \mathbb{S}^1$ be the pencil of oriented lines through $A.$
Consider the application  $\Map{\varphi}{L_A}{\partial U}$ that assigns to $\aline \in L_A$ 
the endpoint of the trace of $U$ on $\aline$ beyond $A$.
Clearly $\varphi$ is a well-defined one-to-one and onto correspondence whose inverse is continuous. 
Therefore it is sufficient to show that $\varphi$ is continuous.
Let $\aline \in L_A$ and let $B$ and $C$ be two points of the interior of~$U$ with $A$ contained in the interior of the line segment joining~$B$ to~$C$. 
As illustrated in Fig.~\ref{boundary} the rays  with origins $A,B$ and $C$ through
$\varphi(\aline)$ leave $U$ at $\varphi(\aline).$ 
\begin{figure}[!htb]
\psfrag{A}{$A$} \psfrag{B}{$B$} \psfrag{C}{$C$} \psfrag{CB}{$U$} \psfrag{phil}{$\varphi(\aline)$}
\psfrag{l}{$\aline$} \psfrag{ln}{$\aline_n$}
\includegraphics[width = 0.50\linewidth]{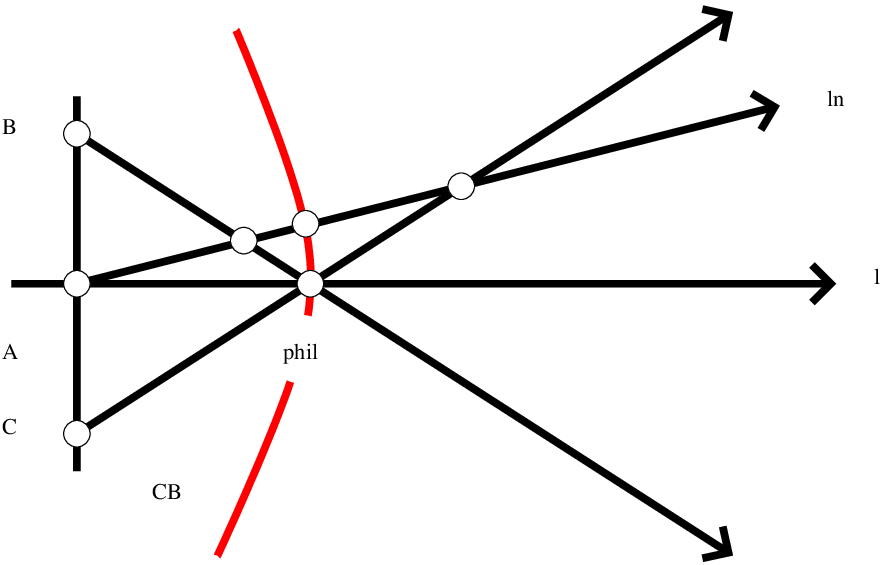}
\caption{ \label{boundary} }
\end{figure}
Let $\aline_1, \aline_2, \ldots, \aline_n, \ldots$, be a sequence of oriented lines of $L_A$
converging to $\aline$ with $\aline_n \neq \aline$ for all $n \geq 1$.  For $n$ large enough the intersection points of the
line $\aline_n$ with the lines $B \vee \varphi(\aline)$ and $C\vee \varphi(\aline)$ are well-defined 
and converge to $\varphi(\aline).$ Consequently, for $n$ large enough, the line
$\aline_n$ intersects the rays  through $\varphi(\aline)$ with origins $B$ and
$C.$ One of these two intersection points is beyond $\varphi(\aline)$ and the
other one is before $\varphi(\aline).$  Therefore $\varphi(\aline_n)$ belongs to the
line segment joining these two points. When $n$ goes to infinity, this line
segment retracts onto $\varphi(\aline).$ Therefore $\varphi$ is continuous. 
\end{proof}

\begin{lemma} Let $U$ be a convex body, let $A$ be a point not in the interior of $U$,  
let  $\XI$ be the set of lines through~$A$ intersecting the interior of $U$, and 
let  $\YI$ be the set of lines through~$A$ intersecting $U$ but not its interior. 
 Then  $\XI$ is a nonempty open interval whose endpoints belong to $\YI$. 
Furthermore $\YI$ is a pair if and only if $A\notin \partial U$.
\end{lemma}
\begin{proof} The set $\XI$ is 
\begin{enumerate}
\item nonempty, because the interior of $U$ is nonempty; 
\item an open subset of the pencil of lines through $A$,  because the join map is open; 
\item connected,  because the interior of $U$ is connected.
\end{enumerate} 
Consequently,  $\XI$ is a nonempty open interval of the pencil of lines through $A$ or 
$\XI$ is the pencil of lines through $A$.

We now show that $\XI$ is not the pencil of lines through $A$.

For $\aline \in \XI$, let $\aline^+$ and $\aline^-$ be the two connected components of $\aline\setminus \{A\}$ with the convention 
that $\aline^+$ is the one that intersects the interior of $U$ and, consequently, $\aline^-$ is the one that misses $U$. 
We set $\RI^+ = \bigcup_{\aline \in \XI} \aline^+$ and $\RI^- = \bigcup_{\aline \in \XI} \aline^-$.
Let $\SI$ be the set of closed line segments contained in the interior of $U$ whose supporting line is not incident to~$A$.
For $I \in \SI$, let $\XI_I$ be the set of lines of $\XI$ intersecting  the interior of $I$ and let $\QI^+(I) = \bigcup_{\aline \in \XI_I} \aline^+$ and $\QI^-(I) = \bigcup_{\aline \in \XI_I} \aline^-$. 
 We leave the verification of the following properties to the reader 
\begin{enumerate}
\item $\SI$ is nonempty;
\item $\XI = \bigcup_{I \in \SI} \XI_I$;
\item $\QI^+(I)$  and $\QI^-(I)$ are open quadrants;
\item $\RI^+ = \bigcup_{I\in \SI} \QI^+(I)$ is open and nonempty;
\item $\RI^- = \bigcup_{I\in \SI} \QI^-(I)$ is open and nonempty.
\end{enumerate}
 
The sets $\RI^+$ and $\RI^-$ are disjoint nonempty open subsets of the point space minus~$A.$ Since this last set is  
connected, there exists a point $E$ neither  in $\RI^+$ nor in $\RI^-.$
The line joining $A$ and $E$ misses the interior of $U.$ Consequently $\XI$ is not
the pencil of lines through $A.$

Finally the endpoints of $X$ belongs to $Y$ since $U$ is compact. The furthermore part follows easily.
\end{proof}

\subsubsection{Duality}
The dual of a convex body $U$ is denoted $\tang{U}$.
\begin{lemma}\label{dualbodybis} 
Let $U$ be a convex body.  Then 
\begin{enumerate}
\item $\tang{U}$ is a simple closed curve in $\Lines$;
\item $\tang{U}$ is a double pseudoline in $\Lines$; 
\item the set of lines intersecting the interior of $U$ is the open crosscap bounded by $\tang{U}$;
\item the set of lines missing $U$ is the one-punctured topological disk bounded by $\tang{U}$.
\end{enumerate} 
\end{lemma}
\begin{proof} We endow the plane with an orientation. Let $\Delta$ be the map that assigns to $A$ not in~$U$ the tangent to $U$ through $A$ with the property that walking along the tangent from~ $A$ to~$U$ we see the convex body $U$ on our right, let $I = [A,B]$ be a closed line segment missing $U$ with the property 
that $\Delta(A) \neq \Delta(B)$, and let $\Gamma$ be a simple closed curve surrounding $U$. We leave the verification of
the following properties to the reader 
\begin{enumerate}
\item $\Delta$ is continuous and onto;
\item the restriction of $\Delta$ to the domain $\Gamma$ is onto;
\item the restriction of $\Delta$ to the domain $I$ and codomain $\Delta(I)$ is a homeomorphism;
\item $\Delta$ is open;
\end{enumerate}
from which it follows that $\tang{U}$ is compact and locally homeomorphic to $\mathbb{R}$, hence a simple closed curve. 

We now prove claims (2), (3), and (4). Let $A$ be an interior point of $U$ and let $\Lines \cup \{\infty\}$ be the one-point compactification of $\Lines$.
Since two pseudolines intersect in at least one point and since the pencil of lines through $A$  is a pseudoline 
that does not intersect $\tang{U}$, it follows that $\tang{U}$ is a double pseudoline of $\Lines \cup \{\infty\}$.
Let $X$ be the set of lines intersecting the interior of $U$. The set $X$ is connected, open, closed in $\Lines \setminus \tang{U}$ and  contains pseudolines. 
Therefore $X$ is the trace on $\Lines$ of the open crosscap bounded by $\tang{U}$ in $\Lines \cup \{\infty\}$.  
It remains to show that the open crosscap bounded by $\tang{U}$ in $\Lines \cup \{\infty\}$ does not contain $\infty$.
This follows from~\cite[Lemma 31.24]{sbghl-cpp-95} which asserts that the set of lines intersecting a compact set of points is compact. 
\end{proof}

\begin{lemma} 
Let $U$ and $V$ be two disjoint convex bodies. Then the double pseudolines $\tang{U}$ and $\tang{V}$ intersect in exactly four points, where they cross. 
\end{lemma}
\begin{proof} We endow the plane with an orientation.
Let $\Delta$ be the map that assigns to $A \notin U$ the tangent to $U$ through $A$ with the property that 
walking along the tangent from  $A$ to $U$ we see the convex body $U$ on our right, let $I = [A,B]$ be a closed line segment missing $U$ with the property 
that $\Delta(A) \neq \Delta(B)$, and let $\Gamma$ be a simple closed curve surrounding $U$.  We leave the verification of
the following properties to the reader 
\begin{enumerate}
\item $\Delta$ is continuous and onto;
\item the restriction of $\Delta$ to the domain $\Gamma$ is onto;
\item the restriction of $\Delta$ to the domain $I$ and codomain $\Delta(I)$ is a homeomorphism;
\item $\Delta(V)$ is a closed interval $[T,T']$, $T\neq T'$, of $\tang{U}$;
\item $\Delta(\interior{V})$ is the open interval $]T,T'[$;
\end{enumerate}
from which it follows that $T$ and $T'$ are  the sole tangents to both $U$ and $V$ such that walking along the tangents  
from $V$ to $U$ we see $U$ on our right (and walking along the tangents from $U$ to $V$ 
we see $V$ on our left or on our right depending on whether we walk on $T$ or on $T'$),
 and that $\tang{U}$ and $\tang{V}$ cross at $T$ and $T'$. Switching the roles of $U$ and $V$ we get a second pair of common tangents to $U$ and $V$. 
This proves the lemma.
\end{proof}


\begin{lemma} Let $U$ and $V$ be two disjoint convex bodies. Then the double pseudolines $\tang{U},\tang{V}$ induce a cellular decomposition of the 
(one-point compactification of) $\Lines$. 
\end{lemma}
\begin{proof} Let $u \in \interior{U}$ and let $v \in \interior{V}$.  We have seen that
\begin{enumerate}
\item $\tang{U}$ and $\tang{V}$ are double pseudolines intersecting in exactly four points---where they cross;
\item $\dual{u}$ and $\dual{v}$ are pseudolines intersecting in exactly one point---where they cross; 
\item $\dual{u}$ is contained in the open crosscap bounded by $\tang{U}$;
\item $\dual{u}$ and $\tang{V}$ intersect in exactly two points---where they cross;
\item $\dual{v}$ is contained in the open crosscap bounded by $\tang{V}$;
\item $\dual{v}$ and $\tang{U}$ intersect in exactly two points---where they cross.
\end{enumerate}
Consequently the arrangements $\{\dual{u},\dual{v}, \tang{U}\}$ and $\{\dual{u},\dual{v}, \tang{V}\}$ are (up to homeomorphism) those depicted in the two leftmost diagrams of 
Fig.~\ref{twocases}. 
\begin{figure}[!htb]
\psfrag{US}{$\tang{U}$} \psfrag{VS}{$\tang{V}$}
\psfrag{u}{$\dual{u}$} \psfrag{v}{$\dual{v}$}
\includegraphics[width =0.98550\linewidth]{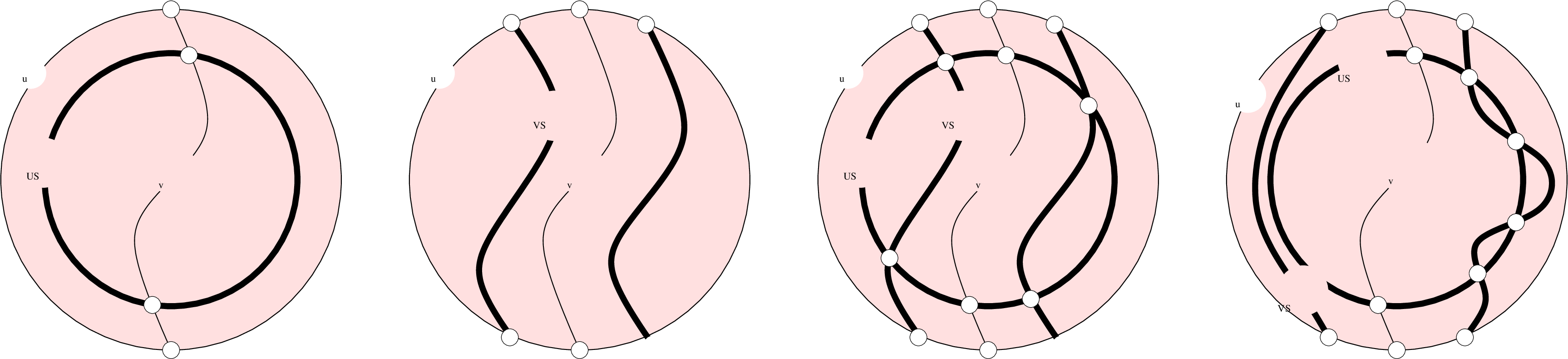}
\caption{\label{twocases}}
\end{figure}
Now it is not hard  to see that the only overlays of these two diagrams fulfilling condition (1) above are the two rightmost diagrams of Fig.~\ref{twocases}.
It remains to observe that the set of lines intersecting both $U$ and $V$ is connected to rule out the rightmost  diagram from our considerations and to conclude that  
the double pseudolines $\tang{U},\tang{V}$ induce a cellular decomposition of the one-point compactification of $\Lines$.
\end{proof}

Observe that this proves that two disjoint convex bodies have a strictly separating line.


\subsection{Convexity and duality in projective planes}
We now work in a projective plane $(\pp,\lpp)$. Recall that a convex body is a closed subset of points with nonempty interior whose intersection with any 
line is a (necessarily closed) line segment, that its polar is the set of lines that miss it, and that its dual is its set of tangent lines or supporting lines 
(i.e., the set of lines intersecting the body but not its interior).  
According to Theorems~\ref{projectivecompletion} and~\ref{theooneneutral} proving Theorem~\ref{theoone}  boils down to prove that the set of lines missing 
two disjoint convex bodies is nonempty. The proof proceeds by a sequence of auxiliary results.

\begin{lemma}\label{lemone} Assume that two of the three sides of a triangular face of a
simple arrangement of three lines are contained in a convex body. Then the triangular face is contained in the convex body. 
\end{lemma}
\begin{proof} Let $\CB$ be a compact subset of points with
nonempty interior whose intersection with any line is a line segment or a  line, and let  $T$ be a triangular face of a
simple arrangement of three lines.
Let $A,B,C$ be the three vertices of the triangular face $T$, as illustrated in the left diagram  of Fig.~\ref{fundone} where the triangular face 
is marked with a little square, let $[AB], [BC]$ and $[CA]$ be the three sides of $T$,
 and assume that $[AB]$ and $[AC]$ are contained in $\CB.$ 
Proving our lemma comes down to proving that $\CB$ contains a line or that $T$ is contained in $\CB.$
\begin{figure}[!htb]
\centering
\psfrag{A}{$A$} \psfrag{B}{$B$} \psfrag{C}{$C$} \psfrag{D}{$D$}
\psfrag{X}{$X$} \psfrag{Xp}{$X'$}
\includegraphics[width=0.75\linewidth]{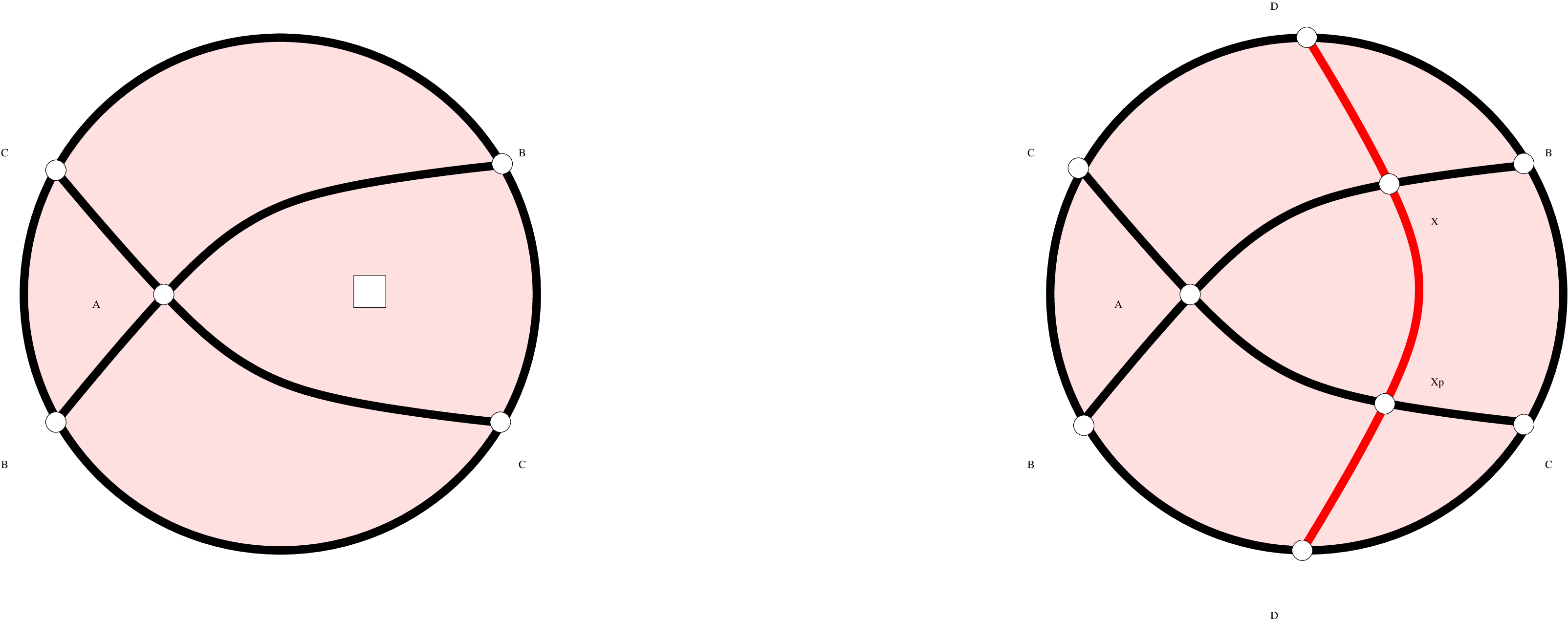}
\caption{\label{fundone} 
}
\end{figure}
Let $D$ be a point of the line $(BC)$ outside the line segment $[BC]$, as illustrated in the right diagram  of Fig.~\ref{fundone}. 
For any $X\in [AB]$, we denote by $X'$ the intersection of the line $(DX)$ with the line $(AC)$ (note that $X'$ ranges over the line segment $[AC]$ and that $B'=C$),
by $[XX']$ the line segment supported by the line $(DX)$ contained 
into $T$, and, for $X\neq A$, by $[X'X]$ the line segment of $(DX)$ contained in the complement of the interior of $T$. 
Let $I$ be the set of $X \in [AB]$, $X\neq A$,  such that $[XX']$ is contained in $\CB$, and let $J$ 
be the set of $X\in [AB]$, $X\neq A$, such that $[X'X]$ is contained in $\CB.$
One can easily check that 
\begin{enumerate}
\item $I \cap J= \emptyset$ unless $\CB$ contains a line $(DX)$ with  $X\in
[AB], X\neq A$;
\item $I \cup J = [AB]\setminus \{A\}$;
\item $I$ and $J$ are both closed in $[AB]\setminus \{A\}.$
\end{enumerate}
Assume now that $I \cap J = \emptyset$, otherwise $\CB$ contains a line $(DX)$ with $X \in [AB], X\neq A$,
and we are done. 
Since $[AB]\setminus \{A\}$ is connected, it follows that $I$  or $J$ is empty. 
In the first case the line  $(DA)$ is contained in $\CB$ since
$\CB$ is compact and in the second case $T = \bigcup_{X \in [AB]} [XX'] \subseteq \CB$. 
In both cases we are done.  \end{proof}

\begin{lemma} \label{lemtwo}
The trace of a line on the interior of a convex body is empty or
is the interior of the trace of the line on the body. 
\end{lemma}
\begin{proof} 
Let $\CB$ be a convex body, let $[AB]$ be the trace of a line on $\CB$ and
assume that $[AB]$ intersects the interior of $\CB$ at point $C$. Let $[DE]$ be
a line segment through $C$, contained in the interior of $\CB$, and not
contained in the line $(AB).$  Clearly the arrangement composed of the six lines joining two of the 
four points $A,B,D,E$ is, up to homeomorphism, the one shown in the rightmost diagram of Fig.~\ref{fundtwo}.  
\begin{figure}[!htb]
\centering
\psfrag{AA}{$A$} \psfrag{BB}{$B$} \psfrag{CC}{$C$} \psfrag{DD}{$D$} \psfrag{EE}{$E$}
\includegraphics[width=0.875\linewidth]{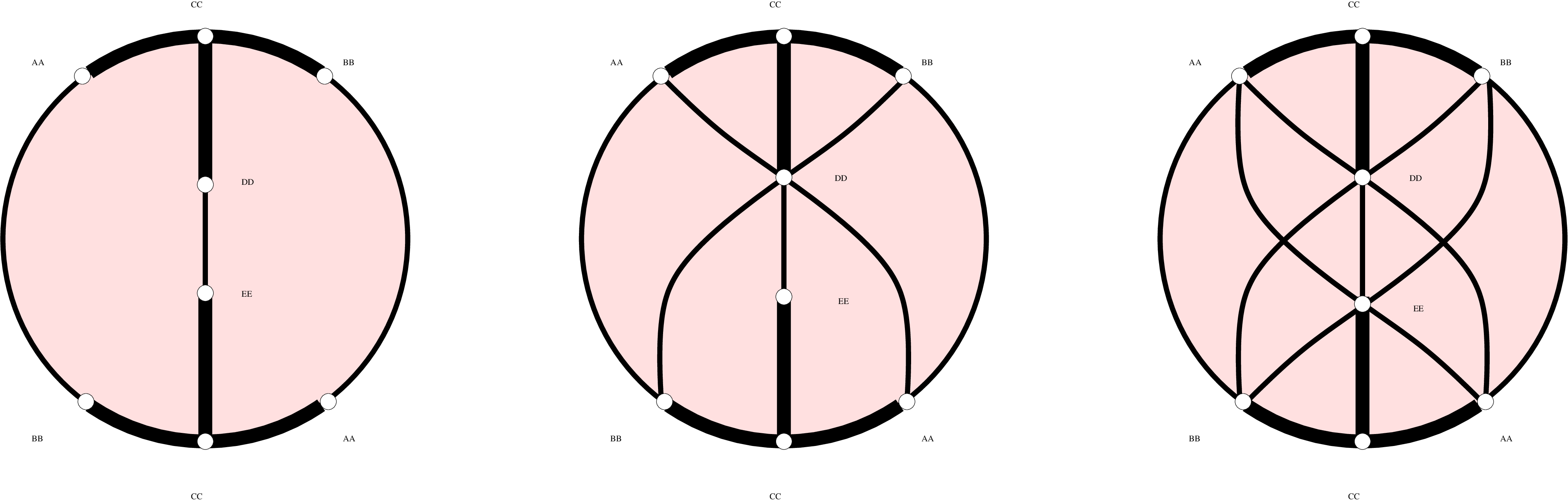}
\caption{\label{fundtwo}}
\end{figure}
According to  Lemma~\ref{lemone} the four triangles $ACD$, $ACE$, $BCD$, and $BCE$ are included in $U$. The lemma follows.
\end{proof}

\begin{lemma} \label{lemtwozer}
Let $\CB$ be a convex body, let $\CL$ be a line intersecting $\CB$ along a nonempty line segment $\II$, let $\TD$ be the closed topological disk obtained by cutting $\pp$ along  $\CL$,  
let $\MapLight{}{\TD}{\pp}$ be the induced canonical map, and  let $\FCB$, $\FCL$  and $\FII$ be the pre-images of 
$\CB$, $\CL$ and $\II$ under $\MapLight{}{\TD}{\pp}.$
Then
\begin{enumerate}
\item $\FII$ has two connected components, denoted $\FIP$ and $\FIM$ thereafter;
\item $\FCB$  has two connected components and the traces of these two connected components on $\FCL$ are the two connected components of $\FII$; 
we denote by $\FCBP$ and $\FCBM$ the connected components of $\FCB$ that contain $\FIP$ and $\FIM$, respectively, and we set 
$\CBP = \FCBP\setminus \FIP$ and $\CBM = \FCBM\setminus \FIM$; 
\item $\CBP$ and $\CBM$ are closed convex subsets of the affine part of $(\pp,\lpp)$ obtained by removing the line $\CL$;
\item 
$\CBP$ or $\CBM$ is  nonempty and  $\CBP$  and $\CBM$ are both nonempty if and only if $\CL$ intersects the interior of $\CB$; 
\item the topological closure of $\CBP$ in $\TD$ is $\CBP \cup \IP$ under the assumption that $\CBP$ is nonempty, and a similar result holds for $\CBM$.
\end{enumerate}
\end{lemma}
\begin{proof} 
Claim (1) is clear since the restriction of $\MapLight{}{\TD}{\pp}$ to the domain $\FCL$ and codomain $\CL$ is a two-covering. 
\begin{figure}[!htb]
\centering
\psfrag{l}{$\CL$}
\psfrag{I}{$\II$} 
\psfrag{IL}{$\FCL$} 
\psfrag{J}{$\FIP$} \psfrag{K}{$\FIM$} \psfrag{B}{$B$}
\psfrag{trace}{$\trace{A}{B}{\CB}$}
\psfrag{AB}{$A_B$} \psfrag{ABS}{$A_B^*$}
\psfrag{V}{$V$} \psfrag{JN}{$J_n$} \psfrag{AN}{$A_n$} \psfrag{A}{$A$} \psfrag{AP}{$A'$} \psfrag{k}{$k$}
\includegraphics[width=0.875\linewidth]{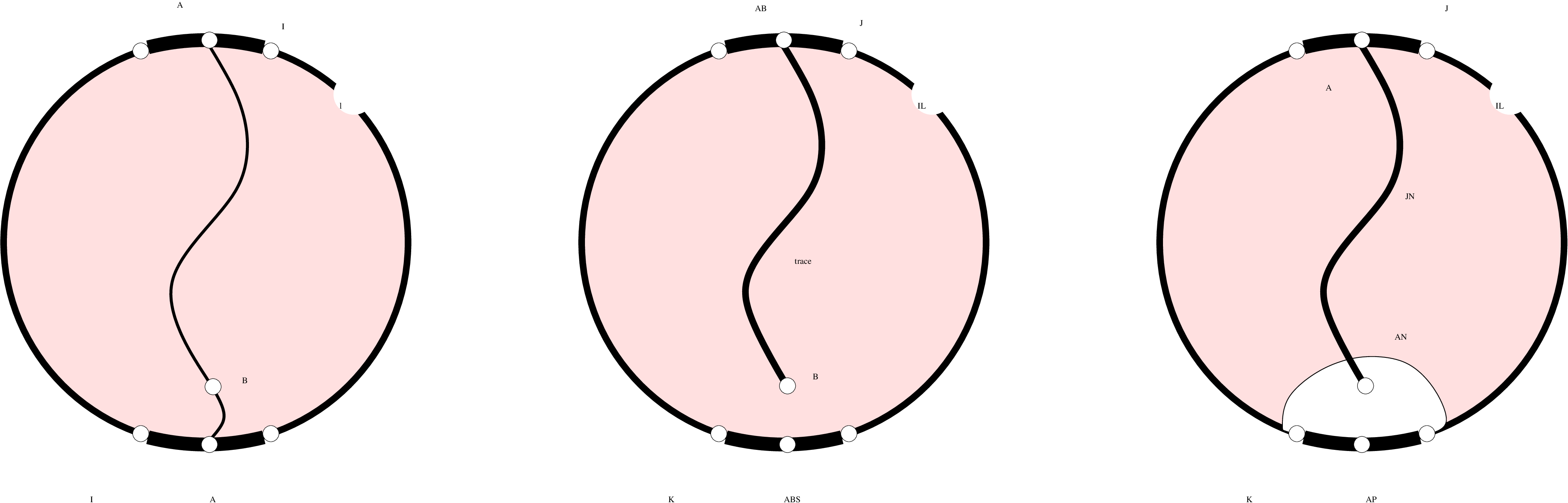}
\caption{\label{fundtwozer}}
\end{figure}
For all $X,Y \in \CB$, $X\neq Y$, we denote by $\trace{X}{Y}{\CB}$ the  line segment joining $X$ and $Y$ contained in $\CB$.  Let  $A \in \II$. 
For all $B\in \CB\setminus \II$ the pre-image  under $\MapLight{}{\TD}{\pp}$ of $\trace{A}{B}{\CB}$ has two connected components: a first line segment reduced to a single point $A^*_B$ and a 
second line segment $(\trace{A}{B}{\CB} \setminus \{A\})\cup A_B$  where $\{A_B^*,A_B\}$ is the pre-image of $A$ under $\MapLight{}{\TD}{\pp}$.
See Fig.~\ref{fundtwozer} for an illustration.
Let $\OCBP$ be the set of $B \in \CB\setminus \II$ such that $A_B \in \FIP$ and, similarly, let $\OCBM$ be the set of $B \in \CB\setminus \II$ such that $A_B \in \FIM.$  
Clearly, by definition, $\OCBP\cup \OCBM = \CB\setminus I$ and $\OCBP\cap \OCBM = \emptyset$.  
We claim that
\begin{enumerate}
\item $\OCBP$ and $\OCBM$ are independent of the choice of $A \in \II$;
\item $\OCBP$ and $\OCBM$ are closed convex subsets of the affine part of $(\pp,\lpp)$ obtained by removing the line $\CL$;
\item $\OCBP$ or $\OCBM$ is nonempty and $\OCBP$ and $\OCBM$ are both nonempty if and only if $\II$ intersects the interior of $\CB$;
\item the topological closure of $\OCBP$ in $\TD$ is $\OCBP \cup \FIP$ under the assumption that $\OCBP$ is nonempty, and a similar result holds for $\OCBM$;
\end{enumerate}
from which it follows that $\FCB$ has two connected components: $\OCBP \cup \FIP$ and $\OCBM \cup \FIM$. 
Claims (1), (2) and (3) are simple applications of Lemmas~\ref{lemone} and ~\ref{lemtwo}. 
Assume now that $\OCBP$ is nonempty. Clearly, the topological closure of $\OCBP$ contains $\OCBP\cup \FIP$ and is contained in $\OCBP\cup \FIP \cup \FIM$.  
Thus we have to prove that the topological closure of $\OCBP\cup \FIP$ avoids $\FIM$. 
Assume the contrary. Then there exists a convergent sequence of points $A_n \in \OCBP$  with limit $A' \in \FIM$. 
Let $A \in \II$, let $J_n = \trace{A}{A_n}{\CB} \setminus \{A\}$ and let   $\aline_n$ be the supporting line of $J_n$.
Without loss of generality 
one can assume that the sequence $\aline_n$ has a limit $\aline'$. Let $B \in \aline'$, $B \notin \II$. There exists a convergent sequence of points $B_n \in \aline_n$   with limit
$B$. For $n$ large enough $B_n$ belongs to $J_n$. Since $\OCBP$ is closed it follows that $B \in \OCBP$ and, consequently, $\aline'$ is a subset of $\CB$. 
This contradicts the assumption that $\CB$ is a convex body. The lemma follows with $\CBP = \OCBP$ and $\CBM = \OCBM$.  \end{proof}

\begin{lemma} \label{lemtwoone} 
Assume that there is a line missing the interior of a convex body. Then there is a line missing the body. 
\end{lemma}
\begin{proof}
Let $\CB$ be a convex body, let $\aline$ be a line missing the interior of $\CB$, let $I$ be the trace of $\aline$ on  $\CB$ and assume
that $I$ is nonempty (otherwise we are done).
Let $\aline_{\infty}$ be a line missing the line segment $I$ and let $Q$ and $Q'$ be the two connected components of the complement of the lines $\aline$ and $\aline_\infty$
in $\pp$, as indicated in the leftmost diagram of Fig.~\ref{fundtwoone}. 
\begin{figure}[!htb]
\centering
\psfrag{Q}{$Q'$} \psfrag{Qp}{$Q$}
\psfrag{linf}{$\aline_{\infty}$}
\psfrag{l}{$\aline$}
\psfrag{I}{$I$}
\psfrag{V}{$V_n$}
\psfrag{JN}{$J_n$}
\psfrag{AN}{$A_n$}
\psfrag{ApN}{$A'_n$}
\psfrag{k}{$k$}
\psfrag{U}{$\NEI$}
\includegraphics[width=0.875\linewidth]{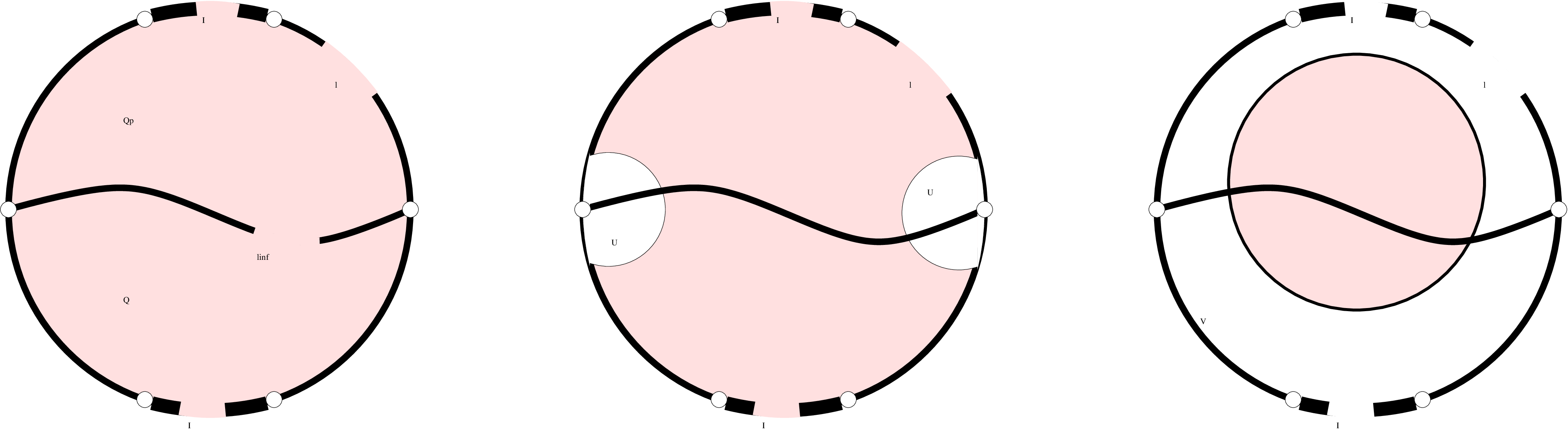}
\caption{\label{fundtwoone}}
\end{figure}
Let $\NEI$ be a neighboorhood of the intersection point of $\aline$ and $\aline_\infty$ disjoint from $\CB$.
Let $V_n$ be a decreasing sequence of open neighborhoods of $\aline$ with $\bigcap_n V_n = \aline$ and let $W_n$ and $W'_n$ be the traces of $V_n$ on $Q \cup \NEI$ and $Q'\cup \NEI$, 
respectively.
According to the previous lemma there exists an $n_0$ such that for all $n\geq n_0$ the trace of $W_n$ on $\CB$ is empty or for all $n\geq n_0$ the trace of $W'_n$ 
on $\CB$ is empty.
Without loss of generality one can assume that the trace of $W'_n$ on $\CB$ is empty for $n\geq n_0$. Using standard compactness arguments we see that there is a line $\aline''$ of the pencil of lines through the intersection point of $\aline$ and $\aline_\infty$ contained in $W'_n$ and we are done.
\end{proof}

\begin{lemma} \label{lemtwotwo} 
Assume that there is a line missing the interiors of two disjoint convex bodies. Then there is a line missing the two bodies. 
\end{lemma}
\begin{proof}
Let $\CB$ and $\CB'$ be two disjoint 
convex bodies,  let $\aline$ be a line missing the interiors of $\CB$ and $\CB'$, let $I$ and $J$ be the traces 
 of $\aline$ on  $\CB$ and $\CB'$ and assume that $I$ and $J$ are nonempty (otherwise we are done, thanks to the previous lemma), as indicated in the leftmost diagram  of Fig.~\ref{fundtwotwo}. Let $\TD$ be the closed 
\begin{figure}[!htb]
\centering
\psfrag{Q}{$Q'$} \psfrag{Qp}{$Q$}
\psfrag{linf}{$\aline_{\infty}$}
\psfrag{l}{$\aline$}
\psfrag{lF}{$\widetilde{\aline}$}
\psfrag{I}{$I$} \psfrag{IP}{$I_+$} \psfrag{IM}{$I_-$}
\psfrag{J}{$J$} \psfrag{JP}{$J_+$} \psfrag{JM}{$J_-$}
\psfrag{V}{$V$} \psfrag{A}{$A$} \psfrag{Ap}{$A'$} \psfrag{AN}{$A_n$} \psfrag{ApN}{$A'_n$}
\psfrag{k}{$k$}
\includegraphics[width=0.75\linewidth]{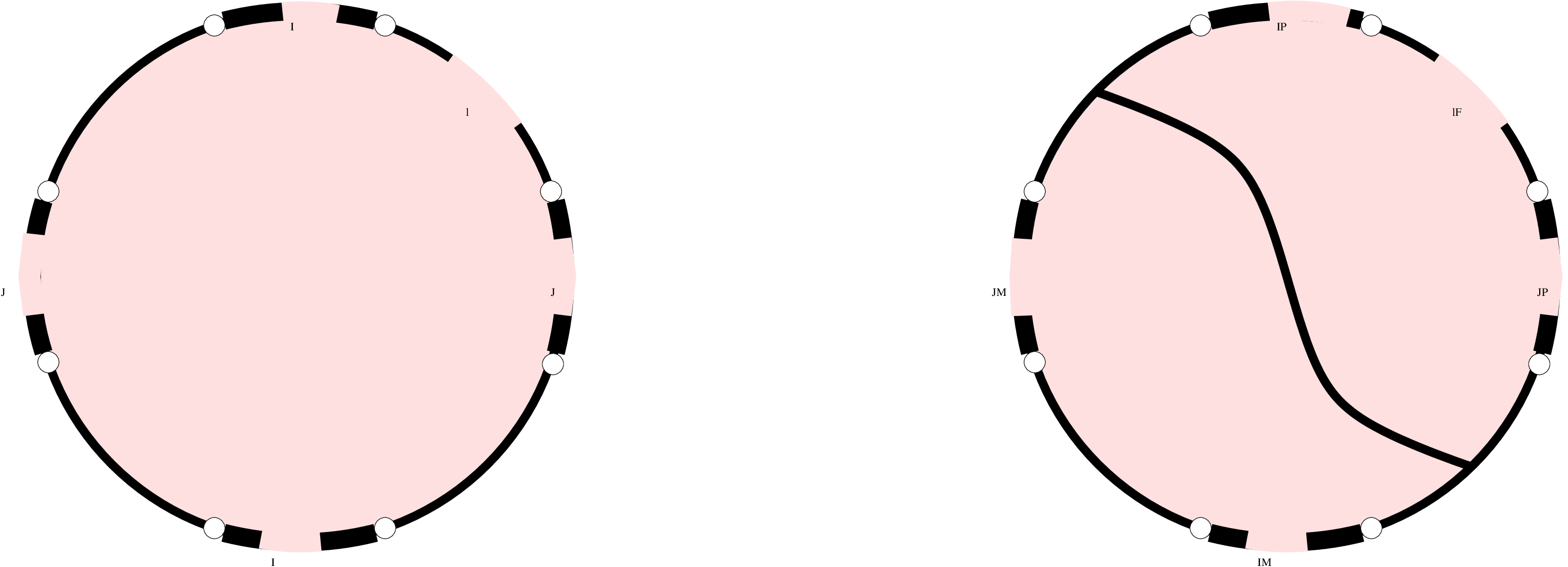}
\caption{\label{fundtwotwo}}
\end{figure}
topological disk obtained by cutting $\pp$ along $\aline$, let $\MapLight{}{\TD}{\pp}$ be the induced canonical map, let $I_+$  and $I_-$ be the 
two connected components of the pre-image of $I$ under $\MapLight{}{\TD}{\pp}$ with the convention that $I_-$ is also a connected component of the pre-image of 
$\CB$ under $\MapLight{}{\TD}{\pp}$, and similarly let $J_+$  and $J_-$ be the
two connected components of the pre-image of $J$ under $\MapLight{}{\TD}{\pp}$ with the convention that $J_-$ is also a connected component of the pre-image of 
$\CB$ under $\MapLight{}{\TD}{\pp}$.  The lemma follows from the simple observation that  there is a line misssing $I$ and $J$ whose pre-image under $\MapLight{}{\TD}{\pp}$
separates $I_+ \cup J_+$ from $I_- \cup J_-$, as indicated in the right diagram  of Fig.~\ref{fundtwotwo}.
\end{proof}

\begin{lemma}\label{lemthr} Any boundary point of a convex body is incident to a line missing 
the interior of the body. 
\end{lemma}
\begin{proof} Let $\CB$ be a convex body and let $A$ be a boundary point of
$\CB.$  The color of an oriented line $\aline$ through $A$  
is defined to be 
\begin{enumerate}
\item  {\it blue} if the line $\aline$ intersects the interior of $\CB$ and if $A$ is the initial
point of the trace on the interior of $\CB$ of the oriented line $\aline;$
\item  {\it white} if the line $\aline$ does not intersect the interior of $\CB$;
\item  {\it red} if the line intersects the interior of $\CB$ and if $A$ is the
terminal point of the trace on the interior of $\CB$ of the oriented line $\aline.$
\end{enumerate}
According to Lemma~\ref{lemtwo} any oriented line through $A$ has a color, and
these colors are mutually exclusive.
The sets of blue and red  oriented lines  are open subsets of the pencil of oriented lines
through $A$. Since none of these two sets is empty  and since the pencil of
oriented lines through $A$ is connected it follows that the set of  white oriented lines  
is nonempty. This proves the lemma. 
\end{proof}

\begin{lemma} \label{lemfou} The set of lines missing a convex body is nonempty. 
\end{lemma}
\begin{proof} Simple consequence  of Lemmas~\ref{lemtwoone} and~\ref{lemthr}.  \end{proof}

\begin{lemma} \label{lemfiv}
The set of lines missing two disjoint convex bodies is nonempty.  
\end{lemma}
\begin{proof} According to Lemma~\ref{lemtwotwo} it is sufficient to prove that the set of lines missing the interiors of two disjoint convex bodies is nonempty.
Let $U$ and $V$ be two disjoint convex bodies and let $\aline$ be a line missing $U.$ 
If $\aline$ avoids the interior of $V$ we are done. Otherwise $\aline$ intersects $V$ along a closed line segment, say $[RS]$, $R\neq S$, 
and $\aline$ intersects the interior of $V$ along the interior of $[RS]$. 

\begin{figure}[!htb]
\centering
\psfrag{Fp}{$\FFF$} \psfrag{Ep}{$\EEE$}
\psfrag{linf}{$\aline$}
\psfrag{g}{$\GGG$}
\psfrag{h}{$\HHH$}
\psfrag{A}{$A$}
\psfrag{B}{$B$}
\psfrag{D}{$D$}
\psfrag{Sone}{$S_1$} \psfrag{Stwo}{$S_2$} \psfrag{Sthr}{$S_3$} \psfrag{Sfou}{$S_4$}
\psfrag{Vone}{$V$} \psfrag{Vtwo}{$V$}
\psfrag{R}{$R$} \psfrag{S}{$S$}
\includegraphics[width=0.9875\linewidth]{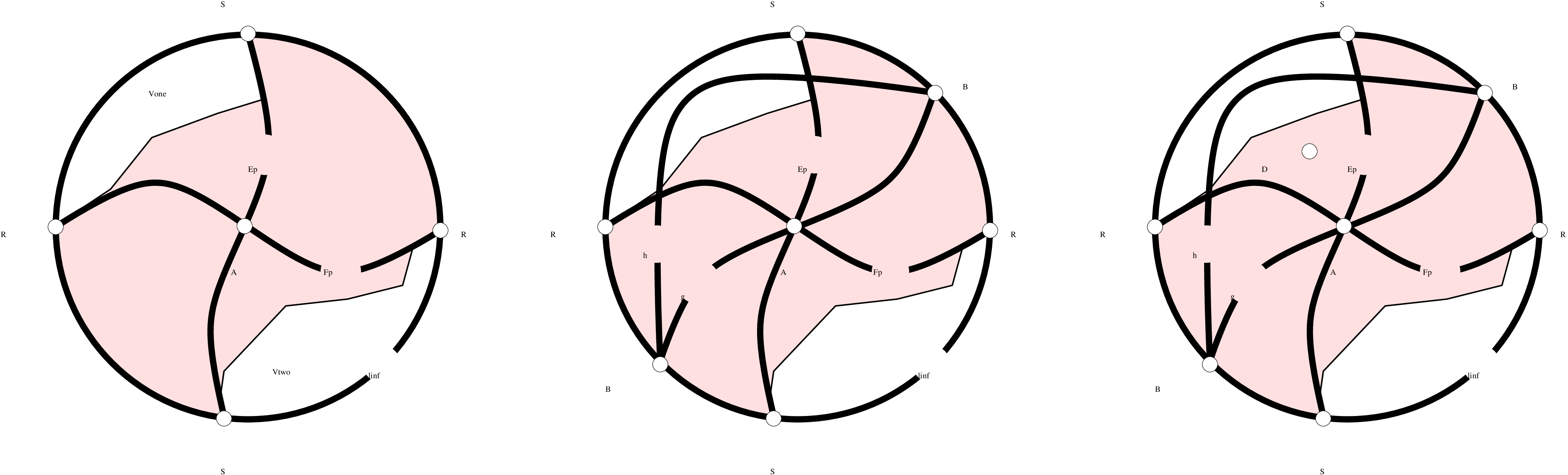}
\caption{\label{fundthr} }
\end{figure}

Let $\EEE$ be a tangent to $V$ at $S$, let $\FFF$ be a tangent to $V$ at $R$ and let $A$ be the intersection point of 
$\EEE$ and $\FFF$. If $\EEE$ or $\FFF$ misses the interior of $U$ we are done. Otherwise we proceed as follows. 

Let $\GGG$ be a line through $A$ that avoids the interior of $V$, 
let $B$ be the intersection point of $\aline$ and $\GGG$, let $\HHH \neq \aline$ be a line through $B$ that avoids $U$ 
but intersects the interior of $V$, and let $W$  be the intersection of  $V$  with the strip delimited by $\GGG$ and $\HHH$ in the affine plane $(\pp\setminus \aline, \lpp\setminus \{\aline\})$. 
Clearly $U$ and $W$ are disjoint convex bodies of the affine plane $(\pp\setminus \aline, \lpp\setminus \{\aline\})$: 
Let $D$ be the intersection point of their interior bitangents. 
We let the reader check that $D$ belongs to the triangle in $\pp\setminus \aline$ delimited by the lines $\EEE,\FFF$ and $\HHH$ and that 
the line through $D$ of the pencil of lines through $B$ avoids the interiors of $U$ and $V$.
\end{proof}
\end{document}